\documentclass[12pt]{amsart}
\usepackage[margin=1in]{geometry}
\usepackage{amsmath,amsthm,amscd,amsfonts,wasysym,amssymb,epic,eepic,bbm,tikz-cd,mathtools,array}
\usepackage{fullpage}
\newcommand{\WL}[1]{{\color{orange} {#1}}}

\usepackage{amsthm,amsfonts,amssymb,amsmath,amsxtra}

\usepackage{xr-hyper}
\usepackage[pagebackref,colorlinks=
   citecolor=Black,
   linkcolor=Red,
   urlcolor=Blue]{hyperref}
\usepackage{verbatim}

\usepackage[all,arc]{xy}
\SelectTips{cm}{}
\usepackage{enumerate}
\usepackage{mathrsfs}
\usepackage{color}   
\usepackage{hyperref}
\hypersetup{
    colorlinks=true, 
    linktoc=all,     
    linkcolor=blue,  
}
\usepackage{amsthm}
\usepackage{amsmath}
\usepackage{graphicx}
\usepackage{wasysym}
\usepackage{pgf,tikz}
\usetikzlibrary{positioning}
\usepackage{ marvosym }
\usepackage{tikz-cd}
\usepackage{ textcomp }
\usepackage{bbm}
\usepackage{ stmaryrd }

\numberwithin{equation}{section}

\usepackage{mathrsfs}

\RequirePackage{xspace}
\RequirePackage{etoolbox}
\RequirePackage{varwidth}
\RequirePackage{enumitem}
\RequirePackage{tensor}
\RequirePackage{mathtools}
\RequirePackage{longtable}
\RequirePackage{multirow}
\RequirePackage{bigstrut}
\RequirePackage{bigints}

\newcommand{\sD}{\ensuremath{\mathscr{D}}\xspace}

\newcommand{\sH}{\ensuremath{\mathscr{H}}\xspace}

\newcommand{\sP}{\ensuremath{\mathscr{P}}\xspace}

\newcommand{\sU}{\ensuremath{\mathscr{U}}\xspace}

\newcommand{\fkg}{\ensuremath{\mathfrak{g}}\xspace}

\newcommand{\fkl}{\ensuremath{\mathfrak{l}}\xspace}

\newcommand{\W}{\mathbf{W}}
\newcommand{\calw}{\mathcal{W}}

\newcommand{\cc}{\mathbb C}

\newcommand{\zz}{\mathbb Z}

\newcommand{\qq}{\mathbb Q}
\newcommand{\rr}{\mathbb R}
\newcommand{\A}{\mathbb A}

\newcommand{\la}{\langle}
\newcommand{\ra}{\rangle}
\newcommand{\lra}{\longrightarrow}
\newcommand{\hra}{\hookrightarrow}

\newcommand{\al}{\alpha}

\newcommand{\ga}{\gamma}
\newcommand{\de}{\delta}
\newcommand{\De}{\Delta}

\newcommand{\lam}{\lambda}
\newcommand{\Lam}{\Lambda}
\newcommand{\vp}{\varpi}
\newcommand{\sig}{\sigma}
\newcommand{\ka}{\kappa}

\DeclareMathOperator{\G}{G}

\DeclareMathOperator{\T}{T}

\DeclareMathOperator{\Tr}{Tr}
\DeclareMathOperator{\sspan}{span}

\DeclareMathOperator{\coker}{coker}

\DeclareMathOperator{\I}{I}
\DeclareMathOperator{\Q}{Q}
\DeclareMathOperator{\supp}{supp}

\newcommand{\fg}{\mathfrak g}

\newcommand{\fc}{\mathfrak c}

\newcommand{\fh}{\mathfrak h}

\newcommand{\fp}{\mathfrak p}
\newcommand{\fq}{\mathfrak q}

\newcommand{\fgl}{\mathfrak{gl}}
\newcommand{\fX}{\mathfrak X}

\newcommand{\calp}{\mathcal{P}}
\newcommand{\calh}{\mathcal{H}}
\newcommand{\cals}{\mathcal{S}}

\newcommand{\calo}{\mathcal{O}}

\newcommand{{\rP}}{\mathrm{P}}
\newcommand{{\rM}}{\mathrm{M}}
\newcommand{{\X}}{\mathrm{X}}
\newcommand{\calv}{\mathcal{V}}
\newcommand{\cale}{\mathcal{E}}

\newcommand{\calf}{\mathcal{F}}

\newcommand{\Gm}{\mathbb{G}_m}
\newcommand{\bbg}{\mathbb{G}}

\newcommand{\iso}{\overset{\sim}{\longrightarrow}}
\newcommand{\car}{\textbf{car}}

\newcommand{\Fbar}{\overline{F}}
\newcommand{\bfun}{\mathbf{1}}
\newcommand{\rH}{\mathrm{H}}
\newcommand{\SF}{\mathcal{S}\mathcal{F}}

\makeatletter
\def\Ddots{\mathinner{\mkern1mu\raise\p@
\vbox{\kern7\p@\hbox{.}}\mkern2mu
\raise4\p@\hbox{.}\mkern2mu\raise7\p@\hbox{.}\mkern1mu}}
\makeatother

\newenvironment{psmatrix}
  {\left(\begin{smallmatrix}}
  {\end{smallmatrix}\right)}

\newtheorem{Thm}{Theorem}[section]
\newtheorem{Prop}[Thm]{Proposition}
\newtheorem{LemDef}[Thm]{Lemma/Definition}
\newtheorem{Lem}[Thm]{Lemma}
\newtheorem{Cor}[Thm]{Corollary}

\newtheorem{Conj}[Thm]{Conjecture}

\newtheorem{Assumption}[Thm]{Assumption}

\theoremstyle{definition}
\newtheorem{Def}[Thm]{Definition}

\theoremstyle{remark}
\newtheorem{Rem}[Thm]{Remark}

\theoremstyle{definition}

\newcommand{\quash}[1]{}

\newcommand{\BA}{\ensuremath{\mathbb {A}}\xspace}

\newcommand{\BC}{\ensuremath{\mathbb {C}}\xspace}

\newcommand{\BG}{\ensuremath{\mathbb {G}}\xspace}

\newcommand{\BR}{\ensuremath{\mathbb {R}}\xspace}

\newcommand{\BZ}{\ensuremath{\mathbb {Z}}\xspace}

\newcommand{\CP}{\ensuremath{\mathcal {P}}\xspace}

\newcommand{\CS}{\ensuremath{\mathcal {S}}\xspace}

\newcommand{\Y}{\ensuremath{\mathrm {Y}}\xspace}

\newcommand{\Ad}{{\mathrm{Ad}}}

\DeclareMathOperator{\diag}{diag}

\DeclareMathOperator{\End}{End}

\DeclareMathOperator{\Gal}{Gal}
\newcommand{\GL}{\mathrm{GL}}

\newcommand{\Hir}{\mathrm{Hir}}

\DeclareMathOperator{\Hom}{Hom}

\newcommand{\inv}{{\mathrm{inv}}}

\DeclareMathOperator{\Lie}{Lie}

\DeclareMathOperator{\Nm}{Nm}

\DeclareMathOperator{\Orb}{Orb}

\newcommand{\PGL}{{\mathrm{PGL}}}

\renewcommand{\Re}{{\mathrm{Re}}}

\DeclareMathOperator{\Res}{Res}

\DeclareMathOperator{\Spec}{Spec}

\newcommand{\SO}{{\mathrm{SO}}}

\DeclareMathOperator{\sgn}{sgn}

\newcommand{\U}{\mathrm{U}}

\DeclareMathOperator{\vol}{vol}

\newcommand{\Herm}{\mathrm{Herm}}

\newcommand{\val}{\mathrm{val}}

\newcommand{\wt}{\widetilde}
\newcommand{\wh}{\widehat}

\newcommand{\pair}[1]{\langle {#1} \rangle}

\newcommand{\ov}{\overline}
\newcommand{\ul}{\underline}

\newcommand{\imp}{\Longrightarrow}

\newcommand{\bs}{\backslash}

\newcommand{\ep}{\varepsilon}

\newtheorem{theorem}{Theorem}

\newtheorem{lem}[theorem]{Lemma}

\newtheorem{cor}[theorem]{Corollary}

\newtheorem{Important}[theorem]{Notational change}

\theoremstyle{definition}
\newtheorem{defn}[theorem]{Definition}

\numberwithin{equation}{section}
\numberwithin{theorem}{section}

\setitemize[0]{leftmargin=*,itemsep=\the\smallskipamount}
\setenumerate[0]{leftmargin=*,itemsep=\the\smallskipamount}

\let\shortmapsto\mapsto
\renewcommand{\mapsto}{%
   \ifbool{@display}{\longmapsto}{\shortmapsto}%
   }
\newlength{\olen}
\newlength{\ulen}
\newlength{\xlen}
\newcommand{\xra}[2][]{%
   \ifbool{@display}%
      {\settowidth{\olen}{$\overset{#2}{\longrightarrow}$}%
       \settowidth{\ulen}{$\underset{#1}{\longrightarrow}$}%
       \settowidth{\xlen}{$\xrightarrow[#1]{#2}$}%
       \ifdimgreater{\olen}{\xlen}%
          {\underset{#1}{\overset{#2}{\longrightarrow}}}%
          {\ifdimgreater{\ulen}{\xlen}%
             {\underset{#1}{\overset{#2}{\longrightarrow}}}
             {\xrightarrow[#1]{#2}}}}%
      {\xrightarrow[#1]{#2}}
   }
\makeatother
\newcommand{\xyra}[2][]{%
   \settowidth{\xlen}{$\xrightarrow[#1]{#2}$}%
   \ifbool{@display}%
      {\settowidth{\olen}{$\overset{#2}{\longrightarrow}$}%
       \settowidth{\ulen}{$\underset{#1}{\longrightarrow}$}%
       \ifdimgreater{\olen}{\xlen}%
          {\mathrel{\xymatrix@M=.12ex@C=3.2ex{\ar[r]^-{#2}_-{#1} &}}}%
          {\ifdimgreater{\ulen}{\xlen}%
             {\mathrel{\xymatrix@M=.12ex@C=3.2ex{\ar[r]^-{#2}_-{#1} &}}}
             {\mathrel{\xymatrix@M=.12ex@C=\the\xlen{\ar[r]^-{#2}_-{#1} &}}}}}%
      {\mathrel{\xymatrix@M=.12ex@C=\the\xlen{\ar[r]^-{#2}_-{#1} &}}}%
   }
\makeatletter
\newcommand{\xla}[2][]{%
   \ifbool{@display}%
      {\settowidth{\olen}{$\overset{#2}{\longleftarrow}$}%
       \settowidth{\ulen}{$\underset{#1}{\longleftarrow}$}%
       \settowidth{\xlen}{$\xleftarrow[#1]{#2}$}%
       \ifdimgreater{\olen}{\xlen}%
          {\underset{#1}{\overset{#2}{\longleftarrow}}}%
          {\ifdimgreater{\ulen}{\xlen}%
             {\underset{#1}{\overset{#2}{\longleftarrow}}}
             {\xleftarrow[#1]{#2}}}}%
      {\xleftarrow[#1]{#2}}
   }
\newcommand{\isoarrow}{%
   \ifbool{@display}{\overset{\sim}{\longrightarrow}}{\xrightarrow\sim}%
   }

\begin{document}

\title[]{Unitary Friedberg--Jacquet periods and their twists:
\\
Fundamental Lemmas}
\author{Spencer Leslie}
\address{Boston College, Department of Mathematics, Chestnut Hill, MA 02467 USA}
\email{spencer.leslie@bc.edu}

\author{Jingwei Xiao}
\address{Beijing, China}
\email{jwxiao922@gmail.com}

\author{Wei Zhang}
\address{Massachusetts Institute of Technology, Department of Mathematics, 77 Massachusetts Avenue, Cambridge, MA 02139, USA}
\email{weizhang@mit.edu}

\date{\today}

\begin{abstract}We formulate a global conjecture for the automorphic period integral associated to the symmetric pairs defined by unitary groups over number fields, generalizing a theorem of Waldspurger's toric period for $\GL(2)$. We introduce a new relative trace formula to prove our global conjecture under some local hypotheses. A new feature is the presence of the relative endoscopy. In this paper we prove the main local theorem: a new relative fundamental lemma comparing certain orbital integrals of functions matched in terms of Hironaka and Satake transforms. 
\end{abstract}

\maketitle

\setcounter{tocdepth}{1}
\tableofcontents

\section{Introduction}

In \cite{W85} Waldspurger proved a formula relating the toric period
integral of a cuspidal automorphic form on $\GL(2)$ (and its inner forms) to the central value of the associated L-function. Since then, there have been various conjectural generalizations of Waldspurger theorem to higher rank groups. One generalization is due to Gross--Prasad (for orthogonal groups $\SO(n)\times\SO(n+2r+1)$) \cite{GrossPrasad} and Gan--Gross--Prasad (for general classical groups including the unitary groups) \cite{GGP} in which the relevant L-functions are Rankin-Selberg convolution L-functions. The  Gan--Gross--Prasad conjecture for unitary groups (and tempered cuspidal automorphic representations) have been mostly proved by now after the work of many authors; in the reductive case which is most closely related to the current paper, the most complete results can be found in \cite{BPLZZ} and \cite{beuzart2020global}, based on the study of the Jacquet-Rallis relative trace formula (abbreviated as RTF henceforth).

In this paper and its companion \cite{LXZ25}, we pursue another generalization in terms of the unitary Friedberg--Jacquet periods and the standard base change L-functions. One of these periods corresponds to pairs of reductive groups 
\begin{equation}\label{eqn: FJ periods}
    \U(W_1)\times \U(W_2)\subset \U(W_1\oplus W_2),
\end{equation}
where $W_1$ and $W_2$ are $n$-dimensional Hermitian vector spaces with respect to a quadratic extension of number fields, and $\U(W_i)$ denotes the associated unitary group. We consider more general forms of such a pair, as described below in terms of certain ``bi-quadratic'' extensions. Similar considerations also appeared in the work of Getz--Wambach \cite{GetzWambach}. In various special cases, our period integrals have appeared in the previous works by Guo \cite{Guo}, Pollack--Wan--Zydor \cite{PollackWanZydor}, Chen--Gan \cite{ChenGan},  Xue--Zhang \cite{XueZhang}. 

We formulate the conjecture in the next subsection, and state our global results in \S \ref{Sec: global results}. In \cite{LXZ25}, we introduce several relative trace formulas which we compare to prove these results (see \S\ref{Sec: RTF intro}). In this paper, we develop the local geometric theory for these trace formula comparisons, formulate the necessary transfer conjectures (toward which we establish partial results), and prove the \emph{three} fundamental lemmas required for the general comparisons (see \S \ref{Sec: fundamental intro}).

\subsection{The global conjecture}\label{Section: biquad conjecture}

Let $F_0$ be a number field, equipped with two quadratic \'etale algebras $F$ and $E_0$. We form the compositum $E=E_0F$, which is of degree $4$ over $F_0$. We obtain a diagram of extensions
\[
\begin{tikzcd}[every arrow/.append style={dash}]
     & E &   \\
    E_0\ar[ur,"\sig"] &       F'\ar[u]\ar[d]      & F \ar[ul, "\nu"] \\
      &      F_0\ar[ul,"\nu"]\ar[ur,"\sig"]       &
\end{tikzcd}
\]
where we have indicated the notation for the corresponding Galois automorphism. We are interested in (a slight generalization of) the following setting: suppose that $(B,\ast)$ as an $F$-central simple algebra of degree $2n$ equipped with an involution of the second kind. For any $\tau\in B^\times$ such that $\tau=\tau^\ast$ we define the unitary group, an $F_0$-algebraic group,
\[
\G:=\{g\in B^\times : g\tau g^\ast =\tau\}.
\]
We now assume that $\iota_0: E_0\lra B$ is an injection of $F_0$-algebras compatible with $\nu, \ast$ and $\tau$; see \S\ref{ss:setup} for the detail. We set 
\[
\rH:=\mathrm{Cent}_{\G}(\iota_0)
\]
to be the subgroup commuting with the image of $\iota_0$. Then $\rH\subset \G$ is a symmetric subgroup and we refer to the pair $(\G,\rH)$ as a unitary symmetric pair.

Fix such a pair $(\G,\rH)$. Let $\pi$ be a cuspidal automorphic representation of $\G(\BA)$, where $\BA$ the ring of adeles of $F_0$. We are interested in the automorphic $\rH$-period integral, defined by
$$
 \sP_{\rH}(\varphi)=\int_{[\rH]}\varphi(h)\,dh,
$$
where $\varphi$ is a form in $\pi$ and $[\rH]=\rH(F_0)\bs \rH(\BA)$. We will say that $\pi$ is $\rH$-distinguished if $ \sP_{\rH}(\varphi)\neq0$ for some $\varphi\in\pi$. These periods generalize the {unitary Friedberg--Jacquet periods}, as the special case of $B = M_{2n}(F)$ and $E_0=F_0\times F_0$ recovers pairs of the form \eqref{eqn: FJ periods}. 

 Let ${\rm BC}_{F}(\pi)$ denote the (weak) base change from $\G$ to $\GL_{2n,F}$ \cite[Section 4.3]{BPLZZ}. Now let $\G'=\GL_{2n,F_0}$, and let $\pi_0$ be a cuspidal automorphic representation of $\G'(\BA)$. By abuse of notation, we also denote by ${\rm BC}_{F}(\pi_0)$ (resp., ${\rm BC}_{F'}(\pi_0)$) the base change from $\GL_{2n,F_0}$ to $\GL_{2n,F}$ (resp., to $\GL_{2n,F'}$) \cite{ArthurClozel}.

If $\pi_0$ is of {\em symplectic type} (i.e., its exterior square L-function has a pole at $s=1$) and its base change ${\rm BC}_{F}(\pi_0)$ is cuspidal, then ${\rm BC}_{F}(\pi_0)$ descends down to a cuspidal automorphic representation  $\pi$ of the quasi-split unitary group $\U_{2n}$ of rank $2n$ in the sense that ${\rm BC}_{F}(\pi_0)={\rm BC}_{F}(\pi)$.

We note the product expansion
$$
L(s,{\rm BC}_{E}(\pi_0))=L(s,{\rm BC}_{F'}(\pi_0) )
L(s,{\rm BC}_{F'}(\pi_0\otimes\eta_{F/F_0}) ),$$
where $\eta_{F/F_0}$ denotes the quadratic character of $\BA^\times/F_0^\times$ associated to $F/F_0$ by class field theory.

\begin{Conj}\label{Conj: biquadratic} Let $\pi_0$ be a cuspidal automorphic representation of $\GL_{2n,F_0}$ of symplectic type. Assume that its base change ${\rm BC}_{F}(\pi_0)$ is cuspidal, and let $\pi$ be a cuspidal automorphic representation of the quasi-split unitary group $\U_{2n}$ of rank $2n$ over $F_0$ such that ${\rm BC}_{F}(\pi_0)={\rm BC}_{F}(\pi)$.
Then the following two assertions are equivalent:
\begin{enumerate}
\item One of $
L(1/2,{\rm BC}_{F'}(\pi_0) )$ and $
L(1/2,{\rm BC}_{F'}(\pi_0\otimes\eta_{F/F_0}) )$  does not vanish;
\item there exists a unitary symmetric pair $(\G,\rH)$ for an inner form\footnote{Here ``pure inner forms" are insufficient; we may need to consider unitary group associated to a central simple algebra over $F$ with an involution of second kind with respect to $F/F_0$. This can already be seen when $n=1$.} $\G$ of $\U_{2n}$ and an automorphic representation $\pi_{\G}$ on $\G(\A)$ nearly equivalent to $\pi$ such that $\pi_{\G}$ is $\rH$-distinguished.
\end{enumerate}

\end{Conj}
We begin with a few remarks.
\begin{Rem}
If $\pi_{\G}$ is $\rH$-distinguished such that  ${\rm BC}_{F}(\pi_{\G})$ is cuspidal, then $\pi_{\G}$ must be as in the conjecture, i.e., there exists $\pi_0$  of {\em symplectic type} such that ${\rm BC}_{F}(\pi_0)={\rm BC}_{F}(\pi_{\G})$.
\end{Rem}
\begin{Rem}There should be a local conjecture dictating the local multiplicity spaces $\Hom_{\rH}(\pi_{\G,v},\BC)$ and local root numbers. But the precise statement seems quite subtle and we do not attempt to write the exact conjecture here. When $F/F_0$ is split at $v$, a precise local conjecture has been formulated by Prasad and Takloo-Bighash \cite{PTB}. Moreover, Chen--Gan have proved results for discrete L-packets in \cite{ChenGan} in some cases.

Additionally, there should be an extension of the global conjecture when ${\rm BC}_{F'}(\pi_0)$ is not cuspidal. Our methods should also play a role in this broader context, relying on a stabilization of the relevant relative trace formulas for unitary symmetric pairs. We postpone the formulation of a precise statement to future work.
\end{Rem}

\begin{Rem}
    This conjecture is compatible with the general notion of \emph{twisted base change} of Getz--Wambach \cite{GetzWambach}. More precisely, their framework predicts a comparison between the $\rH$-distinguished spectrum of $\G$ and automorphic representations $\Pi$ of $\Res_{F/F_0}(\GL_{2n})$ that are distinguished by the two symmetric subgroups
    \[
    \wt{\rH}_1=\Res_{F/F_0}(\GL_n\times\GL_n)\quad\text{ and }\quad \wt{\rH}_2=\U_{2n},
    \]at least in the special case of unitary Friedberg--Jacquet periods \eqref{eqn: FJ periods} (cf. \cite[Conj. 1.1]{LeslieUFJFL}). When restricted to cuspidal representations, results of Friedberg--Jacquet \cite{FriedbergJacquet} and Jacquet \cite{JacquetQuasisplit,FLO} imply that these are precisely the cuspidal automorphic representations $\Pi={\rm BC}_{F}(\pi_0)$ for $\pi_0$ a cuspidal automorphic representation of $\GL_{2n}(\BA)$ of {symplectic type}.
\end{Rem}

\subsection{Main global results}\label{Sec: global results}
The present paper is the first in a series considering cases of the preceding conjecture. In a companion paper \cite{LXZ25}, we introduce a family of relative trace formulas and formulate a comparison designed to attack this conjecture. Though our method and the local results should allow us to obtain global results in the more general bi-quadratic case, for the sake of simplicity we will focus on the following two cases. 

Assume that $B = M_{2n}(F)$ is split, so that $\G = \U(V)$ is the unitary group associated to an $F/F_0$-Hermitian space $V$ of dimension $2n$. We consider the following two situations.
\begin{enumerate}
\item When $E_0=F_0\times F_0$ is split and $F/F_0$ is a quadratic field extension, there exists an orthogonal decomposition 
\[
V=W_1\oplus W_2,\qquad \dim(W_1) = \dim(W_2).
\]
Then $\rH=\U(W_1)\times \U(W_2)$; this case is referred to as the unitary Friedberg--Jacquet period. In this case, $F'=F$ and the first statement of the conjecture simplifies to 
\[
L(1/2,{\rm BC}_{F}(\pi_0) )=L(1/2,{\rm BC}_{F}(\pi_{\G}) )\neq 0.
\]
\item When $E_0=F$ are the same quadratic field over $F_0$, there exists a decomposition 
\[
V=U\oplus U^\perp,\qquad \la-,-\ra|_{U}\equiv 0
\]
so that $\rH = \Res_{F/F_0}\GL(U)$ is a maximal Levi subgroup of $\G$; this case is referred to as the twisted unitary Friedberg--Jacquet period. In this case, $F'= F_0\times F_0$ and the first statement of the conjecture simplifies to the claim that either
\[
L(1/2,\pi_0 )^2 \text{ or }
L(1/2,\pi_0\otimes\eta_{F/F_0} )^2\neq 0.
\]
\end{enumerate}

One reason for considering these two cases is that the local versions occur in the general bi-quadratic setting at infinitely many places $v$ of $F_0$, depending on whether $E_0$ and $F$ are split or inert at $v$. Motivated by this, we will refer to the first (resp. second) case as the {\em split-inert} (resp. {\em inert-inert}) case. With this terminology, the split-split case recovers the linear period of Friedberg--Jacquet \cite{FriedbergJacquet}, while the inert-split case has been treated in \cite{Guo,XueZhang}. Note that these four cases account for all but finitely many local places in the general setting.

Since we will only consider these two cases above, we will make a notational simplification: $F/F_0$ will be renamed as $E/F$.
We impose the following constraints on our number fields. We assume that $E/F$ is a quadratic extension of number fields such that
\begin{enumerate}
    \item\label{field1} $E/F$ is everywhere unramified,
    \item\label{field2} $E/F$ splits over every finite place $v$ of $F$ such that $$p\leq \max\{e(v/p)+1,2\},$$ where  $v|p$ and $e(v/p)$ denote the ramification index of $v$ in $p$.
\end{enumerate}We note that the unramified hypothesis implies that the number of non-split archimedean places is necessarily even (by computing $\eta_{E/F}(-1)=1$ locally).

We first state our main result of \cite{LXZ25} in the  split-inert case.

\begin{Thm}\label{Thm: si}
Let $\pi_{0}$ be a cuspidal automorphic representation of $\G'(\A_{F})$ satisfying 
\begin{enumerate}
    \item there are two (distinct) non-archimedean finite places $v_1,v_2$ such that $\pi_{v_1},\pi_{v_2}$ are supercuspidal,
    \item\label{everywhere unramifed} for each non-split non-archimedean place $v$, $\pi_v$ is unramified,
     \item for each non-split archimedean  place $v$, $\pi_v$ is the descent of the base change of a representation of the compact unitary group $\U_{2n}(\BR)$ that is distinguished by the compact  $\U_{n}(\BR)\times \U_{n}(\BR)$.
\end{enumerate} Assume that $\pi_0$ is of {\em symplectic type}. Then the following two assertions are equivalent:
\begin{enumerate}
\item $L(1/2,{\rm BC}_E(\pi_{0}) )\neq 0$, and 
\item
There exist (non-degenerate) Hermitian spaces $W_1,W_2$ of dimension $n$, and a cuspidal automorphic representation $\pi$ of $\G(\A_{F})$ for $\G=\U(W_1\oplus W_2)$ satisfying
\begin{enumerate}
    \item ${\rm BC}_E(\pi) ={\rm BC}_E(\pi_{0})$,
    \item $\pi$ is $\rH$-distinguished for $\rH=\U(W_1)\times\U(W_2)$ .
\end{enumerate}
\end{enumerate}
\end{Thm}
This establishes the first cases of \cite[Conjecture 1.1]{LeslieUFJFL} (itself a special case of Conjecture \ref{Conj: biquadratic}) for arbitrary $n$, though results in the ``period non-vanishing implies $L$-value non-vanishing'' direction are known in greater generality \cite{ChenGan,PollackWanZydor}.

We also have a version in the inert-inert case.
\begin{Thm}\label{Thm: ii} We further assume that every archimedean place $v$ of $F$ splits in $E$.
   Let $\pi_{0}$ be as above. Assume that $\pi_0$ is of {\em symplectic type}. Then the following two assertions are equivalent:
\begin{enumerate}
\item\label{Thm: ii 1} At least one of $L(1/2,\pi_{0})$ or $L(1/2,\pi_{0}\otimes \eta_{E/F})$ is non-zero, and 
\item There exists a cuspidal automorphic representation $\pi$ of $\G(\A_{F})$ for the quasi-split unitary group $\G=\U(V)$ satisfying
\begin{enumerate}
    \item   ${\rm BC}_E(\pi) ={\rm BC}_E(\pi_{0})$, and
    \item $\pi$ is $\rH$-distinguished, where $\rH$ is the Levi subgroup of the Siegel parabolic subgroup associated 
to a Lagrangian subspace of $V$.
\end{enumerate}
\end{enumerate}
\end{Thm}

The two theorems (as well as stronger results on the direction ``period non-vanishing implies $L$-value non-vanishing" by dropping condition \eqref{everywhere unramifed}) are proved 
 by combining the relative trace formula (RTF) construction in \cite{LXZ25} and the relative fundamental lemmas proved in this paper.

\subsection{The relative trace formulas}\label{Sec: RTF intro}
We give an overview of the relative trace formula construction  and refer to \cite{LXZ25} for more detail.
We return to the general bi-quadratic case. 
As is routine now, for the automorphic period integrals we consider the RTF associated to a unitary symmetric pair $\G\supset \rH$ as defined in \S \ref{Section: biquad conjecture}.
For $f\in C^\infty_c(\G(\A_{F}))$, we define
\begin{equation}\label{eqn: u RTF}
J(f):= \int_{[\rH_1]}\int_{[\rH_2]}K_f(h_1,h_2)\,dh_1\,dh_2,
\end{equation}
where $\rH_1,\rH_2\subset \G$ are pure inner forms of $\rH$.

We will compare this RTF to the following relative trace formula. We consider the pair of $F_0$-groups 
\begin{align*}
   \G' =\GL_{2n},\qquad
	\rH' = \GL_{n}\times\GL_n.
\end{align*}
We need to insert two weight factors. On the first factor $\rH'$, we need the twisted Godement--Jacquet Eisenstein series on the second $\GL_n$-factor in $\rH'$ (see \S\ref{ss:Eis})
$$
E(h, \Phi,s,\eta_{F/F_0}),\quad \Phi\in \CS(\A^{n}), \quad h\in\GL_n(\A) .
$$
On the second factor $\rH_2'(\BA)$, we need the quadratic character $\eta_{F'/F_0}$. Note that the two quadratic characters reflect two different quadratic extensions in general. We thus consider the distribution defined by the integral, for $\wt f=f'\otimes\Phi \in \CS(\G'(\A)\times \A^{n})$,
 \begin{equation}\label{eqn:linear RTF}
I^{(\eta_{F/F_0},\eta_{F'/F_0})}(\wt f):=\displaystyle\int_{[\rH']}\int_{[\rH']}K_{f'}(h_1,h_2) E(h_{1}^{(2)},\Phi,0,\eta_{F/F_0})\,\eta_{F'/F_0}(h_2)\, dh_1\, dh_2,
 \end{equation}
 where $[\rH'] = Z_{\G'}(\A)\rH'({F_0})\backslash\rH'(\A)$. 
 
 In \cite{LXZ25}, we develop simple forms of the spectral and geometric expansions for these RTFs and establish the relationship between the spectral expansion $I^{(\eta_{F/F_0},\eta_{F'/F_0})}(\wt f)$ and the central $L$-values arising in Conjecture \ref{Conj: biquadratic}.


As indicated in \cite{LeslieUFJFL}, the comparison of the geometric expansions of the simple RTFs requires a \emph{stabilized version} of the unitary side. 
In the split-inert case, the ellipticity condition for the simple trace formula causes the geometric side to completely simplify into a sum of {\em stable} orbital integrals (see Corollary \ref{Cor: stable} for another such example). In the inert-inert case, there remains an ``unstable'' term of $\ep$-orbital integrals in addition to the stable one, requiring an endoscopic comparison to handle. To our knowledge, our theorem in the inert-inert case is the first instance where relative endoscopy is applied to establish a generalization of Waldspurger theorem relating period integrals and special values of L-functions. 

In this paper, we consider the local geometric problems required for affecting these comparisons of (stabilized) RTFs. After formulating the comparison of stable orbits, we introduce and analyze the relevant families of orbital integrals. We then formulate the relevant conjectures on the existence of smooth transfers: the first one (Conjecture \ref{Conj: ii simple endoscopy}) relates $\ep$-orbital integrals in the inert-inert case with stable orbital integrals in the split-inert case (hence the latter may be viewed as an endoscopic variety of the former in the sense of \cite{LeslieEndoscopy}). The second transfer conjecture (Conjecture \ref{Conj: smooth transfer}) relates stable orbital integrals in both the split-inert and the inert-inert case to appropriate orbital integrals on the general linear group (Definition \ref{LemDef: linear OI}). In Appendix \ref{Section: proof weak transfer}, we prove some partial results towards these two conjectures (cf. Theorem \ref{Thm: weak smooth transfer} and \ref{Thm: weak endoscopic transfer}). The main results of the present paper are the corresponding fundamental lemmas (Theorems \ref{Thm: fundamental lemma si}, \ref{Thm: fundamental lemma ii}, and \ref{Thm: fundamental lemma varepsilon}); we will give a bit more detail in the next subsection.

\subsection{Relative fundamental lemmas}\label{Sec: fundamental intro}
To prove Theorems \ref{Thm: si} and \ref{Thm: ii}, we establish three relative fundamental lemmas (for  various symmetric spaces) in this paper: Theorem \ref{Thm: fundamental lemma si} (in the split-inert case), \ref{Thm: fundamental lemma ii} (in the inert-inert case), and \ref{Thm: fundamental lemma varepsilon}
(for the relative endoscopic $\ep$-orbital integral). In fact each of these theorems establishes comparisons of orbital integrals for certain functions of arbitrary depth. These statements are reduced to their ``Lie algebraic" versions (stated in Theorem \ref{Thm: fundamental lemma s-i Lie},  \ref{Thm: fundamental lemma stable Lie}, and \ref{Thm: fundamental lemma endoscopic Lie} respectively) in Appendix \ref{Sec: descent}. Miraculously, the latter three results are all eventually reduced to a single statement on ``Lie algebras", namely Theorem \ref{Thm: main local result 2 var}. Given its simplicity, we wish to state this result here.

\begin{Rem}While we only consider the split--inert and inert--inert cases in \cite{LXZ25}, the three fundamental lemmas proved in this paper-- combined with the fundamental lemma in the inert-split case \cite{Guo}-- exhaust the necessary fundamental lemmas for the unit of the Hecke algebra in the general bi-quadratic setting.
\end{Rem}

Let $E/F$ be an unramified quadratic extension of $p$-adic local fields with $\sigma$ the non-trivial Galois involution. Let $\eta:F^\times\to\{\pm1\}$ be the quadratic character associated to $E/F$ by local class field theory; namely, $\eta(x)=(-1)^{v(x)}$. 
Let $(x,w)\in \GL_n(F)\times F_n$ be regular semisimple in the sense that $x$ is regular semisimple and $(x,w)$ is cyclic (namely, $w,wx,\cdots, wx^{n-1}$ form a basis of $F_n$). For $\wt{f} \in C_c^\infty(\GL_n(F)\times F_n)$, consider the integral
\begin{equation*}
    \Orb^{\GL_n(F),\eta,\natural}_s(\wt{f}, (x,w)):=\frac{\De(x,w)}{L(s,T_x,\eta)}\int_{\GL_n(F)}\wt{f}(g^{-1}xg, wg)|g|^s\eta(g)\,dg
\end{equation*}
where $\De(x,w)$ is a transfer factor (cf. \eqref{eqn: correct xiao transfer factor}) and $L(s,T_x,\eta)$ is a local L-factor associated to the centralizer $T_x$ (a torus) of $x$, cf. \S\ref{Section: Lvalue measure}.
This ratio
is holomorphic at $s=0$ and we define
\begin{equation*}
    \Orb^{\GL_n(F),\eta}(\wt{f},x):=\Orb^{\GL_n(F),\eta,\natural}_0(\wt{f}, (x,w))
\end{equation*}
which is independent of $w$. {We will call the above the ``mirabolic" orbital integral.}

Let $\Herm_{n}^{\circ}(F)$ denotes the set of invertible $E/F$-Hermitian $n\times n$ matrices. The group $\GL_n(E)$ acts on $\Herm_{n}^{\circ}(F)$ by $g\cdot x=gxg^\ast, \,g^\ast:=\,^tg^\sigma$. Let $\calh_{K_E}(\Herm_{n}^{\circ}(F))$ denote the ``spherical Hecke algebra" of $\Herm_{n}^{\circ}$, consisting of $K_E=\GL_n(\calo_E)$-invariant elements
in  $C_c^\infty(\Herm_{n}^{\circ}(F))$. It is naturally a module over the spherical Hecke algebra $\calh_{K_E}(\GL_n(E))$ of $\GL_n(E)$. In \S\ref{s:Hir} we will recall that Hironaka \cite{hironaka1999spherical} showed that this module is free of rank $2^n$ and she defined a spherical Fourier transform, which together with the Satake transform on $\GL_n(F)$ defines an isomorphism of $\calh_{K_E}(\GL_n(E))$-modules 
\[
\Hir:\calh_{K_{E}}(\Herm_{n}^{\circ}(F))\lra \calh_{K}(\GL_n(F)).
\]
Here the spherical Hecke algebra $\calh_{K}(\GL_n(F))$ of $\GL_n(F)$ is viewed as a module over $\calh_{K_E}(\GL_n(E))$ via the base change homomorphism $\calh_{K_E}(\GL_n(E))\to \calh_{K}(\GL_n(F))$. The isomorphism also induces an algebra structure on $\calh_{K_{E}}(\Herm_{n}^{\circ}(F))$ (justifying its name), with its multiplication denoted by $\ast$.

The group $\GL_n(E)$ acts on the product $\Herm_{n}^{\circ}(F)\times \Herm_{n}^{\circ}(F)$ by \[g\cdot(x_1,x_2)=(gx_1 g^\ast, (g^\ast)^{-1}x_2g^{-1}).\]
Let $(x_1,x_2)\in\Herm_{n}^{\circ}(F)\times \Herm_{n}^{\circ}(F)$ be regular semi-simple  in the sense that the product $x_1x_2$ is regular semisimple as an element in $\GL_n(E)$. We define the stable orbital integral, for $\phi\otimes\phi'\in C_c^\infty(\Herm_{n}^{\circ}(F)\times \Herm_{n}^{\circ}(F))$ 
$$
\SO^{\GL_n(E)}(\phi\otimes\phi',(x_1,x_2))=\sum_{(x'_1,x'_2)} \int_{T_{(x_1,x_2)}\bs\GL_n(E)}(\phi\otimes\phi')(g^{-1}\cdot(x'_1,x'_2))\,dg,
$$
where $T_{(x_1,x_2)}$ is the stabilizer of $(x_1,x_2)$, and the sum runs over all $\GL_n(E)$-orbits of $ (x'_1,x'_2)$ such that $x_1'x_2'$ has the same characteristic polynomial as $x_1x_2$. We fix our measure conventions in \S \ref{measures}.

We are ready to state the key relative fundamental lemma (Theorem \ref{Thm: main local result 2 var}), relating the stable orbital integrals on the product space $\Herm_{n}^{\circ}(F)\times \Herm_{n}^{\circ}(F)$ to the mirabolic orbital integrals:
\begin{Thm}\label{Thm: FL} Let $\phi, \phi'\in \calh_{K_E}(\Herm_{n}^{\circ}(F))$. Let $(x_1,x_2)\in\Herm_{n}^{\circ}(F)\times \Herm_{n}^{\circ}(F)$ be regular semi-simple and assume $z\in \GL_n(F)$ has the same characteristic polynomial as $x_1x_2$. Then 
	\[
 \SO^{\GL_n(E)}(\phi\otimes\phi',(x_1,x_2))=\Orb^{\GL_n(F),\eta}( \Hir(\phi\ast\phi')\otimes {\bfun}_{\calo_{F,n}},z).
	\]
\end{Thm}

The special case of the theorem  when $\phi$ is the identity ${\bf 1}_{K_E\cdot I_n}={\bf 1}_{\Herm_{n}^{\circ}(\calo_F)}$ is of independent interest. In that case the stable orbital integral is reduced to the stable orbital integral for the adjoint action of the unramified unitary group $\U(V_n)$ on its Lie algebra $\mathfrak{g}$ upon identifying $\mathfrak{g}$ with $\Herm_{n}$. More precisely, for $\phi \in C_c^\infty(\Herm_{n}^\circ(F))$, consider the stable orbital integral for a regular semisimple element $x\in \Herm_{n}^\circ(F)$
$$
\SO^{\U(V_n)}(\phi,x)=\sum_{x'} \int_{T_{x'}\bs\U(V_n)} \phi(h^{-1} x' h)\, dh.
$$Then  for any $\phi \in \calh_{K_E}(\Herm_{n}^{\circ}(F))$, and  regular semi-simple elements $x\in \Herm_{n}^{\circ}(F)$ and $z\in \GL_n(F)$ with equal characteristic polynomial, we have
	\[
	\SO^{\U(V_n)}(\phi,x)=\Orb^{\GL_n(F),\eta}( \Hir(\phi)\otimes \bfun_{\calo_{F,n}},z);
	\]
see Theorem \ref{Thm: main local result 1 var}. This result gives an interpretation of the stable orbital integrals on the Lie algebra of the unramified unitary group in terms of the mirabolic orbital integrals.

\subsection{The organization of the paper} In Part \ref{Part: FL}, we develop the necessary material for our statements on orbital integrals. We describe the orbits on the various symmetric varieties in \S \ref{Section: unitary orbits}, study the local orbital integrals in various settings in \S \ref{Section: orbital integrals}, and formulate the two conjectures on smooth transfer. We state the various fundamental lemmas in \S \ref{Section: fundamental lemmas on variety}. In \S \ref{Section: Lie}, we pass to the ``infinitesimal'' setting, stating the reduction of the fundamental lemmas to this setting in Proposition \ref{Prop: descent fundamentals}. This is proved in Appendix \ref{Sec: descent}. In \S \ref{Section: mirabolic integrals}, we develop the mirabolic orbital integrals and recall a transfer statement in this setting \cite[Theorem 5.1]{Xiaothesis}, which we use to prove Theorems \ref{Thm: weak smooth transfer} and \ref{Thm: weak endoscopic transfer} in Appendix \ref{Section: proof weak transfer}. Turning to the fundamental lemmas, we recall Hironaka's work in \S \ref{s:Hir} and formulate our key relative fundamental lemma Theorem \ref{Thm: main local result 2 var}, to which we reduce the three relative fundamental lemmas in \S \ref{Section: reduce to one statement}.

Part \ref{Part: trace proof global} is devoted to the proof of Theorem \ref{Thm: main local result 2 var}.  The proof is global, relying on a different comparison of relative trace formulas related to the mirabolic orbital integrals and  the stable orbital integrals on $\Herm_{n}^{\circ}$. We crucially rely on the full strength of several previous works on the unitary periods of Jacquet-Ye, notably by Feigon--Lapid--Offen \cite{FLO} and Offen \cite{offenjacquet}. 

\subsection*{Acknowledgments}

We would like to thank Sol Friedberg, Jayce Getz, and Tasho Kaletha for helpful conversations regarding this work.

S. Leslie was partially supported by NSF grant DMS-2200852. W. Zhang was partially supported by NSF grant DMS 1901642 and 2401548, and the Simons foundation.

\section{Preliminaries}



\subsubsection{Local fields} Throughout the paper, $F$ will be a local or global field of characteristic zero. When $F$ is a non-archimedean field, we set $|\cdot|_F$ to be the normalized valuation so that if $\vp$ is a uniformizer, then
\[
|\vp|^{-1}_F= \#(\calo_F/\fp) = :q
\]
is the cardinality of the residue field. Here $\fp$ denotes the unique maximal ideal of $\calo_F.$ 

For any quadratic \'{e}tale algebra $E$ of a local field $F$, we set $\eta=\eta_{E/F}$ for the character associated to the extension by local class field theory. In particular, if $E$ is not a field, then $\eta_{E/F}$ is the trivial character.

Suppose $\G$ a reductive group over a non-archimedean local field $F$. We let $C^\infty_c(\G(F))$ denote the usual space of compactly-supported locally constant functions on $\G(F)$. If $\Omega$ is a finite union of Bernstein components of $\G(F)$, we denote by $C^\infty_c(\G(F))_{\Omega}$ the corresponding summand of $C^\infty_c(\G(F))$ (for the action by left translation).

Throughout the article, all tensor products are over $\cc$ unless otherwise indicated.

{\subsection{Invariant theory and orbital integrals}\label{Section: orbital integrals conventions}
Let $F$ be a field of characteristic zero, with a fixed algebraic closure $\Fbar$. We consider several pairs $(\rH,\X)$ consisting of a reductive group $\rH$  over $F$ and an affine $F$-variety $X$ equipped with a \emph{left} $\rH$-action. We denote by $\X\sslash \rH=\Spec(k[\X]^{\rH})$ the categorical quotient, and let $\pi_{\X}$ denote the natural quotient map. Recall that every fiber of $\pi_\X$ contains a unique Zariski-closed orbit.

Recall that an element $x\in \X(\Fbar)$ is $\rH$-regular semi-simple if $\pi_\X^{-1}(\pi_\X(x))$ consists of a single $\rH(\Fbar)$-orbit. For any field $F\subset E\subset \Fbar$, $x\in \X(E)$ is $\rH$-regular semi-simple if $x$ is $\rH$-regular semi-simple in $X(\ov{F})$. We denote the set of such elements by $\X^{rss}(E)$.	We will write $\rH_x$ for the stabilizer of $x$; for all the varieties considered in this article, the stabilizer of a regular semi-simple point $x\in \X(F)$ will be connected.
\begin{Def} \label{Def: elliptic}
		We say $x\in \X(F)$ is \emph{$\rH$-elliptic} if $x$ is $\rH$-regular semi-simple and $\rH_x/Z(\rH)$ is $F$-anisotropic.
\end{Def}

For all the varieties considered in this article,  $x\in \X(F)$ $\rH$-regular semi-simple will force $\rH_x\subset \rH$ to be a torus. 
For $x,x'\in \X^{rss}(F)$, $x$ and $x'$ are said to lie in the same \emph{stable orbit} if $\pi_\X(x)=\pi_{\X}(x')$. By definition, there exists $h \in \rH(\ov{F})$ such that $h\cdot x=x'$; this implies that the cocycle
\[
(\sig\mapsto h^{-1}h^\sig)\in Z^1(F,\rH)
\]
lies in $Z^1(F,\rH_x)$. Let $\calo_{st}(x)$ denote the set of rational orbits in the stable orbit of $x$.  If $x'\in\calo_{st}(x)$, then $\rH_x$ and $\rH_{x'}$ are naturally inner forms via conjugation by $g$. In particular, when $\rH_x$ is abelian, $\rH_x\simeq\rH_{x'}$. 
 
 A standard computation shows that $\calo_{st}(x)$ is in natural bijection with
\[
\ker^1(\rH_x,\rH;F):=\ker\left[H^1(F,\rH_x)\to H^1(F,\rH)\right],
\]
where $H^1(F,\rH)$ denotes the first Galois cohomology set for $\rH$. We will have need of abelianized cohomology as well.  When $\rH$ is a reductive group over $F$, $H^1_{\mathrm{ab}}(F,\rH)$ denotes the \emph{abelianized} cohomology group of $\rH$ \cite{Borovoi}. Note that $H^1(F,\rH)\simeq H^1_{\mathrm{ab}}(F,\rH)$ when $\rH$ is a torus.

}
Now assume that $F$ is a local field. Let $f \in C_c^\infty(\X(F))$ and $x \in \X^{rss}(F)$. We denote the orbital integral of $f$ at $x$ by
\[\Orb^{\rH}(f,x)=\int_{\rH_x(F) \bs \rH(F)}f(h^{-1}\cdot x)dh,\]
which is convergent since the orbit of $x$ is closed, and is well-defined once a choice of Haar measures on $\rH(F)$ and $\rH_x(F)$ is given (see \S \ref{measures} below). With this, we define the \emph{stable orbital integral}
\begin{equation}\label{eqn: stable orbital integral}
    \SO^{\rH}(f,x) := \sum_{y\in \calo_{st}(x)}\Orb^{\rH}(f,y).
\end{equation}
Note that this sum is finite since $F$ is local.
\subsection{Contractions}\label{Section: contractions} Throughout the paper, we will relate orbital integrals on various spaces to each other through so-called \emph{contraction maps}. We gather here a few elementary facts about this setting to clarify these steps.

{Suppose that a product group $\rH:=\rH_1\times \rH_2$ acts on an affine algebraic variety $\X$.

\begin{Assumption}\label{Assumption: contraction}
  Define $\Y=\Spec(F[\X]^{\rH_2})$ which admits an action of $\rH_1$. Denote $\pi_2: \X \to \Y$ as the natural $\rH_1$-equivariant quotient map. We assume that there exists a Zariski open sub-variety $\X^{invt}$ containing the $\rH_1\times \rH_2$-regular semi-simple locus  $\X^{rss}\subset \X^{invt}\subset \X$ such that the action of $\rH_2\simeq\{1\}\times\rH_2$ on $\X^{invt}$ is free in the sense that $\pi|_{\X^{invt}}$ is a $\rH_2$-torsor.
\end{Assumption} 
A simple family of examples is when $\X=\G$ is a group and $\rH = \rH_1\times \rH_2$ is a product of subgroups acting via left-right multiplication (where $\X^{invt} = \X$). We will also consider forms of the following example:
\[
\X=\GL_{2n}/\GL_n\times \GL_n,\quad \rH=\GL_n\times \GL_n = \rH_1\times \rH_2.
\]
\begin{Lem}\label{Lem: orbits along contraction} The quotient map $\pi$ defines an injection from (stable) $\rH_1\times\rH_2$-regular semi-simple orbits on $\X$ to (stable) $\rH_2$-regular semi-simple orbits on $\Y$. Furthermore, for $\pi(x) = y$, the projection $\rH_1\times \rH_2\to \rH_1$ induces an isomorphism $(\rH_1\times \rH_2)_x\iso \rH_{1,y}$, so that 
\begin{enumerate}
    \item $x$ is $\rH_1\times \rH_2$-elliptic if and only if $y$ is $\rH_1$-elliptic 
    \item if $x$ and $x'$ lie in the same stable $\rH_1\times\rH_2$-orbit, then $\pi(x)$ and $\pi(x')$ lie in the same stable $\rH_1$-orbit
\end{enumerate}
\end{Lem}
\begin{proof}
	Much of this follows directly from the definitions once we prove $(\rH_1\times \rH_2)_x\iso \rH_{1,y}$. This claim follows from the assumption that $\rH_2$ acts freely. We leave the details to the reader.
\end{proof}

Assume now $F$ is a local. Then each stable orbit consists of finitely many orbits. Let $f \in C_c^\infty(\X(F))$ and $x \in \X^{rss}(F)$, and consider the orbital integral $\Orb^{\rH}(f,x)$ 
and the stable orbital integral 
$
\SO^{\rH}(f,x).
$

Under Assumption \ref{Assumption: contraction}, we define the contraction map $\pi_{!}:C_c^\infty(\X(F))\to C^\infty(\Y^{rss}(F))$ as follows: \[\pi_{!}(f)(y)=\begin{cases}
    \displaystyle\int_{\rH_2(F)}f(x\cdot h)dh&: \text{if there exists $x\in \X(F)$ with }\pi(x)=y,\\
    0&: \text{otherwise.}
\end{cases}\] Since $\X^{rss}\subset \X^{invt}$, the $\rH_2(F)$-action is free in the fiber over $y\in \Y^{rss}(F)$ so that the integral is well-defined and we have that
	\[\Orb^{\rH_1\times\rH_2}(f,x)=\Orb^{\rH_1}(\pi_{!}(f),\pi(x)), \quad \SO^{\rH_1\times\rH_2}(f,x)=\SO^{\rH_1}(\pi_{!}(f),\pi(x))\] 
for any regular semi-simple $x\in \X(F)$. Moreover, if $f\in C_c^\infty(\X^{invt}(F))$, then $\pi_{!}(f) \in C_c^\infty(Y(F))$.

 \subsection{Unitary groups and Hermitian spaces}\label{Section: groups and hermitian} 
For a field $F$ and for $n\geq1$, we consider the algebraic group $\GL_n$ of invertible $n\times n$ matrices. We fix the element
\[
w_n:=\left(\begin{array}{ccc}&&1\\&\Ddots&\\1&&\end{array}\right)\in\GL_n(F).
\]
For any $F$-algebra $R$ and $g\in \GL_n(R)$, we define the involution
\[
g^\theta = w_n{}^tg^{-1}w_n,
\]
 where ${}^tg$ denotes the transpose.
 
 Now suppose that $E/F$ is a quadratic \'{e}tale algebra and consider the restriction of scalars $\Res_{E/F}(\GL_n)$. Then for any $F$-algebra $R$ and $g\in \Res_{E/F}(\GL_n)(R)$, we set $g^\sig$ to denote the induced action of the Galois involution $\sig\in \Gal(E/F)$  associated to $E/F$. We denote by ${T}_n\subset \GL_n$ the diagonal maximal split torus, ${B_n}={T_nN_n}$ the Borel subgroup of upper triangular matrices with unipotent radical ${N_n}$.

Note that $w_n$ is Hermitian in sense that that ${}^tw_n^\sig =w_n$. 
For harmony with later notation, we set $\tau_n=w_n$
and denote by $V_n$ the associated Hermitian space; this distinguishes the quasi-split unitary group $\U_n:=\U(V_n)$.
For spectral reasons, we normalize the space of Hermitian matrices with respect to $w_n$ by setting
\[
\Herm_n(F)=\{\tau\in \fgl_n(E): \tau^\ast:=\tau_n{}^t{\tau}^\sig \tau_n^{-1}=\tau\}.
\]
and denote $\Herm_{n}^{\circ}:=\Herm_n\cap \Res_{E/F}(\GL_n)$. Note that $\GL_n(E)$ acts on $\Herm_{n}^{\circ}(F)$ on the left via 
\[
 g \ast \tau=g \tau{g}^\ast,\quad \tau\in\Herm_{n}^{\circ}(F),\: g\in \GL_n(E).
\]
 We let $\calv_n(E/F)$ be a fixed set of $\GL_n(E)$-orbit representatives. For any $\tau\in\Herm_{n}^{\circ}(F),$  then $\tau\tau_n$ is Hermitian and we denote by $V_\tau$ the associated Hermitian space and $\U(V_\tau)$ the corresponding unitary group. In particular, $V_n=V_{I_n}$. We set
 \[
 U(V_\tau) = \U(V_\tau)(F).
 \]Note that $\mathcal{V}_{n}(E/F)$ gives a complete set of representatives $\{V_\tau: \tau\in \calv_n(E/F)\}$ of the isometry classes of Hermitian vector spaces of dimension $n$ over $E$. When convenient, we will abuse notation and identify this set with $\calv_n(E/F)$. 
 We have
 \begin{equation}\label{eqn: Herm orbits}
     \Herm_{n}^{\circ}(F) = \bigsqcup_{\tau\in \calv_n(E/F)}\Herm_{n}^{\circ}(F)_\tau,\quad \Herm_{n}^{\circ}(F)_\tau: = \{\tau\in \Herm_{n}^{\circ}(F): (\Res_{E/F}\GL_n)_\tau\simeq \U(V_\tau)\};
 \end{equation}
 then $\Herm_{n}^{\circ}(F)_\tau = \GL_n(E)\ast \tau\simeq \GL_n(E)/U(V_\tau)$ is the orbit under the action of $\GL_n(E)$.

For any Hermitian space $V_\tau = (V,\la\, ,\,\ra_\tau)$, 
we set 
\[
\Herm_\tau = \{x\in \fgl_n(E): \tau x^\ast=x\tau \},
\]
and 
$\Herm_{\tau}^\circ = \Herm_\tau\cap \Res_{E/F}(\GL_n)$.
 There is an isomorphism of $F$-varieties
 \begin{align*}
     \Res_{E/F}\GL_n/\U(V_\tau)&\iso \Herm_{\tau}^\circ\\
     g&\longmapsto g \tau\tau_n {}^tg^\sig(\tau\tau_n)^{-1} = g\tau g^\ast \tau^{-1},
 \end{align*}
and there is a natural isomorphism $\Herm_{n}^{\circ}\iso \Herm_{\tau}^\circ$, given by the map $y \mapsto   y\tau^{-1}$. It may readily be verified that this map intertwines that $\U(V_\tau)$-action on $\Herm_{n}^{\circ}$ with the conjugation action on $\Herm_{\tau}^\circ$. 

Note that for any $\tau\in \Herm_{n}^{\circ}(F)$, we have an isomorphism
 \begin{align}\label{eqn: twist to different form}
     \U(V_\tau)&\iso  \U(V_{\tau^{-1}})\\
      h&\longmapsto (h^\ast)^{-1}= \tau^{-1}h\tau.\nonumber
 \end{align}
\quash{The following is well-known.


\begin{Lem}\label{Lem: quotient for Herm}
    With respect to the $\U(V_\tau)$-action on $\Herm_{n}^{\circ}$, the following hold.
    \begin{enumerate}
        \item the map $\pi:\Herm_{n}^{\circ}\to \Herm_{n}^{\circ}\sslash\U(V_\tau)$ sending $y\in \Herm_{n}^{\circ}(F)$ to the coefficients of the characteristic polynomial of $y \tau^{-1}$ gives a categorical quotient.
        \item An element $y \in \Herm_{n}^{\circ}(F)$ is $\U(V_\tau)$-regular semi-simple if and only if $y \tau^{-1}$ is regular semi-simple as an element of $\GL_n(E)$. Moreover, {a regular semi-simple element} $y$ is elliptic if and only if the $F$-algebra $F[y\tau^{-1}]=F[X]/(\car_{y\tau^{-1}}(X))$ does not contain the field $E$.
    \end{enumerate}
\end{Lem}}

Finally, we say that $\tau\in \calv_n(E/F)$ is \emph{split} when $\tau= gg^\ast$ for some $g\in \GL_n(E)$, and \emph{non-split} otherwise. This somewhat standard terminology is a bit of a misnomer, as $\tau$ being split only implies that $\U(V_\tau)$ is quasi-split, and when $E/F$ is unramified that $\la,\ra_\tau$ preserves a self-dual lattice.

\subsection{Measure conventions}\label{measures} \quash{Not sure if we should stick with these conventions. What we need is:
\begin{enumerate}
    \item When $F$ is $p$-adic and $\psi$ of conductor $\calo_F,$ we want $K:=\GL_n(\calo_F)$ to have volume $1$. When $E/F$ is also unramified, the same goes for $K_E=\GL_n(\calo_E)\subset \GL_n(E)$ and $\U_n(\calo_F)\subset\U_n(F)$, as well as $\Herm_{n}^{\circ}(\calo_F):=\GL_n(\calo_E)\ast I_n\simeq \GL_n(\calo_E)/\U_n(\calo_F)$. This is for the normalization of the various period integrals. 
    \item When $F$ is $p$-adic, the volume of maximal compact subgroups of regular centralizers (all tori) is one. This is compatible with the measures used in the Jacquet--Rallis fundamental lemma, needed for Theorem \ref{Thm: Xiao FL for d}.
    \item the volume of global elliptic regular centralizers is $2L(0,\T^{\mathrm{op}},\eta)$. This is needed for the comparison of geometric sides of our RTFs
    \item an easy formula for 
    \[
    \frac{\vol(E^\times\bs\A_E^1)}{\vol(F^\times\bs\A_F^1)}.
    \] 
    This is the global scaling factor to compare the RTFs.
\end{enumerate}
    The current setting gives $\Gm$ and other tori two different measures depending on the context.}Suppose now that $E/F$ is an extension of local fields and fix an additive character $\psi: F\to \cc^\times$. By composing with the trace $\Tr_{E/F}$, we also obtain an additive character for $E$. We now fix our normalizations of invariant measures on the linear and unitary groups throughout the paper by following the conventions of \cite[Section 1.5.4]{LeslieUFJFL}, which are compatible with those of \cite{FLO}. 
{
For any non-singular algebraic variety $\mathbf{Y}$ over $F$ of dimension $d$ and gauge form $\boldsymbol{\omega}_\mathbf{Y},$ the Tamagawa measure $dy_{Tam}$ of $Y=\mathbf{Y}(F)$ is defined by transferring the standard Haar measure on $F^d$ to $Y$ by $\boldsymbol{\omega}_\mathbf{Y}$.

For the varieties we consider, we set our measure to be of the form $dy=c(\psi)^{d/2}\boldsymbol{\lam}_\mathbf{Y}dy_{Tam}$, where 
\[
c(\psi) = \begin{cases}q^m&:\text{$F$ non-archimedean and } \mathrm{cond}(\psi)=\vp^m\calo_F,\\|a|_F&: \text{$F$ archimedean and }\psi(x) = e^{2\pi i\Tr_{E/\rr}(ax)}. \end{cases}
\]
For the other terms, we impose the choice that for any $\mathbf{Y},$
\[
\boldsymbol{\omega}_{\Res_{E/F}\mathbf{Y}}=p^\ast(\boldsymbol{\omega}_\mathbf{Y}),
\]
where $p^\ast$ is given in \cite[pg. 22]{Weiladeles}. We now fix $\boldsymbol{\omega}_\mathbf{Y}$: let $n\geq 2$ (see next subsection for $n=1$).
\begin{itemize}
\item
For $\mathbf{Y}=\GL_n$, we take $\boldsymbol{\omega}_{\GL_n}=\frac{\prod_{i,j}dg_{i,j}}{\det(g)^n}$ and take $\boldsymbol{\lam}_{\GL_n} = \prod_{i=1}^nL(i,\bfun_{F^\times})$, where for any character $\chi:F^\times \to \cc^\times,$ $L(s,\chi)$ is the local Tate $L$-factor. We also set $\boldsymbol{\lam}_{\Res_{E/F}(\GL_n)} = \prod_{i=1}^nL(i,\bfun_{E^\times})$.
\item
For $\mathbf{Y}=\mathrm{H}_n$, set $\boldsymbol{\omega}_{\mathrm{H}_n}=\frac{\prod_{i\leq j}dx_{i,j}}{\det(x)^n}$, and take $\boldsymbol{\lam}_{\mathbf{H}_n} = \prod_{i=1}^nL(i,\eta^{i+1})$, where $\eta=\eta_{E/F}$ is the quadratic character associated to $E/F.$
\item
For $\mathbf{Y}=\U(V)$, we take $\boldsymbol{\omega}_{\U(V)}$ to be compatible with $\boldsymbol{\omega}_{\Res_{E/F}(\GL_n)}$ and $\boldsymbol{\omega}_{\mathrm{X}_n}$ in the sense that the isomorphism
\[
\Herm_{n}^{\circ}(F) \cong \bigsqcup_{\tau\in \calv_n}\GL_n(E)/U(V_\tau)
\]
is compatible with these measures. This forces $\boldsymbol{\lam}_{\U(V)} = \prod_{i=1}^nL(i,\eta^i)$.
\end{itemize}
}
When $F$ is $p$-adic and $\psi$ of conductor $\calo_F,$ this choice gives $K:=\GL_n(\calo_F)$ volume $1$. When $E/F$ is also unramified, the same holds for $K_E=\GL_n(\calo_E)\subset \GL_n(E)$ and $\U_n(\calo_F)\subset\U_n(F)$, as well as $\Herm_{n}^{\circ}(\calo_F):=\GL_n(\calo_E)\ast I_n\simeq \GL_n(\calo_E)/\U_n(\calo_F)$. 

When $E/F$ is a quadratic extension of number fields and $\G$ is either a linear or unitary group over $F$, we endow adelic group $\G(\A_F)$ with the product measures with respect to these local choices. As always, discrete groups are equipped with the counting measure. 

\subsubsection{Abelian $L$-functions and centralizers}\label{Section: Lvalue measure}
We now normalize the measures on various tori which occur in the paper. Let $F$ be a local or number field. Assume that $\T$ is a rank $n$ torus such that there exists a finite collection of field extensions $F_i/F$ such that 
\[
\T\simeq \prod_i\Res_{F_i/F}(\Gm);
\]in particular, $\sum_i\deg_F(F_i) = n$. We say that $\T$ is \textbf{simple} if it is of the form $\T=\Res_{L/F}(\Gm)$ for a degree $n$ extension $L/F$. When $F$ is a $p$-adic field, we endow $\T(F)$ with the Haar measure $dt$ normalized to give the (unique) maximal compact subgroup volume $1$. When $F$ is archimedean, we adopt the conventions of Tate's thesis and endow $\rr^\times$ with the measure $\frac{dx}{|x|}$, $\cc^\times$ with $\frac{2dxdy}{x^2+y^2}$, and $S^1$ with the radial measure $d\theta$ so that $\vol_{d\theta}(S^1) = 2\pi$. When $F$ is global, we endow $\T(\A_F)$ with the product measure associated to these local choices.

Fix now a quadratic extension $E/F$, and let $\eta$ be the associated quadratic character of $\A^\times_F$. {Following the notation of \cite{Xiaothesis}, let $S_1=\{i: F_i\not\supset E\}$ and $S_2=\{i: F_i\supset E\}$. By local class field theory, we see that $i\in S_1$ if and only if $\eta_i:=\eta\circ \Nm_{F_i/F_0}\not\equiv 1$.} 

Let $L(s,\T,\eta):=\prod_i L(s,\eta_i)$ denote the $L$-function associated to the character on $T\subset \GL_n$ induced by $\eta$. When $F$ is a number field, it is clear that the order of the pole of $L(s,\T,\eta)$ at $s=1$ is $|S_2|$. 
 \quash{Moreover,
\[
L(s,\T,\eta) = (D_T(\eta))^{1/2-s}L(1-s,\T,\eta),
\]
where $D_{\T} =\prod_i D_i(\eta_i)$, and for each $i$ $D_i(\eta_i)$ is the product of the absolute value of the discriminants of $F_i/\qq$ and the conductor of $\eta_i:=\eta\circ \Nm_{F_i/F}$. }

\quash{We may write 
\begin{equation}\label{eqn: discriminant factor}
    D_{\T}(\eta) = \prod_{v}q_v^{a(T,\eta_v)},
\end{equation} where $q_v$ is the size of the residue field of $F_v$ and $a(T,\eta_v)=0$ for all but finitely many places.}

For a regular $\U(V_\tau)$-semi-simple element $x\in \Herm_{\tau}^\circ(F)$, there is a similar decomposition $$F[x]:=F[X]/(\car_{x}(X))=\prod_{i=1}^mF_i,$$ where $\car_{x}(X)$ denotes the characteristic polynomial of $x$. To construct the centralizer $\T_x\subset \U(V_\tau)$, set $E_i=E\otimes_F F_i$. Then we have
\[
E[x]=\prod_i E_i=\prod_{i\in S_1}E_i\times\prod_{i\in S_2}F_i\oplus F_i,
\]
where $S_1$ and $S_2$ are as above. There is a natural norm map $\Nm: E[x] \lra F[x]$ inducing by the norm map on each factor.
\begin{Lem}\label{Lem: centralizers}
For  a regular $\U(V_\tau)$-semi-simple element $x\in \Herm_{\tau}^\circ(F)$, let $\T_x\subset \U(V_\tau)$ denote the centralizer. Then 
\[
\T_x\cong  \ker[\Res_{E[x]/F}(\Gm)\overset{\Nm}{\lra}\T_x^{\mathrm{op}}:=\Res_{F[x]/F}(\Gm)].
\]
\end{Lem}
\begin{Rem}
    We note that this encompasses all regular stabilizers arising from the unitary symmetric varieties.
\end{Rem}
\begin{proof}
Note that $x\in \Herm_{\tau}^\circ(F)\subset \GL_n(E)$ is also regular semi-simple as an element of $\GL_n(E)$, and its centralizer in this larger group is precisely the group scheme $T_{E}=\Res_{E[x]/E}(\Gm)$.  In particular, if 
\begin{align*}
    \T_x&= T_E\cap \U(V_\tau) = \{t=\sum_{k=0}^{n-1}c_kx^k : \tau t^\ast \tau^{-1} = t^{-1}\}\\
    &=\{t=\sum_{k=0}^{n-1}c_kx^k : t^{-1} = \sum_{k=0}^{n-1}\sig(c_k)x^k\},
\end{align*}
proving the claim.
\end{proof}
When $F$ is $p$-adic, we similarly endow $\T_x(F)$ with the measure $dt$ normalized to give the maximal compact subgroup volume $1$. 

Suppose now that $E/F$ is a quadratic extension of number fields. Assume $x\in \Herm_{\tau}^\circ(F)$ is regular semi-simple and consider the short exact sequence
\[
1\lra \T_x \lra \Res_{E[x]/F}(\Gm)\overset{\Nm}{\lra}\T_x^{\mathrm{op}}:=\Res_{F[x]/F}(\Gm)\lra 1.
\]
As above, we associate to $x$ the $L$-function $L(s,\T_x^{\mathrm{op}},\eta)$.  For each non-split place $v$ of $F$, we similarly consider
\begin{equation}\label{eqn: decomp of stab at v}
    F_v[x]\simeq \prod_{i\in S_v(x)}F_i = \prod_{j\in S_{v,1}(x)}F_j\times \prod_{k\in S_{v,2}(x)}F_k,
\end{equation}
where $S_{v,1}(x)=\{i\in S_v(x): F_i\not\supset E_v\}$ and $S_{v,2}(x)=\{i\in S_v(x): F_i\supset E_v\}$. We adopt the measure convention at archimedean places above, equip $\T_x(\A_F)$ with the product measure, and $[\T_x]:=\T_x(F)\bs\T_x(\A_F)$ with the quotient measure.

\begin{Lem}\label{Lem: vol of centralizer} Let $E/F$ be a quadratic extension of number fields. Assume $x\in \Herm_{\tau}^\circ(F)$ is elliptic, so that $S_2=\emptyset$ and $\T_x$ is anisotropic. With respect to the measure $dt=\prod_vdt_v$ above, 
    \[
    \vol([\T_x]) = 2^{|S_1|-|S_x|}L(0,\T_x^{\mathrm{op}},\eta),
    \]
    where $|S_x| = \sum_{v\in S}|S_{v,1}(x)|$ and $S$ is the set of places of $F$ that ramify in $E$.
\end{Lem}
\begin{proof}
    Note that $\T_x = \prod_{i}\ker[\Nm_{E_i/F_i}: \Res_{E_i/F}(\Gm)\lra \Res_{F_i/F}(\Gm)]$. Similarly, 
    \[
    L(s,\T_x^{\mathrm{op}},\eta) = \prod_i L(s,\Res_{F_i/F}(\Gm),\eta). 
    \]
    It thus suffices to prove the claim when $|S_1| = 1$, so that $F[x]$ and $E[x]$ are fields, and $\T_x$ is the Weil restriction of the norm-$1$ torus for the norm $\Nm_{E[x]/F[x]}$. 

    For any number field $F_0$, our measures on $\A_{F_0}^\times$ and $\rr^\times$ induces a measure on $\A_{F_0}^1= \ker(|\cdot |:\A_{F_0}^\times\to \rr_{>0})$. By Tate's thesis \cite{Tate}, for any Schwartz--Bruhat function on $\A_{F_0}^\times$, we have
    \[
    \Res_{s=0}\int_{\A_{F_0}^\times}f(a)|a|^sd^\times a = -f(0)\vol(F_0^\times\bs \A_{F_0}^1).
    \]
    This claim is independent of the choice of measure on $\A_{F_0}^\times$. Given our specific choice of measure, there exists $f_{F_0}=\bigotimes_v f_{F_0,v}$ such that 
    \begin{enumerate}
    \item $f_{F_0}(0)=1$, 
        \item we have the equality 
        \[
        \int_{\A_{F_0}^\times}f(a)|a|^sd^\times a = \zeta_{F_0}(s),
        \]
        where the left-hand side is the (completed) zeta function for $F_0$.
    \end{enumerate}
 Indeed, we need only take the indicator function $\bfun_{\calo_{F_0,v}}$ at each non-archimedean place (notice that our choice of measure differs from that of Tate at finitely many places) and the choice of Schwartz--Bruhat function contained in \cite{Tate} at each archimedean place. 

 Specializing this discussion to the field extension $E[x]/F[x]$, we find that
 \[
 L(0,\T_{x}^{\mathrm{op}},\eta) = \frac{\zeta_{E[x]}(s)}{\zeta_{F[x]}(s)}\bigg|_{s=0}=\frac{\vol(E[x]^\times\bs\A_{E[x]}^1)}{\vol(F[x]^\times\bs\A_{F[x]}^1)}.
 \]
  
    On the other hand, consider the following exact sequence locally compact topological groups
    \[
    1\lra [\T_x]\lra E[x]^\times \bs \A_{E[x]}^1\lra F[x]^\times \bs \A_{F[x]}^1 \lra \Nm(\A_{E[x]}^\times)F[x]^\times\bs \A_{F[x]}^\times =\zz/2\zz\lra 1.
    \]
    Note that we may endow $\T(\A_F)$ with a fiber measure $dt^{fib}=\otimes_vdt_{v}^{fib}$ with respect to the sequence
    \[
    1\lra \T_x(\A_F)\lra \A_{E[x]}^1\overset{\Nm}{\lra} \A_{F[x]}^1.
    \] This descends to another measure on $[\T_x]$, which we also denote by $dt^{fib}$. If $\vol_{fib}$ denotes the volume of $[\T_x]$ with respect to this measure, then the above exact sequence implies
    \[
    \vol_{fib}([\T_x]) = 2\frac{\vol(E[x]^\times\bs\A_{E[x]}^1)}{\vol(F[x]^\times\bs\A_{F[x]}^1)} = 2 L(0,\T_{x}^{\mathrm{op}},\eta).
    \]
     It remains to compute the constant $C$ such that $dt = Cdt^{fib}$. We work with the archimedean and non-archimedean factors separately.
     
     Let $S_\infty\subset S$ denote the set of ramified archimedean places, and set $S^\infty = S\setminus{S_\infty}$. Our choice of measures ensures that every archimedean place $v$ of $F$ splits in $E$, we have $dt_v = dt^{fib}_{v}$ since $\T_{x,v}\simeq \T^{op}_{x,v}$ in this case. When $F_v\simeq \rr$ and $v$ is a complex, our definitions give $2^{|S_{v,1}(x)|}dt_v=dt_v^{fib}$. This implies that if $dt= dt_{\infty}dt^{\infty}$ and $dt_{fib}= dt_{fib,\infty}dt_{fib}^{\infty}$, then $dt_\infty = 2^{-|S_{x,\infty}|} dt_{fib,\infty}$, where
     \[
     |S_{x,\infty}|= \sum_{v\in S_\infty}|S_{v,1}(x)|.
     \] Now set $C =2^{-|S_{x,\infty}|}C^\infty$, so that $dt^{\infty}=C^\infty dt_{fib}^{\infty}$. To compute $C^\infty$, fix an open neighborhood $K_\infty$ of $1\in (\A_{E[x]}^\times)_\infty$ and set $K=K_\infty\times \prod_v\calo_{E[x]_v}^\times$, where the product runs over all places of $F$. We may choose $K_{\infty}$ small enough that $K\cap E[x]^\times =\{1\}$. Using the above exact sequence to compute volumes of 
     \[
    \begin{tikzcd}
 {[\T_x]}\ar[r]& E[x]^\times \bs \A_{E[x]}^1\ar[r]& F[x]^\times \bs \A_{F[x]}^1 \ar[r]& \zz/2\zz\\
        {[\T_x]} \cap K\ar[r]\ar[u]& K \ar[r]\ar[u]& \Nm(K) \ar[r]\ar[u]& 0,
    \end{tikzcd}
     \]
     we find 
     \[
     \vol_{fib}^\infty(\prod_v \calo_{E[x]_v}^1) = \left[\prod_v\calo_{F[x]_v}^\times:\prod_v\Nm(\calo_{E[x]_v}^\times)\right] = 2^{|S_x^\infty|}.
     \]
     where $     |S_{x}^\infty|= \sum_{v\in S^\infty}|S_{v,1}(x)|$. 
     Since $\vol_{dt^\infty}(\prod_v \calo_{E[x]_v}^1)=1$, we find $C^\infty= 2^{-|S_x^\infty|}$.
\end{proof}

\part{Orbital integrals and the fundamental lemmas}\label{Part: FL}

\section{Orbit calculations and matching} \label{Section: unitary orbits}

 In this section, we gather the invariant-theoretic results necessary to discuss the comparison of orbital integrals. Fix now a field of characteristic zero $F_0$ equipped with two $2$-dimensional \'{e}tale $F_0$-algebras $E_0$ and $F$.  We further form the compositum $E=E_0F$, which is of degree $4$ over $F_0$. We obtain a diagram of extensions
\[
\begin{tikzcd}[every arrow/.append style={dash}]
     & E &   \\
    E_0\ar[ur,"\sig"] &       F'\ar[u]\ar[d]      & F \ar[ul, "\nu"] \\
      &      F_0\ar[ul,"\nu"]\ar[ur,"\sig"]       &
\end{tikzcd}
\]
where we have indicated the notation for the corresponding Galois automorphism. 
In this article, we only consider the cases where $E$ is not a field. 
We will be interested in the following three cases:
\begin{enumerate}
\item (split-split) $E_0=F= F_0\times F_0$,
\item (split-inert) $E_0= F_0\times F_0$, while $F/F_0$ is a field extension,
\item (inert-inert) $E_0=F$ with $F/F_0$ a field extension.
\end{enumerate}
We note that the ``inert-split'' case corresponds to the Guo--Jacquet setting \cite{Guo}; while this case arises in the general bi-quadratic setting, it will not appear in this paper.

\subsection{Uniform set-up}
\label{ss:setup}
Let $(B,\ast)=(M_{2n}(F),\ast)$ with $g^\ast = \tau_{2n}\,{}^Tg^\sig\, \tau_{2n}^{-1}$. {This is a $F$-central simple algebra with an involution of second kind.} %
\begin{Rem}
The notions and notations of this section work for a general $F$-central simple algebra equipped with an involution of second kind, which is needed for the general bi-quadratic setup.
\end{Rem}
\noindent Fix an element $\tau\in\rH_{2n}(F_0):= B^{\times,\ast=1} $ and set
\[
x^\sharp  = \tau x^\ast \tau^{-1}.
\]
We set 
\[
\G:=\U_\sharp(B)=\{g\in B^\times: gg^\sharp = I_B\};
\]
this is a unitary group, which is isomorphic to the general linear group when $F=F_0\times F_0$. We now assume that we have an embedding\begin{equation}\label{eqn: fixed embedding}
\iota_0:E_0\hra B
\end{equation} such that $\iota_0(F_0)\cap Z(B)=F_0$. We assume that $\iota_0$ is $\nu$-equivariant with respect to the two involutions $\ast$ and $\sharp$ in the sense that 
\begin{equation}\label{eqn: equivariance of inclusion}
    \iota_0(t)^\ast =\iota_0(t)^\sharp = \iota_0(t^\nu).
\end{equation}
In particular, $\tau$ centralizes $\iota_0$. Writing $E_0= F_0(\de)$ with $\de^\nu= -\de$ , $\iota_0$ is determined by an element $\de_0:=\iota_0(\de)\in B^\times$. Since $\de_0^2\in F_0^\times \subset Z(B^\times)$, the automorphism
\[
\theta:=\Ad(\de_0): B \lra B
\]
is an involution which does not depend on the choice of $\de\in E_0^\times$. 
\begin{Lem}
    The assumption \eqref{eqn: equivariance of inclusion} implies that $B$ becomes free as an $E$-module. 
\end{Lem}
\begin{proof}
    This is only non-trivial when $E\simeq F\times F$ is not a field, so we assume this is the case. With $\de_0=\iota_0(\de)$ as above, we let $V = F^{2n}$ denote the standard module for $B$. There is a natural non-degenerate $F/F_0$-Hermitian structure $(V,\la\cdot,\cdot\ra)$ such that for any $b\in B^\times$
\[
\la bv,w\ra = \la v,b^\ast w\ra.
\]

Note that $\de_0\notin \iota_0(E_0)\cap Z(B)=F_0$ since $\nu$-equivariance would force $\de_0^\ast =\de_0$. Let $\pm\al\in \Fbar_0$ denote the eigenvalues of $\de_0$. Since $E = F[\de_0]\simeq F\times F$, $\al\in F$ so we may diagonalize 
\[
V= V_{\al}\oplus V_{-\al}.
\]
We claim that $\dim_F(V_\al) = \dim_F(V_{-\al})=n$, which suffices to prove the lemma since $B$ is a free $F\times F$-module with respect to the embedding
\[
(s,t) \longmapsto \begin{psmatrix}
    sI_n&\\&tI_n
\end{psmatrix}.
\] To prove the dimension claim, it is trivial to see that $\la\cdot,\cdot\ra$ restricts trivially to $V_{\pm\al}$ since $\de_0^\ast = -\de_0$. Non-degeneracy then forces that $\la\cdot,\cdot\ra$ induces a perfect pairing $V_{\al}\simeq V^\vee_{-\al}$, proving the dimension claim.
\end{proof}

Fix an embedding \eqref{eqn: fixed embedding} for the remainder of the section; this gives a sub-algebra $E = \iota_0(E_0)F\subset B$. Setting $B_E :=\mathrm{Cent}_{B}(E) = B^\theta$, we define
\[
\rH= \Res_{E_0/F_0}(\U_\sharp(\mathrm{B}_E))=\{h\in (B_E)^\times: hh^\sharp = I_{B}\}.
\]
 As noted above, the assumption that $\iota_0$ is $\nu$-equivariant with respect to the involutions $\ast$ and $\sharp$ implies  $\tau\in\rH(F)$ so that $\tau\in (B_E)^{\times,\ast = 1}$. This implies that $\theta$ preserves $\G$ and induces an algebraic involution over $F_0$ such that $\rH = \G^\theta$. Throughout the paper, we refer to a symmetric pair of the form $(\G,\rH)$ as a \textbf{unitary symmetric pair}. 

\quash{The following lemma is straightforward.

\begin{Lem}\label{Lem: basechange to F}
There are natural compatible bijections
\begin{align}\label{eqn: basechange to F}
    \G(F)\simeq B^\times\simeq\GL_{2n}(F)\qquad\text{and}\qquad \rH(F)\simeq B_E^\times\simeq\GL_n(F)\times \GL_n(F).
\end{align}

\end{Lem}
\begin{proof}
By definition, 
\[
\G(F) = \{b\in (B\otimes_{F_0}F)^\times: bb^{\sharp_F} = I\},
\]
where $\sharp_F=\sharp\otimes I_F: B\otimes_{F_0}F\to B\otimes_{F_0}F$. The inclusion $F\otimes_{F_0}F\subset B\otimes_{F_0}F$ induces a natural map 
\begin{align}\label{eqn: twist split over F}
    \phi:  B\otimes_{F_0}F&\lra  B\times B,\nonumber\\
          (b\otimes t)&\longmapsto (bt,bt^\sig),
\end{align}
extending the $F$-algebra isomorphism $F\otimes_{F_0}F\cong F\times F$. It is easy to check that, with respect to the isomorphism \eqref{eqn: twist split over F}, the involutions $\ast_F$ and $\sharp_F$ correspond to the coordinate-by-coordinate application of the involution composed with the swap involution. Thus, 
\[
\G(F)=\{(b,b^{-\sharp}): b\in B^\times\}\cong B^\times.
\]
Now consider the embedding $\iota_F:=\iota_0\otimes Id_{F_0}: E\otimes_{F_0}F\to B\otimes_{F_0}F$. There is a similar isomorphism $E\otimes_{F_0}F\cong E\times E$ with respect to which
\[
B_E\otimes_{F_0}F\cong B_E\times B_E.
\]
A similar calculation shows that $\rH(F) \cong B_E^\times$.
\end{proof}}
\begin{Def}\label{eqn: unitary symmetric space}
We define the \textbf{linear symmetric space} $\X:=\mathrm{Emb}_{F_0}(E_0,M_{2n})$ to be the variety of embeddings of $F_0$-algebras
\[
\iota: E_0\hra M_{2n}
\]
with the property that $M_{2n}$ is free as an $E_0$-algebra. 
We further define the \textbf{unitary symmetric space} $\Q: = \mathrm{Herm}_\sharp(E_0,B)$ of \textbf{Hermitian embeddings}
\[
\iota: E_0\hra B\qquad\text{satisfying}\qquad \iota(t)^\sharp=\iota(t^\sig),
\]
with the property that $B$ is free as an $E_0$-algebra.
\end{Def}
\quash{Note that there is a natural morphism
\begin{align}
    \phi_F: \Q&\lra \Res_{F/F_0}\X_F\\
    \iota&\longmapsto \iota_F,\nonumber
\end{align}
where $\X_F$ denotes the base change to $F$ and  $\iota_F:E\lra B$ is the induced embedding obtained by sending $F\subset E$ to $Z(B)$. 
\begin{Lem}\label{Lem: isom over F}
 With respect to our fixed base point $\iota_0$, there are natural isomorphisms
\[
\Q\cong\G/\rH \qquad \text{and}\qquad \Q_F\cong\X_F\cong \G_F/\rH_F.
\] 
\end{Lem}
}

\subsection{The unitary symmetric varieties}\label{Section: unitary invariant theory} We now isolate the cases considered in this paper. Thus, fix a quadratic extension of fields $E/F$ and consider different cases for $E_0/F_0$ satisfying $E=E_0F\simeq F\times F$ is not a field. 

\begin{Important}For the rest of the paper, we only work with a single quadratic field extension. To simplify notation, we therefore change the notation of the quadratic extension $F/F_0$ to $E/F$. We let $E_0/F$ denote an auxiliary quadratic \'{e}tale extension such that $E_0\otimes_F E \simeq E\times E$.
\end{Important}
    
Using our embedding $\iota_0:E\times E\hra \fgl_{2n}(E)$, we obtain two idempotent matrices $e_1= \iota_0(1,0)$ and $e_2=\iota_0(0,1)$ such that $e_1+e_2 = I_{2n}$ and we may assume that $\de=\iota_0(1,-1)=e_1-e_2$. Define $B_i = e_iBe_i$.
There are two cases.

\subsubsection{{The split-inert case}} Suppose that $E_0=F\times F$.  In this case, $e_i^\sharp =e_i$ so that
\[
 B_i^\sharp= B_i.
\]
Both factors here are central simple algebras of degree $n$ over $E$ with involution of the second kind $\ast|_{B_i}$. This induces an isomorphism 
\[
B \cong \fgl_2(C),
\]
where $C\cong B_1\cong B_2\simeq \fgl_n(E)$ as central simple algebras with involution. In particular,
\[
B_{E\times E} = \left(\begin{array}{cc}
    C&  \\
     & C
\end{array}\right), \qquad \de_0=\left(\begin{array}{cc}
    I_{n} &  \\
     & -I_{n}
\end{array}\right),
\]
and we may take
\[
\tau= \left(\begin{array}{cc}
    \tau_1 &  \\
     & \tau_2
\end{array}\right)
\]
for any appropriate elements $\tau_1,\tau_2\in C^{\times,\ast=1}$. Then $\tau_i$ are Hermitian forms and if $V_i = V_{\tau_i}$ denotes the $n$-dimensional Hermitian space determined by $\tau_i$, then
\[
\U_{\sharp_i}(C) = \U(V_i).
\]In particular,
\[
 \left(\begin{array}{cc}
    a& b \\
    c &d
\end{array}\right)^\sharp= \left(\begin{array}{cc}
    a^{\sharp_1} & c^{\sharp_{12}} \\
   b^{\sharp_{21}} & d^{\sharp_2}
\end{array}\right),
\]
where $  b^{\sharp_{21}} = \tau_2b^\ast \tau_1^{-1}$ and similar with the notation $c^{\sharp_{12}}$. Under the assumptions and notation above, we have
\[
\rH= \left\{ \left(\begin{array}{cc}
    c_1 &  \\
     & c_2
\end{array}\right): c_i\in \U(V_i) \right\}\cong \U(V_1)\times\U(V_2), \qquad \G\simeq\U(V_{\tau_1}\oplus V_{\tau_2}).
\]

In this setting, an arbitrary element $x\in \Q(F)$ may be written
\begin{equation}\label{eqn: element of split-inert}
x=\left(\begin{array}{cc}a&b\\-b^{\sharp_{12}}&d\end{array}\right),
\end{equation}
where $a\in \Herm_{\tau_1}$, $d\in \Herm_{\tau_2}$, and $b\in \fgl_n(E)$. As in the linear setting, the blocks satisfy the polynomial relations
\[
a^2=I_n-bb^{\sharp_{12}},\: d^2=I_n-b^{\sharp_{12}} b,\: ab=bd, b^{\sharp_{12}} a= db^{\sharp_{12}}.
\]
The following lemma may be found in \cite[Section 5]{Lesliedescent}.
\begin{Lem} \label{Lem: cat quotient si}
\begin{enumerate}
\item
The morphism $\car_{si}: \Q\to \A^n$ which takes as coordinates the coefficients of the polynomial
\begin{align*}
 	x=\left(\begin{array}{cc}a&b\\-b^{\sharp_{12}}&d\end{array}\right)&\longmapsto \car_a(t),
\end{align*}
where $\car_a(t)$ denotes the characteristic polynomial, gives a categorical quotient for $(\rH,\Q)$.
      \item
Let $R: \Q\to \Herm(V_2)$ denote the map 
\[
R\left(\begin{array}{cc}a&b\\-b^{\sharp_{12}}&d\end{array}\right)= d.
\]
Then $(\Herm(V_2), R)$ is a categorical quotient for the $\U(V_1)$-action on $\Q$. The map $R$ is $\U(V_2)$-equivariant with respect to the adjoint action on $\Herm(V_2)$.
\item\label{Lem: centralizer contraction} Let $\Q^{invt}\subset \Q$ be the (Zariski-open) locus consisting of $x\in \Q$ such that $\det(b)\neq 0$; we refer to $\Q^{invt}$ as the \emph{invertible locus}.
The restriction $$R|_{\Q^{invt}}:{\Q}^{invt}\to \Herm_{\tau_2}^{invt}:= \{D\in \Herm_{\tau_2}:\det(I_{V_2}-D^2)\neq 0\}$$ is a non-trivial $\U(V_1)$-torsor. Moreover, for $x\in {\Q}^{invt}(F)$, we have an isomorphism
\[
\rH_x\iso \U(V_2)_{R(x)}.
\]
    \end{enumerate}
\end{Lem}
\begin{proof}
This is contained in  \cite[Section 5]{Lesliedescent} when $F$ is non-archimedean. The arguments extend to the general case of $F$ local or global of characteristic zero, except for the non-triviality of the $\U(V_1)$-torsor $R^{invt}:=R|_{\Q^{invt}}$. This is established in \cite[Lemma 5.]{Lesliedescent}  when $F$ is non-archimedean. The same proof works in the global and archimedean cases by replacing $H^1(F,\U(V_1))$ with the abelianized cohomology $H^1_{\mathrm{ab}}(F,\U(V_1))\simeq F^\times/\Nm(E^\times)$. For example, when $E/F = \cc/\rr$ an element $A=R(x) \in \Herm_{\tau_2}(F)^{invt}$ is constrained so that 
\[
\det(I_n-A^2)\equiv \det(\tau_1)\det(\tau_2)\pmod{\rr_{>0}}.\qedhere
\]
\end{proof}
\subsubsection{{The inert-inert case}}\label{Section: orbits ii}
Now consider the case when $E_0=E$. Then $E_0\otimes_{F}E\cong E\times E$, where $E$ is embedded diagonally, $E_0$ is embedded by
\[
t\mapsto (t,\sig({t})),
\]
and $\sig{(t_1,t_2)} = (\sig({t_2}),\sig({t_1}))$ while $\tau(t_1,t_2) = (t_2,t_1)$. Thus the two idempotents $e_1= \iota_0(1,0)$ and $e_2=\iota_0(0,1)$ satisfy
\[
e_1^\sharp = e_2;
\]
note that $e_1,e_2\in \rH(E)$. In this case, the involution induces a map
\[
\sharp:B_1\iso B_2
\]
identifying $B_2\cong B_1^{\mathrm{op}}$. Thus, we again have 
\[
B =\fgl_2(C),
\]
where again $C\cong B_1\cong B_2^{\mathrm{op}}\simeq\fgl_n(E)$. We embed $E\times E$ as
\[
(s,t)\longmapsto \left(\begin{array}{cc}
    sI_n & \\
    & tI_n
\end{array}\right),
\]
where $E, E_0\subset E\times E$ are identified as above. This indeed satisfies our compatibility assumptions on $\ast$ and $\sharp$, since
\[
 \left(\begin{array}{cc}
    a& b \\
    c &d
\end{array}\right)^\ast= \left(\begin{array}{cc}
    d^{\ast} & b^{\ast} \\
    c^{\ast} & a^{\ast}
\end{array}\right),
\]
where $a^\ast = \tau_n{}^Ta^\sig \tau_n^{-1}$ for any $a\in\fgl_n(E)$.
The restriction on $\iota_0$ being $\sharp$-equivariant now forces 
\[
\tau= \left(\begin{array}{cc}
    \tau' &  \\
     &\tau'
\end{array}\right)
\]
for some Hermitian form $\tau'\in\Herm_{n}^{\circ}(F)$. 
Under these assumptions, we have
\[
\rH = \left\{ \left(\begin{array}{cc}
    c &  \\
     & (c^\sharp)^{-1}
\end{array}\right): c\in \Res_{E/F}(\GL_n) \right\}\simeq \Res_{E/F}(\GL_n),\qquad \G=\U(V_{\tau}).
\]
where $c^\sharp=\tau' c^{\ast}{\tau'}^{-1}$. 

In this setting, any element $x\in \Q(F)$ may be written
\begin{equation}\label{eqn: element of inert-inert}
x=\left(\begin{array}{cc}a&b\\c&a^\sharp\end{array}\right),
\end{equation}
where $a\in \fgl_n(E)$ and $b,c\in \Herm(V)$. As in the previous case, the blocks satisfy the polynomial relations
\begin{equation}\label{eqn: inertinert eqns}
a^2=I_C-bc,\: (a^\sharp)^2=I_C-cb,\: ab=ba^\sharp, ca= a^\sharp c.
\end{equation}
We set $\Q^{invt}\subset \Q$ to be the Zariski-open subscheme on which $\det(b),\det(c)\neq0$, and refer to this as the \emph{invertible locus}.

\begin{Lem}\label{Lem: cat quotient ii}
 The morphism  $\car_{{ii}}: \Q\to \A^n$ which takes as coordinates the coefficients of the polynomial
\begin{align*}
 	x=\left(\begin{array}{cc}a&b\\c&a^\sharp\end{array}\right)&\longmapsto \car_a(t)
\end{align*}
where $\car_a(t)$ denotes the reduced characteristic polynomial, gives a categorical quotient for $(\rH,\Q)$.
\end{Lem}
\begin{proof}
    The only thing to note is that for any $x\in \Q^{invt}(F)$, the polynomial relations \eqref{eqn: inertinert eqns} imply that $a^\sharp = b^{-1}ab$, which implies the characteristic polynomial of $a$ has coefficients in $F$. Since $\Q^{invt}\subset \Q$ is Zariski-open and dense, the claim holds for all $\Q$.
\end{proof}

\subsection{The linear variety}\label{Sec: linear space} Finally, we consider the split-split case where $E=E_0=F\times F$ is not a field; this case occurs in the geometric sides of both relative trace formulas introduced in \cite{LXZ25}.

Clearly, when $E=F\times F$ and $\tau= ({}^{T}\tau_1,-\tau_1),$ there is an isomorphism 
\begin{align*}
\G\simeq  \GL_{2n}, \quad
(g, \tau_1{}^Tg^{-1}\tau_1^{-1})\longmapsto g.
\end{align*}
Moreover, if $E_0=E=F\times F$, then the embedding $\iota_0: E_0\lra \fgl_{2n}(E)$ induces
\[
\rH \simeq \GL_n\times \GL_n,
\]
so that $\Q=\GL_{2n}/{\GL_n\times \GL_n} = \X$. To differentiate this case from the unitary case, we set $\G':=\GL_{2n}$ and $\rH' = \GL_n\times \GL_n$.

We now recall the relevant invariant theory of $\X$. All the results of this section are contained in \cite{Lesliedescent}. Consider the quotient map $s_{\X}:\G'\lra \X$ given on points by $s_{\X}(g) = g\theta(g)^{-1}$; this induces a closed immersion $\X\subset \G$ as $F$-schemes \cite[Lemma 2.4]{Richardson}.  Since $H^1(F,\rH')=0$, we have a surjection on $F$-points
\[
s_{\X}:\G'(F)/\rH'(F)\iso \X(F).
\] Given an element $x=s_{\X}(g)\in \X(F)$, we may write
$$
x= \left(\begin{array}{cc}A&B\\C&D\end{array}\right);
$$ if the embedding $\rH'=\GL_n\times \GL_n\subset \GL_{2n}=\G'$ in a block-diagonal fashion, the action of $(g,h)\in \rH'(F)$ on $x\in \X(F)$ is given by
\begin{equation}\label{eqn:action on X}
   (g,h)\cdot x =\left(\begin{array}{cc}gAg^{-1}&gBh^{-1}\\hCg^{-1}&hDh^{-1}\end{array}\right). 
\end{equation}
Moreover, the block matrices satisfy the polynomial relations
\[
A^2=I_n+BC,\: D^2=I_n+CB,\: AB=BD,\: CA=DC.
\]
As in the unitary cases,  we set the \emph{invertible locus} $\X^{invt}\subset \X$ is the set $x\in \X(F)$ such that $\det(B)\det(C)\neq 0$.

The following lemma may be found in \cite[Section 5]{Lesliedescent}.

\begin{Lem}\label{Lem: quotients natural}
\begin{enumerate}
\item\label{Lem: cat quotient} Consider the $\rH'$-invariant map $\car_{lin}:\X\to \A^n$ given by sending $x\in \X$ to the coefficients of the monic polynomial $\det(tI_n-A)$. The pair $(\A^n,\car_{lin})$ is a categorical quotient for $(\rH',\X)$.
\item\label{Lem: halfway}
Let $R: \X\to \fgl_n$ denote the map 
\[
\left(\begin{array}{cc}A&B\\C&D\end{array}\right)\longmapsto D.
\]
Then $(\fgl_n, R)$ is a categorical quotient for the $ \GL_n\times \{1\}$-action on $\X$. The map $R$ is $\GL_n$-equivariant with respect to the adjoint action on $\fgl_n$.
\end{enumerate}
\end{Lem}

\subsubsection{The extended variety}
For the linear side of the trace formula comparison in \cite{LXZ25}, the geometric expansion is indexed not by the $\rH'$-orbits on $\X(F)$, but certain $\rH'(F)$-orbits on $\X(F)\times F_{n}$; here, $F_{n}$ denotes $1\times n$ matrices over $F$. The $\rH'$-action is given by
\[
h\cdot (x,w) = (h\cdot x,w h_2^{-1}),
\]
where $h=(h_1,h_2)\in \rH'(F)$. For notational purposes, we set
\[
\X_{ext}(F):=\X(F)\times F_{n}.
\]
\begin{Lem}\label{Lem: generic trivial stabilizer}\label{Lem: torsor of vectors}
Suppose that $(x,w)\in \X^{rss}(F)\times F_{n}$ satisfy that
\begin{equation}\label{eqn: span}
    F_{n}=\sspan\{w, wR(x),\ldots, wR(x)^{n-1}\}.
\end{equation}
Then $\rH'_{(x,w)} = 1$ is trivial. Moreover, for any $x\in \X^{rss}(F)$, the set of $w\in F_n$ satisfying \eqref{eqn: span} is a $\rH'_x(F)\cong \T_{R(x)}(F)$-torsor.
\end{Lem}
\begin{proof}
We may pass to the algebraic closure and use the projection $R$ to reduce to the case of $\GL_n(F)$ acting diagonally on $\fgl_n(F)\times F_{n}$.Set $D=R(x)\in \fgl_n(F)$. 

 The assumptions imply that there is an eigenvalue decomposition
\begin{equation}\label{diagonalize}
  F_n =\bigoplus_{i=1}^nFw_i,
\end{equation}
where $w_iD = \lam_iw_i$ and the eigenvalues are distinct. 
The Vandermonde determinant calculation tells us that if $w= \sum_ia_iw_i$, then \eqref{eqn: span} if and only if $\prod_ia_i\neq0$. But $t = \diag(t_1,\cdots t_n)\T_D(F)$ acts by $w\cdot t = \sum_i t_ia_i w_i$. It follows directly that the set of such $w\in F_n$ is a torsor for $\T_D\subset \GL_n$ and  that $\rH'_{(x,w)}=1$.
\end{proof}

\begin{Def}
We say that $(x,w)\in \X(F)\times F_{n}$ is \textbf{strongly regular} if $x\in \X^{rss}(F)$ and $F_{n}=\sspan\{w, wR(x),\ldots,w R(x)^{n-1}\}$. This locus is denoted $\X_{ext}(F)^{sr}=(\X(F)\times F_n)^{sr}$. We say $(\ga,w)\in \G'(F)\times F_{n}$ is strongly regular if $(s_{\X}(\ga),w)$ is. 
\end{Def}

 \section{Orbital integrals and the fundamental lemmas}\label{Section: orbital integrals}
 In this section, we let $E/F$ be a quadratic extension of local fields and introduce the relevant orbital integrals on $\X=\G'/\rH'$ and $\Q=\G/\rH$. We also introduce the necessary transfer factors, formulate two transfer conjectures (Conjectures \ref{Conj: ii simple endoscopy} and \ref{Conj: smooth transfer}), and state our main local theorems (Theorems \ref{Thm: fundamental lemma si}, \ref{Thm: fundamental lemma ii}, and \ref{Thm: fundamental lemma varepsilon}). We also state the local results toward transfer (Theorems  \ref{Thm: weak smooth transfer} and \ref{Thm: weak endoscopic transfer}).

We set $\eta=\eta_{E/F}$. For any character $\chi:F^\times\to \cc^\times$ and $h=(h^{(1)},h^{(2)})\in \rH'(F)$, we set
\[
\chi(h) = \chi(\det(h^{(1)})\det(h^{(2)})^{-1}).
\]
 \subsection{The linear orbital integrals}\label{section: linear orbital ints}
We begin with the orbital integrals that occur in the geometric expansion of the linear RTF introduced in \cite{LXZ25}. After several intermediate steps, the final formula is given in Lemma/Definition \ref{LemDef: linear OI}.

 For $i=0,1,2$, fix now characters $\eta_i:F^\times \to \cc^\times$ and set $\underline{\eta} = (\eta_0,\eta_1,\eta_2)$.
  For $(y,w)\in \X(F)\times F_n$ strongly regular and $\phi\otimes \Phi\in C_c^\infty(\X(F)\times F_{n})$, we consider the local orbital integral
\begin{equation}\label{eqn: basic OI X}
\Orb^{\underline{\eta}}_{s}(\phi\otimes\Phi,(y,w))=\int_{\rH'(F)}\phi(h^{-1}\cdot y)\Phi(wh^{(2)})\eta_{0}(h^{(2)}) (\eta_{1}\eta_2)(h) |h^{(2)}|^{s_0} |h|^{s_1}\,dh
\end{equation}
where
 $\ul{s}=(s_0,s_1)\in \cc^2$ and $\eta(h_1,h_2) =\eta_{1}(h_1) \eta_{2}(h_2)$.

\subsubsection{Transfer factors} 
Setting
\[
\eta(h):=\eta_{0}(h^{(2)}) (\eta_{1}\eta_2)(h),
\]
note that
\begin{equation*}
 \Orb^{\underline{\eta}}_{\ul{s}}(\phi\otimes\Phi,h\cdot(y,w)) =|h_1^{(2)}|^{s_0}|h_1|^{s_1}\eta(h) \Orb^{\underline{\eta}}_{\ul{s}}(\phi\otimes\Phi,(y,w)).
\end{equation*}


\begin{Def}\label{Def: transfer factor on X}
   Let $(y,w)$ be strongly regular, and  write $y= \begin{pmatrix}
       A&B\\C&D
    \end{pmatrix}\in \X^{rss}(F)$. We define the \textbf{transfer factor} on $\X(F)\times F_{n}$, as
\begin{equation}\label{eqn: transfer factor on X}\omega(y,w) = \eta^n_2(BC)({\eta_1\eta_2})(C){\eta_0}(\det(w|wD|\ldots|wD^{n-1})).
\end{equation} 
\end{Def}
\begin{Rem}\label{Rem: correct transfer factor}
    We have added the factor $\eta^n_2(BC)$ to the transfer factor, which is an invariant sign determined by the orbit of $y\in\X(F)$.  This is necessary to make Theorem \ref{Thm: main local result 2 var} hold (see Theorem \ref{Thm: Xiao FL for d}). 
\end{Rem}

\begin{Lem}
  For any strongly regular pair $(y,w)$, the transfer factor $\omega(y,w)$ is well-defined. For any $\wt{\phi}\in C_c^\infty(\X(F)\times F_{n})$, we have
\[
\omega(h\cdot(y,w))\Orb^{\underline{\eta}}_{s}(\wt{\phi},h\cdot (y,w)) =|h^{(2)}|^{s_0}|h|^{s_1}\omega(y,w)\Orb^{\underline{\eta}}_{s}(\wt{\phi},(y,w))
\]
for any $h\in \rH'(F)$.
\end{Lem}
\begin{proof} Writing $y= \begin{pmatrix}
       A&B\\C&D
    \end{pmatrix}\in \X^{rss}(F)$, recall (cf. \cite[Lemma 4.3]{JacquetRallis}) that $B,C\in \GL_n(F)$, so that $\eta^n_2(BC)$ and $({\eta_1\eta_2})(C)$ are well defined. Moreover, strong regularity of $(y,w)$ implies 
\[
F_{n}:= \sspan\{w,wD,\ldots,wD^{n-1}\},
\]
where $D=R(y)$, so ${\eta_0}(\det(w|wD|\ldots|wD^{n-1}))$ is also well defined. 

    As noted above, the sign $\eta^n_2(BC)$ is invariant, so we consider the properties of the other two factors. and set
\begin{equation*}
    \omega_{1}(y,w) = ({\eta_1\eta_2})\left(\det(C)\right).
\end{equation*}
It follows from \eqref{eqn:action on X} that 
\[
  \omega_{1}\left(h\cdot (y,w)\right) = ({\eta_1\eta_2})(h_1)^{-1}\omega_{1}(y,w).
\]
Setting
\begin{equation*}
    \omega_{2}(y,w) =\eta_0\left(\det([w|wR(y)|\cdots|wR(y)^{n-1}])\right),
\end{equation*}
it is easy to check that 
\[
\omega_{2,\chi}\left(h\cdot (y,w)\right) = \chi(\det(h^{(2)}))^{-1}\omega_{2,\chi}(y,w).
\]
The equivariance claim follows.
\end{proof}

For any strongly regular pair $(y,w)\in\X(F)\times F_{n}$, let $T_y\subset \rH'$ denote the stabilizer of $y$ in $\rH'$. For any character $\chi:\rH'(F)\to \cc^\times$, we consider the local (abelian) $L$-function $L(s,T_y,\chi)$ on $T_y(F)$ associated to the restriction of this character.  
{For any strongly regular $(y,w)\in \X(F)\times F_{n}$ and $\wt{\phi}\in C^\infty(\X(F)\times F_{n})$, consider the \emph{normalized orbital integral}
\begin{align}\label{eqn: normalized OI two var}
   \Orb_{s_0,s_1}^{\ul{\eta},\natural}(\wt{\phi},(y,w)):=\omega(y,w)\frac{\Orb^{\underline{\eta}}_{s}(\wt{\phi},(y,w))}{L(s_0, T_y,\eta_0)}.
\end{align}
 \begin{LemDef}\label{LemDef: linear OI}
This normalized integral \eqref{eqn: normalized OI two var} converges for $\Re(s_0)>0$, admits a meromorphic continuation to the whole $s_0$-plane, and is holomorphic at $s_0=0$. 
The value at $s_0=0$ is holomorphic in $s_1\in \cc$ and independent of $w$. We therefore write 
\begin{equation}\label{eqn: linear OI final}
\Orb^{\rH',\ul{\eta}}_{s_1}(\widetilde{\phi},y):=\Orb_{0,s_1}^{\ul{\eta},\natural}(\wt{\phi},(y,w)),
\end{equation} for any $w$ such that $(y,w)$ is strongly regular. 
\end{LemDef}
\begin{proof}For simplicity, we assume that $\wt\phi = \phi\otimes \Phi$. For $\Re(s_0)$ large enough, we rewrite the ratio in \eqref{eqn: normalized OI two var} as
a transfer factor times\[
  \int_{T_y(F)\backslash\rH'(F)}\phi(h^{-1} \cdot y)\left(\frac{\int_{T_y(F)}\Phi(w t h^{(2)})|th^{(2)}|^{s_0}\eta_0(t)dt}{L(s_0, T_y,\eta_0)}\right)\eta(h)|h|^{s_1} dh,
\] 
where $\T_y\subset \rH'$ is the stabilizer of $y\in \X(F)$. We decompose
\[
\T_y\simeq \prod_i \Res_{F_i/F}(\Gm),
\]
where $F_i/F$ is a finite field extension for each $i\in I$. For each $h_2\in \GL_n(F)$, the inner integral is precisely a product of Tate integrals for the characters $$\eta_0\circ\Nm_{F_i/F}:F_i^\times\lra \{\pm1\}.$$ The analytic continuation of the quotient is thus a consequence of Tate's thesis. In particular, the ratio is holomorphic at $s_0=0$ (the value at $s_0=0$ is described in \cite[Proposition 5.3]{Xiaothesis}).

   Finally,  Lemma \ref{Lem: torsor of vectors} implies that if $w'$ is another vector such that $(y,w')$ is strongly regular, then there exists $t\in T_y(F)$ such that
     \[
     (y,w)\cdot t = (y,w t) = (y,w'). 
     \]
     A simple calculation implies that for $\mathrm{Re}(s_0)>0$,
     \[
     \Orb_{s_0,s_1}^{\ul{\eta},\natural}(\widetilde{\phi},(y,w))=|t|^{-s_0} \Orb_{s_0,s_1}^{\ul{\eta},\natural}(\widetilde{\phi},(y,w')),
     \]
     so that the claim follows after setting $s_0=0$.
     \end{proof}

\subsubsection{The split case} We now suppose that $E=F\times F$ and that $\eta_0=\eta_1=\eta_2=1$. The Tate integral appearing in $\Orb_{0,s_1}^{\ul{\eta},\natural}(\phi\otimes\Phi,(y,w))$ then reduces to
\[
\frac{1}{L(s_0,T_y,1)}\int_{T_y(F)}\Phi(wt h^{(2)})|t|^{s_0}dt.
\]
As above, this admits analytic continuation to $s_0=0$, with corresponding value at $s_0=0$ given by $\Phi(0)$. We thus find that
\begin{equation}\label{eqn: orbital int split}
   \Orb^{\rH',\ul{\eta}}_{s_1}(\phi\otimes\Phi,y)=\Phi(0)   \int_{T_y(F)\backslash\rH'(F)}\phi(h^{-1}\cdot y)|h|^{s_1}dh.
\end{equation}
This obviously is independent of $w\in F_{n}$.

\subsection{The unitary orbital integrals}
Now assume that $\rH\subset \G$ is a unitary symmetric pair and set $\Q=\G/\rH$. For $f\in C^\infty(\Q(F))$, we consider the standard stable orbital integral $\SO^{\rH}(f,-)$ as discussed in \S \ref{Section: orbital integrals conventions}. Due to certain cohomological complexities, we consider the two cases separately.
\subsubsection{The split-inert case}In this case, we must work with several pure inner twists at once. Fix $\tau_1,\tau_2\in \calv_n(E/F)$ and consider the symmetric pair
\begin{align*}
	 \G_{\tau_1,\tau_2} &= \U(V_{\tau_1}\oplus V_{\tau_2}),\\
\rH_{\tau_1,\tau_2} &=\U(V_{\tau_1})\times \U(V_{\tau_2}).
\end{align*}
An elementary Galois cohomology argument shows that there is a disjoint union
\begin{equation}\label{eqn: orbits to sym var}
    \Q_{\tau_1,\tau_2}(F) = \bigsqcup_{(\tau_1',\tau_2')}\G_{\tau_1,\tau_2}(F)/\rH_{\tau_1',\tau_2'}(F),
\end{equation}
where $(\tau_1',\tau_2')$ range over pairs such that there is an isometry $V_{\tau_1}\oplus V_{\tau_2}\simeq V_{\tau_1'}\oplus V_{\tau_2'}$. \quash{Let $\pi_{\tau_1',\tau_2'}:\G_{\tau_1,\tau_2}\to \Q_{\tau_1,\tau_2}$ denote the quotient map with fiber $\rH_{\tau_1',\tau_2'}$.
In particular, the map
\begin{align*}
  \bigoplus_{(\tau'_1,\tau'_2)}C_c^\infty(\G_{\tau_1,\tau_2}({F}))&\lra C^\infty_c(\Q_{\tau_1,\tau_2}({F}))\\
  (f^{\tau_1',\tau_2'})_{(\tau'_1,\tau'_2)}&\longmapsto \sum_{(\tau_1',\tau_2')}\pi_{\tau_1',\tau_2',!}(f^{\tau_1',\tau_2'})
\end{align*}
 is surjective. }

 Note that the categorical quotients of these pure inner twists are all canonically identified, so we set $[\Q_{si}\sslash\rH]$ for this common quotient. The next lemma bounds how many inner twists may be considered at once without redundancy on the categorical quotient.
\begin{Lem}\label{Lem: surjective quotient si} Let $E/F$ be a quadratic extension of fields (local or global).  Fix $\tau_2\in \calv_n(E/F)$. \quash{For any $\tau_1\in \calv_n(E/F)$, let $\G^{rrs}_{\tau_1,\tau_2}$ be the points which map under $\pi_{\tau_1,\tau_2}$ to the $\rH_{\tau_1,\tau_2}$-regular semi-simple locus of $\Q_{\tau_1,\tau_2}$.  There is a bijection
    \[
    \bigsqcup_{\tau_1\in \calv_n(E/F)}\bigsqcup_{(\tau_1',\tau_2')}\rH_{\tau_1,\tau_2}(F)\bs\G^{rrs}_{\tau_1,\tau_2}(F)/\rH_{\tau_1',\tau_2'}(F)=\bigsqcup_{\tau_1\in \calv_n(E/F)}\rH_{\tau_1,\tau_2}(F)\bs\Q^{rss}_{\tau_1,\tau_2}(F),
    \]
    where $(\tau_1',\tau_2')$ range over pairs such that there is an isometry $V_{\tau_1}\oplus V_{\tau_2}\simeq V_{\tau_1'}\oplus V_{\tau_2'}$.
Moreover,}
The invariant map $\car_{si}$ from Lemma \ref{Lem: cat quotient si} induces a map
\begin{equation}\label{eqn: regular orbits on the base}
    \car_{si}:\bigsqcup_{\tau_1\in \calv_n(E/F)}\Q^{rss}_{\tau_1,\tau_2}(F)\lra[\Q_{si}\sslash\rH]^{rss}(F),
\end{equation}
satisfying that if $\tau_1\neq \tau_1'$, then 
$
 \car_{si}\left(\Q^{rss}_{\tau_1,\tau_2}(F)\right)\cap \car_{si}\left(\Q^{rss}_{\tau_1',\tau_2}(F)\right)=\emptyset.
$
This map is surjective when $\tau_2$ is split.
\end{Lem}

\begin{proof}
Fix $\tau_1,\tau_2\in \calv_n(E/F)$ and set $\rH=\rH_{\tau_1,\tau_2}$ and $\Q=\G/\rH$.  The contraction map $R: Q\to \Herm_{\tau_2}$ induces an isomorphism
    \[
    {\Q}\sslash\rH\simeq \Herm_{\tau_2}\sslash \U(V_{\tau_2});
    \]
when restricted to $\Q^{invt}$, this induces a bijection (cf. Lemma \ref{Lem: cat quotient si})
\[
\Herm_{\tau_2}(F)^{invt}=\bigsqcup_{\tau_1'}U(V_{\tau_1'})\backslash\Q_{\tau_1',\tau_2}^{invt}(F),
\]
where $\tau_1'\in \calv_n(E/F)$. Since $\Q^{rss}\subset \Q^{invt}$, this gives the second claim. Up to multiplication by a trace-zero element of $F$, we note that $\Herm_{\tau_2}(F)$ is isomorphic to $\Lie(\U(V_{\tau_2}))$. It is classical (e.g. \cite{Kottwitzrational}) that when $\tau_2$ is split (so that $\U(V_{\tau_2})$ is quasi-split) each $F$-rational $\U(V_{\tau_2})$-orbit on $\Herm_{\tau_2}$ has a rational element, so the claim follows for the full invertible locus which contains the regular semi-simple locus.
\end{proof}
With this in mind, the notion of transfer defined in \S \ref{Section: defn of transfer} below compares $C^\infty_c(\X(F)\times F_{n})$ to  collections
    \[
    \underline{\phi}=(\phi_{\tau_1})\in \bigoplus_{\tau_1\in \calv_n(E/F)}C_c^\infty(\Q_{\tau_1,\tau_2}(F)).
    \]
\subsubsection{The inert-inert case}\label{Section: unitary OI ii}We assume that $V_{2n} = L\oplus L^\ast$ is a split Hermitian space with polarization $L\oplus L^\ast$. Then we have
\[
\G= \U(V_{2n})\:\text{ and }\: \rH = \Res_{E/F}(\GL(L)).
\]
Set $\Q_{ii}:=\G/\rH$. The following lemma implies there is no need to deal with inner forms in this case.
 \begin{Lem}\label{Lem: surjective quotient ii}
     The invariant map $\car_{ii}:\Q_{ii}(F)\to [\Q_{ii}\sslash\rH](F)$ is surjective on regular semi-simple orbits.
 \end{Lem}
 \begin{proof}
  We make use of the notation from \S \ref{Section: orbits ii}. We have an $\rH$-equivariant map $R: \Q\to \End(L)=\Res_{E/F}(\fgl_{n})$, which intertwines with the $\rH$-conjugation action on $\End(L)$. Inspecting the characteristic polynomials (cf. Lemma \ref{Lem: cat quotient ii}) if $x\in \Q^{rss}(F)$, then $R(x) = a\in \End(L)^{rss}(F)$. Additionally, there exist $b,c\in \Herm_{n}^{\circ}(F)$ such that
    \[
    a^\sharp = b^{-1}ab = cac^{-1}, \quad a^2 = I_n-bc.
    \]
    In particular, the coefficients of $\car_a(t)$ lie in $F$. On the other hand, this implies that there exists an element $A\in \Herm_{n}^{\circ}(F)$ in the same $\GL_n(F)$-orbit as $a$; indeed, the invariant polynomial $\car_a(t)$ represents a regular semi-simple stable $\U_n$-orbit in $\Herm_{n}^{\circ}\subset \Res_{E/F}(\fgl_n)$, which contains a rational element as $\U_n$ is quasi-split and $\Herm_{n}^{\circ}$ may be identified as a Zariski-open subvariety of $\Lie(\U_n)$. It is now easy to check that 
    \[
    \left(\begin{array}{cc}A&I_n-A^2\\ I_n&A\end{array}\right)\in \Q(F)\mapsto \car_a(t),
    \]
    proving the claimed surjectivity.
 \end{proof}
\subsubsection{Endoscopic comparison}\label{Section: endoscopic}
While only stable orbital integrals of functions $f\in C^\infty_c(\Q_{ii}(F))$ are compared with orbital integrals on $\X(F)$, the global trace formula on $\Q_{ii}$ forces us to consider certain $\ka$-orbital integrals as well. We now considered the required endoscopic comparison.

Consider a regular semi-simple element $x\in \Q_{ii}(F)$. Fix representatives $x_\al\in \calo_{st}(x)$ where $[\inv(x,x_\al)]=\al\in H^1(F,\rH_x)$ and write
  \[
  x_\al = \left(\begin{array}{cc}
      a_\al& b_\al\\c_\al&a_\al^\ast
  \end{array}\right)\in \Q_{ii}^{rss}(F).
  \]
 As in \S \ref{Section: unitary invariant theory}, we know that $b_\al,c_\al\in \Herm_{n}^{\circ}(F)$. Note that if $x= \begin{psmatrix}
      a& b\\c&a^\ast
  \end{psmatrix}$, there is a natural inclusion $\rH_x\subset \U(V_{b})\cap \U(V_c)$.
  
  \begin{Lem}\label{Lem: building special kappa} Set
  \[
  \ker^1_{\mathrm{ab}}(\rH_x,\U(V_b);F)=\ker[H^1(F,\rH_x)\lra H^1_{\mathrm{ab}}(F,\U(V_b))],
  \]
  where $H^1_{\mathrm{ab}}(F,\U(V_0))$ is the abelianized cohomology group. Then
  \[
  \al\in \ker^1_{\mathrm{ab}}(\rH_x,\U(V_b);F)\iff \eta(\det(b_\al)) = \eta(\det(b)).
  \]
  \end{Lem}
  \begin{proof} For any $\al$, we may find by assumption $h\in  \rH(E)$ such that
   \[
  x_\al = h\ast x \implies b_\al = h b h^\ast.
  \]
  Recall $H^1(F,\rH_x) = H^1(\Gal(E/F), \rH_x(E))$, the class $\al$ is represented by the $1$-cocycle $\sig\mapsto h^\ast h$.  
 On the other hand, consider the long exact sequence in Galois cohomology
\[
1\lra U(V_b)\lra\GL(V)\overset{-\ast b}{\lra} \Herm_{n}^{\circ}(F)\lra H^1(F,\U(V_b))\lra 1,
\]
where the image of $h b h^\ast\in \Herm_{n}^{\circ}(F)$ is the same $1$-cocycle $\sig\mapsto h^\ast h$. It follows that $\al\in\ker^1(\rH_x,\U(V_b);F)$ if and only if $b_\al\in\Herm_{n}^{\circ}(F)_b= \GL_n(E)\ast b$. When $F$ is non-archimedean, this holds if and only if 
\[
\det(b_\al)\equiv \det(b)\pmod{\Nm_{E/F}(E^\times)},
\]
as claimed. It only remains to consider the case $F=\rr$, where the abelianization map
\[
\mathrm{ab}^1:H^1(F,\U(V_b))\to H^1_{\mathrm{ab}}(F,\U(V_b))\simeq \rr^\times/\Nm(\cc^\times),
\]
is given by the sign of the determinant.
  \end{proof}

\begin{Lem}
    For any $x\in \Q_{ii}^{rss}(F)$, the map
\begin{align}\label{eqn: special kappa}
\varepsilon: H^1(F,\rH_x)&\lra \cc^\times\\
\al&\longmapsto\eta(\det(b_\al)\det(b))^{-1}\nonumber
\end{align}
 gives a non-trivial character $\varepsilon\in H^1(F,\rH_x)^D$. In particular, when $\rH_x$ is a simple elliptic torus, then $H^1(F,\rH_x)^D= \{1,\varepsilon\}$.
 \end{Lem}
 \begin{proof}
For any $\al$, we may find $h = \diag(h_1,h_2)\in  \rH(E)$ such that
   \[
  x_\al = h\ast x,
  \]
  so that the cocycle $\sig\mapsto h^\ast h$ represents $\al\in H^1(F,\rH_x) = H^1(\Gal(E/F), \rH_x(E))$. Here $h_1,h_2\in \GL_n(\Fbar)$, and cocycle is of the form
\[
\sigma\longmapsto h^\ast h  = \begin{psmatrix}
    h_2^\ast h_1&\\ &h_1^\ast h_2.
\end{psmatrix}
\]
We claim that $\det(h_2^\ast h_1)\in F^\times$ and that $\varepsilon(\al) = \eta(\det (h_2^\ast h_1))$; since the right-hand formula is clearly a character of $H^1(F,\rH_x)$, the claim follows. 

First note that the assertion that $x_\al\in \Q(F)$ implies that 
\begin{equation}\label{rational implication}
 b_\al = h_1 b h_2^{-1} = (h_2^\ast)^{-1}bh_1^\ast.
\end{equation}
In particular, this implies $b(h_1^\ast h_2) = (h^\ast_2 h_1) b$ so that
\[
\det(b)\det(h_1^\ast h_2) = \det(h_2^\ast h_1) \det(b).
\]
This shows that $\det(h_1^\ast h_2)=\overline{\det(h_1^\ast h_2)}$. On the other hand, \eqref{rational implication} also implies
\[
\det(b_\al)\det(b)^{-1} = \det(h_1 h_2^{-1}) \in F^\times.
\]
Thus, 
\[
\det(h_1 h_2^{-1}) \equiv \det(h_2^\ast h_1)\pmod{\Nm_{E/F}(E^\times)},
\]
proving that $\varepsilon(\al) = \eta(\det (h_2^\ast h_1))$ is a character.
  
Finally, when $\rH_x$ is simple we have $|H^1(F,\rH_x)|=2$, so the final claim is immediate from the non-triviality of $\varepsilon$. 
 \end{proof}
 If one considers the theory of endoscopy defined in \cite{LeslieEndoscopy} for $(\G,\Q_{ii})$, the endoscopic variety determined by the pair $(x,\varepsilon)$ is the \emph{split-inert} form (cf. \cite[Proposition 10.2]{LeslieEndoscopy}). Thus we have the following special case of the conjectural existence of endoscopic transfer for $\Q_{ii}$. 

Fix $\tau_1,\tau_2\in \calv_n(E/F)$, and we assume $\tau_2=\tau_n$ is split. We say that $x\in \Q_{ii}^{rss}(F)$ matches $y\in \Q_{\tau_1,\tau_2}^{rss}(F)$ if they have matching invariant polynomials; comparing Lemma \ref{Lem: surjective quotient si} with \ref{Lem: surjective quotient ii}, we see that for any $x\in \Q_{ii}(F)$ there exists a unique $\tau_1$ such that $x$ matches some $y\in \Q_{\tau_1,\tau_2}(F)$. When this occurs we write $x\leftrightarrow_{\tau_1}y$. 
\begin{Conj}\label{Conj: ii simple endoscopy}
 Define the transfer factor $\De_{\varepsilon}: \Q_{ii}^{rss}(F)\to \cc^\times$ by
\[
       \De_{\varepsilon}(x):= \eta(\det(b)), \quad\text{ where   $x = \begin{psmatrix}
      a& b\\c&a^\ast
  \end{psmatrix}$. }
\]
Assume $\tau_2=\tau_n\in \calv_n(E/F)$ is split. For any $\phi\in C^\infty_c(\Q_{ii}(F))$ and $x\in \Q_{ii}^{rss}(F)$, there exists $\ul{\phi}=(\phi_{\tau_1})\in \bigoplus_{\tau_1\in \calv_n(E/F)}C_c^\infty(\Q_{\tau_1,\tau_2}(F))$ such that
     \[
     \De_{\varepsilon}(x) \Orb^{\rH,\varepsilon}(\phi,x) = \SO^{\rH_{\tau_1,\tau_2}}(\phi_{\tau_1},y),
     \]
     for all matching regular semi-simple $x\leftrightarrow_{\tau_1}y$. We say $\phi$ and $\underline{\phi}$ are $\varepsilon$-transfers.
\end{Conj}
We establish a weak form of this conjecture, as well as the corresponding fundamental lemma, in \S \ref{Section: defn of transfer} below.

\quash{\begin{proof}
 If we write  $x = \begin{psmatrix}
      a& b\\c&a^\ast
  \end{psmatrix}$,  we set 
    \[
       \De_{\varepsilon}(x):= \eta(\det(b)),
    \]
    and we define $\phi_{\varepsilon}\in C^\infty_c(\Q^{invt}_{ii}(F))$ by $\phi_{\varepsilon}(x)= \eta(\det(b))\phi(x)$.
    Then we compute
    \begin{align*}
        \De_{\varepsilon}(x) \Orb^{\rH,\varepsilon}(\phi,x) &=\eta(\det(b))\sum_{x_\al\in\calo_{st}(x)}\eta(\det(b_\al))\eta(\det(b))^{-1}\Orb^{\rH}(\phi,x_\al)\\
                                                            &=\sum_{x_\al\in\calo_{st}(x)}\Orb^{\rH}(\phi_\varepsilon,x_\al)=\SO^{\rH}(\phi_{\varepsilon},x).
    \end{align*}
    This proves the first claim.

    Now assume that $E/F$ is unramified and set $\phi = \bfun_{\Q_{ii}(\calo_{F})}$. Assume $x\in \Q_{ii}(F)$ is as in the conjecture. Clearly, $\Orb^{\rH,\varepsilon}(\phi,x)=0$ unless $x$ lies in the stable orbit of an element $x_0\in \Q_{ii}(\calo_{F})$, so we may as well assume that $x\in\Q_{ii}(\calo_{F})$. This implies that $\eta(\det(b)) = 1$. \textcolor{red}{Moreover, if $\eta(\det(b_\al))\neq \eta(\det(b))$, then $\Orb^{\rH}(\phi,x_\al)=0$.}\WL{This line is wrong.} Thus,
    \begin{align*}
        \De_{\varepsilon}(x) \Orb^{\rH,\varepsilon}(\phi,x) =\eta(\det(b))\sum_{\stackrel{x_\al\in\calo_{st}(x)}{\eta(\det(b_\al))= \eta(\det(b))}}\Orb^{\rH}(\phi,x_\al)=\SO^{\rH}(\phi_{\varepsilon},x),
    \end{align*}
    proving the second claim.
\end{proof}
}

\subsection{Matching of orbital integrals}\label{Section: defn of transfer} We continue to assume that $F$ is local. Set $\G' = \GL_{2n}$ and $\rH' = \GL_n\times \GL_n$, viewed as $F$-groups. To unify notation for our cases of interest, we set $(\G_\bullet,\rH_\bullet)$ to be a unitary pair with $\bullet\in\{si, ii\}$.  Set $\X:=\G'/\rH'$ and $\Q_{\bullet} = \G_{\bullet}/\rH_{\bullet}$. We now adopt the conventions of our cases of interest, so that $\eta_1=1$.

We have the quotient diagram
\[
\begin{tikzcd}
\X\ar[dr,"\car_{lin}",swap]&&\Q_\bullet\ar[dl,"\car_{\bullet}"]\\
&\A^n,&
\end{tikzcd}
\]
where $\car_{lin}$ is defined in Lemma \ref{Lem: quotients natural} and $\car_{\bullet}$ is defined in Lemmas \ref{Lem: cat quotient si} and \ref{Lem: cat quotient ii}. 
\begin{Def}\label{Def: Matching orbits}
    We say that regular elements $y\in \X(F)$ and $x\in {\Q}_\bullet(F)$ \emph{match} if
\[
\car_{lin}(y) = \car_{\bullet}(x).
\]
\end{Def}

\begin{Def}[Smooth transfer]\label{Def: transfer} Set $\underline{\eta}=(\eta_0,\eta_2)$.
\begin{enumerate}
    \item (split-inert: $\underline{\eta} = (\eta,\eta)$) Fix $\tau_2\in \calv_n(E/F)$. We say that $\widetilde{\phi}\in C^\infty_c(\X(F)\times F_{n})$ and a collection
    \[
    \underline{\phi}=(\phi_{\tau_1})\in \bigoplus_{\tau_1\in \calv_n(E/F)}C_c^\infty(\Q_{\tau_1,\tau_2}(F))
    \]
    are \emph{$(\eta,\eta)$-transfers} if for any matching regular semi-simple elements $y\in \X(F)$
\begin{enumerate}
\item  if there exists $\tau_1$ and $x\in \Q^{rss}_{\tau_1,\tau_2}(F)$ matching $y$, then
\begin{equation*}
  2^{|S_1^{ram}(y)|}\Orb^{\rH',\ul{\eta}}_{0}(\widetilde{\phi},y)= \SO^{\rH_{\tau_1,\tau_2}}(\phi_{\tau_1},x),
\end{equation*}
where $S_1$ is defined as in \S \ref{Section: Lvalue measure}, and $S_1^{ram}(y)=\{i\in S_1: \eta_i \text{ is ramified on }F_i\}$\footnote{When $F=\rr$, this is equivalent to $\eta_i = \sgn$ on $F_i=\rr$.}.
\item if there does not exist $\tau_1$ and $x\in \Q^{rss}_{\tau_1,\tau_2}(F)$ matching $y$, then
\begin{equation*}
  \Orb^{\rH',\ul{\eta}}_{0}(\widetilde{\phi},y)= 0.
\end{equation*}
\end{enumerate}
    \item (inert-inert: $\underline{\eta} = (\eta,1)$) We say that $\widetilde{\phi}\in C^\infty_c(\X(F)\times F_{n})$ and $\phi\in C_c^\infty(\Q_{ii}(F))$
    are \emph{$(\eta,1)$-transfers} if for any matching regular semi-simple elements $y\in \X(F)$  and $x\in \Q^{rss}_{ii}(F)$, we have
\begin{equation*}
 2^{|S_1^{ram}(y)|} \Orb^{\rH',\ul{\eta}}_{0}(\widetilde{\phi},y)= \SO^{\rH}(\phi,x).
\end{equation*}
   \end{enumerate}
\end{Def}
\begin{Rem}
    We note that the coefficient $2^{|S_1^{ram}(y)|}$ is only visible when $E/F$ is ramified. 
\end{Rem}
To unify notation, we set $\underline{\phi} = \phi\in C_c^\infty(\Q_{ii}(F))$ in the inert-inert case. Then we have $\underline{\phi}=(\phi_{a})_{a\in A}$ where $A= \calv_n(E/F)$ for the split-inert case and $A=\{ii\}$ for the inert-inert case, with $f_{a}\in C_c^{\infty}(\Q_a(F))$.
\begin{Conj}\label{Conj: smooth transfer}
    Set $\underline{\eta}\in \{(\eta,1),(\eta,\eta)\}$. Then for any $\widetilde{\phi}\in C^\infty_c(\X(F)\times F_{n})$ there exists a $\underline{\eta}$-transfer $\underline{\phi}$ as defined above. Conversely, for any collection of functions $\underline{\phi}$ on the unitary variety, there exists its $\underline{\eta}$-transfer $\widetilde{\phi}$.
\end{Conj}
We have the following pieces of evidence for this conjecture, which suffice for our global applications in \cite{LXZ25}.
\subsubsection{The split case}\label{Section: split transfer} We begin by assuming that $E=F\times F$, so that $\underline{\eta}=(1,1)$. In this case, 
\[
\G= \GL_{2n} = \G', \:\text{and}\:\rH=\GL_n\times \GL_n = \rH',
\]
and Conjecture \ref{Conj: smooth transfer} is trivial. More precisely, recall the formula \eqref{eqn: orbital int split}. If $\phi\in C^\infty(\X(F))$ and $\Phi\in C_c^\infty(F_{n})$, we see $\Phi(0)\phi\in C_c^\infty(\Q(F))$ is a smooth $(1,1)$-transfer for $\phi\otimes \Phi$.
\subsubsection{A weak transfer result}\label{Section: matching} We now assume that $E/F$ is a quadratic extension of non-archimedean local fields. We have the following result toward Conjecture \ref{Conj: smooth transfer}.
\begin{Thm}\label{Thm: weak smooth transfer}
    Suppose that $\phi\in C_c^\infty(\X^{invt}(F))$. Then for any $\Phi\in C^\infty_c(F_{n})$, there exists a smooth $\underline{{\eta}}$-transfer $\underline{f}=(f_a)$ satisfying $f_a\in C_c^\infty(\Q_a^{invt}(F))$. Conversely, for any $\underline{f}$ with support in the invertible locus, there exists a smooth transfer $\wt{\phi}\in C_c^\infty(\X^{invt}(F)\times F_{n})$.
\end{Thm}
We also establish the following weak form of the endoscopic transfer statement in Conjecture \ref{Conj: ii simple endoscopy}.
\begin{Thm}\label{Thm: weak endoscopic transfer}
    Suppose that $\phi\in C_c^\infty(\Q_{ii}^{invt}(F))$. Then there exists a $\underline{\phi}=(\phi_{\tau_1})$ such that $\phi_{\tau_1}\in C_c^\infty(\Q^{invt}_{\tau_1,\tau_2}(F))$ and
     \[
     \De_{\varepsilon}(x) \Orb^{\rH,\varepsilon}(\phi,x) = \SO^{\rH_{\tau_1,\tau_2}}(\phi_{\tau_1},y),
     \]
     for all matching regular semi-simple $x\leftrightarrow_{\tau_1}y$.  
\end{Thm}
After introducing the necessary Cayley and contraction maps, we prove this in Appendix \ref{Section: proof weak transfer}. 
\begin{Rem}
  While our global results in \cite{LXZ25} do not use these statements, we include them for two reasons: they provide evidence for Conjectures  \ref{Conj: ii simple endoscopy} and  \ref{Conj: smooth transfer}, and the method of their proof introduces many of the necessary ideas for the proofs of the fundamental lemmas introduced in the next subsection.
\end{Rem}
\subsection{The fundamental lemmas}\label{Section: fundamental lemmas on variety}
We now assume that $E/F$ is an unramified quadratic extension of $p$-adic local fields. 
We adopt the following assumption in this subsection.

\begin{Assumption}
Assume that if $e$ denotes the ramification degree of $F/\Q_p$, then $p>\max\{e+1,2\}$.
\end{Assumption}
This assumption is imposed by \cite[Lemma 8.1]{Lesliedescent}, which is crucial to our descent to the Lie algebra in Appendix \ref{Sec: descent}. It rules out only finitely many places for a given global field. 

\subsubsection{Integral models}\label{Section: integral models} Let $k\in \zz_{\geq0}$. Viewing $(\G',\rH')$ as $\calo_F$-group schemes by using the standard lattice, let $\G'_k$ denote the $k$-th congruence subgroup $\calo_F$-scheme of $\G'$, so that $\G'_k(\calo_F) = I_{2n}+\vp^k\fg'(\calo_F)$. We similarly have the congruence group $\calo_F$-schemes $\rH'_k$, and set $\X_k:=\G'_k/\rH'_k$.

Similarly, assume $(\G,\rH)$ is a unitary symmetric pair consisting of \emph{unramified groups} for which we fix a smooth $\calo_F$-model which we denote by $(\G_0,\rH_0)$. This gives a nice simply-connected symmetric pair over $\calo_F$ in the sense of \cite[Section 4.1]{Lesliedescent}. For each $k> 0$, we also obtain a $k$-th congruence pair of $\calo_F$-schemes $(\G_k,\rH_k)$ (cf. \cite[Definition A.5.12]{KalethaPrasad}); these are all smooth $\calo_F$-schemes with connected special fibers \cite[Proposition A.5.22]{KalethaPrasad}. Set $\Q_k:=\G_k/\rH_k$.

\subsubsection{The split-inert case}
Recall that $\tau_n=w_n$ is our distinguished split Hermitian form, and $V_n = V_{\tau_n}$. In our setting, $|\calv_n(E/F)|=2$ and we set $\calv_n(E/F)= \{\tau_n, \tau_n^{ns}\}$, where $\tau_n^{ns}$ is non-split, and write $\ul{\phi}=(\phi_{s},\phi_{ns})$ in the sense of Definition \ref{Def: transfer}.

The self-dual lattice $\Lam_n\subset V_n$ induces an integral structure on $\G_{\tau_n,\tau_n}$, $\rH_{\tau_n,\tau_n}$ and $\Q_{\tau_n,\tau_n}$ which we denote by $\G_0$, $\rH_0$, and $\Q_{si,0}$. We also obtain $k$-th congruence subgroups $(\G_k,\rH_k)$ as in \S \ref{Section: integral models} and let $\Q_{si,k}(\calo_F)$ denote the $\calo_F$-points of the associated quotient.

\begin{Thm}\label{Thm: fundamental lemma si} Fix $\tau_2=\tau_n\in \calv_n(E/F)$.
For all $k\in \zz_{\geq0}$, the function $\bfun_{\X_k(\calo_F)}\otimes \bfun_{\calo_{F,n}}$ and the collection $((-1)^k\bfun_{\Q_{si,k}(\calo_F)},0)$ are $(\eta,\eta)$-transfers.
    \end{Thm}

\subsubsection{The inert-inert case} Now consider the split form $\tau_{2n}\in \calv_{2n}(E/F)$, and fix a polarization $V_{2n} = L\oplus L^\ast$. We fix a self-dual lattice $\Lam_{2n}\subset V_{2n}$ equipped with the polarization $\Lam_{2n}= \mathcal{L}\oplus \mathcal{L}^\ast$ inducing the polarization of $V_{2n}$. This induces $\calo_F$-structures on $\G = \U(V_{2n})$, $\rH= \GL(L)$, and $\Q_{ii}=\G/\rH$ which we denote by $\G_0$, $\rH_0$ and $\Q_{ii,0}$ as above. We also obtain $k$-th congruence subgroups $(\G_k,\rH_k)$ as above, and let $\Q_{ii,k}(\calo_F)$ denote the $\calo_F$-points of the associated quotient.

\begin{Thm}\label{Thm: fundamental lemma ii}
For all $k\in \zz_{\geq0}$, the functions $\bfun_{\X_k(\calo_F)}\otimes \bfun_{\calo_{{F},n}}$ and $\bfun_{\Q_{ii,k}(\calo_F)}$
    are $(\eta,1)$-transfers.
    \end{Thm}
    \subsubsection{The endoscopic case} Finally, we have the fundamental lemma associated to the endoscopic comparison in Conjecture \ref{Conj: ii simple endoscopy}. Let all notations be as in the previous two subsections.

    \begin{Thm}\label{Thm: fundamental lemma varepsilon}
For all $k\in \zz_{\geq0}$, the functions $\bfun_{\Q_{ii,k}(\calo_F)}$ and the collection $((-1)^k\bfun_{\Q_{si,k}(\calo_F)},0)$ are $\varepsilon$-transfers.
    \end{Thm}



The proofs of Theorems \ref{Thm: fundamental lemma si}, \ref{Thm: fundamental lemma ii}, and \ref{Thm: fundamental lemma varepsilon} will occupy the remainder of the paper. This relies on linearizing via descent to the Lie algebra and then relating the orbital integrals to those on another space via a {contraction map}.

For the remainder of this part, we let $E/F$ be a quadratic extension of $p$-adic fields.

\section{Statements on the Lie algebra}\label{Section: Lie}

In this section, we formulate the necessary objects and results on the Lie algebra. The actual descent to this linearized setting is handled in Appendices \ref{Section: proof weak transfer} and \ref{Sec: descent}.

\subsection{The infinitesimal setting}
We recall the infinitesimal analogues of the symmetric varieties. We also state the relevant invariant theoretic results for these representations. As the proofs are elementary and analogous to the case of the varieties, we omit the details.
\subsubsection{The infinitesimal theory for the linear variety}\label{Section: linear Lie}
Let $y_0= 1\in \X(F)=\G'(F)/\rH'(F)$ denote the unit orbit. The tangent space $T_{y_0}(\X)$ of $\X$ at $y_0$ is the $F$-rational $\rH'$-representation 
\[
\mathfrak{X}:=T_{y_0}(\X)\cong \fgl_n\oplus \fgl_n,
\]
where the action is
\[
(h_1,h_2)\cdot (X,Y) = (h_1Xh_2^{-1}, h_2Yh_1^{-1}).
\]
As in the variety setting, we are interesting in the $\rH'$-action on
\[
\mathfrak{X}_{ext}(F):= \mathfrak{X}(F)\times F_n.
\]
The following lemma sums up the important invariant-theoretic statements.
\begin{Lem}\label{Lem: inft linear invt theory}
\begin{enumerate}
    \item An element $(X,Y)\in \mathfrak{X}(F)$ is $\rH'$-regular semi-simple if and only if $YX\in \GL_n(F)$ is regular semi-simple.
    \item An element $((X,Y),w)\in \mathfrak{X}_{ext}(F)$ is \emph{strongly regular} if and only if $(X,Y)$ is $\rH'$-regular semi-simple and 
\[
\det[w|w(YX)|\cdots |w(YX)^{n-1}]\neq 0;
\]
in this case, $(X,Y)$ is regular semi-simple and $\rH'_{((X,Y),w)}=1$.
\item The natural $\GL_n$-invariant map 
\begin{align*}
   R: \fgl_n\oplus \fgl_n&\lra \fgl_n\\
    (X,Y)&\longmapsto YX
    \end{align*}
     gives a categorical quotient and intertwines the $\rH'$-action with the adjoint action of $\GL_n$ with respect to projecting to the second factor. In particular, the characteristic polynomial of $YX$ gives a categorical quotient for the $\rH'$-action on $\mathfrak{X}$.
\end{enumerate}
\end{Lem}

Later on it will be useful to isolate the following Zariski-open subvariety.
\begin{Def}\label{Def: invertible locus Lie}
    The \emph{invertible locus} $\fX^{invt}(F)\subset \fX(F)$ is the locus $(X,Y)\in \fX(F)$ such that $\det(X)\det(Y)\neq 0$.
\end{Def}
As in the variety case, for any infinitesimal unitary symmetric space $\fq$, we set $\fq^{invt}\subset \fq$ to be the obvious analogue of $\fq^{invt}$.

\subsubsection{The infinitesimal theory for the unitary varieties}
Let $\Q=\G/\rH$ be a unitary symmetric variety. Let $\fg$ denote the Lie algebra of the unitary group $\G$. Consider involution $\theta$ of $\G$ such that $\rH=\G^\theta$. The differential of $\theta$ acts on $\fg$ and induces a $\zz/2\zz$-grading
\[
\fg= \fh\oplus \fq.
\]
Then $\fh=\Lie(\rH)$ and $\rH$ acts on $\fq$ by restriction of the adjoint representation, and the pair $(\rH,\fq)$ is the infinitesimal symmetric space associated to $\Q$, on which we consider the stable orbital integrals (cf. \S \ref{Section: orbital integrals conventions}).

 
\begin{Lem}\label{Lem: unitary Lie}
    \begin{enumerate}
        \item (split-inert case) Fix $\tau_1,\tau_2\in\calv_{n}(E/F)$. We have a natural isomorphism of $H(F)$-representations
\begin{align*}
    \fq_{\tau_1,\tau_2}\simeq \Hom_E(V_{\tau_2},V_{\tau_1})
\end{align*}
where the action on the right-hand side is given by pre- and post-composition. Choosing bases, we identify
$$
\fq_{si}:=\End(V_n)\simeq \Res_{E/F}(\fgl_n),
$$ where $\rH_{\tau_1,\tau_2}=\U(V_{\tau_1})\times \U(V_{\tau_2})$ acts via $(h_1,h_2)\cdot X = h_1 X h_2^\ast.$ 
The map
        \[
        \car_{si}:\fq_{si}(F) \lra F^n,
        \]
        where $\car_{si}(X)\in F^n$ is the coefficients of the characteristic polynomial of $-X^\ast \tau_1^{-1} X\tau_2^{-1}$, is a categorical quotient for the $\rH_{\tau_1,\tau_2}$-action.
\item(inert-inert case) Identify $\G_{ii}=\U(V_{2n})$ and $\rH_{ii}=\Res_{E/F}\GL(L)$ where $V_{2n}=L\oplus L^\ast$, we have a natural isomorphism of $\rH_{ii}$-representations
$$\fq_{ii}:=\Herm(L^\ast,L)\times \Herm(L,L^\ast)\simeq\Herm_n\times \Herm_n,$$ where $g\in \GL(L)$ acts via $g\cdot (x_1,x_2) = (gx_1 g^\ast, (g^\ast)^{-1}x_2 g^{-1})$. The map
        \[
        \car_{ii}:\fq_{ii}(F)\lra F^n,
        \]
        where $\car_{ii}(x_1,x_2)\in F^n$ is the coefficients of the characteristic polynomial of $x_2x_1$, is a categorical quotient for the $\rH_{ii}$-action.
    \end{enumerate}
\end{Lem}

 \subsection{Orbital integrals}\label{Section: linear statements}
We begin with the linear case. For a pair of characters $\underline{\eta}=(\eta_0,\eta_2)$, we define the transfer factor 
\begin{equation}\label{eqn: linear Lie transfer factor}
    \widetilde{\omega}((X,Y),w) := \eta_2^n(YX) \eta_0\left(\det[w|w(YX)|\cdots |w(YX)^{n-1}]\right)\eta_2(Y),
\end{equation}
where $((X,Y),w)$ is strongly regular in $\fX(F)\times F_n$.

 For $\phi\otimes \Phi\in C_c^\infty(\mathfrak{X}(F)\times F_{n})$ and for $[(X,Y),w]$ strongly regular, we consider the regularized orbital integral $ \Orb_{s_0,s_1}^{\rH',\ul{\eta},\natural}(\phi\otimes\Phi,[(X,Y),w])$ given by
    \begin{equation}\label{eqn: Lie alg linear OI with s}
      \frac{\widetilde{\omega}((X,Y),w) }{L(s_0, \eta_0,T_{YX})} \int_{\rH'(F)}\phi(h^{-1}\cdot (X,Y))\Phi(wh^{(2)})\eta_{0}(h^{(2)}) \eta_2(h) |h^{(2)}|^{s_0} |h|^{s_1}\,dh.
\end{equation}
    By linearity, these extend to distributions on $C_c^\infty(\mathfrak{X}(F)\times F_{n})$. The analogue of Lemma/Definition \ref{LemDef: linear OI} holds with a similar proof.
\begin{Def}
 For $\widetilde{\phi}\in C_c^\infty(\mathfrak{X}(F)\times F_n)$ and $(X,Y)\in\fX^{rss}(F)$, we define
\begin{equation}\label{eqn: OI lin Lie}
\Orb_{s_1}^{\rH',\ul{\eta}}(\widetilde{\phi},(X,Y)):=\Orb_{s_0,s_1}^{\rH',\ul{\eta},\natural}(\widetilde{\phi},[(X,Y),w])\bigg|_{s_0=0},
\end{equation}
for any $w\in F_n$ such that $[(X,Y),w]$ is strongly regular.
\end{Def}

Let $(\rH,\fq)$ be an infinitesimal unitary symmetric variety. We say that a regular semi-simple element $X\in \fq$ matches an regular semi-simple element $(X,Y)\in \mathfrak{X}(F)$ if they have matching invariant polynomials (see Lemmas \ref{Lem: inft linear invt theory} and \ref{Lem: unitary Lie}).

\begin{Def}[Smooth transfer] Set $\underline{\eta}=(\eta_0,\eta_2)$.
\begin{enumerate}
    \item (split-inert: $\underline{\eta} = (\eta,\eta)$) Fix $\tau_2\in \calv_n(E/F)$. We say that $\widetilde{\phi}\in C^\infty_c(\fX(F)\times F_{n})$ and a collection
    \[
   \underline{\phi}= (\phi_{\tau_1})\in \bigoplus_{\tau_1\in \calv_n(E/F)}C_c^\infty(\fq_{\tau_1,\tau_2}(F))
    \]
    are \emph{$(\eta,\eta)$-transfers} if for any regular semi-simple element $(X,Y)\in \fX(F)$
\begin{enumerate}
\item  if there exists $\tau_1$ and $Z\in \fq_{\tau_1,\tau_2}(F)$ matching $(X,Y)$, 
\begin{equation*}
  2^{|S_1^{ram}(X,Y)|}\Orb^{\rH',(\eta,\eta)}_{0}(\widetilde{\phi},(X,Y))= \SO^{\rH_{\tau_1,\tau_2}}(\phi_{\tau_1},Z),
\end{equation*}
where $S_1^{ram}(X,Y) = \{i\in S_1: \eta\circ\Nm_{F_i/F}\text{ is ramified}\}$;
\item if there does not exist $\tau_1$ and $Z\in \fq_{\tau_1,\tau_2}(F)$ matching $(X,Y)$, 
\begin{equation*}
  \Orb^{\rH',(\eta,\eta)}_{0}(\widetilde{\phi},(X,Y))= 0.
\end{equation*}
\end{enumerate}
    \item (inert-inert: $\underline{\eta} = (\eta,1)$) We say that $\widetilde{\phi}\in C^\infty_c(\fX(F)\times F_{n})$ and $f\in C_c^\infty(\fq_{ii}(F))$
    are \emph{smooth (ii)-transfers} if for any matching regular semi-simple elements $(X,Y)\in \fX(F)$  and $(x_1,x_2)\in \fq_{ii}(F)$, 
\begin{equation*}
  2^{|S_1^{ram}(X,Y)|}\Orb^{\rH',(\eta,1)}_{0}(\widetilde{\phi},(X,Y))= \SO^{\rH_{ii}}(f,(x_1,x_2)).
\end{equation*}
   \end{enumerate}
\end{Def}

\begin{Conj}\label{Conj: smooth transfer Lie}
    Set $\underline{\eta}=(\eta,\eta)$ or $(\eta,1)$. For any $\widetilde{\phi}\in C^\infty_c(\fX(F)\times F_{n})$ there exists a $\underline{\eta}$-transfer $\underline{f}$. Conversely, for any collection of functions $\underline{f}$ on the infinitesimal unitary variety, there exists a $\ul{\eta}$-transfer $\widetilde{\phi}$.
\end{Conj}

 One can show that Conjecture \ref{Conj: smooth transfer} follows from Conjecture \ref{Conj: smooth transfer Lie}. Unlike the case of  \cite{ZhangFourier}, this requires the use of Harish-Chandra descent due to the lack of flexibility with the Cayley transform (see Lemma \ref{Lem: Cayley map linear} below). We omit this argument, as it is not needed for Theorem \ref{Thm: weak smooth transfer}. 

\subsubsection{The endoscopic comparison}

Consider a regular semi-simple element $(x_1,x_2)\in \fq_{ii}$. Fix representatives $(x_{1,\al},x_{2,\al})\in \calo_{st}(x_1,x_2)$ where $\al\in H^1(F,\rH_x)$. We have the following analogue of Lemma \ref{Lem: building special kappa} with the same proof. 
  \begin{Lem}\label{Lem: building special kappa Lie} The map $(x_{1,\al},x_{2,\al})\mapsto \eta_v(\det(x_{1,\al}))$ satisfies that
  \[
  \al\in \ker^1(\rH_{(x_{1},x_{2})},\U(V_{x_{1}});F)\iff \eta(\det(x_{1,\al})) = \eta(\det(x_{1})). 
  \]
  \end{Lem} 
  In particular, we obtain a non-trivial character $\varepsilon\in H^1(F,\rH_{(x_{1},x_{2})})^D$ given by
  \begin{equation}\label{eqn: special kappa lie}
       \varepsilon(\al) = \eta(\det(x_{1,\al}))\eta(\det(x_{1}))^{-1}.
  \end{equation}

Fix $\tau_1,\tau_2\in \calv_n(E/F)$, and we assume $\tau_2=\tau_n$ is split. We say that $(x_1,x_2)\in \fq_{ii}^{rss}(F)$ matches $X\in \fq_{\tau_1,\tau_2}^{rss}(F)$ if they have matching invariant polynomials. The analogues of Lemmas \ref{Lem: surjective quotient si} with \ref{Lem: surjective quotient ii} hold in this setting, so that for any $(x_1,x_2)\in \fq_{ii}^{rss}(F)$ there exists a unique $\tau_1$ such that $(x_1,x_2)$ matches some $X\in \fq_{\tau_1,\tau_2}(F)$; we write $(x_1,x_2)\leftrightarrow_{\tau_1}X$. 
\begin{Conj}\label{Conj: ii simple endoscopy lie}
 Define the transfer factor $\De_{\varepsilon}: \fq_{ii}^{rss}(F)\to \{\pm1\}$ by
\[
       \De_{\varepsilon}(x_1,x_2):= \eta(\det(x_1)). 
\]
Assume $\tau_2=\tau_n\in \calv_n(E/F)$ is split. For any $f\in C^\infty_c(\fq_{ii}(F))$ and $(x_1,x_2)\in \fq_{ii}^{rss}(F)$, there exists a $\ul{f}=(f_{\tau_1})\in \bigoplus_{\tau_1\in \calv_n(E/F)}C_c^\infty(\fq_{\tau_1,\tau_2}(F))$ such that
     \[
     \De_{\varepsilon}(x_1,x_2) \Orb^{\rH,\varepsilon}(f,(x_1,x_2)) = \SO^{\rH_{\tau_1,\tau_2}}(f_{\tau_1},X),
     \]
     for all matching regular semi-simple $(x_1,x_2)\leftrightarrow_{\tau_1}X$. We say $f$ and $\underline{f}$ are $\varepsilon$-transfers.
\end{Conj}
As before, one can reduce Conjecture \ref{Conj: ii simple endoscopy} to Conjecture \ref{Conj: ii simple endoscopy lie} via descent techniques.

\subsection{Fundamental Lemmas on the Lie algebra}
We assume now that $E/F$ is an unramified extension of $p$-adic fields {of odd residual characteristic}.  

\begin{Rem}
    This restriction to odd residue characteristic is imposed as we rely on \cite[Lemma 3.9]{FLO} in the proof of Theorem \ref{Thm: main local result 2 var} (cf. Lemma \ref{Lem: unramified FLO}). 
\end{Rem}
 We begin with the split-inert case.  Let $V_n=V_{\tau_n}$ be the split Hermitian space and let $\Lam_n\subset V_n$ be a self-dual lattice. Consider $$\fq_{si}(\calo_F):=\End(\Lam_n)\subset \End(V_n)=\fq_{si}(F).$$  Set $\mathfrak{X}(\calo_F):=\fgl_n(\calo_{F})\times \fgl_n(\mathcal{O}_{F})$.
\begin{Thm}\label{Thm: fundamental lemma s-i Lie}
Fix $k\geq0$. The function $\bfun_{\vp^k\fX(\calo_F)}\otimes \bfun_{\calo_{F,n}}$ and the collection $((-1)^k\bfun_{\vp^k\fq_{si}(\calo_F)},0)$ 
    are $(\eta,\eta)$-transfers.
\end{Thm}

Next we consider the inert-inert case. Let $V_{2n}=L\oplus L^\ast$ be a polarization of the split Hermitian space of rank $2n$, and fix a self-dual lattice $\Lam_{2n} = \mathcal{L}\oplus \mathcal{L}^\ast$. This induces an open compact $\calo_E$-lattice
\[
\fq_{ii}(\calo_{F}):=\Herm(\mathcal{L}^\ast,\mathcal{L})\times \Herm(\mathcal{L},\mathcal{L}^\ast)\subset \Herm(L^\ast, L)\times \Herm(L,L^\ast)= \fq_{ii}(F).
\] 
\begin{Thm}\label{Thm: fundamental lemma stable Lie}\label{thm:lie i-i}
The functions $\bfun_{\vp^k\fX(\calo_F)}\otimes \bfun_{\calo_{F,n}}$ and $\bfun_{\vp^k\fq_{ii}(\calo_F)}$
    are $(\eta,1)$-transfers.
\quash{Suppose $[(Y_1,Y_2),w]\in \mathfrak{X}_{ext}^{sr}(F)$ and $(X_1,X_2)\in \fg^{rss}_1$ match. Then
\begin{equation}
    \SO^{\GL(V)}\left(\bfun_{\Herm(\Lam)\times \Herm(\Lam)},{(X_1,X_2)}\right)=\Orb^{\rH'(F),\underline{\eta}}\left(\bfun_{\mathfrak{X}(\calo_F)}\otimes\bfun_{V_n(\calo_F)},[(Y_1,Y_2),w]\right),
\end{equation}
where $\underline{\eta}=(\eta,1)$.  Moreover, when $(Y_1,Y_2)$ does not match an element $X$, the right-hand side vanishes.}
\end{Thm}
Finally, we have the analogue of the endoscopic comparison.
\begin{Thm}\label{Thm: fundamental lemma endoscopic Lie}
The functions $\bfun_{\vp^k\fq_{ii}(\calo_F)}$ and the collection $((-1)^k\bfun_{\vp^k\fq_{si}(\calo_F)},0)$ 
    are $\varepsilon$-transfers.
\end{Thm}

To show that these theorems imply the fundamental lemmas in \S \ref{Section: fundamental lemmas on variety}, we apply the tools from \cite{Lesliedescent}. This is contained in Appendix \ref{Sec: descent}, which proves the following theorem.
\begin{Prop}\label{Prop: descent fundamentals} We have the following implications.
\begin{enumerate}
\item  Theorem \ref{Thm: fundamental lemma s-i Lie} implies Theorem \ref{Thm: fundamental lemma si}. 
\item    Theorem \ref{Thm: fundamental lemma stable Lie} implies Theorem \ref{Thm: fundamental lemma ii}.
    \item Theorem \ref{Thm: fundamental lemma endoscopic Lie} implies Theorem \ref{Thm: fundamental lemma varepsilon}.
\end{enumerate}
\end{Prop}

 In all three Theorems above, we claim that the case of $k>0$ follows from the case $k=0$. For example, consider Theorem \ref{Thm: fundamental lemma s-i Lie} and assume that the case $k=0$ is known. For $k>0$, we wish to show that if $(X,Y)\in \mathfrak{X}(F)\times F_n$ and $Z\in \fq_{si}(F)$ are matching regular semi-simple elements, then 
 \[
 \Orb^{\rH',(\eta,\eta)}_{0}(\bfun_{\vp^k\fX(\calo_F)}\otimes \bfun_{\calo_{F,n}},(X,Y))= \begin{cases}
     (-1)^k\SO^{\rH_{\tau_1,\tau_2}}(\bfun_{\vp^k\fq_{si}(\calo_F)},Z)&: \tau_1=\tau_2 \text{ split,}\\
     \qquad 0&:\text{otherwise}.
 \end{cases}
 \]
Taking the transfer factor \eqref{eqn: linear Lie transfer factor} into account, we have
 \[
 \Orb^{\rH',(\eta,\eta)}_{0}(\bfun_{\vp^k\fX(\calo_F)}\otimes \bfun_{\calo_{F,n}},(X,Y))=(-1)^k\Orb^{\rH',(\eta,\eta)}_{0}(\bfun_{\fX(\calo_F)}\otimes \bfun_{\calo_{F,n}},(\vp^{-k}X,\vp^{-k}Y)),
 \]
 and similarly 
 \[\SO^{\rH_{\tau_1,\tau_2}}(\bfun_{\vp^k\fq_{si}(\calo_F)},Z)=\SO^{\rH_{\tau_1,\tau_2}}(\bfun_{\fq_{si}(\calo_F)},\vp^{-k}Z);
 \]
 moreover, it is clear that $(\vp^{-k}X,\vp^{-k}Y)$ and $\vp^{-k}Z$ match. Thus the claim follows from the $k=0$ case applied to these elements. A similar argument gives the reduction in Theorems \ref{Thm: fundamental lemma stable Lie} and \ref{Thm: fundamental lemma endoscopic Lie}. We therefore only consider the case of $k=0$ in the sequel.

\section{Mirabolic orbital integrals and a transfer result}\label{Section: mirabolic integrals}

In this section, we recall certain orbital integrals studied in \cite{Xiaothesis} which arise in a comparison between the pairs 
\begin{equation}\label{eqn: Jacquets comparison}
   (\GL_n,\GL_n\times \mathbb{G}_{a,n})\longleftrightarrow (\U(V_\tau),\Herm_{\tau}^\circ).
\end{equation}
We then recall the main result of the existence of smooth transfer in this context (see Theorem \ref{Thm: Xiao's transfer} below), which is one of the main results of \cite{Xiaothesis}. 

In the \S \ref{s:Hir}, we formulate two results comparing orbital integrals for certain unramified functions in this setting (Theorems \ref{Thm: main local result 1 var} and \ref{Thm: main local result 2 var}) which we prove imply the fundamental lemmas in Theorems  \ref{Thm: fundamental lemma s-i Lie}, \ref{Thm: fundamental lemma stable Lie}, and \ref{Thm: fundamental lemma endoscopic Lie} (see \S \ref{Section: reduce to one statement}). These results are proved in Part \ref{Part: trace proof global} via global methods, which also rely on Theorem \ref{Thm: Xiao's transfer}. Finally, we use this transfer to prove Theorems \ref{Thm: weak smooth transfer} and \ref{Thm: weak endoscopic transfer} in Appendix \ref{Section: proof weak transfer}.

\begin{Rem}
In the split-inert case, it is natural to see why this comparison is relevant as we have the diagram
\begin{equation}
\begin{tikzcd}
    \fgl_n\times \fgl_n\times \mathbb{G}_{a,n}\ar[r]\ar[d,"R"]& \Res_{E/F}\fgl_n\ar[d,"R"]\ar[l]\\
    \fgl_n\times \mathbb{G}_{a,n}\ar[r]&\Herm_\tau\ar[l],
\end{tikzcd}    
\end{equation}
where the top row is the Lie algebra split-inert comparison, the vertical arrows are contraction maps (cf. \S \ref{Section: contractions}), and the regular semi-simple loci of the top row lands in the invertible locus represented in the comparison \eqref{eqn: Jacquets comparison}. A similar (but more complicated) reduction occurs in the inert-inert case. 
\end{Rem}

 \subsection{Orbital integrals}\label{Section: prelim for Jacquet}
 For any $\tau\in \calv_n(E/F)$, we consider the action of $\U(V_\tau)$ on $\Herm_{\tau}^\circ$.  For $f\in C_c^\infty(\Herm_{\tau}^\circ(F))$ and $y\in \Herm_{\tau}^{\circ, rss}(F)$, the notion of stable orbital integral $\SO^{\U(V_\tau)}(f,y)$ is clear.

On the linear side, consider the action of $\GL_n$ acts on $\GL_n\times \bbg_{a,n}$ given by $g\cdot(x,w)=(gxg^{-1},vg^{-1})$. It is easy to show that an element $(x,w)$ is {regular semi-simple} for the action of $\GL_n$  in the sense of Section \ref{Section: orbital integrals conventions} if $x$ is regular semi-simple in the usual sense and $\det([w|wx|\cdots|wx^{n-1}])\neq0$; we call such pairs \emph{strongly regular} to harmonize with the previous settings. Lemma \ref{Lem: generic trivial stabilizer} implies that the stabilizer of a regular semi-simple pair $(x,w)$ is trivial.

Let $T_x$ be the centralizer of $x$ in $\GL_n$; note that if $F[x]\simeq \prod_i F_i$ for finite field extensions $F_i/F$, then $T_x\simeq \prod_i\Res_{F_i/F}\BG_{m}$. As in \S \ref{Section: Lvalue measure}, we consider the local $L$-factor $L(s,T_x,\eta)$.
\begin{Def} Let $\eta=\eta_{E/F}$ denote the quadratic character associated to the extension $E/F$. For strongly regular $(x,w)\in \GL_n(F)\times F_n$,
    define the transfer factor 
\begin{equation}\label{eqn: correct xiao transfer factor}
    \De(x,w):= \eta(\det(x))^n\eta(\det([w|wx|\cdots|wx^{n-1}])).
\end{equation}
\end{Def}
\begin{Rem}\label{Rem: transfer factor Xiao}
   The factor $\eta(\det(x))^n$ was not present in \cite{Xiaothesis}, but is necessary; see Remark \ref{Rem: correct transfer factor}. 
\end{Rem}
For each $\wt{f} \in C_c^\infty(\GL_n(F)\times F_n)$ and strongly regular element $(x,w)$, let
\begin{equation}\label{eqn: Jacuqet OI with s}
    \Orb^{\GL_n(F),\eta,\natural}_s(\wt{f}, (x,w)):=\frac{\De(x,w)}{L(s,T_x,\eta)}\int_{\GL_n(F)}\wt{f}(g^{-1}xg, wg)|g|^s\eta(g)dg.
\end{equation}
By the same argument as in Lemma/Definition \ref{LemDef: linear OI} (cf. \cite[Proposition 5.3]{Xiaothesis}), this ratio
is holomorphic at $s=0$ and we define
\begin{equation}\label{eqn: linear OI}
    \Orb^{\GL_n(F),\eta}(\wt{f},x):=\Orb^{\GL_n(F),\eta,\natural}_0(\wt{f}, (x,w))
\end{equation}
As in Lemma/Definition \ref{LemDef: linear OI}, this is independent of $w$, as our notation indicates.


\subsubsection{Matching of orbits}
We have the obvious matching of regular semi-simple orbits.
\quash{\begin{LemDef}\label{Lem: matching split}
 We say that $y\in\Herm_{n}^{\circ}(F)$ and $z\in \GL_n(F)$ \emph{match} if they possess the same characteristic polynomial over $F$.

This induces a bijection from regular semi-simple stable $U(V_n)$-orbits on $\Herm_{n}^{\circ}(F)$ to regular semi-simple $\GL_n$-orbits in $\GL_n(F)$. 
\end{LemDef}}
\begin{LemDef}\label{Lem: matching def}
Let $\tau\in \calv_n(E/F)$ and fix regular semi-simple elements $y\in \Herm_{\tau}^{\circ}(F)$ and  $z\in\GL_n(F)$. We say that $y$ and $z$ \emph{match} if they possess the same characteristic polynomial over $F$.

For any $\tau\in \calv_n(E/F)$, this induces an injection from regular semi-simple stable $U(V_\tau)$-orbits on $\Herm_{\tau}^{\circ}(F)$ to regular semi-simple $\GL_n$-orbits in $\GL_n(F)$. If $\tau=\tau_n$ is split, this is a bijection.
\end{LemDef}

\begin{proof}
    Both quotients $\Herm_{n}^{\circ}\sslash\U_n$ and $  \GL_n\sslash\GL_n$ may be identified with the coefficients of $F$-rational characteristic polynomials, giving the matching of orbits, and inducing an injection 
    \[
    \mathrm{Im}(\Herm_{\tau}^{\circ, rss}(F))\subset (\GL_n\sslash \GL_n)(F).
    \] For the claim of bijectivity, recall that when $\tau_n$ is split, $\U(V_n)$ is quasi-split. The map $x \mapsto \xi x$ embeds $\Herm_{n}^{\circ}$ as a Zariski-open subscheme of the Lie algebra $\Lie(\U(V_n))$. As previously noted, the categorical quotient for the $\U(V_n)$-action on its Lie algebra is surjective on $F$-points.
\end{proof}

For $\tau\in \calv_n(E/F)$, if $y\in \Herm_{\tau}^{\circ}(F)$ and $z\in \GL_n(F)$ match, we write $y\leftrightarrow z$.

\subsection{Existence of smooth transfer}\label{Section: orbital ints xiao} Let $E/F$ be a quadratic extension of $p$-adic fields. We now establish the existence of smooth transfer for the comparison \eqref{eqn: Jacquets comparison}. This follows directly from the main results of \cite{Xiaothesis}. 
\begin{Thm}\label{Thm: Xiao's transfer} Fix $\tau\in \calv_n(E/F)$.
	\begin{enumerate}
		\item For any $f \in C_c^\infty(\Herm_{\tau}^{\circ}(F))$, there exists $\wt{f} \in C_c^\infty(\GL_n(F) \times F_n)$ such that $$\SO^{\U(V_\tau)}(f,y)=2^{|S_1^{ram}(z)|}\Orb^{\GL_n(F),\eta}(\wt{f},z)$$ for any matching regular semi-simple orbits $y\leftrightarrow z$ and $\Orb^{\GL_n(F),\eta}(\wt{f},z)=0$ if there is no $y\in \Herm_{\tau}^{\circ}(F)$ matching $z$.
		\item Conversely, for any $\wt{f} \in C_c^\infty(\GL_n(F) \times F_n)$ such that $\Orb^{\GL_n(F),\eta}(\wt{f},z)=0$ if there is no $y\in \Herm_{\tau}^{\circ}(F)$ matching $z\in\GL_n(F)$, there exists $f \in C_c^\infty(\Herm_{\tau}^{\circ}(F))$ such that $$\SO^{\U(V_\tau)}(f,y)=2^{|S_1^{ram}(z)|}\Orb^{\GL_n(F),\eta}(\wt{f},z)$$ for matching regular semi-simple orbits $y\leftrightarrow z$.
	\end{enumerate}
\end{Thm}
\begin{proof} When $\tau=\tau_n$ is split, the theorem is essentially contained in \cite[Theorem 5.1]{Xiaothesis}, once one applies \cite[Proposition 5.3]{Xiaothesis} to identify the orbital integrals used in that theorem with those in Definition \ref{eqn: linear OI}. An important difference with \emph{ibid.} and the current claim is the difference in transfer factor \eqref{eqn: correct xiao transfer factor}. This does not obstruct the existence of smooth transfer, though our definitions differ from those in \cite{Xiaothesis} when $E/F$ is ramified.
 
For the purpose of clarity (especially for Theorem \ref{Thm: Xiao FL for d} below), we recall the construction. For $f\in C_c^\infty(\GL_n(F)\times F_n)$, the second author introduces an orbital integral $\underline{\Orb}^\eta(\wt{f},z)$ (simply denoted $\Orb(f,z)$ in \cite[Section 5.2.1]{Xiaothesis}) in terms of certain nilpotent orbital integrals in the context of the Jacquet--Rallis transfer. Proposition 5.3 of \emph{ibid.} shows that
\[
\Orb_0^\eta(\wt{f},z)=2^{|S_1^\ast|}\underline{\Orb}^\eta(\wt{f},z),
\]
where $S_1$ is defined as in \S \ref{Section: Lvalue measure}, and $S_1^\ast=\{i\in S_1: \eta_i \text{ is unramified on }F_i\}$.

Choosing a function $F\in C_c^\infty(\GL_n(F)\times F^n\times  F_n)$ such that $\wt{f}(z,v) = F(z,0,v)$. With respect to the inclusion $\GL_n(F)\subset \fgl_n(F)$, we may view $F$ as a function on $\fgl_n(F)\times F^n\times  F_n$. Applying smooth transfer of Jacquet--Rallis \cite{ZhangFourier}, we obtain functions $\{F_0,F_1\}$ with $F_i\in C_c^\infty(\Herm(V_i)\times V_i)$, where $\{V_0,V_1\}$ are representatives of the two isometry classes of Hermitian forms for $E/F$ and $V_0=V_n$ is the split form. Set $f_i(y) :=F_i(y,0)$.

Jacquet--Langlands transfer (cf. \cite[Theorem 1.5]{Waldstransfert}) now implies that there exists $JL(f_1)\in C_c^\infty(\Lie(U(V_n)))$ such that
 \[
 \SO^{\U(V_n)}(JL(f_1),y_0) = \SO^{\U(V_\tau)}(f_1,y)
 \] for any regular semi-simple $y\in \Herm_{\tau}^{\circ}(F)$ and $y_0\in \Herm_{n}^{\circ}(F)$ with the same characteristic polynomial and such that $\SO(f_0,y_0)=0$ if there is no such $y\in \Herm_{\tau}^{\circ}(F)$. With this, we define $f_n':= f_0-JL(f_1)$. Finally, $f_n\in C_n^\infty(\Herm_{n}^{\circ}(F))$ is defined by scaling $f'_n$ by an indicator function that vanishes at the determinant locus of $\Herm_n(F)$. Then \cite[Theorem 5.1 and Proposition 5.3]{Xiaothesis} shows that --taking the change of our transfer factor into account (cf. Remark \ref{Rem: transfer factor Xiao})-- the function $f:=\eta((-1)^n)\eta^n\cdot f_n$ is a transfer of $\wt{f}$ in that
 \[
 \SO^{\U(V_n)}({f},y)=2^{|S_1|}\underline{\Orb}^\eta(\wt{f},z)
 \]where $y\leftrightarrow z$ match. Here, 
$
(\eta^n\cdot f_n)(x) = \eta^n(\det(x)) f_n(x).
$

We conclude that for matching regular semi-simple elements  $y\leftrightarrow z$,
\[
 \SO^{\U(V_n)}({f},y)=2^{|S_1|-|S_1^\ast|}\Orb_0^\eta(\wt{f},z)=2^{|S_1^{ram}(z)|}\Orb_0^\eta(\wt{f},z),
\]
proving the claim.
 
 When $\tau$ is non-split, we observe the open inclusion $\xi\cdot\Herm_{\tau}^{\circ}\subset \Lie(\U(V_\tau))$ allows us to recognize that $\SO^{\U(V_\tau)}(f,y)$ as a stable orbital integral for the adjoint action of $\U(V_\tau)$ on its Lie algebra, where we view $f\in C_c^\infty(\Lie(U(V_\tau)))$ via extension-by-zero. Another application of Jacquet--Langlands transfer gives a function $f_0\in C_c^\infty(\Lie(U(V_n)))$ matching $f$. Combining this with the preceding case completes the argument.
\end{proof}

\section{The Hironaka transform and a fundamental lemma}\label{s:Hir}
We now assume $F$ has odd residue characteristic and $E/F$ is unramified. In this section, we recall a beautiful result of Hironaka regarding a certain module for the spherical Hecke algebra of $\GL_n(E)$, and relate the orbital integrals from the previous section to this structure. We then formulate a result (Theorem \ref{Thm: main local result 2 var}) which implies Theorems \ref{Thm: fundamental lemma s-i Lie}, \ref{Thm: fundamental lemma stable Lie}, and \ref{Thm: fundamental lemma endoscopic Lie}. 

We set $K=\GL_n(\calo_F)$ and $K_E=\GL_n(\calo_E)$.

\quash{\subsection{Preliminaries}
 We begin by fixing some notation. 
Denote
$$
\mathbb{P}_n=\{(\lambda_1,...,\lambda_n):\lambda_1\geq \lambda_2\geq \cdots\geq \lambda_n.\}
$$
and 
$
\mathrm{P}_{n,\geq 0}=\{(\lambda_1,...,\lambda_n)\in \BZ^n_{+}: \lambda_n\geq 0.\}
$ Define
$$
|\lambda|=\sum \lambda_i,\quad n(\lambda)=\sum_{i=1}^n (i-1)\lambda_i.
$$
{We denote
$$
w_n(t)=\prod_{i=1}^n(1-t^i),
$$
and a generalization $w_\lambda^{(n)}(t)$ such that  $w_{0}^{(n)}(t)=w_n(t)$. 
}
}
\subsection{The Satake isomorphism} We first recall the Satake isomorphism for $\GL_n(F)$. We let $q$ denote the cardinality of the residue field of $F$, and fix a uniformizer $\varpi\in \calo_F$.

Let $\calh_{K}(\GL_n(F))$ denote the spherical Hecke algebra of $\GL_n(F)$. For any $(s_1,\ldots,s_n)\in \cc^n$, we recall the Satake transform
\begin{equation}\label{eqn: satake transform}
Sat(f)(s_1,\ldots,s_n) = \int_{\GL_n(F)}f(g)\prod_{i=1}^n|a_i|_F^{s_i-\frac{1}{2}(n+1-2i)}dg,
\end{equation}
where $g=nak$ is the Iwasawa decomposition of $g$, $dg$ is our chosen measure from \S \ref{measures}, and $a=\diag(a_1,\ldots,a_n)\in T_n(F).$
This gives an algebra isomorphism 
\begin{equation*}
    Sat:\calh_{K_{n}}(\GL_n(F))\iso \cc[q^{\pm s_1},\ldots,q^{\pm s_n}]^{S_n}.
\end{equation*}
 Setting $t_i=q^{-s_i}$, $ t=\diag(t_1,\ldots,t_n)\in \hat{T}_n\subset \GL_n(\cc)$ is an element of the diagonal split torus in the dual group of $\GL_n(F)$, and $$\cc[q^{\pm s_1},\ldots,q^{\pm s_n}]\cong\cc[\hat{T}_n]\cong \cc[Z_1^{\pm1},\ldots,Z_n^{\pm1}],$$ where 
\begin{equation}\label{eqn: weird normalization}
Z_i(t)=t_i.
\end{equation}

Let $\lam$ be a dominant coweight of $T_n(F)\subset \GL_n(F)$ with respect to $B_n =T_nN_n$ and let $\varpi^\lam$ denote the image of $\varpi$ under $\lam$. Recall (see \cite[pg. 299]{macdonald}) the formula 
\begin{equation}\label{eqn: Satake basis}
Sat(\bfun_{K_{n}\vp^\lam K_{n}}) = q^{\la\lam,\rho\ra}P_\lam(Z_1,\ldots,Z_n;q^{-1}).
\end{equation}
 Here,
\[
P_\lam(x_1,\ldots,x_n;t)= \frac{(1-t)^n}{w_\lambda^{(n)}(t)} \sum_{\sigma\in S_n}\sig\left(x_1^{\lam_1}\cdots x_n^{\lam_n}\prod_{\lam_i>\lam_j}\frac{x_i-tx_j}{x_i-x_j}\right)
\]is the $\lam$-th Hall--Littlewood polynomial \cite[pg. 208]{macdonald}. In particular, we have $
P_{0}(X,t)=1.$ It is well known that these polynomials give a $\zz$-basis for $\zz[t][x_1^{\pm1},\ldots,x_n^{\pm1}]^{S_n}$. 
Note we have an identity \cite[pg.225]{macdonald}
\begin{equation}\label{eqn: geometric series macdonald}
    \sum_{\lambda\in \mathrm{P}_{n,\geq 0}} t^{n(\lambda)} P_\lambda(\ul{X};t)=\prod_{1\leq i\leq n}(1-X_i)^{-1}=\sum_{\lam\geq0}X^\lam,
\end{equation}
where $$\mathrm{P}_n=\{(\lambda_1,...,\lambda_n):\lambda_1\geq \lambda_2\geq \cdots\geq \lambda_n\},$$
$\mathrm{P}_{n,\geq 0} =\{\lam\in \mathrm{P}_n\;:\;\lam_n\geq0\}$, and $X^\lam = \prod_i X_i^{\lam_i}$.

\subsection{The spherical functions for $\Herm_{n}^{\circ}(F)$}\label{Section: spherical hecke}
Let $K_E=\GL_n(\calo_E)$. It follows from \cite{jacobowitz1962hermitian} that the $K_E$-orbits on $\Herm_{n}^{\circ}(F)$ are
\begin{equation}\label{eqn: orbits everywhere}
\Herm_{n}^{\circ}(F)=\bigsqcup_{\lam\in\mathrm{P}_{n}}K_E\ast(\vp^\lam\tau_n).
\end{equation}
We have an action of $\GL_n(E)$ on $C_c^\infty(\Herm_{n}^{\circ}(F))$ given by 
\[
g\ast f(y) = f(g^{-1}\ast y), \quad \text{for any}\quad f\in C_c^\infty(\Herm_{n}^{\circ}(F)), \: g\in \GL_n(E)\;\text{ and }\; y\in \Herm_{n}^{\circ}(F).
\] Set $\calh_{K_E}(\Herm_{n}^{\circ}(F)):=C^\infty_c(\Herm_{n}^{\circ}(F))^{K_E}$ to be the vector space of $K_E$-invariant functions.  Set $\bfun_\lam$ to be the indicator function of the orbit $K_E\ast (\vp^\lam\tau_n).$ The above orbit decomposition implies that $\{\bfun_\lam\}_{\lam\in \mathrm{P}_n}$ is a $\cc$-basis for $\calh_{K_E}(\Herm_{n}^{\circ}(F))$. Note that with this notation 
 \[
 \bfun_0=\bfun_{\Herm_{n}^{\circ}(\calo_F)}.
 \]
 The spherical Hecke algebra $\calh_{K_E}(\GL_n(E))$ acts on this space by
\[
f\ast  \phi(y)= \int_Gf(g^{-1})\phi(g\ast y)dg.
\]
The induced $\calh_{K_E}(\GL_n(E))$-module structure of $\calh_{K_E}(\Herm_{n}^{\circ}(F))$ is well understood thanks to the work of Hironaka. More precisely, Hironaka computes in \cite{hironaka1999spherical} the \emph{normalized spherical function} $\Omega_z\in C^\infty(\Herm_{n}^{\circ}(F))^{K_E}$ (see \cite[Section 6]{offenjacquet} for explicit formulas) and uses this as the kernel for the following integral transform.
\quash{satisfying
\begin{equation}\label{eqn: normalized spherical}
   \Omega_z(\varpi^{\lambda})=(-1)^{n(\lambda)+|\lambda|} q^{n(\lambda)-(n-1)|\lambda|/2} \frac{w^{(n)}_\lambda(-q^{-1})}{w_n(-q^{-1})}\times P_\lambda(q^z,-q^{-1}).
   \end{equation}
}
\begin{Thm}\cite{hironaka1999spherical} The integral transform given by
\[
\mathcal{SF}(\phi)(z) = \int_{\Herm_{n}^{\circ}(F)}\phi(x)\Omega_z(x^{-1})dx
\]
with $\phi\in\calh_{K_{E}}(\Herm_{n}^{\circ}(F))$ induces an isomorphism of $\calh_{K_E}(\GL_n(E))$-modules
\[
\calh_{K_{E}}(\Herm_{n}^{\circ}(F))\cong \cc[Z_1^{\pm1},\ldots,Z_n^{\pm1}]^{S_n}.
\]
In particular, $\calh_{K_E}(\Herm_{n}^{\circ}(F))$ is a free $\calh_{K_E}(\GL_n(E))$-module of rank $2^n$.

\end{Thm}

Combining the spherical Fourier transform with the Satake transform on $\GL_n(F)$, we define an isomorphism of $\calh_{K_E}(\GL_n(E))$-modules 
\[
\Hir:=\mathrm{Sat}^{-1}\circ \mathcal{SF}:\calh_{K_{E}}(\Herm_{n}^{\circ}(F))\lra \calh_{K}(\GL_n(F)),
\]
which we call the \emph{Hironaka transform}.

 The next theorem is one of our main local result on orbital integrals. It constitutes a ``stable'' fundamental lemma for the Hecke module $\calh_{K_E}(\Herm_{n}^{\circ}(F))$; the endoscopic variants for this module were partially established in \cite{LeslieUFJFL}.
\begin{Thm}\label{Thm: main local result 1 var} For any $\phi \in \calh_{K_E}(\Herm_{n}^{\circ}(F))$, the function $\Hir(\phi)\otimes\bfun_{\calo_{F,n}}\in C^\infty_c(\GL_n(F)\times F_n)$ gives a transfer in the sense of Theorem \ref{Thm: Xiao's transfer}. That is, for matching regular semi-simple elements $x\in \Herm_{n}^{\circ}(F)$ and $z\in \GL_n(F)$, we have
	\[
	\SO^{\U(V_n)}(\phi,x)=\Orb^{\GL_n(F),\eta}( \Hir(\phi)\otimes \bfun_{\calo_{F,n}},z).
	\]
\end{Thm}

As we show in \S \ref{Section: implies si}, this statement suffices to prove Theorem \ref{Thm: fundamental lemma s-i Lie}. For its proof and the proofs of the other fundamental lemmas, we need an enhancement involving an algebra structure on $\calh_{K_E}(\Herm_{n}^{\circ}(F))$.

\subsection{A multiplicative fundamental lemma for a Hecke module} \label{Sec: multiplicative fundamental}
The Hironaka transform induces a product structure on $\calh_{K_E}(\Herm_{n}^{\circ}(F))$ as follows: for $\phi, \phi'\in \calh_{K_E}(\Herm_{n}^{\circ}(F))$, we define $\phi\ast\phi'\in \calh_{K_E}(\Herm_{n}^{\circ}(F))$ to be the unique function satisfying
\[
\Hir(\phi\ast\phi') = \Hir(\phi)\ast\Hir(\phi'),
\]
where the right-hand side is the convolution product in $\calh_K(\GL_n(F))$. We call this the Hironaka product of $\phi_1$ and $\phi_2$. We establish an orbital integral characterization for the product structure on the Hecke module $\calh_{K_E}(\Herm_{n}^{\circ}(F))$, extending the product on the Hecke algebra $\sH_{K_E}(\GL_{n}(E))$.

For this, recall from the inert-inert case (cf. Lemma \ref{Lem: unitary Lie}) the action
\[
g\cdot(x_1,x_2) =(gx_1 g^\ast, (g^\ast)^{-1}x_2g^{-1}),
\]
for $g\in \GL_n(E)$ and for $(x_1,x_2)\in \Herm_{n}^{\circ}(F)\times \Herm_{n}^{\circ}(F) = \fq^{invt}_{ii}$. For $\Phi\in C_c^\infty(\Herm_{n}^{\circ}\times \Herm_{n}^{\circ})$ and $(x_1,x_2)\in \Herm_{n}^{\circ}(F)\times \Herm_{n}^{\circ}(F)$ regular semi-simple, we consider the  stable orbital integral $\SO^{\GL_n(E)}( \Phi,(x_1,x_2))$ as before.
\begin{Def}\label{Def: lie algebra matching redo}
    We say that a regular semi-simple element $(x_1,x_2)\in \Herm_{n}^{\circ}(F)\times \Herm_{n}^{\circ}(F)$ \emph{matches} $z\in \GL_n(F)$ if $x_1x_2$ has the same characteristic polynomial as $z\in \GL_n(F)$. 
\end{Def}
Our main local result is the following matching of orbital integrals.
\begin{Thm}\label{Thm: main local result 2 var} Let $\phi, \phi'\in \calh_{K_E}(\Herm_{n}^{\circ}(F))$. Let $(x_1,x_2)\in\Herm_{n}^{\circ}(F)\times \Herm_{n}^{\circ}(F)$ be regular semi-simple and assume $z\in \GL_n(F)$ matches $(x_1,x_2)$. Then 
	\[
 \SO^{\GL_n(E)}(\phi\otimes\phi',(x_1,x_2))=\Orb^{\GL_n(F),\eta}( \Hir(\phi\ast\phi')\otimes {\bfun}_{\calo_{F,n}},z).
	\]
\end{Thm}

	 The following is immediate from combining the statements of Theorems \ref{Thm: main local result 1 var} and \ref{Thm: main local result 2 var}. We say a $\GL_n(E)$-regular semi-simple $(y_1,y_2)\in \Herm_{n}^{\circ}(F)\times \Herm_{n}^{\circ}(F)$ matches a $\U(V_n)$-regular semi-simple $y\in \Herm_{n}^{\circ}(F)$ if $y_1y_2$ has the same characteristic polynomial as $y$. 
\begin{Cor}\label{Cor: product orbital ints} Let $\phi, \phi'\in \calh_{K_E}(\Herm_{n}^{\circ}(F))$. Assume that a regular semi-simple $(y_1,y_2)\in \Herm_{n}^{\circ}(F)\times \Herm_{n}^{\circ}(F)$ matches $y\in \Herm_{n}^{\circ}(F)$.   Then 
	\[
 \SO^{\GL_n(E)}(\phi\otimes\phi',(y_1,y_2))=\SO^{\U(V_n)}(\phi\ast\phi',y).
	\]
\end{Cor}
We postpone the proof of Theorem \ref{Thm: main local result 2 var} to Part \ref{Part: trace proof global}.

\section{Reduction to Theorem \ref{Thm: main local result 2 var}}\label{Section: reduce to one statement}
We now show how Theorems \ref{Thm: fundamental lemma s-i Lie}, \ref{Thm: fundamental lemma stable Lie}, and \ref{Thm: fundamental lemma endoscopic Lie} follow from Theorem \ref{Thm: main local result 2 var}.
\quash{
\begin{Thm}
	\label{thm uFL gp}The unit
	${\bf 1}_{G(O_{F})}$ matches ${\bf 1}_{(G'\times V_0)(O_{F})}$.
\end{Thm}

\begin{Thm}
	\label{thm uFL Lie}
	The function ${\bf 1}_{ (\Herm_n\times \Herm_n)(O_{F})}$ matches ${\bf 1}_{(\fkg\fkl_n\times \fkg\fkl_n\times V)(O_{F})}.$
\end{Thm}

\begin{lem} Then the
	Lie algebra version implies the group version.
\end{lem}

\begin{proof}
	{\bf TBA}
\end{proof}
}
\subsection{Theorem \ref{Thm: main local result 2 var} $\imp$ Theorem \ref{Thm: main local result 1 var}}
The first reduction is a simple argument in Galois cohomology. We begin with the following general lemma relating the orbital integrals in \S \ref{Sec: multiplicative fundamental} to those in Theorem \ref{Thm: Xiao's transfer}.

 Let $f=f_1\otimes f_2\in C^\infty_c(\Herm_{n}^{\circ}(F)\times \Herm_{n}^{\circ}(F))$ and assume that $\supp(f_{1})\subset \Herm_{n}^{\circ}(F)_\tau$ for some $\tau\in \mathcal{V}_n(E/F)$. Then there exists $f_\tau\in C^\infty_c(\GL_n(E))$ such that 
   \[
   f_{1}(y) =r_{!}^\tau(f_\tau)(y)=\int_{U(V_\tau)}f_\tau(g_yu)du,
   \]
   where where $r^\tau(g) = g\tau g^\ast$ and $y = r^\tau(g_y)$.

 \begin{Lem} For any regular semi-simple element $(x_1,x_2)\in \Herm_{n}^{\circ}(F)_\tau\times \Herm_{n}^{\circ}(F)$, we have the identity
   \begin{equation}\label{eqn: unwind to Hn}
   \SO^{\GL_n(E)}(f_{1}\otimes f_2, (x_1,x_2)) = \SO^{\U(V_{\tau^{-1}})}(f^\ast_\tau\ast f_2, x\tau),
   \end{equation}
   where $x_1 = g_1 \tau g_1^\ast$ and $x= g_1^\ast x_2 g_1\in \Herm_{n}^{\circ}(F)$, $f^\ast(g) = f(g^\ast)$, and where
   \[
   (f\ast f_2)(y) = \int_{\GL_n(E)}f(g^{-1})f_2(gyg^\ast)dg.
   \]
   Note that $x\tau\in \Herm_{\tau^{-1}}^{\circ}(F)$, so that the right-hand side is a stable orbital integral arising in the comparison of Theorem \ref{Thm: Xiao's transfer}.
   \end{Lem}
   \begin{proof}
First note that 
   \[
    \SO^{\GL_n(E)}(f_{1}\otimes f_2, (x_1,x_2)) = \sum_{\al\in H^1(F, T_{x_1,x_2})}\Orb^{\GL_n(E)}(f_{1}\otimes f_2, (x_1,x_2)_\al),
   \]
   where $T_{x_1,x_2} = \{t\in \GL_n(E): (tx_1 t^\ast, (t^\ast)^{-1}x_2 t^{-1}) = (x_1,x_2)\}$. Clearly, $$T_{x_1,x_2}\subset \U(V_{x_1}) = g_1\U(V_\tau)g^{-1}_{1}.$$ We have $T_{x}= g_1^{-1}T_{x_1,x_2}g_1\subset \U(V_\tau)$, and the right-hand side of \eqref{eqn: unwind to Hn} is 
\[
\SO^{\U(V_{\tau^{-1}})}(f^\ast_\tau\ast f_2, x) = \sum_{\al\in \ker^1(T_{x},\U(V_\tau);F)}\Orb^{\U(V_{\tau^{-1}})}(f^\ast_\tau\ast f_2, x_\al).
\]

By Lemma \ref{Lem: building special kappa Lie}, under the identification $H^1(F,T_{x_1,x_2}) \simeq H^1(F,T_{x})$, the representative $(x_1,x_2)_\al\in \Herm_{n}^{\circ}(F)_\tau\times \Herm_{n}^{\circ}(F)$ if and only if $\al\in \ker^1(T_{x},\U(V_\tau);F)$. In particular, $$\Orb^{\GL_n(E)}(f_{1}\otimes f_2, (x_1,x_2)_\al)=0$$ when $\al\notin \ker^1(T_{x},\U(V_\tau);F)$.

To complete the proof, it remains to relate the two orbital integrals when $\al\in \ker^1(T_{x},\U(V_\tau);F)$. For simplicity, we verify the identity on orbital integrals for $\al\in \ker^1(T_{x},\U(V_\tau);F)$ for $\al = 1$; the general claim follows identically for $(x_1,x_2)_\al = (y_1,y_2)$ with appropriate notational changes.

By definition,
   \begin{align*}
  \Orb^{\GL_n(E)}(f_{1}\otimes f_2, (x_1,x_2))&= \int_{T_{x_1,x_2}\backslash\GL_n(E)}\left(\int_{U(V_\tau)}f_\tau(g^{-1}g_1u)du\right)f_2(g^\ast x_2 g)dg,
   \end{align*}
    Now exchange the order of integration and affect a change of variables to obtain
     \begin{align*}
   \int_{T_{x}\backslash U(V_\tau)}&\int_{\GL_n(E)}f_\tau(g^{-1})f_2(g^\ast u^{\ast}x u g)dgdu\\
   &= \int_{T_{x}\backslash U(V_\tau)}[f^\ast_\tau\ast f_2]( u^{\ast}x u )du
   \end{align*} We now make a change of variables $u\mapsto (u^\ast)^{-1}$, use \eqref{eqn: twist to different form} to identify the integral with one over $U(V_{\tau^{-1}})$, and apply the intertwining map $\Herm_{n}^{\circ}\to \Herm_{\tau^{-1}}^{\circ}$ given by $ x\mapsto x\tau$ to obtain
   \[
   \int_{T_{x}\backslash U(V_{\tau})}[f^\ast_\tau\ast f_2]( u^{\ast}x u )du = \Orb^{\U(V_{\tau^{-1}})}(f^\ast_\tau\ast f_2, x\tau)
   \]This proves \eqref{eqn: unwind to Hn}.
   \end{proof}
\begin{Prop}
Theorem \ref{Thm: main local result 2 var} implies Theorem \ref{Thm: main local result 1 var}. 
\end{Prop}
\begin{proof}
We assume that for any $\phi_1,\phi_2\in \calh_{K_E}(\Herm_{n}^{\circ}(F))$ and suppose that for any regular semi-simple $z\in \GL_n(F)$ and matching pair $(x_1,x_2)\in\Herm_{n}^{\circ}(F)\times \Herm_{n}^{\circ}(F)$, we have
	\[
 \SO^{\GL_n(E)}(\phi_1\otimes\phi_2,(x_1,x_2))=\Orb^{\GL_n(F),\eta}(\Hir(\phi_1\ast\phi_2)\otimes {\bf 1}_{\calo_{F,n}},z).
	\] 
    
    For any regular semi-simple $z\in \GL_n(F)$ and $x\in\Herm_{n}^{\circ}(F)$ such that $z$ is conjugate in $\GL_n(E)$ to $x$, the claim is that
 	\[
 \SO^{\U_n(F)}(\phi,x)=\Orb^{\GL_n(F)}(\Hir(\phi)\otimes {\bf 1}_{\calo_{F,n}},z).
	\] 
   Note that $(I_n,x)\in \Herm_{n}^{\circ}(F)\times \Herm_{n}^{\circ}(F)$ is $\GL_n(E)$-regular semi-simple. It thus suffices to take $\phi_1= \bfun_{0} = \bfun_{\Herm_{n}^{\circ}(F)}$ and show that 
    \[
    \SO^{\GL_n(E)}(\bfun_{0}\otimes \phi,(I_n,x))=  \SO^{\U_n(F)}(\phi,x).
    \]
    This is a special case of \eqref{eqn: unwind to Hn}.  Indeed, recall that $V_n = V_{I_n}$, so that $h\mapsto (h^{\ast})^{-1}$ reduces to the identity in this setting. Noting this, we need the simple computation $\bfun_0^\ast\ast\phi = \phi$, which follows from the identification $\Herm_{n}^{\circ}(\calo_F)=\GL_n(\calo_E)/\U_n(\calo_F)$ and our normalization of measures.
\end{proof}

\subsection{Theorem \ref{Thm: main local result 1 var} $\imp$ Theorem \ref{Thm: fundamental lemma s-i Lie}}\label{Section: implies si}
We now consider the split-inert case. This result only needs Theorem \ref{Thm: main local result 1 var}, which follows from \ref{Thm: main local result 2 var}. Recall that in this case, $\underline{\eta}=(\eta,\eta)$.

\quash{The proof of this result indicates that what is needed is the linear OI
\[
 \int_{\GL_n(F)^2}\bfun_{\fgl_n(\calo_F)}(h_1^{-1}z_1h_2)\bfun_{\fgl_n(\calo_F)}(h_2^{-1}z_2h_1)\Phi(w h_1)|h_1|^s\eta(h_2)dh,
\]
which corresponds to $\eta_1=\eta_3 = \eta$ in the notation of the RTF section}

Recalling that $\fq_{si}(\calo_F) = \fgl_n(\calo_E)$, we le $\bfun_{\fgl_n(\calo_E)}$ denote the function in the $k=0$ case of Theorem \ref{Thm: fundamental lemma s-i Lie}. Set $\Phi :=r_!(\bfun_{\fgl_n(\calo_E)})$, where $r$ is the contraction map $r(X) = -X^\ast X\in \Herm_n(F)$.  $\Phi^n:=r_!\bfun_{\End(\Lam_n)}$.This gives a function on $r(\fq_{si}^{rss}(F))\subset \Herm_{n}^{\circ}(F)$, which we view as a locally constant function on $\Herm_{n}^{\circ}(F)$ by extending-by-zero over the compliment of $r(\fq_{si}^{rss}(F))$. The resulting $K_E$-invariant function is not compactly supported, but may be viewed as an element of the \emph{completion} of the Hecke module $\calh_{K_E}(\Herm_{n}^{\circ}(F))$.

\begin{Prop}\label{Prop: implies FLSI}
Theorem \ref{Thm: main local result 1 var} implies Theorem \ref{Thm: fundamental lemma s-i Lie}. 
\end{Prop}
\begin{proof}
The goal is to show that for any regular semi-simple $x\in \fgl_{n}(E)$ and $(z_1,z_2)\in \fgl_n(F)\oplus \fgl_n(F)$ such that $x$ is $\GL_n(E)$-conjugate to $z_1z_2$,
 \begin{equation}\label{eqn: to show implies FLSI}
    \SO^{\U_n(F)\times \U_n(F)}\left(\bfun_{\fgl_n(\calo_E)}, x\right) =  \Orb^{\rH'(F),\ul{\eta}}\left(\bfun_{\fgl_n(\calo_F)\oplus\fgl_n(\calo_F)}\otimes{\bf 1}_{\calo_{F,n}}, (z_1,z_2)\right).
\end{equation}
  Consider the commutative diagram of $\calh_{K_E}(\GL_n(E))$-modules, 
\begin{equation*}
\begin{tikzcd}
&\calh_{K_E}(\GL_n(E))\ar[dl,swap,"-\ast \bfun_0"]\ar[dr,"BC"]&\\
\calh_{K_E}(\Herm_{n}^{\circ}(F))\ar[rr,"\Hir"]&&\calh_{K}(\GL_n(F)).
\end{tikzcd}
\end{equation*}
Since the $\U_n\times\U_n$-regular semi-simple locus of $\fgl_n(E)$ is contained in $\GL_n(E)$, we need only consider the restriction of $\bfun_{\fgl_n(\calo_E)}$ to this locus. This may be viewed as an element of a completion of the Hecke algebra $\calh_{K_E}(\GL_n(E))$. The maps of this diagram naturally extend to these completions and $\bfun_{\fgl_n(\calo_E)}\ast\bfun_0 = \Phi$ \cite[Lemma 3.11]{LeslieUFJFL}. In particular, in terms of the natural extensions of the spherical transform and Satake map to completions of the Hecke algebras, we have the identity $$\mathcal{SF}(\Phi) = \mathrm{Sat}(\mathrm{BC}(\bfun_{\fgl_n(\calo_E)})).$$ It follows from \cite[Proof of Proposition 3.12]{LeslieUFJFL} that
\begin{equation}\label{eqn: si identity}
\mathcal{SF}(\Phi)(Z) = \prod_{1\leq i \leq n}\frac{1}{(1-q^{(n-1)/2}Z_i)(1+q^{(n-1)/2}Z_i)},
\end{equation}
where $Z_i=q^{-s_i}$ encode the Satake parameters.

We now consider the linear side. 

Set $z=z_1z_2$. Up to multiplication by the factor $\frac{\widetilde{\omega}((z_1,z_2),w) }{L(s,\eta,T_y)}$, the linear orbital integral in the proposition is the value at $s=0$ of 
\begin{align*}
    & \int_{\rH'(F)}\bfun_{\fgl_n(\calo_F)}(h_1^{-1}z_1h_2)\bfun_{\fgl_n(\calo_F)}(h_2^{-1}z_2h_1){\bf 1}_{\calo_{F,n}}(w h_2)|h_2|^s\eta(h_1)dh\\
    =&\int_{\GL_n(F)}\left(\int_{\GL_n(F)}\bfun_{\fgl_n(\calo_F)}(h_1^{-1}z_1h_2)\bfun_{\fgl_n(\calo_F)}(h_2^{-1}z_2h_1)\eta(h_2^{-1}h_1)dh_1\right){\bf 1}_{\calo_{F,n}}(w h_2)|h_2|^s\eta(h_2)dh_2\\
       =&\eta(z_2)\int_{\GL_n(F)}\left(\int_{\GL_n(F)}\bfun_{\fgl_n(\calo_F)}(h_1^{-1})\bfun_{\fgl_n(\calo_F)}(h_2^{-1}zh_2h_1)\eta(h_2^{-1}zh_2h_1)dh_1\right){\bf 1}_{\calo_{F,n}}(w h_2)|h_2|^s\eta(h_2)dh_2\\
     =&\eta(z_2)\int_{\GL_n(F)}\left(\left[\bfun_{\fgl_n(\calo_F)}\ast(\eta\cdot\bfun_{\fgl_n(\calo_F)})\right](h_2^{-1}zh_2)\right){\bf 1}_{\calo_{F,n}}(w h_2)|h_2|^s\eta(h_2)dh_2,
\end{align*}
where $(\eta\cdot\bfun_{\fgl_n(\calo_F)})(z) = \eta(z)\bfun_{\fgl_n(\calo_F)}(z)$. Note that the factor $\eta(z_2)$ cancels with the corresponding term in \eqref{eqn: linear Lie transfer factor} so that we obtain \eqref{eqn: correct xiao transfer factor} in the reduction. 
 
 We thus need to compute the convolution
\begin{align*}
    \Phi^{\GL}(z) 
    :=\int_{\GL_n(F)}\bfun_{\fgl_n(\calo_F)}(zh^{-1})\bfun_{\fgl_n(\calo_F)}(h)\eta(h)dh.
\end{align*}
Since $(z_1,z_2)$ is regular semi-simple, $z=z_1z_2\in \GL_n(F)$. We thus only need to consider the restriction of $\bfun_{\fgl_n(\calo_F)}$ to $\GL_n(F)$. We thus decompose
\[
\bfun_{\fgl_n(\calo_F)}|_{\GL_n(E)}=\sum_{\lam\in\zz_{\geq0}^n}\bfun_{K\vp^\lam K}.
\]
Combining \eqref{eqn: Satake basis} and \eqref{eqn: geometric series macdonald}, it follows that 
\[
\mathrm{Sat}(\bfun_{\fgl_n(\calo_F)})(Z)= \prod_{1\leq i\leq n}\frac{1}{1-q^{(n-1)/2}Z_i}.
\]
 On the other hand, as in the proof of \cite[Proposition 3.12]{LeslieUFJFL},
\[
\mathrm{Sat}(\eta\cdot\bfun_{\fgl_n(\calo_F)})(Z) = \sum_{\mathbf{m}\in\zz_{\geq0}^n}q^{|\mathbf{m}|(n-1)/2}(-Z)^{\mathbf{m}} =  \prod_{1\leq i\leq n}\frac{1}{1+q^{(n-1)/2}Z_i},
\]
so that $$\mathrm{Sat}(\Phi^{\GL})(Z) = \prod_{1\leq i\leq n}\frac{1}{({1-q^{(n-1)/2}Z_i})(1+q^{(n-1)/2}Z_i)},$$ proving that $\Hir(\Phi) = \Phi^{\GL}$. 

 Thus, for any matching $x\in \fgl_n(E)$ and $(z_1,z_2)\in \fgl_n(F)\oplus \fgl_n(F)$, set $d= \val(\det(z_1z_2))$. Conjugating $(z_1,z_2)$ if necessary, we are free to assume that $z_i\in\fgl_n(\calo_F)$ since otherwise both sides of \eqref{eqn: goal of implies FLII} vanish. Then
 \begin{align*}
     \Orb^{\rH'(F),(\eta,\eta)}&\left(\bfun_{\fgl_n(\calo_F)\oplus\fgl_n(\calo_F)}\otimes {\bf 1}_{\calo_{F,n}},(z_1,z_2)\right)\\&=\Orb^{\GL_n(F),\eta}(\Hir(\Phi) \otimes{\bf 1}_{\calo_{F,n}},z_1z_2)\\
     &{=}\SO^{\U(V_n)}(\Phi,r(x))\quad\qquad\quad\qquad \text{(Thm \ref{Thm: main local result 1 var})}\\
     &=\SO^{\U(V_n)\times\U(V_n)}(\bfun_{\fgl_n(\calo_E)} ,x),
 \end{align*}
proving the claim.
\end{proof}
\subsection{Theorem \ref{Thm: main local result 2 var} $\imp$ Theorem \ref{Thm: fundamental lemma stable Lie}}
We now consider the inert-inert case. Recall that in this case, $\ul{\eta} = (\eta,1)$.
Let $\bfun_{\Herm_n(\calo_F)}$ denote the characteristic function of $\Herm_n(F)\cap\fgl_n(\calo_E)$, so that $\bfun_{\fq_{ii}(\calo_F)}=\bfun_{\Herm_n(\calo_F)}\otimes\bfun_{\Herm_n(\calo_F)}$ is the function in the $k=0$ case of Theorem \ref{Thm: fundamental lemma stable Lie}.

\begin{Prop}\label{prop: implies FLII}
Theorem \ref{Thm: main local result 2 var} implies Theorem \ref{Thm: fundamental lemma stable Lie}.
\end{Prop}
\begin{proof} The goal is to show that for any regular semi-simple $(x_1,x_2)\in \Herm_{n}(F)\times \Herm_n(F))$ and $(z_1,z_2)\in \fgl_n(F)\oplus \fgl_n(F)$ such that $y_1y_2$ is $\GL_n(E)$-conjugate to $z_1z_2$,
 \begin{equation}\label{eqn: goal of implies FLII}
    \SO^{\GL_n(E)}\left(\bfun_{\fq_{ii}(\calo_F)},(x_1,x_2)\right) =  \Orb^{\rH'(F),(\eta,1)}\left(\bfun_{\fgl_n(\calo_F)\oplus\fgl_n(\calo_F)}\otimes {\bf 1}_{\calo_{F,n}},(z_1,z_2)\right).
\end{equation}
Up to multiplication by the factor $\frac{\widetilde{\omega}((z_1,z_2),w) }{L(s,\eta,T_y)}$, the right-hand side of \eqref{eqn: goal of implies FLII} is the value at $s=0$ of
\begin{align*}
    & \int_{\rH'(F)}\bfun_{\fgl_n(\calo_F)}(h_1^{-1}z_1h_2)\bfun_{\fgl_n(\calo_F)}(h_2^{-1}z_2h_1){\bf 1}_{\calo_{F,n}}(w h_2)|h_2|^s\eta(h_2)dh\\
    =&\int_{\GL_n(F)}\left(\int_{\GL_n(F)}\bfun_{\fgl_n(\calo_F)}(h_1^{-1}z_1h_2)\bfun_{\fgl_n(\calo_F)}(h_2^{-1}z_2h_1)dh_1\right){\bf 1}_{\calo_{F,n}}(w h_2)|h_2|^s\eta(h_2)dh_2\\
     =&\int_{\GL_n(F)}\left(\left[\bfun_{\fgl_n(\calo_F)}\ast\bfun_{\fgl_n(\calo_F)}\right](h_2^{-1}z_1z_2h_2)\right){\bf 1}_{\calo_{F,n}}(w h_2)|h_2|^s\eta(h_2)dh_2.
\end{align*}
Setting $z =z_1z_2$, we consider the convolution
\begin{align*}
    \left(\bfun_{\fgl_n(\calo_F)}\ast\bfun_{\fgl_n(\calo_F)}\right)(z)=\int_{\GL_n(F)}\bfun_{\fgl_n(\calo_F)}(zh^{-1})\bfun_{\fgl_n(\calo_F)}(h)dh.
\end{align*}
As in the split-inert case, we only need to consider the restriction of $\bfun_{\fgl_n(\calo_F)}$ to $\GL_n(F)$. Setting $\bfun_d = \bfun_{\fgl_n(\calo_F),\val(\det)=d}$, we decompose
\[
\bfun_{\fgl_n(\calo_F)}|_{\GL_n(E)}=\sum_{d\in\zz_{\geq0}}\bfun_{d}.
\]
\quash{Combining \eqref{eqn: Satake basis} and \eqref{eqn: geometric series macdonald}, it follows that under the natural extension of the Satake isomorphism to power series rings,
\[
\mathrm{Sat}(\bfun_{\fgl_n(\calo_F)})(Z)= \prod_{1\leq i\leq n}\frac{1}{1-q^{(n-1)/2}Z_i},
\]
so that 
\[
\mathrm{Sat}( \Phi^{\GL})(Z) =\prod_{1\leq i\leq n}\frac{1}{(1-q^{(n-1)/2}Z_i)^2}.
\]}

On the Hermitian side, for $d\in \zz_{\geq0}$ consider 
\begin{equation}\label{eqn: det slice at d}
  \Herm_{n}^{\circ,d}(\calo_F):= \{y\in \Herm_n(\calo_F): |\det(y)|_F = q^{-d}\},  
\end{equation}
and let $\Phi_d\in C_c^\infty(\Herm_{n}^{\circ}(F))$ denote its characteristic function.  Then \cite[Lemma 6.7]{offenjacquet} shows 
\[
\Hir(\Phi_d) = (-1)^{nd}\bfun_d;
\]
this immediately implies $\Hir(\bfun_{\Herm_n(\calo_F)}) = \eta^n\cdot \bfun_{\fgl_n(\calo_F)}$, where we abuse notation and let both sides denote the restrictions to invertible elements. An easy calculation shows that
\[
\Hir(\Phi_i \ast\Phi_j ) = (-1)^{n(i+j)}(\bfun_i\ast\bfun_j),
\]
so that $\Hir(\bfun_{\Herm_n(\calo_F)} \ast\bfun_{\Herm_n(\calo_F)} ) = \eta^n\cdot  \left(\bfun_{\fgl_n(\calo_F)}\ast\bfun_{\fgl_n(\calo_F)}\right)$. With this in mind, we re-write the preceding integral as
\[
\eta^n(z_1z_2)\int_{\GL_n(F)}\Hir(\bfun_{\Herm_n(\calo_F)} \ast\bfun_{\Herm_n(\calo_F)} ) (h_2^{-1}z_1z_2h_2){\bf 1}_{\calo_{F,n}}(w h_2)|h_2|^s\eta(h_2)dh_2.
\]
Note that the transfer factor \eqref{eqn: linear Lie transfer factor} does not possess a factor of the form $\eta^n(z_1z_2)$, but the transfer factor \eqref{eqn: correct xiao transfer factor} does. This discrepancy is accounted for by the factor above.

 Thus, for any matching $(x_1,x_2)\in \Herm_{n}(F)\oplus \Herm_n(F)$ and $(z_1,z_2)\in \fgl_n(F)\oplus \fgl_n(F)$, set $d= \val(\det(z_1z_2))$. Conjugating $(z_1,z_2)$ if necessary, we are free to assume that $z_i\in\fgl_n(\calo_F)$ since otherwise both sides of \eqref{eqn: goal of implies FLII} vanish. Then
 \begin{align*}
     \Orb^{\rH'(F),(\eta,1)}&\left(\bfun_{\fgl_n(\calo_F)\oplus\fgl_n(\calo_F)}\otimes {\bf 1}_{V_n(\calo_{F})},(z_1,z_2)\right)\\&=\sum_{i=0}^{d}\Orb^{\GL_n(F),\eta}((-1)^{nd}\bfun_{d-i}\ast\bfun_i,z_1z_2)\\
     &{=}\sum_{i=0}^{d}\SO^{\GL_n(E)}(\Phi_{d-i}\otimes\Phi_i,(x_1,x_2))\quad\qquad \text{(Thm \ref{Thm: main local result 2 var})}\\
     &=\SO^{\GL_n(E)}(\bfun_{\Herm_n(\calo_F)} \otimes\bfun_{\Herm_n(\calo_F)} ,(x_1,x_2)),
 \end{align*}
proving the claim.
\end{proof}
 \subsection{Theorem \ref{Thm: main local result 2 var} $\imp$ Theorem \ref{Thm: fundamental lemma endoscopic Lie}}\label{twisted thing}
 

For the endoscopic comparison, we begin with the following calculation in $\cc[[X_1,\ldots X_n]]$.
\begin{Lem}\label{Lem: power series identity}
    Set for any $d\geq 0$, set 
    \[T_d:=(-1)^{nd}\sum_{\lam\in  \zz_{d,+}^n} X^{\lam}
\]where $ \zz_{d,+}^n =\{\lam\in \zz^n_{\geq0}: \sum_i \lam_i=d\}$ and $X^{\bf m}=\prod_{i=1}^n X_i^{{\lam}_i}$. Then
    \begin{align}\label{eqn: power series identity}
\sum_{d\geq 0}\left(\sum_{0\leq i\leq d} (-1)^i  T_i T_{d-i}\right)=\prod_{1\leq i\leq n}\frac{ 1}{(1-X_i  )(1+X_i)}.
\end{align}
\end{Lem}
\begin{proof}
	Consider the generating series 
	\[
 \phi(u)=\sum_{\lambda\geq 0}  (uX)^\lambda,
    \]
    so that $ \phi((-1)^nu) = \sum_{d\geq 0}T_du^d$. The left-hand side of \eqref{eqn: power series identity} is precisely $ \phi(-u)\phi(u),$ with $u=(-1)^n$. On the other hand,  it is easy to see using \eqref{eqn: geometric series macdonald} that
	\[
	\phi(u)=\prod_{1\leq i\leq n}(1-uX_i)^{-1},
	\]
which gives the claim for any $n$.
\end{proof}
\begin{Prop}\label{Prop: implies FL endo}
Theorem \ref{Thm: main local result 2 var} implies Theorem \ref{Thm: fundamental lemma endoscopic Lie}. 
\end{Prop}
\begin{proof} The goal is to show that for any regular semi-simple $(x_1,x_2)\in \Herm_n(F)\times \Herm_n(F)$, if $(x_1,x_2)$ matches with $X\in \fgl_n(E)$ with respect to $\tau_1=\tau_n$, then
\begin{equation}\label{eqn: endoscopic goal}
\De_\varepsilon(x_1,x_2)\Orb^{\GL_n(E),\varepsilon}\left(\bfun_{\Herm_n(\calo_F)}\otimes\bfun_{\Herm_n(\calo_F)},(x_1,x_2)\right) = \SO^{\U(V_n)}(\bfun_{\fgl_n(\calo_E)},X),
\end{equation}
and if it matches $X$ with respect to $\tau_1$ non-split, then
\[
\Orb^{\GL_n(E),\varepsilon}\left(\bfun_{\Herm_n(\calo_F)}\otimes\bfun_{\Herm_n(\calo_F)},(x_1,x_2)\right)=0.
\]
Note that this is determined by $\val(\det(x_1x_2))$ modulo $2$.

Recalling the characteristic function $\Phi_d\in C_c^\infty(\Herm_{n}^{\circ}(F))$ of \eqref{eqn: det slice at d},
if $\val(\det(x_1x_2))=d$ then we have
\[
\De_\varepsilon(x_1,x_2)\Orb^{\GL_n(E),\varepsilon}\left(\bfun_{\Herm_n(\calo_F)}\otimes\bfun_{\Herm_n(\calo_F)},(x_1,x_2)\right)= \sum_{i=0}^{d}(-1)^i\SO^{\GL_n(E)}(\Phi_{i}\otimes\Phi_{d-i},(x_1,x_2)).
\]
Indeed, if we set $\val(x) = \val(\det(x))$, then the left-hand side is given by
\quash{\[
\sum_{\al}\eta(x_{1,\al})\Orb^{\GL_n(E)}\left(\bfun_{\Herm_n(\calo_F)}\otimes\bfun_{\Herm_n(\calo_F)},(x_1,x_2)\right),
\]
and for each $\al\in H^1(F,\rH_{(x_1,x_2)}$, there is a unique $0\leq i\leq d$
\[
\Orb^{\GL_n(E)}\left(\bfun_{\Herm_n(\calo_F)}\otimes\bfun_{\Herm_n(\calo_F)},(x_1,x_2)\right) = \Orb^{\GL_n(E)}(\Phi_{i}\otimes\Phi_{d-i},(x_{1,\al},x_{2,\al})).
\]
In particular,}
\begin{align*}
    \sum_{\al}(-1)^{\val(x_{1,\al})}\Orb^{\GL_n(E)}&\left(\bfun_{\Herm_n(\calo_F)}\otimes\bfun_{\Herm_n(\calo_F)},(x_{1,\al},x_{2,\al})\right)\\ = \sum_{i=0}^{d}(-1)^i&\sum_{{\val(x_{1,\al})=i}}\Orb^{\GL_n(E)}(\Phi_{i}\otimes\Phi_{d-i},(x_{1,\al},x_{2,\al}))\\
    &= \sum_{i=0}^{d}(-1)^i\SO^{\GL_n(E)}(\Phi_{i}\otimes\Phi_{d-i},(x_1,x_2)).
\end{align*}
By Theorem \ref{Thm: main local result 2 var} (in particular, its Corollary \ref{Cor: product orbital ints}), this  is equal to 
$$
 \sum_{i=0}^{d}(-1)^i\SO^{\U(V_n)}(\Phi_{i}\ast\Phi_{d-i},X)
$$
We are reduced to considering the stable orbital integrals of the function
\[
\Phi_{\varepsilon}:=\sum_{d\geq0}\left(\sum_{i=0}^d(-1)^i\Phi_{i}\ast\Phi_{d-i}\right)\in C^\infty(\Herm_{n}^{\circ}(F))^{K_E};
\]
by \cite[Lemma 6.7]{offenjacquet}, this satisfies that  
\[
\SF(\Phi_{\varepsilon})=\sum_{d\geq 0}\left(\sum_{0\leq i\leq d} (-1)^i  T_i T_{d-i}\right)\in \cc[[X_1,\ldots,X_n]],
\]
where $X_i = q^{(n-1)/2}Z_i$ for $Z_i$ as in \eqref{eqn: weird normalization}. Lemma \ref{Lem: power series identity} and \eqref{eqn: si identity} now imply that
\[
\SF(\Phi_{\varepsilon}) = \prod_{1\leq i\leq n}\frac{1}{({1-q^{(n-1)/2}Z_i})(1+q^{(n-1)/2}Z_i)}=\SF(\Phi),
\]
where $\Phi = r_!(\bfun_{\fgl_n(\calo_E)})$. This proves the both the identity \eqref{eqn: endoscopic goal} and vanishing claims.
\end{proof}

\part{Proof of Theorem \ref{Thm: main local result 2 var}}\label{Part: trace proof global}

In this part, we prove Theorem \ref{Thm: main local result 2 var}. The proof is global, relying on a comparison of relative trace formulas built on the transfer of orbital integrals in Theorem \ref{Thm: Xiao's transfer} between $\GL_n$ and $\Herm_{n}^{\circ}$.


\section{Main result on orbital integrals}
Recall the orbital integrals and matching of orbits introduced in \S \ref{Section: prelim for Jacquet} and Theorem \ref{Thm: Xiao's transfer}. In this section, we state our main local result Theorem \ref{Thm: first step main}, and verify that this implies Theorem \ref{Thm: main local result 2 var}.

\subsection{Important change of convention}
Our proof of Theorem \ref{Thm: main local result 2 var} relies on a global comparison of trace formulas. Due to the nature of this argument, it is simpler to change our formalism on group actions and contractions.

\begin{Important}\label{Important}
    For the remainder of the paper, we use the \textbf{right} action of $\GL_n(E)$ on $\Herm_n^\circ(F)$, given as $h\ast g = g^\ast h g$ with $h\in \Herm_n^\circ(F)$ and $g\in \GL_n(E)$. In particular, we set
    \[
    \U(V_\tau):=\{u\in \Res_{E/F}\GL_n: u^\ast \tau u = \tau\}.
    \]
    This conflicts with the previous notation when $\tau\neq\tau^{-1}$ (but changes nothing in the case of $\U(V_n)$ where $\tau = I_n$).
\end{Important} 
The justification for switching is that the right action is more natural in the spectral setting below due to our use of results of \cite{FLO}, and keeping our prior conventions would introduce unpleasant notational complications. 

This is a purely notational change since $\U(V_\tau)\simeq \U(V_{\tau^{-1}})$. In particular, this change does not change the set $\calv_n(E/F)$ nor the do the sets
\[
\Herm_{n}^{\circ}(F)_\tau=\{x\in\Herm_{n}^{\circ}(F)\;:\; (\Res_{E/F}\GL_n)_x\simeq \U(V_\tau)\}
\]
change. The natural isomorphism $\Herm_{n}^{\circ}\iso \Herm_{\tau}^\circ$ intertwining the $\U(V_\tau)$-action on $\Herm_{n}^{\circ}$ with the conjugation action on $\Herm_{\tau}^\circ$ is given by $x\mapsto \tau^{-1} x$.

This does change our notation for quotient maps $p_\tau:\GL_n(E)\to \Herm_{n}^{\circ}(F)$, now given by $p_\tau(g) = g^\ast \tau g$. This implies that the contraction map
\[
p_{\tau_!}:C^\infty_c(\GL_n(E))\lra C^\infty_c(\Herm_{n}^{\circ}(F))
\]
is now
\[
p_{\tau,!}(f)(g^\ast \tau g) =\int_{U(V_{\tau})}f(ug)du.
\]

For $\tau\in \calv_n(E/F)$, we say that a function $f\in C_c^\infty(\Herm_{n}^{\circ}(F))$ is of type $\tau$ if 
\[
\supp(f)\subset\Herm_{n}^{\circ}(F)_\tau.
\]
Clearly this is equivalent to $f= p_{\tau,!}(f^\tau)$ for some $f^\tau\in C_c^\infty(\GL_n(E))$. 
The following lemma is important for later.
\begin{Lem}\label{Lem: invariant lifts}Let $\tau\in \Herm_{n}^{\circ}(F)$. Suppose a spherical function $\phi\in \calh_{K_E}(\Herm_{n}^{\circ}(F))$ is of type $\tau$. There exists a right $K_E$-invariant function $\phi^\tau\in C_c^\infty(\GL_n(E))$ such that 
    \[
    \phi=p_{\tau,!}(\phi^\tau).
    \]
\end{Lem}
\begin{proof}
Fix \emph{some} lift $f^\tau\in C_c^\infty(\GL_n(E))$ and define
\[\phi^\tau(g)=\int_{K_E}f^\tau(gk)dk.\] Clearly, $\phi^\tau$ is right $K_E$-invariant and it is also a lift of $\phi$ since $\phi$ is $K_E$-invariant. 
\end{proof}


\subsection{Reformulation of the main result}
Suppose now that $E/F$ is an unramified extension of non-archimedean local fields. We first recall the following fundamental lemma for the unit element from \cite[Theorem 5.1]{Xiaothesis}. 
\begin{Thm}\label{Thm: xiao FL}
The characteristic function $\bfun_{0}$ of $\Herm_{n}^{\circ}(\calo_F)$ and $\bfun_{\GL_n(\calo_F)}\otimes \bfun_{\calo_{F,n}}$ are smooth transfers in the sense of Theorem \ref{Thm: Xiao's transfer}. 
\end{Thm}
\begin{proof}
	This is contained in \cite[Theorem 5.1]{Xiaothesis} with the normalizations discussed above. Note that $\eta^n\equiv 1$ on the supports of the test functions, so that the change in transfer factor is not relevant.
\end{proof}

We require slightly more. For $d\in \zz_{\geq0}$, recall the characteristic function $\Phi_d\in C_c^\infty(\Herm_{n}^{\circ}(F))$ of
\[
\Herm_{n}^{\circ,d}(\calo_F):= \{y\in \Herm_{n}^{\circ}(F)\cap \fgl_{n}(\calo_E): |\det(y)|_F = q^{-d}\},
\]
and the characteristic function $\bfun_d\in C_c^\infty(\GL_n(F))$ of
\[
\fgl_{n}^d(\calo_F):=\{z\in  \fgl_{n}(\calo_F): |\det(z)|_F = q^{-d}\}.
\]
\begin{Thm}\label{Thm: Xiao FL for d}
The characteristic function $\Phi_d$ is a transfer of $(-1)^{nd} \bfun_d\otimes \bfun_{\calo_{F,n}}$. 
\end{Thm}
\begin{Rem}
   As noted in the proof of Proposition \ref{prop: implies FLII},  $\Hir(\Phi_d) =(-1)^{nd}\bfun_d$ so that this is a special case of Theorem \ref{Thm: main local result 1 var}.
\end{Rem}
\begin{proof}
    This follows from the same argument as the fundamental lemma ($d=0$), relying on the fundamental lemma of Jacquet--Rallis. First we claim that these functions are \emph{almost} Jacquet--Rallis transfers of each other. Recalling the Jacquet-Rallis fundamental lemma on the Lie algebra of \cite{YunJR, beuzart2019new}, the functions $\bfun_{\fgl_{n}(\calo_F)}$ matches with $\{\bfun_{\Herm_n(\calo_F)},0\}$ with respect to Jacquet--Rallis transfer. An easy corollary of this is that $\bfun_d$ and $\{\Phi_d,0\}$ are transfers for any $d\geq 0$. Following the construction of transfer in the proof of Theorem \ref{Thm: Xiao's transfer}, this implies that
    \[
    \eta((-1)^n)\eta^n\cdot \Phi_d = \eta^n\cdot \Phi_d
    \] is a transfer for $\bfun_d\otimes \bfun_{\calo_{F,n}}$, where we have used that $E/F$ is unramified. Noting that $\eta^n\cdot \Phi_d=(-1)^{nd}\Phi_d$, the theorem follows by moving the sign to the other side. 
\end{proof}

Our main goal for the rest of the paper is to prove the following:
\begin{Thm}\label{Thm: first step main} Fix $\tau_1,\tau_2\in \Herm_{n}^{\circ}(F)$. For $i\in \{1,2\}$, assume that $\phi_i\in \calh_{K_E}(\Herm_{n}^{\circ}(F))$ is of type $\tau_i$, and let $\phi_{1}^{\tau_1}$ (resp., $\phi^{\tau_2}_2$) denote left-$K_E$-invariant lift to $\GL_n(E)$ as in Lemma \ref{Lem: invariant lifts}.

With respect to the action of $\U(V_{\tau_2^{-1}})$ on $\Herm_{\tau_2^{-1}}^{\circ}(F)$, the function $ \Phi_{\tau_1,\tau_2}^{\phi_1,\phi_2}:=p_{\tau_1,!}\left(\phi_{1}^{\tau_1}\ast\phi^{\tau_2,\ast}_{2}\right)$is transfer of ${\Hir(\phi_1\ast\phi_2)}\otimes \bfun_{\calo_{F,n}}$ in the sense of Theorem \ref{Thm: Xiao's transfer}.
\end{Thm}

We now verify that this implies Theorem \ref{Thm: main local result 2 var}.
\begin{Cor}\label{Cor: main local results}
    Theorem \ref{Thm: main local result 2 var} holds. That is, for any $\phi_1,\phi_2\in \calh_{K_E}(\Herm_{n}^{\circ}(F))$ and any regular semi-simple $(x_1,x_2)\in\Herm_{n}^{\circ}(F)\times \Herm_{n}^{\circ}(F)$, we have
	\[
 \SO^{\GL_n(E)}(\phi_1\otimes\phi_2,(x_1,x_2))=\Orb^{\GL_n(F),\eta}( \Hir(\phi_1\ast\phi_2)\otimes {\bf 1}_{\calo_{F,n}},z),
	\]
 for $z\in \GL_n(F)$ matching $(x_1,x_2)$. 
\end{Cor}
\begin{proof}
     \quash{By linearity it is enough to consider the case when $\phi_1$ is of type $\tau_1$ and $\phi_2$ is of type $\tau_2$. As noted above, Theorem \ref{Thm: first step main} may be rephrased as the identity
     \[
     \SO^{\U(V_1)\times \U(V_2)\times \GL_n}(\phi_1^{\tau_1}\otimes\phi_2^{\tau_2},(g_1,g_2)) = \Orb^{\GL}( \Hir(\phi_1\ast\phi_2)\otimes {\bf 1}_{\calo_{F,n}},z),
     \]
     where $(g_1,g_2)\in \GL_n(E)\times \GL_n(E)$ satisfies $p_{\tau_1}(g_1)= x$ and $p_{\tau_2}(g_2)=y$. Contracting along $p_{\tau_1}\times p_{\tau_2}$ gives 
 and the identity
     \[
      \SO^{\GL_n}( \phi_1\otimes\phi_2,{(x,y)})= \SO^{\U(V_1)\times \U(V_2)\times \GL_n}(\phi_1^{\tau_1}\otimes\phi^{\tau_2}_2,(g_1,g_2)),
     \]
     implying the matching of orbital integrals in Theorem \ref{Thm: main local result 2 var}.}
Assume that a $\GL_n(E)$-regular semisimple pair $(x_1,x_2)\in \Herm_{n}^{\circ}(F)\times \Herm_{n}^{\circ}(F)$ matches $z\in \GL_n(F)$.  
      By linearity it is enough to consider the case when $\phi_1$ is of type $\tau_1$ and $\phi_2$ is of type $\tau_2$. It is trivial that $\SO^{\GL_n(E)}( \phi_1\otimes\phi_2,{(x_1,x_2)})=0$ if the stable orbit of $(x_1,x_2)$ does not meet $\Herm_{n}^{\circ}(F)_{\tau_1}\times \Herm_{n}^{\circ}(F)_{\tau_2}$. We thus assume $(x_1, x_2)\in \Herm_{n}^{\circ}(F)_{\tau_1}\times \Herm_{n}^{\circ}(F)_{\tau_2},$ and write
      \[
      (x_1,x_2) = (g_1^\ast\tau_1g_1, g_2^\ast\tau_2g_2)
      \]
      for $g_1,g_2\in \GL_n(E)$.
      
We now consider following diagram of contractions: 

\begin{equation}\label{eqn: contraction diagram}
\begin{tikzcd}
    &\ar[dr, "p^{2,3}"]\ar[ddr,swap,"p^{2,4}"]\Res_{E/F}(\GL_n\times \GL_n)\ar[dl,"p^{2,1}"]&&\\
    \Herm_{n}^{\circ}\times \Herm_{n}^{\circ}& & \Res_{E/F}\GL_n\ar[d,"p^{3,2}"]&&\\
    && \Herm_{n}^{\circ}\ar[r,"\tau^{-1}\cdot"]&\Herm_{\tau}^{\circ}.
\end{tikzcd}
\end{equation}
Here, the \emph{right} actions are as follows:
\begin{enumerate}
	\item\label{5}
	$\Res_{E/F}\GL_{n}$ acting on $\Herm_{n}^{\circ} \times \Herm_{n}^{\circ}$ via $(x_1,x_2)\cdot g=(g^{-1}x_1(g^\ast)^{-1},g^\ast x_2g)$.
    \item\label{4} $\U(V_1) \times \U(V_2) \times \Res_{E/F}\GL_{n}$ acting on $\Res_{E/F}(\GL_{n}\times \GL_{n})$ via $$(y_1,y_2)\cdot (h_1,h_2,g)=(h_1^{-1}y_1(g^{\ast})^{-1},h_2^{-1}y_2g).$$
    	\item\label{3} $\U(V_1) \times \U(V_2)$ acting on $\Res_{E/F}\GL_{n}$ via $ y\cdot (h_1,h_2)=h_1^{-1}y (h_2^\ast)^{-1}$.
	\item\label{1} $\U(V_\tau)$ acting on $\Herm_{n}^{\circ}$ via $y\cdot h=h^\ast yh$.
    \item $\U(V_\tau)$ acting on $\rH_{\tau}$ via $y\cdot h=h^{-1} yh$.
\end{enumerate}
The morphisms indicated are the natural quotient morphisms. For example, the map 
\begin{align*}
   p^{2,3}: \Res_{E/F}(\GL_n\times \GL_n)&\lra \Res_{E/F}\GL_n\\
     (y_1,y_2)&\longmapsto  y_1 y_2^\ast.
\end{align*} It is an easy check that various maps are (compositions of) the contraction maps in the sense of \S \ref{Section: contractions}.

Our assumption is that $(x_1,x_2)=p^{2,1}_{\tau_1,\tau_2}(g_1,g_2)$. Then
\[
p^{2,4}_{\tau_1}(g_1,g_2) = g_2 x_1 g_2^\ast\in \Herm_{n}^{\circ}(F).
\]
    Now a simple computation using contractions along the Diagram \ref{eqn: contraction diagram} and the identity \eqref{eqn: unwind to Hn} gives the equality
      \begin{align*}
             \SO^{\GL_n(E)}( \phi_1\otimes\phi_2,{(x_1,x_2)}) &=\SO^{\U(V_{\tau_1})\times \U(V_{\tau_2})\times \GL_n(E)}(\phi_1^{\tau_1}\otimes\phi_2^{\tau_2},(g_1,g_2))\\
              &=\SO^{\U(V_{\tau_1})\times \U(V_{\tau_2})}(\phi_1^{\tau_1}\ast\phi_2^{\tau_2,\ast},g_1g_2^\ast)\\
             &= \SO^{\U(V_{\tau_2^{-1}})}\left(\Phi_{\tau_1,\tau_2}^{\phi_1,\phi_2}, \tau_2g_2 x_1 g_2^\ast\right),
      \end{align*}
          where $\Phi_{\tau_1,\tau_2}^{\phi_1,\phi_2}=p_{\tau_1,!}(\phi_{1}^{\tau_1}\ast \phi^{\tau_2,\ast}_{2})$ is as in Theorem \ref{Thm: first step main}. The final integral is an orbital integral in the context of Theorem \ref{Thm: Xiao's transfer}, where we use the map $\Herm_{n}^{\circ}\to \Herm_{\tau_2^{-1}}^{\circ}$ given by $x\mapsto \tau_2x$. \quash{,  and obtain
      \[
      \SO^{\U(V_{\tau_1^{-1}})}\left(\Phi_{\tau_1,\tau_2}^{\phi_1,\phi_2}, g_1^\ast x_2 g_1\tau_1\right).
      \]} We claim that the points $ \tau_2g_2 x_1 g_2^\ast$ and $z$ match in the sense of Lemma/Definition \ref{Lem: matching def}. Indeed, this is equivalent to 
\begin{equation*}
\det(tI_n - \tau_2g_2 x_1 g_2^\ast)= \det(tI_n -  (g_2^{\ast})^{-1} x_2x_1g_2^{\ast})  =\det(tI_n-x_2x_1)    
\end{equation*}
 agreeing with $\det(tI_n-z)$, which is true by assumption that $(x_1,x_2)$ matches $z$.  Theorem \ref{Thm: first step main}  now 
     implies the matching of orbital integrals in Theorem \ref{Thm: main local result 2 var}.
\end{proof}
\begin{Rem}  As seen in the proof of Corollary \ref{Cor: main local results}, the idea of Theorem \ref{Thm: first step main} is use the diagram \eqref{eqn: contraction diagram} to identify the correct test function on $\Herm_{n}^{\circ}(F)$ to compare to the linear side. In particular, the critical identity is the relation
\[
\Orb^{\GL_n}( \phi_1\otimes\phi_2,{(x_1,x_2)}) =  \Orb^{\U(V_{\tau_1})\times \U(V_{\tau_2})\times \GL_n(E)}(\phi_1^{\tau_1}\otimes\phi_2^{\tau_2}, (g_1,g_2)),
\]
obtained by considering the contraction map $p^{2,1}$ in \eqref{eqn: contraction diagram}. One then contracts along $p^{2,4}$ to obtain the formula above. 
  \end{Rem}

\quash{\begin{Thm}\label{Thm: first step main redux}Fix two Hermitian forms $\tau_1,\tau_2\in \Herm_{n}^{\circ}(F)$. For $i\in \{1,2\}$, assume that $\phi_i\in \calh_{K_E}(\Herm_{n}^{\circ}(F))$ is of type $\tau_i$, and let $\phi_{1}^{\tau_1}$ (resp., $\phi^{\tau_2}_2$) denote right-$K_E$-invariant lift to $\GL_n(E)$ as in Lemma \ref{Lem: invariant lifts}.

With respect to the action of $\U(V_{\tau_2^{-1}})$ on $\Herm_{\tau_2^{-1}}^{\circ}(F)$, the function $ \Phi_{\tau_1,\tau_2}^{\phi_1,\phi_2}:=p_{\tau_1,!}\left(\phi_{1}^{\tau_1}\ast(\phi^{\tau_2}_{2})^\ast\right)$ is transfer of ${\Hir(\phi_1\ast\phi_2)}\otimes \bfun_{\calo_{F,n}}$.
\end{Thm}}

\subsection{Fixing a central character}\label{Section: dealing with center}
Let $Z\subset \GL_n$ and $Z_{E}\subset \Res_{E/F}\GL_n$ denote the centers of the respective groups. The proof of Theorem \ref{Thm: first step main} relies on a  global comparison of trace formulas. This requires we include actions of the centers by fixing a central character. We now modify the transfer statement of Theorem \ref{Thm: Xiao's transfer} to incorporate this action. Consider the natural norm map
\[
\Nm:=\Nm_{E/F}: Z_{E}\lra Z,
\]
which is not surjective on points, globally and locally. 

We now assume that $E/F$ is a quadratic extension of local or global fields of characteristic zero. Fixing $\tau\in \Herm_{n}^{\circ}(F)$, consider the $Z_{E}(F)$-action on $\Herm_\tau^\circ(F)$ given by
\[
(x,a) \mapsto x\Nm(a).
\]
Clearly $Z_E^1:=\ker\Nm$ acts trivially under the above action. 

Fix $\tau\in \calv_n(E/F)$. A simple calculation shows that if $x\in \Herm_{\tau}^{\circ}(F)$ is regular semi-simple, then the trace $\Tr(x)$ is scaled by $\Nm(z)$, so that the norm class $[\Tr(x)]\in F^\times/\Nm(E^\times)$ is invariant. We say $x\in \Herm_{\tau}^{\circ}(F)$ is $Z$-regular semi-simple if it is regular semi-simple in $\Herm_{\tau}^{\circ}(F)$ and if $\Tr(x)\in F^\times$. This gives a Zariski-open, dense subset of $\Herm_{\tau}^{\circ}(F)$. 

The proof of the next lemma is straightforward and essentially contained in  \cite[Proposition 5.5]{Xiaothesis}.
\begin{Lem}\label{Lem: Z-reg on unitary side}
If $x\in \Herm_{\tau}^{\circ}(F)$ is $Z$-regular semi-simple, then its centralizer under the $Z_{E}(F)\times U(V_\tau)$-action is $T_x\times {Z}_E^1$ and its orbit is closed. In particular, a $Z$-regular semi-simple element is $Z_{E}\times \U(V_\tau)$-regular semi-simple.
\end{Lem}


For the linear case, we similarly consider the action of $Z$ on $\GL_n$. We say that $z\in \GL_{n}(F)$ is $Z$-regular semi-simple if it is regular semi-simple under the $\GL_n(F)$-action and $\Tr(z)\neq0$. A similar but easier argument (cf. \cite[Proposition 5.7]{Xiaothesis}) now shows that if $z$ is $Z$-regular semi-simple and $(z,w)$ strongly regular, then $(z,w)$ has trivial stabilizer under $Z\times \GL_n$ on $\GL_n(F)\times F_n$ and has a closed orbit.

The following lemma follows immediately from the preceding definitions.
\begin{Lem}\label{matching via center}
 If $x\in \Herm_{\tau}^{\circ}(F)$ and $z\in \GL_n(F)$ are $Z$-regular semi-simple elements that match in the sense of Definition \ref{Lem: matching def}, and $a\in Z_{E}(F)$, then $x\Nm(a)$ matches $z a$.
\end{Lem}

Now assume $E/F$ is local, and assume that $x\in \Herm_{\tau}^{\circ}(F)$ is $Z$-regular semi-simple. For any central character $\omega':Z_{E}(F)\to \cc^\times$, we note that the integral
\[
\int_{Z_{E}(F)}\Orb^{\U(V_\tau)}(f',ax)\omega'(a)da
\]
is absolutely convergent by the closed orbit assertion of Lemma \ref{Lem: Z-reg on unitary side} and vanishes unless $\omega'$ is trivial on $Z_{E}^1(F)$. In this case, $\omega'=\omega\circ \Nm$ for some character $\omega: Z(F)\to \cc^\times$, and we set
\begin{align}\label{eqn: unitary OI final}
 \Orb^{\U(V_\tau)}_{\omega'}(f',x)&:=\frac{1}{\vol(Z_E^1(F))}\int_{T_x(F)\backslash U(V_\tau)}\int_{Z_{E}(F)}f'(h^{-1}x h\Nm(t))\omega'(t)dtdh\nonumber\\
 &=\int_{\Nm(Z_E(F))}\Orb^{\U(V_\tau)}(f',xa)\omega(a)d{a}.
\end{align}
Note that the integrand is stated in terms of the variable $a\in Z_{E}(F)$, but that it depends only on $\Nm(a)$, and is independent of the lift. We also have the stable version $\SO^{\U(V_\tau)}_{\omega'}(f,x)$. 

In light of Lemma \ref{matching via center}, for any central character $\omega:Z(F)\to\cc^\times$ and $f\otimes\Phi\in C_c^\infty(\GL_n(F)\times F_n)$, we restrict $\omega$ to the open subgroup 
\[
\Nm(Z_E(F))\cong \Nm_{E/F}(E^\times)
\]
and for any $Z$-regular semi-simple element $z\in \GL_n(F)$ set
\begin{align}\label{eqn: twisted OI final}
    \Orb^{\GL_n(F),\eta}_\omega(f\otimes \Phi,z)
    &=\int_{\Nm(Z_{E}(F))}\Orb^{\GL_n(F),\eta}(f\otimes \Phi,z a)\omega(a)da.
\end{align}
The integration is absolutely convergent. 

\begin{Cor}\label{Cor: transfer with central character}
Assume that $E/F$ is a quadratic extension of $p$-adic local fields. Suppose that $\wt{f}\in C_c^\infty(\GL_n(F)\times F_n)$ and $f'\in C_c^\infty(\Herm_{\tau}^{\circ}(F))$ are transfers in the sense of Theorem \ref{Thm: Xiao's transfer}. Let $\omega$ denote a central character for $\GL_n(F)$ and let $\omega'=\omega\circ \Nm$ denote its base change to a central character for $\GL_n(E)$. For any matching $Z$-regular semi-simple orbits $z\leftrightarrow x$, we have
\begin{equation}\label{eqn: relevant transfer with center}
2^{|S_1^{ram}(z)|}\Orb^{\GL_n(F),\eta}_\omega(\wt{f},z)=\SO^{\U(V_\tau)}_{\omega'}(f',x).
\end{equation}
\end{Cor}
\begin{proof}
    For any $a\in Z_{E}(F)$, Lemma \ref{matching via center} tells us $z a$ matches $x\Nm(a)$. It is easy to see that the transfer factor $\omega$ is invariant under scaling by $\Nm(Z_E(F))$.  The corollary now follows from the formulas \eqref{eqn: unitary OI final} and \eqref{eqn: twisted OI final}.
\end{proof}

In \S \ref{Section: final proof}, we require the following converse to Corollary \ref{Cor: transfer with central character}.
\begin{Lem}\label{Lem: converse with center}
Let $E/F$ be an unramified extension of $p$-adic local fields.  Let $\omega$ denote a central character for $\GL_n(F)$ and let $\omega'=\omega\circ \Nm$ denote its base change to a central character for $\GL_n(E)$. For the notation in Theorem \ref{Thm: first step main} and each unitary $\omega$, suppose that $ \Phi_{\tau_1,\tau_2}^{\phi_1,\phi_2}$ and ${\Hir(\phi_1\ast\phi_2)}\otimes \bfun_{\calo_{F,n}}$ are $\omega$-transfers in the sense that for any matching $Z$-regular semi-simple orbits $z\leftrightarrow x$, we have
\begin{equation*}
\Orb^{\GL_n(F),\eta}_\omega(\Hir(\phi_1\ast\phi_2)\otimes \bfun_{\calo_{F,n}},z)=\SO^{\U(V_{\tau_1^{-1}})}_{\omega'}(\Phi_{\tau_1,\tau_2}^{\phi_1,\phi_2},x).
\end{equation*} 
Then Theorem \ref{Thm: first step main} holds.
\end{Lem}
\begin{proof}The proof is simple and mirrors the proof of \cite[Lemma 11.2]{LeslieUFJFL}. The key points are the local constancy of the orbital integrals in question, Zariski-density of the $Z$-regular locus, and a simple Fourier inversion argument to remove the central character. We leave the details to the interested reader.
\end{proof}
 \subsection{The split case}\label{Section: split transfer 2}
For our global applications, we also consider the split case of the transfer in Theorem \ref{Thm: Xiao's transfer}, where $E=F\times F$ is the split quadratic $F$-algebra. We assume $F$ is a local field of characteristic zero.

 Here, $\calv_n=\{\tau_n\}$ is a singleton and $\eta$ is trivial. We fix an isomorphism $\GL_n(E)\cong \GL_n(F)\times \GL_n(F)$ such that the unitary group $\U_n\cong \GL_n\hra\GL_n\times\GL_n$ is sent to
\[
\U_n\cong\{(g,g^\theta)\in \GL_n(F)\times \GL_n(F): g\in \GL_n(F)\},
\]
where we recall that for $g\in \GL_n(E)$, $g^\theta = w_n{}^tg^{-1}w_n$. In particular,
\[
\Herm_{n}^{\circ}(F)\cong\{(g,g^{-\theta})\in \GL_n(F)\times \GL_n(F): g\in \GL_n(F)\}
\]
and the $\U_n$ action identifies with conjugation of $\GL_n(F)$ on itself via projection to the first coordinate. Under this identification, a regular semi-simple element $(g_1,g_2)\in \GL_n(E)$ matches $z\in \GL_n(F)$ if $g_1^{-1}g_2^{\theta}$ and $z$ are conjugate in $\GL_n(F)$. 

\begin{Prop}\label{Prop: split transfer}
Assume $E=F\times F$ and fix $\Phi\in C_c^\infty(F_n)$ such that $\Phi(0)=1$. Then the functions $f'=f_1\otimes f_2\in C_c^\infty(\GL_n(E))$ and $f_1^\vee\ast f_2^{\theta} \otimes \Phi\in C_c^\infty(\GL_n(F)\times F_{n})$ are smooth transfers, where $\ast$ denotes convolution $f^\vee(g) = f(g^{-1})$, and 
$f^{\theta}(g) = f(g^{\theta}).$
\end{Prop}
\begin{proof}
 By \cite[Proposition 5.3]{Xiaothesis}, the orbital integrals on $\GL_n(F)\times F_{n}$ simplify in the split case to
\begin{equation}\label{eqn: orbital split}
    \Orb^{\GL_n(F),\eta}(f\otimes \Phi,x)=\Phi(0)\displaystyle\int_{T_x\backslash\GL_n(F)}f(g^{-1}xg)dg.
\end{equation}
The result now follows from the simple calculation
\begin{align*}
  \int_{T_{x,y}\bs \GL_n(F)\times \GL_n(F)}f_1(g_1^{-1}xg_2)f_2(g_1^{-\theta}yg_2^\theta)dg_1dg_2 
  &=\int_{T_{x^{-1}y^{\theta}}\bs \GL_n(F)}\left(f_1^\vee\ast f_2^{\theta}\right)(g^{-1}x^{-1}y^{\theta}g)dg,
\end{align*}
where $T_{x,y}\subset \GL_n(F)\times \GL_n(F)$ is the stabilizer of $(x,y)\in  \GL_n(F)\times \GL_n(F)$.
\end{proof}

Now assume $F$ is non-archimedean and $E=F\times F$. For $\phi=\phi_1\otimes \phi_2\in \calh_{\GL_n(\calo_E)}(\GL_n(E))$, we consider the $\GL_n(\calo_F)$-invariant function
\begin{equation}\label{eqn: split unram transfer}
    \widetilde{\phi} = \left(\phi_1^\vee\ast\phi_2^{\theta}\right)\otimes \bfun_{\calo_{F,n}}\in C_c^\infty(\GL_n(F)\times F_n).
\end{equation}
Proposition \ref{Prop: split transfer} implies that $\phi=\phi_1\otimes\phi_2$ and $\widetilde{\phi}$ are matching test functions.
%



\section{Spectral Preliminaries}
We now recall the necessary spectral results to set up our trace formulas. 

\quash{We now return to the global setting and recall the properties of certain period integrals and their relations to (twisted) standard and exterior square $L$-functions. We also recall the necessary local representation theoretic results. When $F$ denotes a fixed number field, let $\A_{F}$ denote its ring of adeles. Let $\G'=\GL_{2n}$, $\rH'=\GL_n\times\GL_n$, and $\pi$ a cuspidal automorphic representation of $\G'(\A_{F})$ with central character $\omega=\omega_{\pi}$.  We use $V_{\pi}$ to denote the vector space of the representation $\pi$. As before, for any character $\chi$, for $h=(h^{(1)},h^{(2)})\in \rH'(A_{F})$ we set $$\chi(h) = \chi(\det(h^{(1)})(\det(h^{(2)})^{-1}).$$
Finally, we write $[\rH'] = Z_{\G'}(\A_{F})\rH'({F})\backslash\rH'(\A_{F})$, and similar notation for other groups. Our measure conventions are found in \S \ref{measures}.
}
\subsection{Whittaker models}\label{Sec: Whittaker}
Suppose that $F$ is a local field. For any non-trivial additive character $\psi: F\to \cc^\times,$
we denote by $\psi_0$ the generic character of $N_n(F)$
\[
\psi_0(u)=\psi\left(\sum_iu_{i,i+1}\right).
\]
For any irreducible generic representation $\pi$, we denote by $\pi^\vee$ the abstract contragredient representation. 
Set $\calw^\psi(\pi)$ to be the Whittaker model of $\pi$ with respect to the generic character $\psi_0$. The action is given by
\[
\pi(g) W(h) = W(hg),\quad g,h\in \GL_n(F),\: W\in \calw^{\psi}(\pi).
\]
Given a generic representation $\pi$, set $\calw(\pi):=\calw^\psi(\pi)$ to be the Whittaker model of $\pi$ with respect to the induced generic character $\psi_0:N_{n}(F)\to \cc^\times$. Then there an isomorphism
\[
\hat{(\cdot)}:\calw(\pi)^\theta\lra \calw^{\psi^{-1}}(\pi^\vee),
\]
given by $\hat{W}(g) = W(g^\theta),$  where $g^{{\theta}}=w_n{}^tg^{-1}w_n$.

If $F$ is global, $\pi$ a cuspidal automorphic representation of $\GL_n(\A_{F})$, we denote by $\W^\varphi$ the $\psi_0$-Fourier coefficient of $\varphi\in V_{\pi}$:
\[
\W^\varphi(g) = \int_{[N_n]}\varphi(ng)\psi_0^{-1}(n)dn,
\]
where $\psi_0$ is our generic character of the unipotent subgroup $N_n(\A_{F})$.

Suppose that $S$ is a finite set of places, containing the archimedean ones, such that $\pi_v$ is unramified and $\psi_{0,v}$ has conductor $\calo_{F_v}$ for $v\notin S$. Let $\varphi\in V_\pi$ be such that $W^\varphi$ is factorizable (for simplicity, we will say that $\varphi$ is factorizable), write $\W^\varphi(g) = \prod_vW_v(g_v)$, where $W_v\in \calw^{\psi_v}(\pi_v)$.  We may assume that for all $v\notin S$, $W_v$ is spherical and normalized so that $W_v(I_{n})=1$.

\subsection{Peterson inner product}\label{Section: inner product part 3}
 Suppose $\pi$ is a cuspidal automorphic representation of $\GL_n(\A_F)$, and let $\hat{\pi}\cong \pi^\vee$ denote the contragredient representation of $\pi$ realized on the space of functions $\{\phi^{{\theta}}:\phi\in \pi\}$. For $\varphi$ and $\hat{\varphi}$ as above, consider the inner product 
\begin{equation}\label{eqn: Peterrson norm center}
    \la\varphi,\hat{\varphi}\ra_{Pet} = \int_{Z_{\GL_n}(\A_F)\GL_n(F)\backslash \GL_n(\A_F)}\varphi(g)\hat{\varphi}(g)dg.
\end{equation}
this is a $\GL_n(\A_F)$-invariant inner product on $\pi.$

As mentioned above, we adopt a different convention for the local inner products in this part, which we now describe. 
\quash{Suppose now that $S$ is a finite set of places, containing the archimedean ones, such that $\pi_v$ is unramified and $\psi_{0,v}$ has conductor $\calo$ for $v\notin S$. Let $\phi\in \pi$ be such that $W^\phi$ is factorizable (for simplicity, we will say that $\phi$ is factorizable), write $\W^\phi(g) = \prod_vW_v(g_v)$, where $W_v\in \calw^{\psi_v}(\pi_v)$. Similarly, let $\hat{\phi}\in \hat{\pi}$ be factorizable and set $\W^{\hat{\phi}}(g) = \prod_v\hat{W}_v(g_v)$, where $\hat{W}_v\in \calw^{\psi^{-1}_v}(\hat{\pi}_v)=\calw^{\psi^{-1}_v}(\pi^\vee_v)$. We may assume that for all $v\notin S$, $W_v$ and $\hat{W}_v$ are spherical and normalized so that $W_v(e)=\hat{W}_v(e)=1$.}
Recall the canonical inner product \cite[Appendix A]{FLO}
\[
[\cdot,\cdot]_{\pi_v}:\calw^{\psi_v}(\pi_v)\otimes\calw^{\psi^{-1}_v}(\hat{\pi}_v)\lra \cc.
\]
It is defined by considering the integral
\begin{equation*}
I_s(W_v,\hat{W}_v')=L(n,\bfun_{F_v^\times})\int_{N_n(F_v)\backslash P_n(F_v)}W_v(h)\hat{W}'_v(h)|\det(h)|_F^sdh,
\end{equation*}
where $W_v,W_v'\in \calw^{\psi_v}(\pi_v)$ and $P_n$ is the mirabolic subgroup of $\GL_n.$ The integral converges for $\mathrm{Re}(s)\gg0$, and has meromorphic continuation. It is known for any local field of characteristic zero (see \cite[Appendix A]{FLO} and the references therein) that this continuation is holomorphic at $s=0$ and gives a non-degenerate $\GL_n(F_v)$-invariant pairing. We set
\begin{equation}\label{eqn: local inner product}
    [W_v,\hat{W}'_v]_{\pi_v}:=I_0(W_v,\hat{W}_v').
\end{equation}
%
\begin{Prop}\cite[Section 10.3]{FLO}\label{Prop: nice inner product}
Assume that $\varphi\in V_\pi$ is factorizable in the sense that  $\W^\varphi=\bigotimes_vW_v$. There is a corresponding factorization
\begin{align}\label{eqn: inner product}
   \la\varphi,\hat{\varphi}\ra_{Pet} =  \frac{\Res_{s=1}L(s,\pi\times\pi^\vee)}{\vol(F^\times\backslash \A_F^1)}\prod_{v}[W_v,\hat{W}_v]^\natural_{\pi_v},
\end{align}
where \begin{equation}\label{eqn: normalized inner product}
    [W_v,\hat{W}_v]^\natural_{\pi_v}=\frac{[W_v,\hat{W}_v]_{\pi_v}}{L(1,\pi_v\times\pi^\vee_v)},
\end{equation}
where $L(s,\pi_v\times \pi_v^\vee)$ denotes the local Rankin-Selberg $L$-factor. When $\pi_v$ is unramified and $W_v$ is the spherical vector normalized so that $W_v(e) =1,$ we have $[W_v,\hat{W}_v]^\natural_{\pi_v}=1$.\end{Prop}


\subsection{An Eisenstein series}\label{ss:Eis}
{Let $\eta$ be a Hecke character such that $\eta|_{\BR_+}=1$ under the  usual decomposition $\BA^\times_{F}=\BA^1_{F}\times \BR_+$ (for example, if $\eta$ is of finite order).
We  define
 \begin{align}\label{eqn: our Eisenstein}
 E(h,\Phi,s,\eta):=&|\det(h)|^s\eta(\det(h))\int_{[\GL_1]}\sum_{0\neq v\in F_{n}}\Phi(avh)|a|^{ns}\eta(a^n)da
 \\=&\int_{[\GL_1]}  \sum_{0\neq v\in F_{n}}\Phi(avh)|\det(ah)|^{s}\eta(\det(ah))da  \notag
 \end{align}
 where $\Phi\in \CS(\A_{F}^n)$ is in the space of Schwartz functions, and $h\in \GL_n(\A_{F})$. 

 We have the following proposition recording the basic properties of the Eisenstein series.
 \begin{Prop}\cite[Proposition 2.1]{CogdellGLn}\label{Prop: Eisenstein properties}
The Eisenstein series $ E(h,\Phi,s,\eta)$ converges absolutely whenever $\mathrm{Re}(s)>1$ and admits a meromorphic continuation to the all of $\cc$. It has (at most) simple poles $s=0,1$.

As a function of $h$, it is smooth of moderate growth and as a function of $s$ it is bounded in vertical strips away from the poles, uniformly in $g$ in compact sets. Moreover, we have the functional equation
\[
 E(h,\Phi,s,\eta) =  E({}^th^{-1},\widehat{\Phi},1-s,\eta^{-1}),
\]
where $\widehat{\Phi}=\mathcal{F}_{n}(\Phi)$ and $\mathcal{F}_{n}:\mathcal{S}(\A_{F}^n)\lra \mathcal{S}(\A_{F}^n)$ is the Fourier transform.
The residue at $s=1$ (resp. $s=0$) is 
\[
 \eta(h)\frac{\int_{[\GL_1]^1}\eta(a)^n da}{n}\wh{\Phi}(0)\quad
 \left(\text{resp.} \quad
 \eta(h)\frac{\int_{[\GL_1]^1}\eta(a)^n da}{n}\Phi(0) \right).
 \]
 \end{Prop}
 }

\subsection{Twisted Rankin--Selberg period}\label{Sec: Rankin selberg}
Now assume $E/F$ is a quadratic extension of number fields, with $\eta$ the associated quadratic character. Suppose that $\pi$ and $\pi'$ are  irreducible cuspidal representations for $\GL_n(\A_F)$. If $\omega$ and $\omega'$ denote the respective central characters, we assume $\omega\omega'=1$. Consider the Rankin-Selberg integral
\begin{equation*}
    \calp_\eta(s,\varphi,\varphi',\Phi):=\displaystyle\int_{Z_n(\A_F)\GL_n(F)\bs\GL_n(\A_F)}\varphi(g)\varphi'(g) E(g,\Phi,s,\eta) dg,
\end{equation*}
where $\varphi\in V_\pi$, $\varphi'\in V_\pi'$, and $E(g,\Phi,s,\eta)$ is the Eisenstein series \eqref{eqn: our Eisenstein} with $\Phi\in C_c^\infty(\A_{F,n})$ and $\eta=\eta_{E/F}$ is non-trivial. 

The local variant period is defined by
\[
\Psi_{\eta}(s,W_v,W_v',\Phi_v) = \int_{N_n(F_v)\backslash\GL_n(F_v)}W_v(h)W_v'(h)\Phi_v(e_n h)|\det(h)|^s\eta_v(h)dh,
\]
for $s\in \cc$, $W_v\otimes W_v'\in \calw^{\psi_v^{-1}}(\pi_{v})\boxtimes\calw^{\psi_v}(\pi_{v}')$, and $\Phi_v\in C_c^\infty(F_{v,n})$.

\begin{Prop}\label{Prop: Rankin} Let $\varphi\in V_\pi$, $\varphi'\in V_{\pi'}$, and $\Phi\in C_c^\infty(\A_{F,n})$.
   \begin{enumerate}
       \item\cite[Theorem 2.2]{CogdellGLn} The integral $\calp_\eta(s,\varphi,\varphi',\Phi)$ is meromorphic in $s$ and bounded in vertical strips away from the poles of the Eisenstein series. If $\varphi\in V_\pi$ and $\varphi'\in V_{\pi'}$ have factorizable Whittaker functions and $\Phi=\bigotimes_v\Phi_v$, then
    \[
   \calp_\eta(s,\varphi,\varphi',\Phi)= \prod_v\Psi_{\eta}(s,W_v,W_v',\Phi_v) 
    \]
    with convergence absolute and uniform away from the poles.
    \item\label{eqn: factor rankin}  The integral  $\calp_\eta(s,\varphi,\varphi',\Phi)$ satisfies the functional equation
    \[
     \calp_\eta(s,\varphi,\varphi',\Phi) =\calp_{1-s}(\widetilde{\varphi},\widetilde{\varphi}',\widehat{\Phi}),
    \]
    where $\widetilde{\varphi}(g) = \varphi({}^tg^{-1})$, and $\widehat{\Phi}$ denotes the Fourier transform. In particular, for $\Re(s)<0$, we have
    \[
    \calp_\eta(s,\varphi,\varphi',\Phi)=\prod_v\Psi_{\eta}(1-s,\widetilde{W}_v,\widetilde{W}_v',\widehat\Phi_v),
    \]
    where $\widetilde{W}_v(g) = {W}_v(w_n{}^tg^{-1})$.
        \item\label{no pole in RS} If $\pi\not\simeq \pi\otimes \eta$, then $\calp_\eta(s,\varphi,\hat{\varphi},\Phi)$ is holomorphic at $s=0$, where $\hat{\varphi}\in {\hat{\pi}}$
       \item\label{item 2: unramified} \cite[Theorem 3.3]{CogdellGLn} Assume that $v$ is a finite place such that the local extension $E_v/F_v$ is unramified. Assume both $\pi_v$ and $\psi_v$ are unramified, and that $W_v$ and $\hat{W}_v$ are the spherical Whittaker vectors normalized so that $W_v(1)=\hat{W}_v(1)=1$. If  $\Phi=\bfun_{\calo_{F_v,n}}$, then
\[
\Psi_{\eta}(s,W_v,\hat{W}_v,\Phi_v) =L (s,\pi_v \times \hat{\pi}_{v} \otimes \eta_v).
\]
   \end{enumerate} 
\end{Prop}
\begin{proof}
    The only statement not explicitly found in \cite{CogdellGLn} is \eqref{no pole in RS}. This claim is immediate when $n$ is odd, as the Eisenstein series has no poles in this case. When $n$ is even, the Eisenstein series has a pole at $s=0$, and the residue at $s=0$ of $\calp_\eta(s,\varphi,\varphi',\Phi)$ is (cf. \cite[Section 2.3.2]{CogdellGLn})
    \[
    -c\Phi(0) \int_{Z_n(\A_F)\GL_n(\A_F)\backslash\GL_n(\A_F)} \varphi(g)\hat{\varphi}(g)\eta(g)dg.
    \]
    The assumption on $\pi$ ensures that this inner product between $\pi$ and $\hat{\pi}\otimes \eta$ vanishes, giving the claim.
\end{proof}
We also consider the normalized variant
\[
\Psi_{\eta}^\natural(s,W_v,W_v',\Phi_v)= \frac{\Psi_{\eta}(s,W_v,W_v',\Phi_v)}{L (s,\pi_v \times \pi'_{v} \otimes \eta_v)}.
\]
\subsection{Unitary periods and FLO functionals}
Now let $\Pi$ be a cuspidal automorphic representation of $\GL_n(\A_E)$. For any Hermitian form $\tau\in \Herm_{n}^{\circ}(F)$, we have the associated unitary group $\U(V_\tau)\subset \Res_{E/F}\GL_n$. For $\varphi\in V_{\Pi}$, consider the \emph{unitary period} $\calp_\tau(\varphi)$ given by the integral
\begin{equation*}
\calp_\tau(\varphi)= \int_{[U(V_\tau)]}\varphi(h)dh.
\end{equation*}
By\cite{JacquetQuasisplit,FLO}, there is a cuspidal automorphic representation $\pi$ of $\GL_n(\A_F)$ such that $\Pi=\mathrm{BC}_E(\pi)$ if and only if $\Pi$ is {distinguished by $U(V_\tau)$} for any $\tau\in \Herm_{n}^{\circ}(F)$.

\subsubsection{Local functionals} Now consider a quadratic \'{e}tale extension $E/F$ of local fields. Let $\Pi\in \mathrm{Temp}(\GL_n(E))$ and denote by $\mathcal{E}(\Herm_{n}^{\circ}(F),\mathcal{W}^{\psi'}(\Pi)^\ast)$ the set of all maps 
\[
\al:\Herm_{n}^{\circ}(F)\times \mathcal{W}^{\psi'}(\Pi)\to \cc,
\] 
which are continuous and  satisfy $\al(\tau\ast g, W) = \al(\tau,\Pi(g)W)$. Note that  we have an isomorphism
\begin{align*}
\mathcal{E}(\Herm_{n}^{\circ}(F),\mathcal{W}^{\psi'}(\Pi)^\ast)&\iso \bigoplus_{\tau\in\calv_n(E/F)}\Hom_{U(V_\tau)}(\mathcal{W}^{\psi'}(\Pi),\cc)\\
\al&\mapsto (\al_\tau)_{\tau\in\calv_n(E/F)},
\end{align*}
where $\al_\tau:=\al(\tau,-)$. For any such $\al$, we consider the \emph{twisted Bessel character} $J^\al_{\Pi}:C_c^\infty(\Herm_{n}^{\circ}(F))\to \cc$ given by
\[
J_{\Pi}^\al(f') = \la \Pi(f') \al,\lam_1^\vee\ra
\]
where $\Pi(f') \al$ is the smooth functional
\[
W\mapsto \int_{\Herm_{n}^{\circ}(F)}f'(x)\al(x,W)dx,
\]
which we identify with an element of $\mathcal{W}^{{\psi'}^{-1}}(\Pi^\vee)$ via the pairing $[\cdot,\cdot]_{\Pi}$, and $\lam_1^\vee$ denotes the functional $\hat{W}\mapsto \hat{W}(e).$ Note that we may also write this as (cf. \cite[Section 2.1]{FLO})
\[
J_{\Pi}^\al(f')=\sum_{\tau\in\mathcal{V}_n(E/F)}\sum_{W'}\al_\tau(\Pi(f^\tau)W')\hat{W}'(e),  
\] where $W'$ runs over an orthonormal basis of of the Whittaker model of $\Pi$, and where $f^\tau\in C_c^\infty(\GL_n(E))$ satisfy
\[
f' = \sum_\tau p_{\tau,!}(f^\tau).
\]
Similarly, for $\pi\in\mathrm{Temp}(\GL_n(F))$, we define the \emph{Bessel character} $I_{\pi}: C_c^\infty(\GL_n(F))\to \cc$ by
\[
I_{\pi}(f) = \la \pi(f) \lam_{w_n},\lam_1^\vee\ra=\sum_{W}\lam_{w_n}(\pi(f)W)\hat{W}(e),
\]
where $\pi(f) \lam_{w_n}$ denotes the smooth functional
\[
W\mapsto \int_{\GL_n(F)}f(g)W(w_ng)dg.
\]

One of the main local results of \cite{FLO} is the following theorem.
\begin{Thm}\label{Thm: FLO functionals}
For every $\pi\in \mathrm{Temp}(\GL_n(F))$, there exists a unique \[
\al^{\pi}\in \mathcal{E}(\Herm_{n}^{\circ}(F),\mathcal{W}^{\psi'}(\mathrm{BC}_E(\pi))^\ast)
\]such that the identity
\begin{equation*}
J_{\mathrm{BC}_E(\pi)}^{\al^{\pi}}(f') =I_{\pi}(f)
\end{equation*}
holds for all pairs of test functions $(f,f')$ which are transfers in the sense defined in \cite[Section 3]{FLO}.
\end{Thm}
We refer to the functionals $\al^{\pi}$ as \textbf{FLO functionals}. When $E\cong F\times F$ is split, so that $\mathrm{BC}_E(\pi) \cong \pi\otimes \pi$, these functionals are very simple \cite[Corollary 7.2]{FLO}:
\begin{equation}\label{eqn: split FLO identity}
\al_{(h,{}^th)}^{\pi}(W'\otimes W'') = \left[\mathcal{W}(h,\pi)W',\mathcal{W}(w_n,\hat{\pi})\hat{W}''\right]_{\pi}
\end{equation}
for any $h\in \GL_n(F)$ and $W',W''\in \mathcal{W}^{\psi}(\pi)$. 
\begin{Lem}\cite[Lemma 3.9]{FLO}\label{Lem: unramified FLO}
Assume that $F$ is non-archimedean of odd residue characteristic. Further assume that $E/F$ is an unramified extension and that $\psi'$ has conductor $\calo_{F}$. Let $\Pi=\mathrm{BC}_E(\pi)$ be unramified and let $W_0\in \mathcal{W}^{\psi'}(\Pi)$ denote the normalized spherical vector. Then for any $\tau\in \Herm_{n}^{\circ}(\calo_{F})$, we have
\[
\al_{\tau}^{\pi}(W_0) = L(1,\pi\times {\pi}^\vee\cdot \eta).
\]
\end{Lem}

We thus re-normalize and define
\begin{equation}\label{spherical FLO}
\al_\tau^{\pi,\natural}(W) = \frac{\al_\tau^{\pi}(W)}{ L(1,\pi\times {\pi}^\vee\cdot \eta)},
\end{equation}
 for $W\in \mathcal{W}^{\psi'}(\Pi)$. Returning to the global extension of number fields $E/F$, we have the following factorization of unitary periods of cusp forms.
\begin{Prop}\cite[Theorem 10.2]{FLO}\label{Prop: unitary periods factor}
Let $\pi$ be an irreducible cuspidal automorphic representation of $\GL_n(\A_F)$, and let $\Pi = \mathrm{BC}_E(\pi)$. Then for any $\tau\in \Herm_{n}^{\circ}(F)$ and any factorizable $\varphi\in V_{\Pi}$, we have
\begin{equation}\label{eqn: period product}
\calp_\tau(\varphi) = 2L(1,\pi\times {\pi}^\vee\cdot \eta) \prod_v\al_\tau^{\pi_v,\natural}(W_v),
\end{equation}
where $W^{{\Pi}}(g: \varphi) = \prod_v{W}_v(g)$.
\end{Prop}
\quash{
\begin{Thm}[FLO] For any $\tau$, there exists  $\alpha_\tau^\pi \in \Hom_{\U_\tau}(W_\Pi,\BC)$ that is characterized by the following identity. For any matching functions $f_F \in C_c^\infty(\GL_n(F))$ and $f_E \in C_c^\infty(\GL_n(E))$ in the sense of Jacquet-Ye, we have \[L(1,\pi \times \pi^\vee \otimes \eta)\sum_{\varphi \in \textup{OB}(W_\pi)} (\pi(f_F)\varphi)(e)\ov{\varphi(e)}=\sum_{\varphi \in \textup{OB}(W_\Pi)} \alpha_\tau^\pi(\Pi(f_E)\varphi)\ov{\varphi(e)}.\]
	Here, we use the normalized inner product $(-,-)$ on the Whittaker models.
\end{Thm} 
We remark that the extra $L(1,\pi \times \pi^\vee \otimes \eta)$ from \cite{FLO} is because we use different normalization on the inner product, where the unnormalized version is used in \cite{FLO}. Also, it will not be important for us to define the transfer of Jacquet-Ye. What matters is that the following functions are transfers:
}

\subsection{Global relative characters}
Let $E/F$ be a quadratic extension of number fields.
\begin{Def}\label{Def: unitary rel char}
Fix a pair $\tau_1,\tau_2\in\Herm_{n}^{\circ}(F)$ and let $f'\in C_c^\infty(\GL_n(\A_E))$. For any unitary cuspidal representation $\Pi$ of $\GL_{n}(\A_E)$, define
\begin{align*}
J^{\tau_1,\tau_2}_\Pi(f')
&=\sum_{\varphi \in \textup{OB}(\Pi)}\frac{\CP_{\tau_1}(\Pi(f')\varphi)\CP_{\tau_2^{-1}}(\hat{\varphi})}{\la\varphi,\hat{\varphi}\ra_{Pet}}
\end{align*}
where the sum runs over an orthogonal basis for $\Pi$\quash{, and $y^\tau = y^{-\ast} = w_n{}^T{y}^{-\sig}w_n$}. This gives a distribution on $C_c^\infty(\GL_n(\A_E))$, called the \textbf{unitary relative character of $\Pi$}. 
\end{Def}

We also have the global relative character on the linear side.
\begin{Def}\label{Def: relative char linear}
For any cuspidal automorphic representation $\pi$ of $\GL_{n}(\A_F)$ and $f\in C_c^\infty(\GL_n(\A_F))$ and $\Phi\in C_c^\infty(\A_{F,n})$, define 
\begin{align*}I_{\pi}(f\otimes \Phi)&=\sum_{\varphi \in \textup{OB}(\pi)}\frac{   \calp_0(\pi(f)\varphi,\hat{\varphi},\Phi)}{\la\varphi,\hat{\varphi}\ra_{Pet}}\\
&=\sum_{\varphi \in \textup{OB}(\pi)} \frac{ 1}{\la\varphi,\hat{\varphi}\ra_{Pet}}\int_{Z_n(\A_F)\GL_n(F)\bs\GL_n(\A_F)}(\pi(f)\varphi)(g)\hat{\varphi}{(g)}E(g,\Phi,s,\eta)dg \bigg \vert_{s=0},
\end{align*}
where the sum runs over an orthogonal basis for $\pi$. This extends linearly to give a distribution on $C_c^\infty(\GL_n(\A_F)\times \A_{F,n})$, called  the \textbf{linear relative character of $\pi$}. 
\end{Def}

\subsection{Local relative characters}
Now assume that $F$ is a local field and let $E/F$ be a quadratic \'{e}tale $F$-algebra. When $F$ is Archimedean, we always assume $E=F\times F$.

We now define the local relative characters, 
beginning with the unitary side. 
\begin{defn} Let $\Pi=\mathrm{BC}(\pi)$ be a generic irreducible admissible representation of $\GL_n(E)$. For any choice of $\tau_1,\tau_2\in\Herm_{n}^{\circ}(F)$ and $f'\in C^\infty_c(\GL_n(E))$, we define the \textbf{(local) unitary spherical character of $\Pi$} by 
	\begin{equation}\label{eqn: unitary local} 
	    J_{\Pi}^{\tau_1,\tau_2}(f')=\sum_{W \in \textup{OB}(W_\Pi)}\frac{\alpha_{\tau_1}^{\pi}(\Pi(f')W)\alpha_{\tau_2^{-1}}^{\pi^\vee}(\hat{W})}{[W,\hat{W}]_{\Pi}}.
	\end{equation}
\end{defn}
\begin{Lem}\label{Lem: unramified unitary local rel char}
     Assume that $F$ is has odd residue characteristic and $E/F$ is an unramified extension. Assume that $\Pi$ is unramified and that $\psi$ has conductor $\calo_F$. When $\tau_1,\tau_2\in\Herm_{n}^{\circ}(\calo_F)$,
    \[
    J_{\Pi}^{\tau_1,\tau_2}(\bfun_{\GL_n(\calo_E)})=\frac{L(1,\pi\times \pi^\vee\otimes \eta)}{L(1,\pi\times \pi^\vee)},
    \]
    where $\pi$ is an unramified admissible representation of $\GL_n(F)$ satisfying $\Pi = \mathrm{BC}(\pi)$.
\end{Lem}
\begin{proof}
  This follows directly from \eqref{eqn: local inner product}, Lemma \ref{Lem: unramified FLO}, and the identity
 \begin{equation}\label{L value relation}
L(s,\Pi\times \Pi^\vee) = L(s,\pi\times \pi^\vee)L(s,\pi\times \pi^\vee\cdot \eta).\qedhere
\end{equation}
\end{proof}
We also define the \emph{normalized} relative character $J^{\tau_1,\tau_2,\natural}_{\Pi}$ by replacing each functional by its normalized variant. When $\tau_1,\tau_2\in \Herm_{n}^{\circ}(\calo_F)$, then $J_{\Pi}^{\tau_1,\tau_2,\natural}(\bfun_{\GL_n(\calo_E)}) =1$.

 We now consider the linear relative character. 

\begin{Def} Let $\pi$ be a generic irreducible admissible representation of $\GL_n(F)$. The \textbf{(local) linear relative character of $\pi$} is defined as 
	\begin{align}\label{eqn: linear local rel char}
 I_{\pi}(f \otimes \Phi)&= \sum_{W\in \textup{OB}(W_{\pi})} \frac{\Psi_\eta(1,\hat{\pi}(f^\theta)\hat{W},{W},w_0\cdot\hat{\Phi})}{[W,\hat{W}]_{\pi}}\\
 &=
 \sum_{W\in \textup{OB}(W_{\pi})} \frac{\int_{N_n(F) \bs \GL_n(F)}(\pi(f)W)(h^\theta){W}{(h)}\hat{\Phi}(e_nhw_n)|h|\eta(h)dh}{[W,\hat{W}]_{\pi}}.\nonumber
 \end{align}  
	Here $\hat{\Phi}$ denotes the Fourier transform induced by $\psi$. This extends linearly to $C_c^\infty(\GL_n(F)\times F_n)$.
\end{Def}
\begin{Lem}\label{Lem: linear rel char unit}
    Assume that $F$ is non-archimedean and $E/F$ is an unramified extension. Assume that $\pi$ is unramified and that $\psi$ has conductor $\calo_F$. Then 
    \[
    I_{\pi}(\bfun_{\GL_n(\calo_F)} \otimes \bfun_{\calo_{F,n}})=\frac{L(1,\pi\times \pi^\vee\otimes \eta)}{L(1,\pi\times \pi^\vee)}.
    \]
\end{Lem}
\begin{proof}
Let $W_0$ denote the normalized spherical Whittaker function for $\pi$. Our assumption on $\psi$ gives $\widehat{\bfun_{\calo_{F,n}}} = \bfun_{\calo_{F,n}}$, and clearly $w_n\cdot \bfun_{\calo_{F,n}}= \bfun_{\calo_{F,n}}$. Thus
\[
I_{\pi}(\bfun_{\GL_n(\calo_F)} \otimes \bfun_{\calo_{F,n}}) =\frac{\int_{N_n(F) \bs \GL_n(F)}W_0(g)\hat{W}_0{(g)}\bfun_{\calo_{F,n}}(e_ng)|g|\eta(g)dg}{[W_0,\hat{W}_0]_{\pi}}.
\]
Recalling that the denominator is $L(1,\pi \times \pi^\vee)$, it suffices to compute the numerator. This is given in Proposition \ref{Prop: Rankin}\eqref{item 2: unramified}, proving the claim.
\end{proof}
We also define the \emph{normalized} relative character $I^\natural_{\pi}$ by replacing each functional by its normalized variant. Then $I^\natural_{\pi}(\bfun_{\GL_n(\calo_F)} \otimes \bfun_{\calo_{F,n}}) =1$.

 \subsection{Factorizable distributions} We now show our global distribution on both sides decompose into local ones.  We begin on the linear side. 

\begin{Prop}\label{Prop: factorize linear} Assume that the automorphic base change $\Pi=\mathrm{BC}(\pi)$ is cuspidal. For any $\widetilde{f}=\bigotimes_v\widetilde{f}_v\in C_c^\infty(\GL_n(\A_F)\times \A_{F,n})$, we have
	\[I_\pi(\wt{f})=L(1,\pi \times \pi^\vee \otimes \eta)\frac{\vol(F^\times \bs \BA_F^1)}{\textup{Res}_{s=1}L(s,\pi \times \pi^\vee)}\prod_v I^\natural_{\pi_v}(\wt{f_v}).\]
\end{Prop}
\begin{proof} Note that the assumption that $\Pi$ is cuspidal implies $\pi \neq \pi \otimes \eta$ so that the Ranking--Selberg integral is holomorphic at $s=0$ and $1$. By linearity, it suffices to consider pure tensors $f\otimes \Phi$. For each function $\Phi\in C_c^\infty(\A_{F,n})$, we write $\wh{\Phi}$ as its Fourier transform. We also define $\hat{f}(g) = f({}^tg^{-1})$. 

Recall that $I_\pi$ is defines as a specialization at $s=0$. Consider now the evaluation at $s=1$. These are related by the functional equation of the Rankin-Selberg integral.
More precisely, we apply  \eqref{eqn: inner product} and Proposition \ref{Prop: Rankin}(\ref{eqn: factor rankin}) to obtain the factorization
	\begin{equation}\label{eqn: almost factorization linear}
	    {I}_\pi(f\otimes \Phi)=L(1,\pi \times \pi^\vee \otimes \eta)\frac{\vol(F^\times \bs \BA_F^1)}{\textup{Res}_{s=1}L(s,\pi \times \pi^\vee)}\prod_v  \wh{I}^\natural_{\pi_v}({f}_v \otimes {\Phi}_v),
	\end{equation}
    where the local factor on the right-hand side is the normalized version of the relative character
    \[
    \wh{I}_{\pi_v}({f}_v \otimes {\Phi}_v)= 
 \sum_{W\in \textup{OB}(W_{\pi})} \frac{\Psi_{\eta}(1,\widetilde{{\pi}({f}_v)W_v},\widetilde{\hat{W}}_v,\widehat\Phi_v)}{[W_v,\hat{W}_v]_{\pi}}.
    \]
    A simple calculation with the local Rankin-Selberg integrals reduces this to
 \begin{align*}
     \sum_{W\in \textup{OB}(W_{\pi})} \frac{\Psi_{\eta}(1,\hat{\pi}({f}^\theta_v)\hat{W}_v,{{W}}_v,w_n\cdot\widehat\Phi_v)}{[W_v,\hat{W}_v]_{\pi}}
\end{align*}
This last expression is precisely $I_{\pi_v}({f_v}\otimes\Phi_v)$. Renormalizing the various functionals, the claimed identify follows.
\end{proof}
We now consider the unitary side. 

\begin{Prop}\label{Prop: factorize unitary}
 Fix $\tau_1,\tau_2\in \calv_n(E/F)$. Consider an irreducible cuspidal automorphic representation $\Pi$ of $\GL_n(\A_E)$. If $J_{\Pi}^{\tau_1,\tau_2}$ is not identically zero, let $\pi$ be an irreducible cuspidal automorphic representation of $\GL_n(\A_F)$ such that $\Pi=\mathrm{BC}(\pi)$.  For $f'\in C_c^\infty(\GL_n(\A_E))$, we have
	\[J_\Pi^{\tau_1,\tau_2}(f')=4L(1,\pi\times\pi^\vee\otimes \eta)\frac{\vol(E^\times \bs \BA_E^1)}{\textup{Res}_{s=1}L(s, \pi \times \pi^\vee)}\prod_v J_{\Pi_v}^{\tau_1,\tau_2}(f'_{v}).\]
\end{Prop}
\begin{proof} This follows directly from \eqref{eqn: inner product}, Proposition  \ref{Prop: unitary periods factor},  and the definitions.
\end{proof}

\section{Comparison of relative trace formulas}\label{Section: comparison proof}
In this section, we establish a comparison of relative trace formulas first studied in \cite{Xiaothesis}, the most refined global statement being Theorem \ref{Thm: main global tool}. We then turn to comparing local relative characters, establishing a comparison in the split case, the unramified case, and a weak general comparison.
\subsection{Relative trace formulas}
Let $E/F$ be a quadratic extension of number fields such that each archimedean place of $F$ splits in $E$. Let $S$ denote the finite set of places of $F$ which ramify in $E$. 
\quash{\begin{Rem}\label{Rem: our globalization}
    The constraint that $E/F$ be everywhere unramified is artificial, and is included to simplify certain statements. We remark that our application to proving Theorem \ref{Thm: first step main} allows for such a restriction. Indeed, we only require in \S \ref{Section: final proof} that $E/F$ is split at all archimedean places and globalizes a given unramified extension of $p$-adic fields at a certain place. The existence of such a globalization which is also unramified everywhere is assured (for example) by \cite[Theorem 1]{Moret-Bailly}.
\end{Rem}}Let $Z\subset \GL_n$  and $Z_E\subset \Res_{E/F}(\GL_n)$ the central tori. We fix a unitary central character $\omega: [Z]\lra \cc^\times$, and denote by $\omega'=\omega\circ\Nm_{E/F}$ the induced character on $[Z_E]$.  

\subsubsection{Nice transfer}
Fix $\tau_0\in \calv_n(E/F)$ and consider the $\U(V_0)$-action on $\Herm_{n}^{\circ}(F)$. Let $f'=\otimes_vf'_v\in C^\infty_c(\Herm_{n}^{\circ}(\A_F))$. By Theorem \ref{Thm: Xiao's transfer}, we obtain $\wt{f}\in C_c^\infty(\GL_n(\A_F)\times \A_{F,n})$ such that $f'_v$ matches $\wt{f}_v$ at each place $v$. We may find $f^\tau\in C_c^\infty(\GL_n(\A_E))$ for each $\tau\in \calv_n(E/F)$ such that
\[
f' = \sum_x p_{\tau,!}(f^\tau);
\] we also say that $\wt{f}$ and $\{f^\tau\}_\tau$ are transfers. Note that $f^\tau\equiv 0$ for all but finitely many $\tau\in \calv_n(E/F)$ by the fundamental lemma. 

We introduce the constraints on test functions to affect our simple trace formula.
\begin{Def}\label{Def: nice transfer 2} Fix a central character $\omega$ for $\GL_n(\A_F)$ and let $\omega'$ denote its base change.    We say that the matching test functions $\wt{f}=\bigotimes_v\wt{f}_v$ and $f'=\bigotimes_v{f}'_v$ form a \textbf{($\omega$-)nice transfer} if the following conditions are satisfied:
    \begin{enumerate}
        \item There exists a non-archimedean place $v_1$ which splits in $E$ and a finite union $\Omega$ of cuspidal Bernstein components of $\GL_n(F_{v_1})$ such that (noting that $\calv_{n,v}=\{\tau_n\}$ in the split case)
\[
f^{\tau_n}_{v_1} = \Phi_{v_1}(0)f_{v_1}\in C^\infty_c(\GL_n(F_{v_1}))_{\Omega},
\]
where $\wt{f}_{v_1} = f_{v_1}\otimes \Phi_{v_1}$ for appropriate $\Phi_{v_1}\in C_c^\infty(F_{v_1,n})$ with $\Phi_{v_1}(0)\neq 0$.
        \item There exists a non-archimedean place $v_2$ which splits in $E$ such that the local transfer is \emph{regular elliptic} in the sense that $f^{\tau_n}_{v_2} = \Phi_{v_2}(0)f_{v_2}\in C^\infty_c(\GL_n(F_{v_2}))$ with
        \[
         \supp(f_{v_2})\subset \{x\in \GL_n(F_{v_2})\::\: x\text{ is $Z$-regular semi-simple, } T_x\simeq \Res_{L_x/F_{v_2}}(\Gm) \},
        \]
        where $L_x$ is a degree $n$ field extension of $F_{v_2}$ and $\wt{f}_{v_2} = f_{v_2}\otimes \Phi_{v_2}$ with $\Phi_{v_2}(0)\neq 0$.
    \end{enumerate}
    In this case, we say that $\wt{f}$ (resp., $f'$ or $\{f^\tau\}_\tau$) is an $\omega$-nice (resp., $\omega'$-nice) test function. 
\end{Def}
It is clear that nice transfers exist.
\subsubsection{Recollection on Galois cohomology}
 Due to our use of nice transfers in the trace formula, we do not need to consider the full pre-stabilization of the elliptic locus (cf. \cite[Section 10]{LeslieEndoscopy}). Nevertheless, we need to recall some basic notations from abelianized Galois cohomology from \cite{LabesseBasechange}. Assume that $\I\subset \rH$ is an inclusion of connected reductive groups.

  For a place $v$ of $F$, we let $F_v$ denote the localization. We set
  \[
  \mathfrak{C}(\I,\rH;F_v):=\ker[H^1_{\mathrm{ab}}(F_v,\I)\to H^1_{\mathrm{ab}}(F_v,\rH)];
  \]
  We also have the global variant
  \begin{equation*}
       \mathfrak{C}(\I,\rH;\A_{F}/F):=\coker[H^0_{\mathrm{ab}}(\A_{F},\rH)\to H^0_{\mathrm{ab}}(\A_{F}/F,\I\bs\rH)],  
  \end{equation*}
 equipped with a natural map 
  \begin{equation}\label{eqn: adelic map on C-groups}
    \mathfrak{C}(\I,\rH;\A_{F}):=\prod_v\mathfrak{C}(\I,\rH;F_v)\lra   \mathfrak{C}(\I,\rH;\A_{F}/F).  
  \end{equation}
This induces a localization map $ \mathfrak{C}(\I,\rH;\A_{F}/F)^D\to  \prod_v\mathfrak{C}(\I,\rH;F_v)^D$ on Pontryagin dual groups.
  
The group $\mathfrak{C}(\I,\rH;\A_{F}/F)$ sits in an exact sequence \cite[Proposition 1.8.4]{LabesseBasechange}
\begin{equation}\label{eqn: exact sequence Galois cohom}
  \ker^1_{\mathrm{ab}}(F,\rH)\to    \mathfrak{C}(\I,\rH;\A_{F}/F) \lra H^1_{\mathrm{ab}}(\A_{F}/F, \I)\lra H^1_{\mathrm{ab}}(\A_{F}/F, \rH),  
\end{equation}
where $\ker^1_{\mathrm{ab}}(F,\rH):=\ker[H^1_{\mathrm{ab}}(F,\rH)\to H^1_{\mathrm{ab}}(\A_{F},\rH)]$. Note that since $\I$ and $\rH$ are connected, this implies $\mathfrak{C}(\I,\rH;\A_{F}/F)$ is a finite group. Moreover, $\ker^1_{\mathrm{ab}}(F,\rH)=0$ if and only if $\rH$ satisfies the Hasse principle \cite[Corollaire 1.6.11]{LabesseBasechange}. We note that for any unitary symmetric pair $(\G,\rH)$ considered in this paper, it is well-known that $\rH$ satisfies the Hasse principle so that $\ker^1_{\mathrm{ab}}(F,\rH)=0$.

\subsubsection{Global distributions}
For $f\in C_c^\infty(\GL_n(\A_E))$, we consider the kernel function
\[
K_{f,\omega'}(g_1,g_2) = \int_{[Z_E]}\left(\sum_{\ga\in \GL_n(E)}f(g_1^{-1}\ga g_2a)\right)\omega'(a)da
\]
For the rest of this section, we fix $\tau_0\in \calv_n(E/F)$ and set $V_0:=V_{\tau_0^{-1}}$ (compare with Theorem \ref{Thm: first step main}).
\begin{Prop}\label{Prop: unitary RTF}
   Suppose that $\{f^\tau\}_\tau$ are  $\omega'$-nice. The integral
    \[
    J^{\tau_0}_{\omega'}(f'):= \sum_{\tau\in \calv_n(E/F)}\displaystyle\int_{[\U(V_{\tau})]}\int_{[\U(V_{0})]}K_{f^\tau,\omega'}(h_1,h_2)dh_1dh_2
    \]
    is absolutely convergent and admits the geometric expansion
    \[
      J^{\tau_0}_{\omega'}(f'):={(2^{1-|S|}L(0,\eta))}\sum_{x\in \Nm(E^\times)\backslash\Herm_{n}^{\circ}(F)/ U(V_0)}\vol([T_x])\prod_v\Orb^{\U(V_0)}_{\omega'}(f'_v,x).
    \]
    Finally, we have the spectral expansion
    \[
    J^{\tau_0}_{\omega'}(f')=\sum_{\tau\in \calv_n(E/F)}\sum_{\Pi}J^{\tau,\tau_0}_\Pi(f^\tau),
    \]
    where the sum ranges over irreducible cuspidal representations $\Pi$ of $\GL_n(\A_E)$ with central character $\omega'$, and where $J^{\tau,\tau_0}_{\Pi}(f^\tau)$ is the relative character defined in Def. \ref{Def: unitary rel char}.
\end{Prop}
\begin{proof}
    This is a special case of the simple trace formula established in \cite[Theorem 18.2.5]{GetzHahnbook}. Variants of the convergence and geometric expansion statements are proved in \cite[Props 5.5 and 5.6]{Xiaothesis}, with the factor $\vol([Z_E^1]) = {2^{1-|S|}L(0,\eta)}$ mistakenly dropped. The final statement regarding the spectral expansion is a standard consequence of the definition of $\omega'$-nice text functions (cf. \cite[Prop. 9.6]{LeslieUFJFL}).
\end{proof}

In fact, our assumption that $f'= \sum_\tau p_{\tau,!}(f^\tau)= \bigotimes_v f'_v$ has elliptic support at a split prime implies that only terms $x\in \Herm_n^\circ(F)$ with $T_y$ $E$-simple contribute the geometric expansion. Without this constraint, the distribution $J^{\tau,\tau_0}$ requires stabilization (cf. \cite[Section 3]{LeslieUFJFL}, where the endoscopic fundamental lemma for the Hecke algebra is worked out in this case). 

Recall that the map $\car_{0}:\Herm_n^\circ\to \A^n$ sending $x$ to the coefficients of $\det(tI_n-\tau_0x)$ is the categorical quotient for the $\U(V_0)$-action on $\Herm_n^\circ$; we let $\A^{n,rss}:=\Herm_n^{\circ,rss}\sslash\U(V_0)\subset \A^n$ denote the quotient of the $\U(V_0)$-regular semi-simple locus. Finally, we note that the action of $Z_E\simeq\Res_{E/F}(\Gm)$ on $\Herm_n^\circ$ descends to an action on $\A^{n,rss}$.

\begin{Cor}\label{Cor: stable}
    Suppose that $f' = \sum_\tau p_{\tau,!}(f^\tau)$ is $\omega'$-nice. Then 
    \[
          \frac{1}{2^{1-|S|}L(0,\eta)}J^{\tau_0}_{\omega'}(f')=\sum_{[x]\in \Nm(E^\times)\backslash\A^{n,rss}(F)}\vol([T_x])\prod_v\SO^{\U(V_0)}_{\omega'}(f'_v,x).
    \]
    where the stable orbital integral is defined in \S \ref{Section: dealing with center}.
\end{Cor}
\begin{proof}
    As noted above, the elliptic assumption support at $v_2$ ensures that if a $Z$-regular semi-simple element $x\in \Herm_{n}^{\circ}(F)$ contributes to the above expansion, then $$F[x] = F[T]/(\car_{\tau_0x}(T))$$ must be a field extension of $F$ linearly disjoint from $E$. This implies that 
    \begin{equation}\label{simple stabilizer}\T_x \simeq \Res_{F[x]/F}\Res_{E[x]/F[x]}^1(\Gm),
    \end{equation}
    so that $|H^1(\A_{F}/F,\T_x)| = 2$.
    
     Considering the geometric expansion of $J_{\omega'}^{\tau_0}(f')$, we may collect terms in stable $\U(V_0)$-orbits to give the expression
     \[
    \frac{1}{2^{1-|S|}L(0,\eta)}J^{\tau_0}_{\omega'}(f')= \sum_{x\in \Nm(E^\times)\backslash\Herm_{n}^{\circ}(F)/_{st} U(V_0))}\vol([\T_x])\sum_{\stackrel{y\in \Herm_{n}^\circ(F)/U(V_0)}{x\sim_{\mathrm{st}}y}}\prod_v\Orb^{\U(V_0)}_{\omega'}(f'_v,y),
     \]
     where $\T_x\simeq\T_{y}$ for any $y$ in the fiber. The outer sum is naturally identified with the sum over $\Nm(E^\times)\backslash\A^{n,rss}(F)$ by identifying stable regular semi-simple orbits with the fibers of $\car_{0}$. 
     
     The pre-stabilization of the inner sum may now be worked out via standard tools of (abelianized) Galois cohomology \cite[Chapter 1]{LabesseBasechange} and Fourier inversion on the finite group $\mathfrak{C}(\T_x,\U(V_0);\A_{F}/F)^D$ using the exact sequence \eqref{eqn: exact sequence Galois cohom} (see \cite[Proposition 4.2.1]{LabesseBasechange}), and we find
         \[
     \sum_{\stackrel{y\in \Herm_{n}^\circ(F)/U(V_0)}{\car_{0}(x)=\car_{0}(y)}}\prod_v\Orb^{\U(V_0)}_{\omega'}(f'_v,y)= \frac{\tau(\U(V_0))}{\tau(\T_x)d_x}\sum_\ka\prod_v\Orb^{\U(V_0),\ka}_{\omega'}(f'_v,x),
     \]
     where $\ka\in \mathfrak{C}(\T_x,\U(V_0),\A_{F}/F)^D$, $\tau(\rH)$ denotes the Tamagawa number of $\rH$, and 
     \[
     d_x = \#\coker[H^1(\A_{F}/F, \T_x)\to H^1_{\mathrm{ab}}(\A_{F}/F, \U(V_0))].
     \] The local $\ka$-orbital integrals correspond to the image $(\ka_v)_v$ of $\ka$ along the map dual to \eqref{eqn: adelic map on C-groups}. This product is well defined by \cite[Proposition 4.11]{Lesliedescent}, which implies that all but finitely many of the terms in the product reduce to $1$. More precisely, any global representative $y\in\Herm_n^\circ(F)$ of the class $[x]$ lies in $\Herm_n^\circ(\calo_{F_v})$ and gives an absolutely semi-simple element  for all but finitely many localizations. The claim now follows from \emph{ibid.} for any such place where $f'_v= \bfun_{\Herm_n^\circ(\calo_{F_v})}$.

We compute $$|H^1_{\mathrm{ab}}(\A_{F}/F, \U(V_0))|=\tau(\U(V_0))=2.$$ Since we only consider stable orbits such that \eqref{simple stabilizer} holds, the map on Galois cohomology is a bijection, so that $d_x=1$ and $\tau(\T_x) =2$. Thus, the right-hand side simplifies as
     \[
      \frac{\tau(\U(V_0))}{\tau(\T_x)d_x}\sum_\ka\Orb_{\omega'}^{\U(V_0),\ka}(f',x)= \prod_v\SO^{\U(V_0)}_{\omega'}(f'_v,x),
     \]
     proving the claim.
\end{proof}

We now consider the linear side. For $f\otimes \Phi\in C_c^\infty(\GL_n(\A_F)\times \A_{F,n})$, we consider the corresponding automorphic kernel $K_{f,\omega}(x,y)$ for $f$. When $\{f^\tau\}_\tau$ and $f\otimes \Phi$ is $\omega$-nice, $K_{f,\omega}(x,x)$ is of rapid decay \cite[Prop. 5.6]{Xiaothesis}. We may thus consider the following integral:
\begin{equation}
    I_\omega(f\otimes \Phi,s):= \displaystyle\int_{Z(\A_F)\GL_n(F)\bs\GL_n(\A_F)}K_{f,\omega}(g,g)E(g,\Phi,s,\eta)dg,
\end{equation}
where 
$
E(g,\Phi,s,\eta)
$ is defined in \S \ref{eqn: our Eisenstein}. Proposition \ref{Prop: Eisenstein properties} along with the rapid decay of $K_{f,\omega}(g,g)$ imply that $I_\omega(f\otimes \Phi,s)$ is absolutely convergent away from the poles of the Eisenstein series, and defines a meromorphic function in $s\in \cc$. For a fixed $s\in \cc$, this extends linearly to give a distribution on $C^\infty_c(\GL_n(\A_F)\times \A_{F,n})$.

\begin{Prop}
     Suppose that $\wt{f}=\bigotimes_v\wt{f}_v$ is $\omega$-nice. 
     \begin{enumerate}
         \item\label{item: local OIs aa} Fix $(z,w)\in \GL_n(F)\times F_n$ strongly regular. Recall the local orbital integrals from \eqref{eqn: Jacuqet OI with s}. For all but finitely many places $v$, we have the equality
         \[
               \Orb^{\GL_n(F),\eta,\natural}_s(\wt{f}_v,(z,w))= 1.
         \]
         \item  When $\Re(s)>1$, we have the geometric expansion
         \begin{equation*}
                  I_\omega(\wt{f},s) = \displaystyle\int_{Z(F)\backslash Z(\A_F)}\sum_{(z,w)/\sim}L(s,T_z,\eta)\prod_v \Orb^{\GL_n(F),\eta,\natural}_s(\wt{f}_v,(za,w)) \omega(a)da,
         \end{equation*}
         where the sum ranges over strongly regular pairs $(z,w)$ with $z\in \GL_n(F)$ elliptic and the equivalence relation on $(z,w)$ is modulo $\GL_n(F)$.
     \end{enumerate}
\end{Prop}
\begin{proof}
    The first claim is established in \cite[Props. 5.3 and 5.8]{Xiaothesis}. The second claim is established in \cite[Prop. 5.6 and Section 5.4.3]{Xiaothesis} under slightly different constraints on the test function $\wt{f}$, but with an identical proof once we explain why only strongly regular pairs $(z,w)$ with $z\in \GL_n(F)$ regular elliptic appear in the sum. This follows from the ellipticity assumption at the split place $v_2$, which by definition requires that $F[z]\simeq L$ for some degree-$n$ field extension $L/F$. Since the span of $\{w,wz,\ldots,wz^{n-1}\}\subset F_n$ is a quotient of $F[z]$, it follows that this span must be $n$-dimensional and thus all of $F_n$. But this is precisely the constraint of $(z,w)$ being strongly regular.
\end{proof}

\begin{Prop}\label{Prop: Eisenstein RTF}
Let $\omega$ be a unitary central character for $\GL_n(\A_F)$. If $\wt{f}=\bigotimes_v\wt{f}_v$ is $\omega$-nice, then $I_\omega(\wt{f},s)$ is holomorphic at $s=0$ and we have the equality
\begin{equation}\label{eqn: linear RTF}
  2\sum_{(z,w)/\sim}L(0,T_z,\eta)\prod_v \Orb^{\GL_n(F),\eta,\natural}_{\omega_v}(\wt{f}_v,z)= I_\omega(\wt{f})+I_{\omega\eta}(\wt{f})=\sum_{\pi}I_{\pi}(\wt{f}),
\end{equation}
where the first sum ranges over $(z,w)$ modulo $\GL_n(F)\times\Nm(Z_E(F))$ and the second sum runs over irreducible cuspidal automorphic representations of $\GL_n(\A_F)$ with central character either $\omega$ or $\omega\eta$.
\end{Prop}
\begin{proof}
We begin by noting that for any strongly regular $(z,w)\in \GL_n(F)\times F_n$ with $z$ $Z$-regular semi-simple, the product $\prod_v \Orb^{\GL_n(F),\eta,\natural}_s(\wt{f}_v,(z,w))$ is holomorphic at $s=0$. Indeed, all but finitely many factors are equal to $1$ by the preceding proposition and the holomorphicity of the remaining factors was established in \S \ref{Section: prelim for Jacquet} (cf. \cite[Proposition 5.3]{Xiaothesis}). 

Fix now $f'\in C_c^\infty(\Herm_{n}^{\circ}(\A_F))$ so that $\wt{f}$ and $f'$ are a $\omega$-nice transfer. We now claim that if the value of $\prod_v \Orb^{\GL_n(F),\eta,\natural}_s(\wt{f}_v,(z,w))$ at $s=0$ is non-zero, then $L(s,T_x,\eta)$ is holomorphic at $s=0$. This follows from the elliptic assumption at the inert place $v_2$.

More precisely, suppose $x\in \Herm_{n}^{\circ}(F)$ is a $Z$-regular semi-simple element matching $z\in \GL_n(F)$. By definition of matching orbits, there is an isomorphism
\[
F[x]:=F[T]/(\car_{{\tau_0x}}(T))=F[T]/(\car_z(T))\simeq F[z].
\]
The assumption of elliptic support at $v_2$ implies that if 
\[
2^{|S_{v_2,1}^{ram}(z)|}\Orb^{\GL_n(F),\eta}(\wt{f}_{v_2},z)=\SO^{\U(V_0)}(f'_{v_2},x)\neq0,
\]
then $F[x]$ is a field not containing $E$. In particular, if $\prod_{v}\Orb^{\GL_n(F),\eta,\natural}_0(\wt{f}_v,(z,w))\neq0$, it follows that $\SO^{\U(V_0)}(f_{v_2}',x)\neq0$ so that $E\not\subset F[x]$. This implies that $L(s,\T_z,\eta)=L(s,\T_x^{\mathrm{op}},\eta)$ is holomorphic at $s=0$. 

Given this, we obtain the geometric expansion of $I_\omega(\wt{f},0)+I_{\omega\eta}(\wt{f},0)$ as
\begin{align*}
    \displaystyle\int_{Z(F)\backslash Z(\A_F)}\sum_{(z,w)/\sim}L(0,\T_z,\eta)&\prod_v \Orb^{\GL_n(F),\eta}(\wt{f}_v,za) \left(\omega(a)+\omega\eta(a)\right)da\\
    =2\displaystyle\int_{Z(F)\backslash Z(F)\Nm(Z_E(\A_F))}&\sum_{(z,w)/\sim}L(0,\T_z,\eta)\prod_v \Orb^{\GL_n(F),\eta}(\wt{f}_v,za) \omega(a)da\\
    =2\sum_{(z,w)/\sim'}L(0,\T_z,\eta)&\int_{\Nm(Z_E(\A_F))}\prod_v \Orb^{\GL_n(F),\eta}(\wt{f}_v,za) \omega(a)da\\
    =2\sum_{(z,w)/\sim'}L(0,\T_z,\eta)&\prod_v \Orb_{\omega_v}^{\GL_n(F),\eta}(\wt{f}_v,z),
\end{align*}
where $\sim'$ denotes modulo  $\GL_n(F)\times\Nm(Z_E(F))$.

Finally, the proof of the second equality in \eqref{eqn: linear RTF} is standard due to the cuspidal assumption at $v_1$. 
\end{proof}


\subsubsection{Comparison}

We now extract the global comparison of relative characters. For the rest of this section, we assume that $E/F$ is a quadratic extension of number fields such that every archimedean place $v|\infty$ of $F$ splits in $E$. We begin with the foundational main result of \cite[Chapter 5]{Xiaothesis}. 
\begin{Thm}
     Suppose that $f' = \sum_\tau p_{\tau,!}(f^\tau)$ and $\wt{f}$ form a $\omega$-nice transfer. Then
     \[
     \frac{1}{2^{1-|S|}L(0,\eta)}J_{\omega'}^{\tau_0}(f') = I_\omega(\wt{f})+I_{\omega\eta}(\wt{f}).
     \]
\end{Thm}
\begin{proof}
We have seen that the two geometric expansions range over matching orbits of $Z$-regular elliptic regular semi-simple elements. For each such matching pair $z\in  \GL_n(F)$ and $x\in \Herm_{n}^{\circ}(F)$, the contribution to the right-hand side is
\[
2L(0,\T_z,\eta)\prod_v \Orb^{\GL_n(F),\eta}_{\omega_v}(\wt{f}_v,z),
\]
while the contribution to the left-hand side is 
\[
\vol([\T_x])\prod_v\SO_{\omega'}(f'_v,x).
\]
The result now follows by combining Corollary \ref{Cor: transfer with central character} with the calculation in Lemma \ref{Lem: vol of centralizer} that $\vol([\T_x])=2^{1-|S_x|}L(0,\T_x^{\mathrm{op}},\eta)$ and the count
\[
|S_x| =\sum_{v\in S}|S_{v,1}(x)|= \sum_{v\in S}|S_{v,1}^{ram}(z)|
\]
Indeed, at each place $v\in S$, decompose $F[x]_v$ as in \eqref{eqn: decomp of stab at v}. Since $E_v/F_v$ is ramified, the assumption that $\eta\circ\Nm_{F_i/F_v}\not\equiv 1$ forces this character to be ramified on $F_i^\times,$ implying that $S_{v,1}(x)=S_{v,1}^{ram}(z)$. Finally, we have the identification of $L$-functions $L(s,\T_z,\eta) = L(s,\T_x^{\mathrm{op}},\eta)$, giving the necessary equality. 
\end{proof}

An immediate corollary of Propositions \ref{Prop: unitary RTF} and \ref{Prop: Eisenstein RTF} is the following.
\begin{Cor}\label{Cor: comparing full RTFs}
     Suppose that $f' = \sum_\tau p_{\tau,!}(f^\tau)$ and $\wt{f}$ is a nice transfer. Then we have the identity
     \[
     \frac{1}{2^{1-|S|}L(0,\eta)} \sum_{\tau\in \calv_n(E/F)}\sum_{\Pi}J^{\tau,\tau_0}_\Pi(f^\tau) =\sum_{\pi}I_{\pi}(\wt{f}),
     \]
     where the left-hand sum ranges  over irreducible cuspidal representations $\Pi$ of $\GL_n(\A_E)$ with central character $\omega'$ and the right-hand sum ranges over irreducible cuspidal automorphic representations of $\GL_n(\A_F)$ with central character $\omega$ or $\omega\eta$.
\end{Cor}

To obtain the necessary refined comparison, we make another use of \cite[Theorem A]{ramakrishnan2018theorem}. We introduce the following terminology to simplify a few statements.
\begin{Def}\label{Def: efficient}
    Fix a character $\omega$ of $Z(\A_F)$ and set $\omega'=\omega\circ\Nm_{E/F}$. Fix a split place $v_1$ and a supercuspidal representation $\pi_{v_1}$ of $\GL_n(F_{v_1})$. We say that the functions
\[
\text{$\wt{f}\in C_c^\infty(\GL_n(\A_F)\times \A_{F,n})$ and $\{f^{\tau}\}_{\tau\in\calv_n(E/F)}$}\subset C_c^\infty(\GL_n(\A_E))
\]
are \textbf{efficient transfers} for $(\omega,\pi_{v_1})$ if they give a $\omega$-nice transfer such that $f_{v_1}\in C_c^\infty(\GL_n(F_{v_1}))_\Omega$ where $\Omega$ is the block of $\pi_{v_1}$ ($f_{v_1}$ is said to be essentially a matrix coefficient of $\pi_{v_1}$ in this case) and that $f{\tau_n}_{v_1}$ is related to $f_{v_1}$ as in Proposition \ref{Prop: split transfer}; note that this implies that $f^{\tau}_{v_1}$ is essentially a matrix coefficient of the base change $\Pi_{v_1}\simeq \pi_{v_1}\otimes \pi_{v_1}$ of $\pi_{v_1}$.
\end{Def}
Making use of the strong multiplicity one theorem \cite[Theorem A]{ramakrishnan2018theorem}, we may separate the terms in the preceding sum. For a cuspidal automorphic representation $\Pi$ of $\GL_n(\A_E)$, we denote by $\mathcal{B}(\Pi)$ the (finite) set of cuspidal automorphic representations $\pi$ of $\GL_n(\A_F)$ such that $\Pi=\mathrm{BC}_E(\pi)$. We have the following comparison of relative characters.
\begin{Prop}\label{Prop: simple character id}
 For almost all split places $v$, we fix an irreducible unramified representation $\pi_v^0$. For a fixed split place $v_1$, we fix a supercuspidal representation $\pi_{v_1}$ of $\GL_n(F_{v_1})$. Then there exists at most one cuspidal automorphic representation $\Pi$ of $\GL_n(\A_E)$ such that if $\wt{f}$ and $\{f^\tau\}_{\tau\in\calv_n(E/F)}$ are efficient transfers for $(\omega,\pi_{v_1})$, then
\begin{equation*}
 \frac{1}{2^{1-|S|}L(0,\eta)}\sum_{\tau\in \calv_n(E/F)}J_{\Pi}^{\tau,\tau_0}(f^\tau)=\sum_{\pi\in \mathcal{B}(\Pi)}I_\pi(\wt{f}),
\end{equation*}
where the sum on the right runs over all cuspidal automorphic representations $\pi$ of $\GL_n(\A_F)$ such that
\begin{enumerate}
\item\label{property1redux}$\pi_v\cong \pi_v^0$ for almost all split $v$, 
\item\label{property2redux} $\pi_{v_1}$ is our fixed supercuspidal representation,
\end{enumerate}
and where $\Pi=\mathrm{BC}_E(\pi)$ is the base change of any $\pi$ appearing in the sum.
\end{Prop}
\begin{Rem}
    Note that the comparison above is vacuous unless $\pi=\bigotimes_v\pi_v$ is $\GL_n$-elliptic at some split place $v_2$ (this is built into the requirement that $\wt{f}$ and $\{f^\tau\}_{\tau\in\calv_n(E/F)}$ are efficient transfers). By Lemma \ref{Lem: supercuspidal non-vanishing}, it suffices for $\pi_{v_2}$ to be supercuspidal.
\end{Rem}
\begin{proof} The proof is standard, with the argument mirroring that of \cite[Proposition 2.10]{ZhangFourier} and \cite[Proposition 9.9]{LeslieUFJFL}. We include the details for completeness. 

{Let $\wt{f}$ and $\{f^\tau\}_{\tau}$ be efficient transfers as in the statement of the proposition. We may assume that all test functions are factorizable. Let $S\supset S_\infty$ be a finite set of places containing all infinite places such that all Hermitian spaces $V_\tau$ with $f^\tau\neq 0$ are unramified outside $S$. Enlarging $S$ if necessary, we may assume that
\begin{enumerate}
\item\label{inert} for any inert place $v\notin S$, $\wt{f}_v=f_v\otimes \Phi_v$ and $f^\tau_{v}$  are the functions that match by the fundamental lemma (Theorem \ref{Thm: xiao FL});
\item\label{split} for any split place $v\notin S$, $\wt{f}_v=f_v\otimes \Phi_v$ and $f^\tau_{v}$ are unramified functions which match in the sense of Proposition \ref{Prop: split transfer}. More precisely, we may assume that they match in the sense of \eqref{eqn: split unram transfer}.
\end{enumerate}
 Write $\wt{f}=\wt{f}_S\otimes [f^S\otimes \Phi^S]$, where $f^S\in \calh_{{K}^S}(\GL_n(\A_F^S))$, where $\A_F^S=\prod_{v\notin S}F_v$ and ${K}^S=\prod_{v\notin S}K_v$; similarly, we write $f^\tau=f^\tau_{S}\otimes f^{\tau,S}$ with $ f^{\tau,S}\in \calh_{{K}_E^S}(\GL_n(\A_E^S))$, where $\A_E^S=\prod_{{w}}E_w$, where $w$ ranges over places of $E$ such that $w|v$ for some $v\notin S$.

With these notations, Corollary \ref{Cor: comparing full RTFs} gives the identity
\[
  \sum_{x\in \calv_n(E/F)}\sum_{\Pi}J^{\tau,\tau_0}_\Pi(f^\tau_{S}\otimes f^{\tau,S}) =\sum_{\pi}I_{\pi}(\wt{f}_S\otimes [f^S\otimes \Phi^S]).
\]
where $\Pi$ and $\pi$ run over cuspidal automorphic representations with the prescribed central characters and supercuspidal component at $v_1$. For the unramified representations $\Pi^S$ (resp. $\pi^S$), let $\lam_{\Pi^S}$ (resp. $\lam_{\pi^S}$) be the Hecke-trace functionals of $\calh_{{K}_E^S}(\GL_n(\A_E^S))$ (resp. $\calh_{{K}^S}(\GL_n(\A_F^S))$). Then we observe (cf. \cite[proof of Proposition 2.10]{ZhangFourier}) that
\[
I_{\pi}(\wt{f}_S\otimes [f^S\otimes \Phi^S])=\lam_{\pi^S}(f^S)I_\pi(\wt{f}_S\otimes [\bfun_{K^S}\otimes \Phi^S]),
\]
and 
\[
J^{\tau,\tau_0}_\Pi(f^\tau_{S}\otimes f^{\tau,S})=\lam_{\Pi^S}(f^{\tau,S})J^{\tau,\tau_0}_\Pi(f^\tau_{S}\otimes \bfun_{K_E^{S}}).
\]
Since we only allow non-identity elements of the local Hecke algebras at at S or places of $F$ that split in $E$, we may view the above two equations as identities of linear functionals on the Hecke algebra $\calh_{{K}^{S,split}}(\GL_n(\A_F^{S,split}))$, where the superscript $split$ indicates that we only take the product over the split places outside of $S$. To see this, note that the matching \eqref{eqn: split unram transfer} sends $f^\tau_v = \phi_1\otimes \phi_2\in \calh_{K_v}(\GL_n(E_v))$ to $\phi_1^{\vee}\ast\phi^{\theta}_2\otimes\bfun_{\calo_{F,n}}$ as a special case of Proposition \ref{Prop: split transfer}, so that we obtain a linear map
\[
\calh_{{K}_E^{S,split}}(\GL_n(\A_E^{S,split}))\lra \calh_{{K}^{S,split}}(\GL_n(\A_F^{S,split}))
\]
which we compose with $I_\pi(f_S\otimes -)$ to obtain a linear functional on $\calh_{{K}_E^{S,split}}(\GL_n(\A_E^{S,split}))$.

By the infinite linear independence of Hecke characters (see \cite[Appendix]{BadulescuJRnotes} for a short proof), for any fixed $\otimes_v \pi^0_v$ we obtain the sum
\[
\sum_{x\in \calv_n(E/F)}\sum_\Pi J_{\Pi}^{\tau,\tau_0}(f^\tau)=\sum_{\pi\in \mathcal{B}}I_\pi(\wt{f}),
\]
where $\mathcal{B}$ is the set of cuspidal automorphic representations satisfying \eqref{property1redux} and \eqref{property2redux}, and where $$\Pi\in \{\Pi:  \text{for almost all split primes, }\Pi_v= \mathrm{BC}_E(\pi_v)\text{ for some }\pi\in \mathcal{B}\}.$$ Applying \cite[Theorem A]{ramakrishnan2018theorem}, we see that there is at most one representation appearing on the left-hand side. Furthermore, this implies that $\mathcal{B}=\mathcal{B}(\Pi)$.}
\end{proof}


\begin{Thm}\label{Thm: main global tool}
Fix $\tau\in\Herm_{n}^{\circ}(F)$. Let $\Pi$ be a cuspidal automorphic representation of $\GL_n(\A_E)$ with central character $\omega'$ such that
\begin{enumerate}
    \item $\Pi\cong\Pi^\sig$, and
    \item there is a split place $v_1$ and a supercuspidal representation $\pi_1$ of $\GL_n(F_{v_1})$ such that $\Pi_{v_1}\simeq \pi_1\otimes\pi_1$ is the (supercuspidal) base change of $\pi_1$.
\end{enumerate} Consider a $\omega'$-nice factorizable function $f\in C_c^\infty(\GL_n(\A_E))$ satisfying that $f_{v_1}$ is essentially a matrix coefficient of $\Pi_{v_1}$. There exists a $\omega'$-nice factorizable function $\wt{f}\in C_c^\infty(\GL_n(\A_F)\times \A_{F,n})$ matching $\{f^{\tau'}\}_{\tau'}$, where
\[
f^{\tau'} =\begin{cases}
    f&: \tau'=\tau,\\
    0&: \text{otherwise,}
\end{cases}
\] such that $(\wt{f},\{f^{\tau'}\})$  are efficient transfers for $(\omega,\pi_{v_1})$. We have the identity
\begin{equation}\label{eqn: global spectral transfer}
 \frac{1}{2^{1-|S|}L(0,\eta)}J_{\Pi}^{\tau,\tau_0}(f)=\sum_{\pi\in \mathcal{B}(\Pi)}I_\pi(\wt{f}).
\end{equation}
 If $f = \bigotimes_vf_{v}$, then $\wt{f}=\bigotimes_v\wt{f}_v$ may be chosen so that
\begin{equation}\label{eqn: super important}
2^{|S|}\prod_vJ^{\tau,\tau_0,\natural}_{\Pi_v}(f_{v}) =\prod_vI^\natural_{\pi_v}(\wt{f}_v),
\end{equation}
where $\Pi=\mathrm{BC}(\pi)$.
\end{Thm}

\begin{proof}
That such a transfer $\wt{f}\in C_c^\infty(\GL_n(\A_F)\times \A_{F,n})$ exists follows from Theorem \ref{Thm: Xiao's transfer}, Proposition \ref{Prop: split transfer}, and the properties of the $Z$-regular semi-simple loci. We now apply the previous proposition to the unramified representation $\bigotimes_v\pi_v^0$ where $v$ runs over those split places over which $\Pi$ is unramified and $\pi_v^0$ is determined by
\[
\Pi_v=\mathrm{BC}_E(\pi_v^0) \cong \pi_v^0\otimes \pi_v^0.
\]
This gives (\ref{eqn: global spectral transfer}).

Now fix $\pi\in \mathcal{B}(\Pi)=\{\pi,\pi\otimes\eta\}$. By our assumptions on the global extension $E/F$, we may assume that for each place $v$ of $F$,
\[
\supp(\wt{f}_v) \subset \{(g,w)\in \GL_n(F_v)\times F_{F_v,n}\::\: \eta_v(\det(g)) = \eta(\det(x))\}
\]
Since the  Hermitian form $x$ is global, this implies that $\wt{f}=\eta\cdot \wt{f}$, where $(\eta\cdot \wt{f})(g,w) = \eta(\det(g))\wt{f}(g,w)$.

Considering the local distribution $I_{\pi_v}$, we have
\[
I_{\pi_v\eta_{v}}(\wt{f}_v) =I_{\pi_v}(\eta_v\cdot\wt{f}_v).
\]
Combining this with the product formula in Proposition \ref{Prop: factorize linear} implies that
\[
\sum_{\pi\in \mathcal{B}(\Pi)}I_\pi(\wt{f})=I_\omega(\wt{f},0)+I_{\omega\eta}(\eta\cdot\wt{f},0)=2I_\pi(\wt{f}), 
\]
so that
\[
\frac{2^{|S|}}{2L(0,\eta)}J_{\Pi}^{\tau,\tau_0}(f)=2I_\pi(\wt{f})
\]
whenever $f$ and $\wt{f}$ are matching functions as in Proposition \ref{Prop: simple character id}. Combining this with the product formulas in Propositions \ref{Prop: factorize linear} and \ref{Prop: factorize unitary} gives \eqref{eqn: super important}, where we use the volume calculation
\[
\frac{\vol(E^\times \backslash\A_E^1)}{L(0,\eta)}=\vol(F^\times \backslash\A_F^1).\qedhere
\]
\end{proof}
\subsection{Comparison of local relative characters}
We now let $F$ denote a local field and establish a comparison of local relative characters in two settings: at an inert place with unramified representations and at a split local place. We also recall the necessary elliptic support result in the split setting.

\subsubsection{An unramified relative character identity} In this section, $E/F$ is an unramified extension of $p$-adic fields. Lemmas \ref{Lem: linear rel char unit} and \ref{Lem: unramified unitary local rel char} combine to give
\[
    I_{\pi}(\bfun_{\GL_n(\calo_F)} \otimes \bfun_{\calo_{F,n}})=   J_{\mathrm{BC}(\pi)}^{\tau_1,\tau_2}(\bfun_{\GL_n(\calo_E)})
\] whenever $\tau_1,\tau_2\in\Herm_{n}^{\circ}(\calo_F)$. We now generalize this identity to all spherical functions. For this, we recall the following crucial matching of relative orbital integrals in the context of Jacquet--Ye transfer.
\begin{Prop}\label{Prop: FLO fundamental lemma}
	Assume $E/F$ is an unramified extension of $p$-adic fields with $p$ odd and that $\psi'$ has conductor $\calo_F$. The Hironaka transform
 \[
 \Hir:\calh_{K_E}(\Herm_{n}^{\circ}(F))\iso \calh_{K}(\GL_n(F))
 \] realizes Jacquet--Ye transfer. More precisely, for any $\phi\in \calh_{K_E}(\Herm_{n}^{\circ}(F))$, the pair $\phi$ and $\Hir(\phi)$ are transfers in the sense defined in \cite[Section 3]{FLO}. 
\end{Prop}
\begin{proof}
When restricting to the image of $\calh_{K_E}(\GL_n(E))$ with respect to the diagram
\begin{equation*}
\begin{tikzcd}
&\calh_{K_{E}}(\GL_n(E))\ar[dl,swap,"-\ast \bfun_0"]\ar[dr,"BC"]&\\
\calh_{K_{E}}(\Herm_{n}^{\circ}(F))\ar[rr,"\Hir"]&&\calh_{K}(\GL_n(F)),
\end{tikzcd}
\end{equation*}
 this is a theorem of Jacquet \cite{JacquetKloosterman2}. The statement for the full Hecke algebra is a theorem of Offen \cite[Theorem 10.1]{offenjacquet}.
\end{proof}

We now state the full result.
\begin{Prop}\label{Prop: unramified rel char} Suppose $E/F$ and fix $\tau_1,\tau_2\in\Herm_{n}^{\circ}(F)$. Assume $\pi$ is unramified. Let $\phi_1,\phi_2\in \calh_{K_E}(\Herm_{n}^{\circ}(F))$, and assume that $\phi_1$ is of type $\tau_1$ and $\phi_2$ is of type $\tau_2$. Then 
	\[I_{\pi}(\Hir(\phi_1\ast\phi_2)) \otimes \bfun_{\calo_{F,n}})=J^{\tau_1,\tau_2}_{\Pi}(\varphi_1\ast\varphi_2^{\ast}),\]
\end{Prop}
where $\varphi_1,\varphi_2$ are right-$\GL_n(\calo_E)$-invariant lifts of $\phi_1$ and $\phi_2$ as in Lemma \ref{Lem: invariant lifts}.
\begin{proof} 
	We begin by unwinding the definition. One the left side, we have 
	\[I_{\pi}(\Hir(\phi_1\ast\phi_2) \otimes \bfun_{\calo_{F,n}})=\sum_W \int_{N_n(F)\bs\GL_n(F)} \hat{\pi}(\Hir(\phi_1\ast\phi_2)^\theta)\hat{W}(h){W}(h) \bfun_{\calo_{F,n}}(e_nh)|h|\eta(h)dh.\] First recall that for any $f\in \calh_K(\GL_n(F))$, one has
\begin{equation}\label{eqn: Satake involution}
    \lam_{\pi}(f) = \lam_{\pi^\vee}(f^\vee).
\end{equation} If we write $\lambda_{\pi}$ as the Hecke character defined by $\pi$, the previous sum collapses to 
 \[
 \lambda_{\pi}(\Hir(\phi_1))\lambda_{\pi}(\Hir(\phi_2))\frac{L(1,\pi\times \pi^\vee\otimes \eta)}{L(1,\pi\times \pi^\vee)},
 \]since any spherical function $\phi$ projects onto the unramified line, scales by $\lam_\pi$, and satisfies $\phi^\theta = \phi^\vee$.
 
 On the right side, we similarly have 
\[
{J^{\tau_1,\tau_2}}_{\Pi} (\varphi_1\ast (\varphi_2)^{\ast})=\frac{\alpha_{\tau_1}^{\pi}(\Pi(\varphi_1)W_0)\alpha_{\tau_2^{-1}}^{\pi^\vee}(\Pi^\vee({(\varphi^{\ast\vee}_2)})\hat{W}_0)}{L(1,\Pi\times \Pi^\vee)}
\]
where the 
equality follows from the basic property of Bessel-type distributions (cf. \cite[Lemma 2.1]{FLO}) that evaluating on a left- or right-$K_E$-invariant function $f'$ collapses the sum \eqref{eqn: unitary local} to the unramified term (this is stated only for bi$-K_E$-invariant functions in \emph{ibid.}, but the generalization is straightforward).

Note that $p_{\tau_2^{-1},!}(\varphi^{\ast\vee}_2) = \phi_2^\vee$, where 
\[
\phi_2^\vee(g^\ast \tau_2^{-1} g) = \phi_2(g^{-1}\tau_2(g^\ast)^{-1}).
\]

Comparing the two expressions, it suffices to show that
\begin{enumerate}
    \item for any $\phi\in \calh_{K_E}(\Herm_{n}^{\circ}(F))$ of type $\tau$, we have
	\begin{equation}\label{cor spectral formula}
 {\alpha_\tau^\pi(\Pi(\phi^\tau)W_0)}=L(1,\pi \times \pi^\vee \otimes \eta)\lambda_{\pi}(\Hir(\phi)),\end{equation} 
 where $\phi = p_{\tau,!}(\phi^\tau)$, and
 \item for any  $\phi\in \calh_{K_E}(\Herm_{n}^{\circ}(F))$ of type $\tau$ and unramified $\pi$, 
\begin{equation}\label{matching with involution}
\alpha_{\tau^{-1}}^{\pi^\vee}(\Pi^\vee({(\varphi^{\ast\vee})})\hat{W}_0)=L(1,\pi \times \pi^\vee \otimes \eta)\lam_{\pi}(\Hir(\phi)),
\end{equation}
where $p_{\tau^{-1},!}(\varphi^{\ast\vee}) = \phi^\vee$.
\end{enumerate} 
To establish \eqref{cor spectral formula}, Theorem \ref{Thm: FLO functionals} and  Proposition \ref{Prop: FLO fundamental lemma} combine to produce the identity 
 \begin{align*}
   \sum_{W \in \textup{OB}(W_\Pi)} \alpha_\tau^\pi(\Pi(\phi^\tau)W')\hat{W'}{(e)} &=  J_{\mathrm{BC}_E(\pi)}^{\al^{\pi}}(\phi)\\ &\overset{\text{Prop.  \ref{Prop: FLO fundamental lemma}}}{=}I_{\pi}(\Hir(\phi))=\sum_{W \in \textup{OB}(W_\pi)} (\pi(\Hir(\phi))W)(w_n)\hat{W}{(e)}
 \end{align*}
 On both side, only the normalized spherical Whittaker function appear, so that
 \[
 \frac{\alpha_\tau^\pi(\Pi(\phi^\tau)W'_0)}{L(1,\Pi\times\Pi^\vee)}=\frac{(\pi(\Hir(\phi))W_0)(w_n)}{L(1,\pi\times \pi^\vee)}=\frac{\lambda_{\pi}(\Hir(\phi))}{L(1,\pi\times \pi^\vee)},
 \]
 which directly implies \eqref{cor spectral formula}.
    
For the second statement, a simple calculation with Jacquet--Ye transfer (see \cite[Section 4]{offenjacquet}, for example) shows that if $\phi$ and $f$ are Jacquet--Ye transfers with respect to an additive character $\psi$, then  $\phi^\vee$ and $f^\vee$ are Jacquet-Ye transfers with respect to an additive character $\psi_{-1}$, where $\psi_{-1}(t) = \psi(-t)$. Applying this to $\phi\in\calh_{K_E}(\Herm_{n}^{\circ}(F))$ with Proposition \ref{Prop: FLO fundamental lemma} gives that $\phi^\vee$ and $\Hir(\phi)^\vee$ are Jacquet-Ye transfers with respect to $\psi_{-1}$. 

The value $\alpha_{\tau^{-1}}^{\pi^\vee}(\Pi^\vee({(\varphi^{\ast\vee})})\hat{W}_0)$ is independent of this change in additive character by \cite[(9.7)]{FLO} since both are of conductor $\calo_F$ and $\pi$ is unramified. We may therefore run the same argument via  Theorem \ref{Thm: FLO functionals} and  Proposition \ref{Prop: FLO fundamental lemma} as above to see that
\[
\alpha_{\tau^{-1}}^{\pi^\vee}(\Pi^\vee({(\varphi^{\ast\vee})})\hat{W}_0)=L(1,\pi \times \pi^\vee \otimes \eta)\lambda_{\pi^\vee}(\Hir(\phi)^\vee),
\]
giving the result by \eqref{eqn: Satake involution}.
\end{proof}
\begin{Rem}
    Proposition \ref{Prop: FLO fundamental lemma} only requires that the additive character have conductor $\calo_F$, so holds for both $\psi$ and $\psi_{-1}$. In particular, both $\Hir(\phi)^\vee$ and $\Hir(\phi^\vee)$ are transfers of $\phi$. It follows that
    \[
    \Hir(\phi)^\vee=\Hir(\phi^\vee),
    \]
since they have the same value under the spherical character $\lam_\pi$ for all tempered unramified $\pi$. 
\end{Rem}

\subsection{The split case}
We now assume that $F$ is a local field and that $E=F\times F$ is the split quadratic $F$-algebra. Recall our conventions from \S \ref{Section: split transfer 2}. \quash{an isomorphism $\GL_n(E)\cong \GL_n(F)\times \GL_n(F)$ such that the unitary group $\U_n\cong \GL_n(F)\hra\GL_n(E)$ is sent to
\[
\U_n(F)\cong\{(g,g^\theta)\in \GL_n(F)\times \GL_n(F): g\in \GL_n(F)\},
\]
where we recall that for $g\in \GL_k(E)$, $g^\theta = w_k{}^tg^{-1}w_k$. In particular,
\[
\Herm_{n}^{\circ}(F)\cong\{(g,g^{-\theta})\in \GL_n(F)\times \GL_n(F): g\in \GL_n(F)\}
\]
and the $\U_n$ action identifies with conjugation of $\GL_n(F)$ on itself via the first coordinate. Under this identification, a regular semi-simple element $x=(g_1,g_2)\in \GL_n(E)$ matches $z\in \GL_n(F)$ if $g_1^{-1}g_2^{\theta}$ and $z$ are conjugate in $\GL_n(F)$. Additionally, conjugacy and stable conjugacy are identified.
}
For reasons to be made clear below, we assume that $\Phi\in C_c^\infty(F_n)$ satisfies $\Phi(0)\neq 0$.

\begin{Prop}
Consider matching smooth functions $\Phi(0)(f_1\otimes f_2)\in C_c^\infty(\GL_n(F)\times \GL_n(F))$ and $f\otimes \Phi$ where $f=f_1^\vee\ast f_2^{\theta}\in C_c^\infty(\GL_n(F))$. Then for any irreducible generic representation $\pi$ of $\GL_n(F)$,
\[
J_{\mathrm{BC}_E(\pi)}(f_1\otimes f_2) = I_\pi(f\otimes \Phi).
\]
\end{Prop}
\begin{proof}
   For a generic representation $\pi$, we recall identity \cite[Appendix A]{FLO}
\begin{equation}
  \int_{N_n(F)\bs \GL_n(F)}W(g)\hat{W}{(g)}\wh{\Phi}(e_ng)|g|dg =   {\Phi}(0)[W,\hat{W}]_{\pi}.
\end{equation}
For simplicity of notation, assume $\Phi(0)=1$. Using the definition of $I_{\pi}$, one now obtains
	\begin{align}\label{eqn: linear local rel char formula}
 I_{\pi}(f \otimes \Phi) 
 &=
 \sum_{W\in \textup{OB}(W_{\pi})} \frac{[\hat{\pi}(f^\theta)\hat{W},{W}]_{\hat\pi}}{[W,\hat{W}]_{\pi}}=\sum_{W\in \textup{OB}(W_{\pi})} \frac{[W,\hat{\pi}(f^\theta)\hat{W}]_{\pi}}{[W,\hat{W}]_{\pi}},
 \end{align}  
  where we used the simple identity $[W,\hat{W}]_{\pi} = [\hat{W},W]_{\hat\pi}$. 
 
On the unitary side, \eqref{eqn: split FLO identity} implies
\begin{equation*}
\al_{(w_n,w_n)}^{\hat{\pi}}(\hat{W}_1\otimes \hat{W}_2) =  \left[\calw(w_n,\hat{\pi})\hat{W}_1,\calw(w_n,{\pi}){W}_2\right]_{\pi}=\left[\hat{W}_1,W_2\right]_{\hat{\pi}}.
\end{equation*}
This implies that
\[
\al^{\pi}_{(w_n,w_n)}(\pi(f_1)W_1\otimes \pi(f_2)W_2) = [W_1,\hat{\pi}((f_1^\vee\ast f_2^\theta)^\theta)\hat{W_2}]_{\pi}.
\]
In particular,
 \begin{align*}
     J_{\mathrm{BC}(\pi)}(f_1\otimes f_2)&= \sum_{W_1\otimes W_2}\frac{\al^{\pi}_{(w_n,w_n)}(\pi(f_1)W_1\otimes \pi(f_2)W_2)\al^{\hat{\pi}}_{(w_n,w_n)}(\hat{W}_1\otimes\hat{W}_2)}{[W_1,\hat{W}_1]_{\pi}[W_2,\hat{W}_2]_{\pi}}\\
  &=   \sum_{W_1}\frac{[W_1,\hat{\pi}((f_1^\vee\ast f_2^\theta)^\theta)\hat{W_2}]_{\pi}}{[W_1,\hat{W}_1]_{\pi}} = I_\pi((f_1^\vee\ast f^{\theta}_2)\otimes\Phi).
 \end{align*}
 The claim now follows by the assumption that $f=f_1^\vee\ast f_2^{\theta}$.
\end{proof}

\subsubsection{Elliptic support} In the global comparison of Proposition \ref{Prop: simple character id} and Theorem \ref{Thm: main global tool}, we impose two local assumptions: at a split place $v_1$ we assume that $f_{v_1}$ is essentially a matrix coefficient of a supercuspidal representation and at a second split place $v_2$ we assume that $f_{v_2}$ has elliptic support. The next lemma ensures that these constraints still allow for non-vanishing global orbital integrals in our relative trace formulas.

We recall  that for a $Z$-regular element $z\in \GL_n(F)$ to be {elliptic}, $F[z]=F[X]/(\car_z(X))$ must be a field.
As noted in the proof of Proposition \ref{Prop: split transfer}, the orbital integrals on $\GL_n(F)\times F_{n}$ simplify in the split case to
\begin{equation*}
    \Orb^{\GL_n(F),\eta}(f\otimes \Phi,z)=\Phi(0)\displaystyle\int_{\T_z(F)\backslash\GL_n(F)}f(g^{-1}zg)dg,
\end{equation*}
which is just the classical orbital integral of $f$ at $z\in \GL_n(F)$.  

\begin{Lem}\label{Lem: supercuspidal non-vanishing}
Assume that $F$ is a non-archimedean local field. Suppose that $\pi$ is a supercuspidal representation of $\GL_n(F)$ with central character $\omega$. Then there exists a matrix coefficient $c_\pi$ of $\pi$ and $f\in C_c^\infty(\GL_n(F))$ such that
\begin{itemize}
\item $f_\omega(z) = \displaystyle\int_{F}f(za)\omega(a)^{-1}da =c_\pi(z)$ for all $z\in \GL_n(F)$, and 
\item there exists a $Z$-regular elliptic element $z\in \GL_n(F)$ such that $\Orb_\omega(f\otimes\Phi,z)\neq0.$
\end{itemize}
\end{Lem}

\begin{proof}
The first requirement follows from the surjectivity of the map
\begin{align*}
C_c^\infty(G(F))&\lra C_c^\infty(Z_G(F)\backslash G(F),\omega)\\
f&\longmapsto f_\omega.
\end{align*}
Since $\pi$ is supercuspidal, any matrix coefficient $c_{\pi}$ lies in $C_c^\infty(Z_G(F)\backslash G(F),\omega)$, so there exists an $f$ satisfying $f_\omega=c_{\pi}$.

We now use \cite[Section 3.2]{KimShinTemplier} and Proposition 3.2 of \emph{ibid.} to see that $c_\pi$ may be chosen so that for $z\in \GL_n(F)$ $Z$-regular elliptic,
\begin{align*}
\Orb_\omega(f\otimes\Phi,z)&=\int_{F^\times}\Orb(f\otimes \Phi, za)\omega(a)^{-1}da\\ &=\int_{\T_z(F)\backslash\GL_n(F)}\int_{F^\times}f(g^{-1}zag)\omega(a)^{-1}dadg\\&= \Orb(c_\pi,z) \neq0.\qedhere
\end{align*}
\end{proof}

In particular, for $\pi$ and $f$ as in the Lemma, $f$ satisfies the supercuspidality constraint of Definition \ref{Def: nice transfer 2} without vanishing on the elliptic locus. Moreover, \eqref{eqn: linear local rel char formula} shows that $c_\pi$ may be chosen so that $I_\pi(f\otimes\Phi)\neq 0$.

\subsection{Weak transfer of relative characters}
While we do not prove a general comparison of local relative characters when $E/F$ is a quadratic extension of $p$-adic fields, the global identity of Theorem \ref{Thm: main global tool} implies the following weak statement, which is sufficient for our aims.
\begin{Lem}\label{Prop: weak transfer}
Assume that $E/F$ is a quadratic extension of number fields such that every archimedean place $v|\infty$ of $F$ splits in $E$. Let $\Pi=\mathrm{BC}_E(\pi)$ be an irreducible cuspidal automorphic representation of $\GL_n(\A_E)$ with central character $\omega'$ and $\tau_0,\tau\in \Herm_{n}^{\circ}(F)$ such that there exists a {$\omega'$-nice} test function $f$ such that $$J_{\Pi}^{\tau,\tau_0}(f)\neq 0.$$

For any non-split place $v_0$ of $F$, there exists a non-zero constant $C(\Pi_{v_0},\tau_0,\tau)\in \cc^\times$ depending only on $\tau_0$, $\tau$, and $\Pi_{v_0}$ such that for any $\omega'$-nice transfer $\wt{f}_{v_0}\in C_c^\infty(\GL_n(F_{v_0})\times F_{v_0,n})$ and $\{f^{\tau'}\}_{\tau'}$ satisfying
\[
f^{\tau'} =\begin{cases}
    f_{v_0}&: \tau'=\tau,\\
    0&: \text{otherwise,}
\end{cases}
\]we have
\begin{equation*}
J_{\Pi_{v_0}}^{\tau,\tau_0,\natural}(f_{v_0})=C(\Pi_{v_0},\tau,\tau_0)I^\natural_{\pi_{v_0}}(\wt{f}_{v_0}).
\end{equation*}
\end{Lem}

\begin{proof}
Let $\A_E^{v_0}$ denote the adeles away from the place $v_0$ and let \[{f}^{v_0}= \bigotimes_{v\neq v_0}f_{v}\in C_c^\infty(\GL_n(\A_E^{v_0}))\] be a factorizable test function. Proposition \ref{Prop: factorize unitary} implies that there exists a constant
\[
J_{\Pi}^{\tau,\tau_0}(f_{v_0}\otimes {f}^{v_0}) = CJ_{\Pi_{v_0}}^{\tau,\tau_0,\natural}(f_{v_0}).
\]
Since $J_{\Pi}^{\tau,\tau_0}\not\equiv 0$, we may choose ${f}^{v_0}$ so that $C\neq0$ and $f = f_{v_0}\otimes {f}^{v_0}$ is a nice test function for the central character $\omega'=\omega\circ\Nm$ of $\Pi$. In particular, there is a finite split place $v_1$ (necessarily distinct from $v_0$) such that $\Pi_{v_1}$ is supercuspidal and we may assume that $f_{v_1}\in C_c^\infty(\GL_n(E_{v_1}))$ is supported in the supercuspidal Bernstein component associated to $\Pi_{v_1}$. Additionally, there exists a second split place $v_2$ such that the local test function $f_{v_2}$ is supported in the $Z$-regular elliptic locus. 

By Theorem \ref{Thm: main global tool}, there exists $\wt{f}^{v_0}=\prod_{v\neq v_0}\wt{f}_v\in C_c^\infty(\GL_n(\A_F^{v_0})\times \A_{F,n}^{v_0})$ such that for any function $\wt{f}_{v_0}$ as in the statement of the proposition matching $f_{v_0}$, the function $\wt{f}=\wt{f}_{v_0}\otimes \wt{f}^{v_0}$ is $\omega$-nice and 
\[
 \frac{1}{2^{1-|S|}L(0,\eta)}J_{\Pi}^{\tau,\tau_0}(f)=2I_\pi(\wt{f}).
\]
Since we chose ${f}^{v_0}$ such that $C\neq0$, Proposition \ref{Prop: factorize linear} implies that there is a non-zero constant $C'$ such that 
\begin{equation*}
CJ^{\tau,\tau_0,\natural}_{\Pi_{v_0}}(f_{v_0})=J_{\Pi}^{\tau,\tau_0}(f)=\frac{4L(0,\eta)}{2^{|S|}}I_\pi(\wt{f})=C'I^\natural_{\pi_{v_0}}(\wt{f}_{v_0}).
\end{equation*}
Since the initial test function $f_{v_0}$ was arbitrary, the constant $$C(\Pi_{v_0},\tau,\tau_0) = C^{-1}C'\neq0$$ is independent of functions $f_{v_0}$ and $\wt{f}_{v_0}$, finishing the proof.
\end{proof}

Combining this with the matching of orbital integrals in Theorem \ref{Thm: Xiao FL for d} gives the following corollary.
\begin{Cor}\label{Cor: unramified spectral transfer}
Let $\Pi= \bigotimes_v\Pi_v$ be as in the previous Proposition, and fix $\tau_1,\tau_2\in \Herm_{n}^{\circ}(F)$. For an inert place $v_0$, if $\Pi_{v_0}$ is unramified, then $C(\Pi_{v_0},\tau_1,\tau_2)=1$.
\end{Cor}
\begin{proof}
    By the proposition, for any pair of matching functions $f_{v_0}\in C_c^\infty(\GL_n(E_{v_0}))$ and $f_{v_0}\in C_c^\infty(\GL_n(F_{v_0})\times F_{{v_0},n})$ we have
\begin{equation*}
J_{\Pi_{v_0}}^{\tau_1,\tau_2,\natural}(f_{v_0})=C(\Pi_{v_0},\tau_1,\tau_2)I^\natural_{\pi_{v_0}}(\wt{f}_{v_0}),
\end{equation*}
for some $C(\Pi_{v_0},\tau_1,\tau_2)\in\cc^\times$.

Assume now that $\Pi_{v_0}$ is unramified. When both $\tau_1,\tau_2$ are split, then Theorem \ref{Thm: xiao FL} implies we may use $f_{v_0}=\bfun_{K_{E_{v_0}}}$ and $\wt{f}_{v_0}=\bfun_{K_{v_0}}\otimes \bfun_{\calo_{F_{v_0},n}}$. Then Lemmas \ref{Lem: linear rel char unit} and \ref{Lem: unramified unitary local rel char} combine to give
\[
J_{\Pi_{v_0}}^{\tau_1,\tau_2,\natural}(f_{v_0})=I^\natural_{\pi_{v_0}}(\wt{f}_{v_0})=1,
\] giving the claim in this case.

When $\tau_2$ is split by $\tau_1$ is non-split at $v_0$, we take $f_{v_0} = \phi^{\tau_1}$ right-$K_{E_{v_0}}$-invariant so that $p_{\tau_1,!}(\phi^{\tau_1}) = \Phi_{d}$ for some odd $d\in \zz_{\geq0}$, where $\Phi_d$ is defined before Theorem \ref{Thm: Xiao FL for d}  and $\wt{f}_{v_0}=(-1)^{nd}\bfun_{d}\otimes \bfun_{\calo_{F_{v_0},n}}$. Then Proposition \ref{Prop: unramified rel char} with $\phi^{\tau_2} = \bfun_{K_{E_{v_0}}}$  similarly implies that
\[
J_{\Pi_{v_0}}^{\tau_1,\tau_2,\natural}(f_{v_0})=I^\natural_{\pi_{v_0}}(\wt{f}_{v_0}) =(-1)^{nd}\lam_{\pi_{v_0}}(\bfun_d) =(-1)^{nd}q_{v_0}^{d(n-1)/2}\sum_{\mathbf{m}\in  \zz_{d,+}^n}\prod_{i=1}^n q_{v_0}^{-m_is_i},
\]where $ \zz_{d,+}^n =\{{\bf m}\in \zz^n_{\geq0}: \sum_i m_i=d\}$ and where $(q_{v_0}^{-s_1},\ldots,q_{v_0}^{-s_n})$ are the Satake parameters for $\pi_{v_0}$. For a fixed $\pi_{v_0}$, the right-hand side cannot vanish for all odd integers $d$, showing that $C(\Pi_{v_0},\tau_1,\tau_2)=1$ in this case as well.

On the other hand, if $\tau_2$ is non-split but $\tau_1$ is split at $v_0$, we use the identity (cf. \cite[Lemma 2.1]{FLO})
\[
J_{\Pi_{v_0}}^{\tau_1,\tau_2}(f_{v_0}) = J_{\Pi^\vee_{v_0}}^{\tau_2,\tau_1}({f}^\vee_{v_0})
\]
to reduce to the prior case (noting that $\Pi^\vee_{v_0}$ is also unramified).

Finally, suppose that $\tau_1$ and $\tau_2$ are both non-split. We choose global Hermitian forms  $\mathcal{T}_1$ and $\mathcal{T}_2$ such that 
\begin{enumerate}
    \item for $i=1,2$, $\mathcal{T}_i$ corresponds to the global Hermitian space that is non-split exactly at the places $v_0$ and another inert place $v_i$, and split at all other places, and
 \item\label{ref 2}  the inert places $v_1$ and $v_2$ are distinct. 
\end{enumerate}
We also fix the global split form $\mathcal{T}_0$. Running the argument from the proof of Lemma \ref{Prop: weak transfer}, we obtain two $\omega'$-nice test functions 
\[
f^{(0)} = \bigotimes_v f^{(0)}_v,\;\; f^{(2)}=\bigotimes_vf^{(2)}_v\in C^\infty_c(\GL_n(\A_E))
\] and an $\omega$-nice test function $\wt{f}= \bigotimes_v\wt{f}_v\in C^\infty_c(\GL_n(\A_F)\times \A_{F,n})$ such that
\begin{enumerate}
    \item\label{ref 1} for $v\neq v_0,v_1,v_2$, $f^{(0)}_v = f_v^{(2)}$, 
    \item for $j=0,2$, $\wt{f}$ and $f^{(j)}$ match in the sense of Theorem \ref{Thm: main global tool} with $\tau_0 = \mathcal{T}_j$. In particular, we have fixed auxiliary split places $v_{cusp}$ and $v_{ell}$ (necessarily distinct from $v_0,v_1,v_2$) to define nice transfers.
\end{enumerate}
With these assumptions, we obtain the global identity
\begin{equation*}
    \frac{1}{2^{1-|S|}L(0,\eta)}J_{\Pi}^{\mathcal{T}_1,\mathcal{T}_2}(f^{(2)})=I_\pi(\wt{f})+I_{\pi\eta}(\wt{f})= \frac{1}{2^{1-|S|}L(0,\eta)}J_{\Pi}^{\mathcal{T}_1,\mathcal{T}_0}(f^{(0)}),
\end{equation*}
which combines with Proposition \ref{Prop: factorize unitary} and \eqref{ref 1} to give
\begin{equation}\label{three terms}
    \prod_{i\in \{0,1,2\}}J_{\Pi_{v_i}}^{\mathcal{T}_1,\mathcal{T}_2,\natural}(f^{(2)}_{v_i})=\prod_{i\in \{0,1,2\}}J_{\Pi_{v_i}}^{\mathcal{T}_1,\mathcal{T}_0,\natural}(f^{(0)}_{v_i})
\end{equation}
Note that the choice of the places $v_1$ and $v_2$ is free among places inert in $E$, so we may assume that there exist $\Pi$ as in the statement of Lemma \ref{Prop: weak transfer} such that $\Pi_{v}$ is unramified at $v_0$, $v_1$, and $v_2$. By the preceding cases and assumption \eqref{ref 2}, we thus know that
\begin{align*}
  J_{\Pi_{v_0}}^{\mathcal{T}_1,\mathcal{T}_2,\natural}(f^{(2)}_{v_0}) \prod_{i\in \{1,2\}}I^\natural_{\pi_{v_i}}(\wt{f}_{v_i})&=\prod_{i\in \{0,1,2\}}J_{\Pi_{v_i}}^{\mathcal{T}_1,\mathcal{T}_2,\natural}(f^{(2)}_{v_i})\\
  &\overset{\eqref{three terms}}{=}\prod_{i\in \{0,1,2\}}J_{\Pi_{v_i}}^{\mathcal{T}_1,\mathcal{T}_0,\natural}(f^{(0)}_{v_i})=\prod_{i\in \{0,1,2\}}I^\natural_{\pi_{v_i}}(\wt{f}_{v_i}),
\end{align*}
where the first line uses the split-inert case of the present lemma at the places $v_1$ and $v_2$, and the final line uses both the split--split and non-split--split cases. Since the relative characters $I^\natural_{\pi_{v_i}}(\wt{f}_{v_i})$ are non-zero distributions, we conclude
\[
 J_{\Pi_{v_0}}^{\mathcal{T}_1,\mathcal{T}_2,\natural}(f^{(2)}_{v_0})=I^\natural_{\pi_{v_0}}(\wt{f}_{v_0}).\qedhere
\]
\end{proof}

\section{The proof of Theorem \ref{Thm: first step main}}\label{Section: final proof}
We now have the necessary ingredients to prove  Theorem \ref{Thm: first step main}. Fix two Hermitian forms $\tau_1,\tau_2\in \Herm_{n}^{\circ}(F)$. For $i\in \{1,2\}$, assume that $\phi_i\in \calh_{K_E}(\Herm_{n}^{\circ}(F))$ is of type $\tau_i$, and let $\phi_{1}^{\tau_1}$ (resp., $\phi^{\tau_2}_2$) denote right-$K_E$-invariant lift to $\GL_n(E)$ as in Lemma \ref{Lem: invariant lifts}. Set $\U(V_2):=\U(V_{\tau_2^{-1}})$. By Lemma \ref{Lem: converse with center}, it suffices to prove that
\begin{equation}\label{eqn: unram relevant transfer with center}
\SO^{\U(V_2)}_{\omega'}(\Phi_{\tau_1,\tau_2}^{\phi_1,\phi_2},y)=\Orb^{\GL_n(F),\eta}_\omega(\Hir(\phi_1\ast\phi_2)\otimes \bfun_{\calo_{F,n}},\ga),
\end{equation} for each unitary central character $\omega$ and any matching $Z$-regular semi-simple orbits $y\leftrightarrow \ga$.

The approach is similar to the proof of \cite[Theorem 11.1]{LeslieUFJFL}, albeit with different trace formulas. Let $E/F$ be our unramified extension of $p$-adic fields with $p$ odd. We globalize the problem as follows: 

Fix a quadratic extension of number fields $\cale/\calf$ satisfying that every archimedean place of $\mathcal{F}$ splits in $\mathcal{E}$ and possessing a fixed finite place $v_0$ such that $\cale_{v_0}/\calf_{v_0}\cong E/F$. We now choose hermitian forms $\mathcal{T}_1,\mathcal{T}_2\in \Herm_{n}^{\circ}(\mathcal{F})$ globalizing $\tau_1,\tau_2$ in the following way
\begin{enumerate}
	\item if $\tau_{i}$ induces a split Hermitian space, we choose $\mathcal{T}_i$ corresponding to the global split Hermitian space,
	\item if $\tau_{i}$ induces a non-split Hermitian space, we choose $\mathcal{T}_i$ to be non-split exactly at the places $v_0$ and a single auxiliary inert place $v_i$, and
 \item if \textbf{both} $\tau_1$ and $\tau_2$ are non-split, we assume that the inert places $v_1$ and $v_2$ are distinct. 
\end{enumerate}
Such global forms clearly exist.

\subsection{Preliminaries}

 Fix unramified unitary character $\omega$ of $Z_n(F)$, and by an abuse of notation we let $\omega=\prod_v\omega_v$ be a global unitary character of $[Z_n]$ such that $\omega_{v_0}=\omega$. Let $\omega'$ denote the base change of $\omega$ to $\cale$. Toward establishing \eqref{eqn: unram relevant transfer with center}, we fix matching $Z$-regular semi-simple orbits $y\in \Herm_n^\circ(F)\leftrightarrow z\in \GL_n(F)$.   We also set aside two split places $v_{cusp}$ and $v_{ell}$. To simplify notations, for any place $v$ of $\calf$, set $K_v=\GL_n(\calo_{\calf_v})$. When $v$ is inert, we set $K'_v=\GL_n(\calo_{\cale_v})$.

 Let $\pi_{v_{cusp}}$ be a supercuspidal automorphic representation of $\GL_n(\calf_{v_{cusp}})$. Recalling the notation from Theorem \ref{Thm: main global tool}, we consider a $\omega'$-nice factorizable function $f\in C_c^\infty(\GL_n(\A_E))$ satisfying that $f_{v_{cusp}}$ is essentially a matrix coefficient of $\Pi_{v_{cusp}}=\mathrm{BC}(\pi_{v_{cusp}})$. Then there exists a $\omega'$-nice factorizable function $\wt{f}\in C_c^\infty(\GL_n(\A_{\calf})\times \A_{{\calf},n})$ matching $\{f^{\tau}\}_{\tau}$, where
\[
f^{\tau} =\begin{cases}
    f&: \tau=\mathcal{T}_1,\\
    0&: \text{otherwise,}
\end{cases}
\] and $(\wt{f},\{f^{\tau}\})$  forms an efficient transfer for $(\omega,\pi_{v_1})$. We extract from such a efficient transfer the local transfers at $v_{cusp}$ and $v_{ell}$.

Using Lemma \ref{Lem: supercuspidal non-vanishing}, there exists a $Z$-regular elliptic element $z_{cusp}$ such that
\[
\Orb^{\GL_n(F),\eta}_\omega(\wt{f}_{v_{cusp}},z_{cusp})\neq0.
\]
\quash{for appropriate $\Phi_{v_{cusp}}$. Let $f'_{v_{cusp}}$ be essentially a matrix coefficient of $\mathrm{BC}(\pi_{cusp})$ matching $f_{v_{cusp}}\otimes \Phi_{v_{cusp}}$; it is clear that such an $f'_{v_{cusp}}$ exists. Then we know
\[
\Orb^{\GL_n(F),\eta}_\omega(f_{v_{cusp}}\otimes \Phi_{v_{cusp}},\ga_{cusp}) = \Orb_\omega(f'_{v_{cusp}},\de_{cusp})\neq0,
\]
for matching orbits $\de_{cusp}\leftrightarrow z_{cusp}$.}
For the place $v_{ell}$, we may assume $f_{v_{ell}}$ has elliptic support, and we let $z_{ell}$ a $Z$-regular elliptic semi-simple element such that
\[
\Orb^{\GL_n(F),\eta}_\omega(\wt{f}_{v_{ell}},z_{ell})\neq0.
\]

 To study orbital integrals at $y$ and $z$ by global means, we first approximate these points by global elements, using the fact that the diagonal embeddings
\[
\GL_n(\calf)\hra \GL_n(F)\times \GL_n(\calf_{v_{cusp}})\times \GL_n(\calf_{v_{ell}})
\]
are dense, as well as local constancy of the relevant orbital integrals. We state this as the following lemma.
 \begin{Lem}\label{Lem: approx glob} Consider our pair of matching $Z$-regular semi-simple orbits $y\leftrightarrow z$. There exist matching global $Z$-regular semi-simple elements 
$Y\in \Herm_{n}^{\circ}(\calf)$ and $Z\in \GL_n(\calf)$ such that
\[
\Orb^{\GL_n(F),\eta}_{\omega'}(\wt{f}_{v_{cusp}},Z)=\Orb^{\GL_n(F),\eta}_\omega(\wt{f}_{v_{cusp}},z_{cusp})\neq0,
\]
\[
\Orb^{\GL_n(F),\eta}_{\omega'}(\wt{f}_{v_{ell}},Z)=\Orb^{\GL_n(F),\eta}_\omega(\wt{f}_{v_{ell}},z_{ell})\neq0,
\]
\[
\Orb^{\GL_n(F),\eta}_\omega(\Hir(\phi_1\ast\phi_2)\otimes \bfun_{\calo_{F,n}},Z)=\Orb^{\GL_n(F),\eta}_\omega(\Hir(\phi_1\ast\phi_2)\otimes \bfun_{\calo_{F,n}},z),
\]
and
\[
\SO^{\U(V_{2})}_{\omega'}(\Phi_{\tau_1,\tau_2}^{\phi_1,\phi_2},Y)=\SO^{\U(V_{2})}_{\omega'}(\Phi_{\tau_1,\tau_2}^{\phi_1,\phi_2},y).
\]
 \end{Lem}

We now choose global test functions to analyze these orbital integrals.

 \begin{Prop}\label{Prop: good global test function}
 \begin{enumerate}
     \item There exist functions $f=\bigotimes_v f_v\in C_c^\infty(\GL_n(\A_{\cale}))$ and $\wt{f}=\bigotimes_v \wt{f}_v\in C_c^\infty(\GL_n(\A_{\calf})\times \A_{\calf,n})$ which are an efficient matching pair of functions for $(\omega, \pi_{v_{cusp}})$ such that
\begin{align}\label{eqn: isolation matching}
2I_\omega(\wt{f}) &= 2L(0,T_Z,\eta)\prod_v \Orb^{\GL_n(F),\eta}_{\omega_v}(\wt{f}_v,Z)\\&= \vol([T_Y])\prod_v\SO_{\omega'}^{\U(V_2)}(\pi_{\mathcal{T}_1,!}(f_v),Y)= \frac{1}{2^{1-|S|}L(0,\eta)}J_{\omega'}^{\mathcal{T}_1,\mathcal{T}_2}(f) \neq0,\nonumber
\end{align}
where $\eta=\eta_{\cale/\calf}$.
\item\label{Final change to prop} With $f$ and $\wt{f}$ as above, set \[\text{ $\underline{\wt{f}}=\Hir(\phi_1\ast\phi_2)\otimes\bigotimes_{v\neq v_0}\wt{f}_{v}$ and $\underline{f}=[\phi_1^{\tau_1}\ast(\phi^{\tau_2}_2)^\ast]\otimes\bigotimes_{v\neq v_0}f_v$.}\]
Then $\underline{\wt{f}}$ and $\underline{f}$ are  nice functions satisfying 
\[
 \Orb^{\GL_n(F),\eta}_\omega(\underline{\wt{f}},Z')= \SO^{\U(V_2)}_{\omega'}(\pi_{\mathcal{T}_1,!}(\underline{f}),Y')=0
\]
for any global match regular semi-simple elements $Z'\leftrightarrow Y'$ lying in distinct orbits from $Z \leftrightarrow Y$. In particular, we have
\[
I_\omega(\ul{\wt{f}}) = L(0,T_Z,\eta)\prod_v \Orb^{\GL_n(F),\eta}_{\omega_v}(\ul{\wt{f}}_v,Z)\]\text{   and    }\[\frac{1}{2^{1-|S|}L(0,\eta)}J_{\omega'}^{\mathcal{T}_1,\mathcal{T}_2}(\underline{f})=\vol([T_Y])\prod_v\SO_{\omega'}^{\U(V_2)}(\pi_{\mathcal{T}_1,!}(\ul{f}_v),Y).
\]
 \end{enumerate}
 \end{Prop}

 \begin{proof}
Let $S$ be a finite set of places of $\calf$ containing all infinite places and  the places $v_0$, $v_1$ (if $\tau_1$ is non-split), $v_2$ (if $\tau_2$ is non-split), $v_{cusp}$, and $v_{ell}$ such that for each $v\notin S$, $Y\in \Herm_n^\circ(\calo_{F_v})$ and $Z\in \GL_n(\calo_{F_v})$. We already have chosen matching functions at $v_{cusp}$ and $v_{ell}$ satisfying the claims of Lemma \ref{Lem: approx glob}.  For every $v\in S\setminus{\{v_{cusp},v_{ell}\}}$, select matching $\wt{f}_v$  and $f_{v}$ such that
\[
\Orb^{\GL_n(F),\eta}_{\omega}(\wt{f}_{v},Z) = \SO^{\U(V_2)}_{\omega'}(\pi_{\mathcal{T}_1,!}(f_v),Y)\neq0.
\]
For all places $v\notin S$, we take $\wt{f}_v= \bfun_{K_v}\otimes \bfun_{\calo_{F_v,n}}$ and $f_{v}=\bfun_{K'_v}$ to be the unit spherical functions. Theorem \ref{Thm: xiao FL} implies that
\[
\Orb^{\GL_n(F),\eta}(\wt{f}_v,Z) = \SO^{\U(V_2)}(\pi_{\mathcal{T}_1,!}(f_v),Y)\neq0
\]
for all $v\notin S$. The non-vanishing for $v\notin S$ follows since $f_{v}=\bfun_{K_{v,E}}$ is a non-negative function with $\pi_{\mathcal{T}_1,!}(f_v)(Y)\neq0$. 
When taking the central character into account, it is easy to see (cf. \cite[Lemma 11.5]{LeslieUFJFL}) that if $\omega_v'$ is a unramified unitary central character, then
 \[
 \SO_{\omega'}(\pi_{\mathcal{T}_1,!}(\bfun_{K'_v}),Y) =  \SO(\pi_{\mathcal{T}_1,!}(\bfun_{K'_v}),Y)\neq0.
 \]

Now for each place $v\in S$, let $C_v\subset \GL_n(\calf_v)\times \calf_{v,n}$ be a compact set containing the support of $\wt{f}_v$ and assume that $C_{v_0}$ is large enough to contain the support of $\Hir(\phi_1\ast\phi_2)$; set $$C=\prod_{v\in S}C_v\times \prod_{v\notin S}K_v\times \calo_{\mathcal{F}_v,n}\subset \GL_n(\A_{\calf})\times \A_{\calf,n}.$$

Now set $f= \bigotimes_vf_v$ and $f'=\bigotimes_vf'_{v}$.  As discussed in the proof of Theorem \ref{Thm: main global tool}, we may assume that for each place $v$ 
\[
\supp(\wt{f}_v) \subset \{(g,w)\in \GL_n(F_v)\times F_{v,n}\::\: \eta_v(\det(g)) = \eta(\det(\mathcal{T}_1))\}.
\]
Since $\mathcal{T}_1$ is global, this implies that $\wt{f}=\eta\cdot \wt{f}$, where $(\eta\cdot \wt{f})(g,w) = \eta(\det(g))\wt{f}(g,w)$.
Our choices ensure that $f$ and $f'$ are an efficient matching pair of functions for $(\omega, \pi_{v_{cusp}})$ and that 
\[
2L(0,T_Z,\eta)\prod_v \Orb^{\GL_n(F),\eta}_{\omega_v}(\wt{f}_v,Z)= \vol([T_Y])\prod_v\SO_{\omega'}^{\U(V_2)}(\pi_{\mathcal{T}_1,!}(f_v),Y)\neq0.
\]

We claim that we may augment $f$ and $f'$ so that this is the only non-vanishing global orbital integral for $f$. Note that the support of $f_{v_{ell}}$ already reduces this to $Z$-regular elliptic orbits. 

Recall that the matching of orbits may be characterized by matching of characteristic polynomials of $y\in \Herm_{n}^{\circ}(\mathcal{F})$ and $z\in \GL_n(\mathcal{F})$. Let $\{c_i\}_{i=1}^{n}$ denote the coefficients for $\GL_n$ and $\{c_i^{{\mathcal{T}_1,\mathcal{T}_2}}\}_{i=1}^{n}$ for $\Herm_{n}^{\circ}$. Let $p: \GL_n(\A_{\calf})\times \A_{\calf,n}\to \GL_n(\A_{\calf})$ denote the projection to the first factor. Consider the compact set $p(C)\subset \GL_n(\A_{\calf})$. For each $i$,  $c_i(C)\subset \A_{\calf}$ is a compact set. Since $\calf\subset \A_{\calf}$ is discrete, the intersection $c_i(p(C))\cap \calf$ is finite for each $i=1,\ldots, n$. Since $\supp(\wt{f})\subset C$, this implies that $\prod_v \Orb^{\GL_n(F),\eta}_{\omega_v}(\wt{f}_v,Z')=0$ outside of a finite set of orbits
\[
Q_f\subset \GL_n(\calf)\backslash [\GL_n(\calf)\times \calf_{n}]^{Z-ell}/\Nm(Z_{\cale}(\calf)).
\]

\quash{Let $S_{aux}$ be a set of $|Q_f|-1$ unramified places of $\calf$ disjoint from $S$, and fix a bijection between $S_{aux}$ and $Q_f-\{a\}$. 
For each $v\in S_{aux}$, set $\widetilde{f}_{v} = \bfun_{G_{v}[l_v]}\cdot\bfun_{K_{v}}$ and $\widetilde{f}'_{v} = \bfun_{G'_{v}[l_v]}\cdot\bfun_{K'_{v}}$, where
\[
G_{v}[l_v]=\{g\in G(\calf_{v}): c_i(g)\in c_i(a)+\mathfrak{p}_{v}^{l_v}, \text{ for all $i=1,\ldots,2n+1$}\}
\]
and
\[
G'_{v}[l_v]=\{g\in G'(\calf_{v}): c^{w_n,w_{n+1}}_i(g)\in c^{w_n,w_{n+1}}_i(b)+\mathfrak{p}_{v}^{l_v}, \text{ for all $i=1,\ldots,2n+1$}\}.
\]
}
Let $S_{aux}$ be a set of $|Q_f|-1$ unramified places of $\calf$ disjoint from $S$, and fix a bijection between $S_{aux}$ and $Q_f-\{Z\}$. For each $v\in S_{aux}$, set $\widetilde{f'}_{v} = \bfun_{\GL_n(\calf_{v})[l_v]}\cdot\bfun_{K_{v}\times \calo_{\calf_V,n}}$ and ${f'}_{v} = \bfun_{\bfun_{\GL_n(\cale_{v})[l_v]}}\cdot\bfun_{K'_{v}}$, where 
\[
\GL_n(\calf_{v})[l_v]=\{g\in \GL_n(\calf_{v}): c_i(g)\in c_i(Z)+\mathfrak{p}_{v}^{l_v}, \text{ for all $i=1,\ldots,n-1$}\}.
\]
and
\[
\GL_n(\cale_{v})[l_v]=\{g\in \GL_n(\cale_{v}): c_i^{\mathcal{T}_1,\mathcal{T}_2}(g)\in c_i^{\mathcal{T}_1,\mathcal{T}_2}(Y)+\mathfrak{p}_{v}^{l_v}, \text{ for all $i=1,\ldots,n-1$}\}.
\]
It is clear that that ${f'}_{v}$ and $\widetilde{f}'_{v}$ are transfers for any choice of $l_v\in \zz_{\geq0}$ and that they have non-vanishing orbital integrals at $Y$ and $Z$, respectively.

Recalling out compact set $C\subset \GL_n(\A_{\calf})\times \A_{\calf,n}.$, we set $C[l]$ to be the compact set with $K_v$ replaced with $K_v\cap G_{v}[l_v]$ for all $v\in S_{aux}$. For each $v\in S_{aux}$, we now choose $l_v$ large enough that if $a(v)$ represents the orbit in $Q_f-\{Y\}$ associated to $v\in S_{aux}$, then $c_i(a(v))\notin c_i(C[l])$ for some $i$. Replacing $\wt{f}_v$ and $f_v$ by $\widetilde{f'}_{v}$ and ${f}'_{v}$ for each $v\in S_{aux}$, this assumption forces all orbital integrals but those at $Y$ (and $Z$) to vanish and we conclude that $\wt{f}=\bigotimes_v \wt{f}_v$ and $f=\bigotimes_v f_v$ satisfy \eqref{eqn: isolation matching}.

To prove the second claim \eqref{Final change to prop}, we note by inspection of our construction of $f=\bigotimes_vf_v$ and $f'=\bigotimes_vf_v'$ below, the support of the factors $\wt{f}_{v_0}$ and $f_{v_0}'$ do not play a role in the vanishing statements. The final expressions are then just special cases of the simple trace formulas.
\end{proof}

We are now ready to complete the proof.
\begin{proof}[Proof of Theorem \ref{Thm: first step main}.]
    Using the test functions $\underline{\wt{f}}$ and $\underline{f}$ from Proposition \ref{Prop: good global test function}, it suffices to prove that
    \[
    2I_{\omega}(\underline{\wt{f}}) = \frac{1}{2^{1-|S|}L(0,\eta)}J_{\omega'}^{\mathcal{T}_1,\mathcal{T}_2}(\underline{f});
    \]
    indeed, this implies that 
    \[
 L(0,T_Z,\eta)\prod_v \Orb^{\GL_n(F),\eta}_{\omega_v}(\ul{\wt{f}}_v,Z)=\vol([T_Y])\prod_v\SO_{\omega'}^{\U(V_2)}(\pi_{\mathcal{T}_1,!}(\ul{f}_v),Y).
    \]
    which combines with our choice of matching functions at all places $v'\neq v_0$, Lemma \ref{Lem: approx glob}, and Lemma \ref{Lem: converse with center} to prove the claim.

    Relying on the spectral expansions of Propositions \ref{Prop: unitary RTF} and \ref{Prop: Eisenstein RTF}, it suffices to prove
    \begin{equation}
         \frac{1}{2^{1-|S|}L(0,\eta)}J_{\Pi}^{\mathcal{T}_1,\mathcal{T}_2}(\underline{f})=\sum_{\pi\in \mathcal{B}(\Pi)}I_\pi(\underline{\wt{f}}),
    \end{equation} 
    for each cuspidal automorphic representation $\Pi$ of $\GL_n(\A_{\cale})$ with central character $\omega'=\omega\circ \Nm_{\mathcal{E}/\calf}$ that arises as a base change $\Pi=\mathrm{BC}(\pi)$ of a cuspidal automorphic representation $\pi$ of $\GL_n(\A_{\calf})$. Note that the choice of $\wt{f}_{v_{cusp}}$ and $f_{v_{cusp}}$ ensures that the left-hand side vanishes for any other $\Pi$.
    By our choice of $\wt{f}$ and $f$ in Proposition \ref{Prop: good global test function}, this is the same as verifying that
    \begin{equation}\label{eqn: super important proof}
2^{|S|}J^{\tau_1,\tau_2,\natural}_{\Pi_{v_0}}([\phi_1^{\tau_1}\ast(\phi^{\tau_2}_2)^\ast])\prod_{v\neq v_0}J^{\mathcal{T}_1,\mathcal{T}_2,\natural}_{\Pi_v}(f_{v}) =I^\natural_{\pi_{v_0}}(\Hir(\phi_1\ast\phi_2)\otimes \bfun_{\calo_{F,n}})\prod_{v\neq v_0}I^\natural_{\pi_v}(\wt{f}_v).
\end{equation}
    Note that if $\pi$ is not unramified at $v_0$, then both sides are zero. We thus assume that $\pi_{v_0}$ is unramified.

    Given the matching of the un-augmented functions $\wt{f}$ and $f$ in Proposition \ref{Prop: good global test function}, if we choose a function $f^\circ\in C_c^\infty(\GL_n(E))$ matching $\Hir(\phi_1\ast\phi_2)\otimes \bfun_{\calo_{F,n}}$,  Theorem \ref{Thm: main global tool} implies
      \begin{equation}\label{eqn: super important almost}
2^{|S|}J^{\tau_1,\tau_2,\natural}_{\Pi_{v_0}}(f^\circ)\prod_{v\neq v_0}J^{\mathcal{T}_1,\mathcal{T}_2,\natural}_{\Pi_v}(f_{v}) =I^\natural_{\pi_{v_0}}(\Hir(\phi_1\ast\phi_2)\otimes \bfun_{\calo_{F,n}})\prod_{v\neq v_0}I^\natural_{\pi_v}(\wt{f}_v).
\end{equation}
Moreover, Proposition \ref{Prop: unramified rel char} and Corollary \ref{Cor: unramified spectral transfer} implies that 
\begin{align*}
J^{\tau_1,\tau_2,\natural}_{\Pi_{v_0}}(f^\circ)&=I^\natural_{\pi_{v_0}}(\Hir(\phi_1\ast\phi_2)\otimes \bfun_{\calo_{F,n}})\\
&=J^{\tau_1,\tau_2,\natural}_{\Pi_{v_0}}(\phi_1^{\tau_1}\ast(\phi^{\tau_2}_2)^\ast),
\end{align*}
and we conclude \eqref{eqn: super important proof}.
\end{proof}

\appendix

\section{Weak smooth transfer}\label{Section: proof weak transfer}
\quash{In \S \ref{Section: linear statements}, we introduced five spaces of orbital integrals:
\begin{enumerate}
    \item two spaces of linear orbital integrals 
    \[
    \cals^{\ul{\eta}}(\fX(F)/\rH'):=   \{S_1^{ram}(X,Y)\Orb^{\rH',\ul{\eta}}_{0}(\widetilde{\phi},(X,Y)):\;\wt{\phi}\in C_c^\infty(\fX(F)\times F_{n})\},
    \]for $\underline{\eta}= (\eta,1)$ or $(\eta,\eta)$,
    \item two spaces of stable orbital integrals on infinitesimal unitary symmetric varieties
    \[
    \cals_{\tau_2}(\fq_{si}(F)/\rH_{si})=\{\SO^{\rH_{\tau_1,\tau_2}}(f_{\tau_1},Z):\; \underline{f}= (f_{\tau_1})\in \bigoplus_{\tau_1\in \calv_n(E/F)}C_c^\infty(\fq_{\tau_1,\tau_2}(F))\},
        \]
        and 
            \[
    \cals(\fq_{ii}(F)/\rH_{ii})=\{\SO^{\rH_{ii}}(f,(x_1,x_2)):\; {f}\in C_c^\infty(\fq_{ii}(F))\},
        \]
    \item and one space of $\varepsilon$-orbital integrals
              \[
    \cals^{\varepsilon}(\fq_{ii}(F)/\rH_{ii})=\{\De_\varepsilon(x_1,x_2)\Orb^{\rH_{ii},\varepsilon}(f,(x_1,x_2)):\; {f}\in C_c^\infty(\fq_{ii}(F))\}.
        \]
\end{enumerate}
Conjectures \ref{Conj: smooth transfer Lie} and \ref{Conj: ii simple endoscopy lie} amount to the identifications
\[
 \cals^{{(\eta,1)}}(\fX(F)/\rH') =     \cals(\fq_{ii}(F)/\rH_{ii})\quad\text{ and }\quad \cals^{{(\eta,\eta)}}(\fX(F)/\rH')= \cals_{\tau_2}(\fq_{si}(F)/\rH_{si})=  \cals^{\varepsilon}(\fq_{ii}(F)/\rH_{ii}),
\]
for $\tau_2\in\calv_n(E/F)$ split.
}
In this appendix, we prove Theorems \ref{Thm: weak smooth transfer} and \ref{Thm: weak endoscopic transfer}. \quash{ by descending to the linear setting and showing that all of these spaces of orbital integrals agree over the invertible locus.}  For this, we apply a Cayley transform to descend orbital integrals to the linearized setting and reduce both theorems to the transfer in Theorem \ref{Thm: Xiao's transfer}.
 \subsection{Cayley transform}\label{Section: Cayley} Let $E/F$ be a quadratic extension of local fields of characteristic zero. Fix an $E$-vector space $V$ of dimension $n$, and let $W = V\oplus V$; fix also an $F$-vector space $V_0$ such that $V_0\otimes_{F}E =V$ and set $W_0=V_0\oplus V_0$. For any $\nu\in E$, we define
\[
D_\nu=\{X\in \End(W): \det(\nu I_{2n}-X)=0\}.
\]
\begin{Lem}\label{Lem: Cayley map linear}
 Consider the Cayley transform
\begin{align*}
\fc_{\pm1}:\End(W)-D_1(F)&\lra \GL(W)\\
			X&\longmapsto \mp(1+X)(1-X)^{-1}.
\end{align*}
\begin{enumerate}
\item When restricted to $\fX(F)\subset \End(W_0)\subset \End(W)$, this induces a $\rH'=\GL(V_0)\times \GL(V_0)$-equivariant isomorphism between $\mathfrak{X}(F)-D_1(F)$ and $\X(F)-D_{\pm1}(F)$. The images of $\mathfrak{X}(F)-D_1(F)$ under $\fc_{\pm}$ form a finite cover by open subsets of $\X(F)-(D_1\cap D_{-1})(F)$. In particular, the images contain the regular semi-simple locus of $\X(F)$.
\item For any unitary symmetric variety $\Q=\G/\rH$, if we restrict $\fc_{\pm}$ to $\fq\subset \End(W)$, this gives an $\rH$-equivariant isomorphism between $\fq(F)-D_1(F)$ and $\Q(F)-D_{\pm1}(F)$. The images of $\fq(F)-D_1(F)$ under $\fc_{\pm}$ form a finite cover by open subsets of $\Q(F)-(D_1\cap D_{-1})(F)$. In particular, the images contain the regular semi-simple locus of $\Q(F)$.
\end{enumerate}
\end{Lem}
\begin{proof}
    This is proved in \cite[Lemma 7.2]{Lesliedescent} in the split-inert case. The proof immediately extends to all cases, but we recall the reason for needing to restrict $\nu\in E$ to $\nu=\pm1$, as this is an important difference from the argument in \cite[Section 3]{ZhangFourier}.

Suppose $\theta(g) = \ep g\ep$ is the involution on $\G$ such that $\rH=\G^\theta$ (note that this also applies to the linear case).
We now show that for the transform to be compatible with the involution $\theta$, we need $\nu=\pm1$. If $X\in \fq(F)$, then $\theta(X)=\ep X \ep =-X$. A simple calculation shows
\[
\theta(\fc_\nu(X)) = -\nu(1-X)(1+X)^{-1}= ({-\nu}^{-1}(1+X)(1-X)^{-1})^{-1}=\fc_{\nu^{-1}}(X)^{-1}.
\]
Thus, $\theta(\fc_\nu(X))=\fc_\nu(X)^{-1}$ if and only if $\nu=\pm1$. 
\end{proof}
We now consider the relationship between transfer factors under the Cayley transform.
\begin{Lem}\label{Lem: tranfer factor cayley} Fix $\nu\in\{\pm1\}$. Suppose that $[(X,Y),w]\in \mathfrak{X}^{sr}_{ext}(F)$. Then $(y_{\nu},w):=[\fc_{\nu}(X,Y),w]\in \X^{sr}_{ext}(F)$. For any choice $\underline{\eta}=(\eta_0,\eta_2)$ there is a sign $c_{\nu}^{\underline{\eta}}=\pm1$ independent of $(X,Y)$ such that
\[
\widetilde{\omega}((X,Y),w)=c_{\nu}^{\underline{\eta}}\cdot\eta_2(I-XY) {\omega}(y_{\nu},w),
\]
which is well-defined since $(X,Y)\notin D_1(F)$.
\end{Lem}
\begin{proof}
We write $(X,Y)\in\fX(F)\subset \fgl(W)$ as
\[
 Z=\left(\begin{array}{cc}
     & X \\
    Y & 
\end{array}\right),
\]
and calculate $y_\nu=\fc_{\nu}(Z)$ to be
\begin{equation}\label{cayley calculation}
     -\nu\left(\begin{array}{cc}(I+XY)(I-XY)^{-1}&-2X(I-Y X)^{-1}\\-2Y(I-XY)^{-1}&(I+Y X)(I-Y X)^{-1}\end{array}\right)=\left(\begin{array}{cc}
    A&  B \\
    C &   D
\end{array}\right). 
\end{equation}
Recalling the formula 
\begin{align*}
\omega(y_\nu,w)= \eta^n_2(BC)\eta_2(C)\eta_0(\det\{w,wD,\ldots,wD^{n-1}\}),
\end{align*}
an elementary calculation (cf. \cite[proof of Lemma 3.5]{ZhangFourier}) shows that
\[
\det(w|w\fc_{\nu}(YX)|\ldots|w\fc_{\nu}(YX)^{n-1}) = 2^n\nu^{n(n-1)/2}\det(w|w(YX)|\ldots|w(YX)^{n-1}),
\]
where we abuse notation writing $\fc_\nu:\End(V_0)-D_1(F)\lra \GL(V_0)$ for the Cayley transform on the lower-right factor. 
This allows us to express the above as 
\[
 \eta^n_2(YX(I-Y X)^{-1}(I-XY)^{-1})\eta_2((I-XY)^{-1})\eta_0\left(\det[w|w(YX)|\cdots |w(YX)^{n-1}]\right)\eta_2(Y),
\]
times the constant $c_{\nu}^{\underline{\eta}}=\eta_0(2^n\nu^{n(n-1)/2})\eta_2((\nu2)^n)$.

Since
\[
\widetilde{\omega}((X,Y),w)=\ep(T_{YX},\eta_0,\psi) \eta_2^n(YX) \eta_0\left(\det[w|w(YX)|\cdots |w(YX)^{n-1}]\right)\eta_2(Y), 
\]
we obtain
\begin{equation*}
\omega(y_\nu,w) = c_{\nu}^{\underline{\eta}}\eta_2((I-XY)^{-1})\widetilde{\omega}((X,Y),w).\qedhere    
\end{equation*}
\end{proof}
\begin{cor}
    For ${\varphi}\in C_c^\infty(\fX(F)-D_1(F))$, $\Phi\in C^\infty_c(F_{n})$, and $(X,Y)\in \X^{rss}(F)$, we have the identity
    \[
 \Orb^{\rH',\ul{\eta}}_{0}({\phi}_\nu,\Phi,y_\nu)=    c_{\nu}^{\underline{\eta}}\cdot\eta_2(I-XY)  \Orb^{\rH',\ul{\eta}}_{0}({\varphi},\Phi,(X,Y))
    \]
    where $y_\nu=\fc_{\nu}(X,Y)$, and ${\phi}_\nu = {\varphi}\circ \fc_\nu^{-1}\in C_c^\infty(\X(F)-D_\nu(F)).$
\end{cor}
\begin{proof}
    This follows directly from the preceding two lemmas, and comparing the formulas \eqref{eqn: basic OI X} and \eqref{eqn: Lie alg linear OI with s}.
\end{proof}
\begin{Lem}\label{Lem: descent on cayley}
    Conjecture \ref{Conj: smooth transfer Lie} implies Conjecture \ref{Conj: smooth transfer} for functions  $\phi\otimes\Phi \in C_c^\infty(\X(F)\times F_{n})$ where ${\phi}\in C_c^\infty(\X(F)-(D_1\cap D_{-1})(F)).$
\end{Lem}
\begin{proof}  For ${\phi}\in C_c^\infty(\X(F)-(D_1\cap D_{-1})(F))$, we may use a partition of unity to write ${\phi} = {\phi}_++{\phi}_-$, where ${\phi}_\nu\in C_c^\infty(\X(F)-D_{\nu}(F))$. Then there exist functions ${\varphi}_\nu\in C_c^\infty(\fX(F)-D_1(F))$ such that 
${\phi}_\nu = {\varphi}_\nu\circ \fc_\nu^{-1}$. 

Assume now that $\underline{f}_{\nu}$ is a $\underline{\eta}$-transfer for ${\varphi}_\nu\otimes \Phi$. Again, using a partition of unity argument argument (cf. \cite[Lemma 3.6]{ZhangFourier}), we are free to assume that for each $f_{\nu,a}\in C_c^\infty(\fq_a(F)-D_1(F))$, where $a\in A$ as before the statement of Conjecture \ref{Conj: smooth transfer}. Setting
\[
f_{a} = c_{\nu}^{\underline{\eta}}\eta_2(I-XY)(f_{+,a}\circ \fc_+^{-1}+ f_{-,a}\circ \fc_-^{-1}),
\]
one now easily sees that $\underline{f} = (f_a)$ gives an $\underline{\eta}$-transfer for $\phi\otimes \Phi$.
\end{proof}

\subsection{Proof of Theorems \ref{Thm: weak smooth transfer} and \ref{Thm: weak endoscopic transfer}}\label{Section: proof weak transfer arg}
We now assume that $F$ is non-archimedean.   It follows immediate from \eqref{cayley calculation} that $\fc_{\nu}(\fX^{invt}) = \X^{invt}$. Thus to prove the theorem, it suffices by Lemma \ref{Lem: descent on cayley} to prove the corresponding claim between $\fX^{invt}$ and $\fq^{invt}$. 

\begin{Lem}\label{Lem: same ois on invt linear}
    Let ${\phi}\in C_c^\infty(\fX^{invt}(F))$ and $\Phi\in C^\infty_c(F_n)$.  For $\underline{\eta}\in\{(\eta,\eta),(\eta,1)\}$, there exists $\varphi\in C_c^\infty(\GL_n(F))$ such that
\begin{equation}\label{eqn: contraction linear}
\Orb^{\rH',\underline{\eta}}_0(\phi\otimes\Phi,(X,Y))=\Orb^{\GL_n(F),{\eta}}(\varphi\otimes\Phi, YX).    
\end{equation}
Conversely, for any $\varphi\in C_c^\infty(\GL_n(F))$, there exists $\wt{\phi}\in C_c^\infty(\fX^{invt}(F)\times F_{n})$ satisfying \eqref{eqn: contraction linear}.
\end{Lem}
\begin{proof}
To handle both cases at once, we denote $\ul{\eta}=(\eta_0,\eta_2)$. Recall that the contraction map 
   \begin{align*}
       R:\fX&\lra \fgl_n\\
        (X,Y)&\longmapsto YX
   \end{align*}
   is equivariant with respect to the $\GL_n$-action (via the second factor of $\rH'=\GL_n\times \GL_n$). Let $\rH' =\rH_{(1)}\times \rH_{(2)}$, with $\rH_{(i)}\simeq \GL_n$. {For $\Re(s_0)>0$, we take $s_1=0$ and  consider the integral}
   \[
 R_{!}^{\eta_2}(\phi)(YX):=\int_{\rH_{(1)}(F)}\phi({h^{(1)}}^{-1}X,Y h^{(1)})  \eta_2(Yh^{(1)})dh^{(1)},
   \]
   where the factor $\eta_2(Y)$ comes from the transfer factor.
Then
\begin{align*}
\eta_2(Y)\int_{\rH'(F)}\phi(h^{-1}\cdot (X,Y))&\Phi(wh^{(2)})\eta_{0}(h^{(2)}) \eta_2(h) |h^{(2)}|^{s_0}\,dh\\& =\int_{\GL_n(F)}R_{!}^{\eta_2}(\phi)(g^{-1}YXg)\Phi(wg)\eta_{0}(g) |g|^{s_0}\,dg.    
\end{align*}
We set $\varphi:=R_{!}^{\eta_2}(\phi)$. Fixing $w\in F_n$ so that $((X,Y),w)$ is strongly regular, and recalling that $T_{(X,Y)}\iso T_{YX}$ with respect to the projection $\rH_{(1)}\times \rH_{(2)}\to \rH_{(2)}$, we obtain an equality
   \begin{align*}
        \Orb_{s_0,0}^{\rH',\ul{\eta},\natural}(\phi\otimes\Phi,[(X,Y),w])=\Orb^{\GL_n(F),\eta_0,\natural}_{s_0}(\varphi\otimes\Phi, (YX,v));
   \end{align*}
 setting $s_0=0$ we obtain $\Orb^{\rH',\underline{\eta}}_0(\phi\otimes\Phi,(X,Y))=\Orb^{\GL_n(F),{\eta_0}}(\varphi\otimes\Phi, YX)$. 

  Noting that 
   \[
   R_!^{\eta_2}:C^\infty_c(\fX^{invt}(F))\lra C^\infty_c(\GL_n(F))
   \] is surjective, the converse also holds.
\end{proof}
\begin{proof}[Proof of Theorem \ref{Thm: weak smooth transfer}]
Combining Lemma \ref{Lem: same ois on invt linear} with Theorem \ref{Thm: Xiao's transfer}, we need only show that the stable orbital integrals on the spaces $\fq_{si}(F)$ and $\fq_{ii}(F)$ may be computed as stable orbital integrals on $\rH_{\tau}(F)$ when the test function is supported on the invertible locus.   We handle the two cases separately.\\

   \noindent\underline{\textbf{split-inert case $\underline{\eta} = (\eta,\eta)$:}} Fix $\tau_1,\tau_2\in \calv_n(E/F)$ and let $\fq=\fq_{\tau_1,\tau_2}=\Res_{E/F}(\fgl_n)$ be the infinitesimal symmetric variety for $\rH_{\tau_1,\tau_2}=\U(V_{\tau_1})\times \U(V_{\tau_2})$. In this case, consider the contraction map 
\begin{align*}
    r^{\tau_1}:\fq&\lra \Herm_{\tau_2}(F),\\
    X&\longmapsto-X^\ast\tau_1^{-1} X\tau_2^{-1};
\end{align*} this is designed to intertwine the $\rH_{\tau_1,\tau_2}$-action on $\fq$ given in Lemma \ref{Lem: unitary Lie} with the conjugation action of $\U(V_{\tau_2})$ on $\Herm_{\tau_2}$.

As in the linear setting, this gives a categorical quotient of the $\U(V_{\tau_1})$-action on $\fq$ and the restriction $$r|_{\fq^{invt}}: \fq^{invt}\to \Herm_{\tau_2}^{\circ}$$ is a $\U(V_{\tau_1})$-torsor, so that the contraction map 
\begin{equation}\label{eqn: contraction torsor}
r_!:\bigoplus_{\tau_1\in \calv_n(E/F)}C^\infty_c(\fq^{invt}_{\tau_1,\tau_2}(F))\lra C^\infty_c(\Herm_{\tau_2}^{\circ}(F))    
\end{equation}
is surjective and  for any $\underline{f}\in\bigoplus_{\tau_1\in \calv_n(E/F)}C^\infty_c(\fq^{invt}(F))$ and $Z\in \fq^{rss}(F)\subset \fq^{invt}(F)$, we have 
\[
\SO^{\rH_{\tau_1,\tau_2}}(f_{\tau_1},Z) = \SO^{\U(V_{\tau_2})}(r_{!}(f),r(Z)),
\]
where $r_!(f)= \sum_{\tau_1}r^{\tau_1}_!(f_{\tau_1})\in C_c^\infty(\Herm_{\tau_2}^{\circ}(F))$. Combining Lemma \ref{Lem: same ois on invt linear} with Theorem \ref{Thm: Xiao's transfer} proves the claim in this direction.

Now fix $\tau_2\in \calv_n(E/F)$ and let ${\phi}\in C_c^\infty(\fX^{invt}(F))$ and $\Phi\in C^\infty_c(F_n)$. We assume that for any $(X,Y)\in\fX^{rss}(F)$ for which there does not exist a matching $Z\in \fq_{\tau_1,\tau_2}^{rss}$ for any  $\tau_1\in \calv_n(E/F)$, we have $\Orb^{\rH',\underline{\eta}}_0(\phi\otimes\Phi,(X,Y))=0$. Let $\varphi=R^{\eta}_!(\phi)$. By properties of the contraction map $r$ and the identity
$$\Orb^{\rH',\underline{\eta}}_0(\phi\otimes\Phi,(X,Y))=\Orb^{\GL_n(F),\eta_0}(\varphi\otimes\Phi, YX),$$
we see that $\Orb^{\GL_n(F),\eta_0}(\varphi\otimes\Phi, z)=0$ for any $z=YX\in \GL_n(F)$ such that there is no matching semi-simple element $y\in \Herm_{\tau_2}^{\circ}(F)$. By Theorem \ref{Thm: Xiao's transfer}, there exists $f\in C^\infty_c(\Herm_{\tau_2}^{\circ}(F))$ such that for any matching $x\leftrightarrow y$
\[
2^{|S_1^{ram}(z)|}\Orb^{\GL_n(F),\eta_0}(\varphi\otimes\Phi, z)=\SO^{\U(V_{\tau_2})}(f,y).
\]
By surjectivity of \eqref{eqn: contraction torsor}, we obtain a collection $\underline{f} = (f_{\tau_1})$ with $f_{\tau_1}\in C^\infty_c(\fq^{invt}_{\tau_1,\tau_2})$ such that for any matching $(X,Y)\in \fX^{rss}(F)$ and $Z\in \fq_{\tau_1,\tau_2}^{rss}$
\begin{align*}
   2^{|S_1^{ram}(z)|} \Orb^{\rH',\underline{\eta}}_0(\phi\otimes\Phi,(X,Y))&=2^{|S_1^{ram}(z)|}\Orb^{\GL_n(F),\eta_0}(\varphi\otimes\Phi, YX)\\& = \SO^{\U(V_{\tau_2})}(f,-Z^\ast \tau_1^{-1} Z\tau_2^{-1})=\SO^{\rH_{\tau_1,\tau_2}}(f_{\tau_1},Z).
\end{align*}

Thus $\underline{f}$ is an $(\eta,\eta)$-transfer for $\phi\otimes \Phi$. \\

   \noindent\underline{\textbf{inert-inert case $\underline{\eta}=(\eta,1)$:}} Recalling the notation from Lemma \ref{Lem: unitary Lie}, we have
   \[
   \fq_{ii}(F) \simeq\Herm_n(F)\times \Herm_n(F);
   \]
   We thus obtain an isomorphism $ \fq_{ii}^{invt}(F) \simeq\Herm_{n}^{\circ}(F)\times \Herm_{n}^{\circ}(F)$. Recall now the decomposition
   \[
   \Herm_{n}^{\circ}(F)=\bigsqcup_{\tau\in\calv_n(E/F)}\Herm_{n}^{\circ}(F)_\tau,
   \]
   where the (open) subset $\Herm_{n}^{\circ}(F)_\tau \simeq \GL_n(E)/\U(V_\tau)$ is the set of $x\in \Herm_{n}^{\circ}(F)$ such that $\GL_n(E)_x\simeq \U(V_\tau)$. 

  Let $f=f_1\otimes f_2\in C^\infty_c(\Herm_{n}^{\circ}(F)\times \Herm_{n}^{\circ}(F))$. We may write
   \[
   f_1 = \sum_{\tau}f_{1,\tau},
   \]
   where $\supp(f_{1,\tau})\subset \Herm_{n}^{\circ}(F)_\tau$, so that
$f = \sum_{\tau}f_{1,\tau}\otimes f_2.$ By linearity, it suffices to assume that $f_1 = f_{1,\tau}$ for some $\tau$. Then there exists $f_\tau\in C^\infty_c(\GL_n(E))$ such that 
   \[
   f_{1,\tau}(y) =r_{!}^\tau(f_\tau)(y)=\int_{U(V_\tau)}f_\tau(g_yu)du,
   \]
   where where $r^\tau(g) = g\tau g^\ast$ and $y = r^\tau(g_y)$.
   
Recall from \eqref{eqn: unwind to Hn} that for any regular semi-simple element $(x_1,x_2)\in \Herm_{n}^{\circ}(F)_\tau\times \Herm_{n}^{\circ}(F)$, we have the identity
   \begin{equation*}
   \SO^{\GL_n(E)}(f_{1,\tau}\otimes f_2, (x_1,x_2)) = \SO^{\U(V_{\tau^{-1}})}(f^\ast_\tau\ast f_2, x\tau),
   \end{equation*}
   where $x_1 = g_1 \tau g_1^\ast$ and $x= g_1^\ast x_2 g_1\in \Herm_{n}^{\circ}(F)$, $f^\ast(g) = f(g^\ast)$, and where
   \[
   (f\ast f_2)(y) = \int_{\GL_n(E)}f(g^{-1})f_2(gyg^\ast)dg.
   \] We note that $x\tau\in \Herm_{\tau^{-1}}^{\circ}(F)$, so that the right-hand side is a stable orbital integral arising in the comparison of Theorem \ref{Thm: Xiao's transfer}.

Thus, given $f=f_{1,\tau}\otimes f_2$, let $\widetilde{\varphi}\in C^\infty_c(\GL_n(F)\times F_n)$ be a smooth transfer of $f_\tau^\ast\ast f_2\in C^\infty_c(\Herm_{n}^{\circ}(F))$ in the sense of Theorem \ref{Thm: Xiao's transfer}. This lifts via $R_!^{1}$ to a function $\widetilde{\phi}\in C_c^\infty(\fX^{invt}(F)\times F_n)$, which we claim give a $(\eta,1)$-transfer for $f$. For any $(x_1,x_2)\in \fq_{ii}^{rss}(F)$ matching $(X,Y)\in \fX^{rss}(F)$, we have
\[
\SO^{\GL_n(E)}(f_{1,\tau}\otimes f_2, (x_1,x_2)) = \SO^{\U(V_{\tau^{-1}})}(f^\ast_\tau\ast f_2, x\tau),
\]
and 
\[
\Orb_{0,0}^{\rH',\ul{\eta}}(\wt{\phi},(X,Y))= \Orb^{\GL_n(F),\eta}(\wt{\varphi}, YX),
\]
we only need to check that $x\in \Herm_{n}^{\circ}(F)$ matches $YX\in \GL_n(F)$ in the sense of Lemma/Definition \ref{Lem: matching def}. This is the case if and only if 
\begin{equation*}
\det(tI_n - x\tau)= \det(tI_n -  g_1^{\ast} x_2(g_1\tau g_1^\ast)(g_1^{\ast})^{-1})  =\det(tI_n-x_2x_1)    
\end{equation*}
 agrees with $\det(tI_n-YX)$, which is true by assumption that $(x_1,x_2)$ matches $(X,Y)$. 

Conversely, suppose that $\phi\in C^\infty_c(\fX^{invt}(F))$ and $\Phi\in C^\infty_c(F_n)$. Let $f\in C^\infty_c(\Herm_{n}^{\circ}(F))$ be a transfer of $R_!^1(\phi)\otimes \Phi$ in the sense of Theorem \ref{Thm: Xiao's transfer}. We can write
\[
f= \sum_\tau f_\tau = \sum_{\tau}r_{!}^\tau(f'_\tau),
\]
for appropriate $f'\in C^\infty_c(\GL_n(E))$, so that $r_{!}^\tau(f')\in C^\infty_c(\Herm_{n}^{\circ}(F)_\tau)$. We claim that $r_{!}^\tau(f'_\tau) = cf'_\tau\ast 1_{K'\ast \tau}$, where $K'\subset \GL_n(E)$ is a small enough open compact subgroup that $f'_\tau$ is constant on $K'$-double cosets and $c\neq 0$ is a constant depending on $K'$. Indeed, both functions vanish off $\Herm_{n}^{\circ}(F)_\tau$ and 
\begin{align*}
    f'_\tau\ast 1_{K'\ast \tau}(g_1\tau g_1^\ast) &= \int_{\GL_n(E)}f_\tau'(g^{-1}) \bfun_{K'\ast \tau}(gg_1\tau g_1^\ast g^\ast)dg\\
                                                    &= \int_{\GL_n(E)}f_\tau'(g_1g^{-1}) \bfun_{K'\ast \tau}(g\tau  g^\ast)dg\\
                                                    &= \int_{K'U(V_\tau)}f_\tau'(g_1g^{-1}) dg\\
                                                    &= \vol_{dx}(K'\ast \tau)\int_{U(V_\tau)}f_\tau'(g_1u) du = c^{-1}r_{!}^\tau(f'_\tau)(g_1\tau g_1^\ast).
\end{align*}
Thus, there are functions $f_{1,\tau}\in C^\infty_c(\GL_n(E))$ and $f_{2,\tau}\in C_c^\infty(\Herm_{n}^{\circ}(F))$ such that $f= \sum_{\tau}f_{1,\tau}^\ast\ast f_{2,\tau},$
so that for any matching $(X,Y)\in \fX^{rss}(F)$ and $(x_1,x_2)\in \fq_{ii}^{rss}(F)$ with $x_1=g_1\tau g_1^\ast\in \Herm_{n}^{\circ}(F)_\tau$
\begin{align*}
    \Orb_{0,0}^{\rH',\ul{\eta}}({\phi},\Phi,(X,Y))&= \Orb^{\GL_n(F),\eta}({\varphi},\Phi, YX)\\
                                                   &=\SO^{\U(V_{\tau^{-1}})}( f_{1,\tau}^\ast\ast f_{2,\tau}, g_1^\ast x_2 g_1)\\
                                                   &=\SO^{\GL_n(E)}(r_{!}^\tau(f_{1,\tau})\otimes f_{2,\tau},(x_1,x_2)).
\end{align*}
Thus $\sum_\tau f_{1,\tau}\otimes f_{2,\tau}$ gives an $(\eta,1)$-transfer for $\phi\otimes\Phi$. 
\end{proof}
\begin{proof}[Proof of Theorem \ref{Thm: weak endoscopic transfer}]
As in the prior section, it suffices by Lemma \ref{Lem: descent on cayley} to prove the corresponding claim between $\fq_{ii}^{invt}(F)$ and $\fq_{si}^{invt}(F)$. \quash{Thus, we show that for any $f\in C_c^\infty(\fq_{ii}^{invt}(F))$, there exists a $\underline{f}=(f_{\tau_1})$ such that $f_{\tau_1}\in C_c^\infty(\fq^{invt}_{\tau_1,\tau_2}(F))$ and
     \[
     \De_{\varepsilon}(x_1,x_2) \Orb^{\rH,\varepsilon}(f,(x_1,x_2)) = \SO^{\rH_{\tau_1,\tau_2}}(f_{\tau_1},X),
     \]
     for all matching regular semi-simple $(x_1,x_2)\leftrightarrow_{\tau_1}X$.  Conversely, for any $\underline{f}$ with support in the invertible locus, there exists an $\varepsilon$-transfer ${f}\in C_c^\infty(\fq_{ii}^{invt}(F))$.

    } By Lemma \ref{Lem: same ois on invt linear} and the preceding proof, it clearly suffices to produce a function $f'\in  C_c^\infty(\fq_{ii}^{invt}(F))$ such that
     \[
     \De_{\varepsilon}(x_1,x_2) \Orb^{\rH,\varepsilon}(f,(x_1,x_2)) = \SO^{\rH_{ii}}(f',(x_1,x_2)),
     \]
     as we may then pull back to functions on $\bigoplus_{\tau_1} C_c^\infty(\fq^{invt}_{\tau_1,\tau_2}(F))$.   But this is immediate on the invertible locus as we may set $f'(x_1,x_2) := \De_{\varepsilon}(x_1,x_2)f(x_1,x_2)$. The claim now follows.
\end{proof}



\section{Descent to the Lie algebra}\label{Sec: descent} 
In this appendix, we prove Proposition \ref{Prop: descent fundamentals}. In this appendix, $E/F$ is an unramified quadratic extension of $p$-adic fields and $\eta=\eta_{E/F}$ is the associated quadratic character. We assume throughout that if $e$ denotes the ramification degree of $F/\Q_p$, then $p>\max\{e+1,2\}$ (cf. \cite[Lemma 8.1]{Lesliedescent}). In particular, we assume $p$ is odd. We fix a uniformizer $\vp\in \calo_{F}$.
\subsection{Descent for linear orbital integrals}Let $\X=\G'/\rH'$ as before. For each $k\geq 0$, we have $\X_k(\calo_F) = \X(F)\cap \G'_k(\calo_F)$ and $\mathfrak{X}_k(\calo_F)=\mathfrak{X}\cap \vp^k\fg'(\calo_F)$. Set $\Phi_0 =\bfun_{\calo_{F,n}}$. 

Fix $\ul{\eta}= (\eta_0,\eta_1)$ of unramified characters of $F^\times$. On the variety, we consider the orbital integrals $\Orb^{\rH',\underline{\eta}}_{s_0,s}(\bfun_{\X_k(\calo_F)}\otimes\Phi_0,(y,w))$ given by the formula
\[
 \int_{\rH'(F)}\bfun_{\X_k(\calo_F)}(h^{-1}\cdot y)\Phi_0(wh^{(2)})\eta_{0}(h^{(2)}) \eta_2(h) |h^{(2)}|^{s_0} |h|^{s}\,dh,
\]
where $(y,w)\in \X(F)\times F_n$ is strongly regular; on the Lie algebra we consider the orbital integrals $\Orb^{\rH',\underline{\eta}}_{s_0,s}(\bfun_{\fX_k(\calo_F)}\otimes\Phi_0,[(X,Y),w])$ given as
\[
 \int_{\rH'(F)}\bfun_{\fX_k(\calo_F)}(h^{-1}\cdot (X,Y))\Phi_0(wh^{(2)})\eta_{0}(h^{(2)}) \eta_2(h) |h^{(2)}|^{s_0} |h|^{s}\,dh.
\]
Both of these expressions are holomorphic in $(s_0,s)\in \cc^2$, and the Fundamental lemmas are concerned with the values (cf. Lemma/Definition \ref{LemDef: linear OI})
\begin{equation}\label{limit unit}
    \Orb_{s}^{\rH',\ul{\eta}}(\bfun_{\X_k(\calo_F)}\otimes\Phi_0,y):= \omega(y,w)\frac{\Orb^{\rH',\underline{\eta}}_{s_0,s}(\bfun_{\X_k(\calo_F)}\otimes\Phi_0,(y,w))}{L(s_0, \eta_0,T_y)}\bigg|_{s_0=0}
\end{equation}

and (cf. Definition \ref{eqn: OI lin Lie})
\begin{equation}\label{limit unit lie}
\Orb_{s}^{\rH',\ul{\eta}}(\bfun_{\fX_k(\calo_F)}\otimes\Phi_0,(X,Y)):=\widetilde{\omega}((X,Y),w) \frac{\Orb^{\underline{\eta}}_{s_0,s}(\bfun_{\fX_k(\calo_F)}\otimes\Phi_0,[(X,Y),w])}{L(s_0, \eta_0,T_{YX})}\bigg|_{s_0=0}.
\end{equation}

In this section, we show how to relate
\subsubsection{Cayley transform and the $\heartsuit$-locus}\label{Sec: descent linear} Recall the discussion of the Cayley transforms $\fc_{\nu}$ with $\nu=\pm1$ in \S \ref{Section: Cayley}.

\quash{
An inspection of the proof of the Lemma 7.2 of \cite{Lesliedescent} immediately gives the following.
\begin{Lem}\label{Lem: Cayley map linear}
 The Cayley transform
\begin{align*}
\fc_{\pm1}:\End(V)-D_1(F)&\lra \GL(V)\\
			X&\longmapsto \mp(1+X)(1-X)^{-1}
\end{align*}
induces a $\GL(V_1)\times \GL(V_2)$-equivariant isomorphism between $\mathfrak{X}(F)-D_1(F)$ and $\X(F)-D_{\pm1}(F)$. The images of $\mathfrak{X}(F)-D_1(F)$ under $\fc_{\pm}$ form a finite cover by open subsets of $\X(F)-(D_1\cap D_{-1})(F)$. In particular, the images contain the regular semi-simple locus of $\X(F)$.
\end{Lem}

It is well-known that the Cayley transform $\fc_{\pm1}: \mathfrak{X}(F)-D_1(F)\to\X(F)-D_{\pm1}(F)$ is equivariant with respect to the $\rH'$-actions and takes regular elements to regular elements; see \cite[Section 7]{Lesliedescent} for a discussion in the unitary setting which carries over \emph{mutatis mutandis} to our setting.
\begin{Lem}
Suppose that $[(X,Y),w]\in \mathfrak{X}^{sr}_{ext}(F)$. Then $(y_{\pm},w):=[\fc_{\pm1}(X,Y),w]\in \X^{sr}_{ext}(F)$. Moreover,
\[
b
\]
\end{Lem}

}
We define certain open subsets of $\X(\calo_F)$.  We define the \textbf{very regular locus} of $\mathfrak{X}(\calo_F)$ 
\[
\mathfrak{X}^\heartsuit(\calo_F)=\{\de\in \mathfrak{X}^{rss}(\calo_F)-D_{1}(\calo_F): \overline{\de}\in \End(\Lam/\vp\Lam)-D_{1}(k) \},
\]
where $\overline{\de}\in\End(\Lam/\vp\Lam)$ denotes the image of $\de$ under the modular reduction map. Similarly, we define  the $\nu$-\textbf{very regular locus} of $\X(\calo_F)$
\[
\X^{\heartsuit,\nu}(\calo_F)=\{x\in\X^{rss}(\calo_F): \overline{x}\in \X(k)-D_{\nu}(k)\}.
\]
\begin{Lem}\label{Lem: very reg count}
Suppose that $x\in \X(F)$ and suppose $x=\fc_\nu(\de)$ for $\de\in \mathfrak{X}(F)$. Then $x\in \X^{\heartsuit,\nu}(\calo_F)$ if and only if $\de\in\mathfrak{X}^{\heartsuit}(\calo_F)$.
\end{Lem}
\begin{proof} The proof is simple, and mirrors the argument in \cite[Lemma 7.7]{Lesliedescent}.
\end{proof}
This shows that
\[
\X^{\heartsuit,\nu}(\calo_F)=\fc_\nu\left(\mathfrak{X}^\heartsuit(\calo_F)\right).
\]
In particular, for any $x\in \X^{\heartsuit,\nu}(\calo_F)$, the reduction $\overline{x}\in\X(k)$ is in the image of the Cayley transform. It is easy to see that  for any $k\geq 1$,  $\fX_k(\calo_F)\subset \mathfrak{X}^\heartsuit(\calo_F)$ and $\X_k(\calo_F) \subset \X^{\heartsuit,-1}(\calo_F)$, with
\[
\X_k(\calo_F)=\fc_{-1}\left(\mathfrak{X}_k(\calo_F)\right).
\]

Define $\X^{\heartsuit, \nu}(F)$ to be the open subset of $\X^{rss}(F)$ consisting of elements in the same stable orbit as an element in $\X^{\heartsuit,\nu}(\calo_F)$; we similarly define $\fX^{\heartsuit}(F)$. This locus may be characterized as those elements of $\X(F)$ with eigenvalues $\lam\in \calo^\times_{F^{alg}}$ all satisfying 
\[
|\lam-\nu|=1.
\]
The next lemma is now straightforward.

\begin{Lem}\label{Lem: descent linear heart}
Suppose that $y=\fc_{\nu}(X,Y)\in\X^{\heartsuit,\nu}(F)$. Then for any $w\in F^n$ such that $(y,w)\in \X^{sr}_{ext}(F)$ and any $s\in \cc$, we have
\[
\Orb_{s}^{\rH',\ul{\eta}}(\bfun_{\X_k(\calo_F)}\otimes\Phi_0,y)=   \Orb_{s}^{\rH',\ul{\eta}}(\bfun_{\fX_k(\calo_F)}\otimes\Phi_0,(X,Y))
\]
for all $k\geq0$.
\end{Lem}
\begin{proof} We verify that all the components of the formulas \eqref{limit unit} and \eqref{limit unit lie} agree for all $(s_0,s)$, which clearly suffices. 
The only thing to check is the relationship between the transfer factors, as the identity
\[
\Orb^{\rH',\underline{\eta}}_{s_0,s}(\bfun_{\X_k(\calo_F)}\otimes\Phi_0,(y,w))= \Orb^{\rH',\underline{\eta}}_{s_0,s}(\bfun_{\fX_k(\calo_F)}\otimes\Phi_0,[(X,Y),w])
\]
for $y=\fc_{\nu}(X,Y)\in\X^{\heartsuit,\nu}(F)$ follows easily from Lemma \ref{Lem: very reg count}. But if $y=\fc_\nu(X,Y)$, then Lemma \ref{Lem: tranfer factor cayley} shows that
\[
\widetilde{\omega}((X,Y),w)=c_{\nu}^{\underline{\eta}}\cdot\eta_2(I-XY) {\omega}(y_{\nu},w).
\]
where $c_{\nu}^{\underline{\eta}} =\eta_0(2^n)$ in this unramified setting; thus $c_{\nu}^{\underline{\eta}}=1$ if the residual characteristic of $F$ is odd. Since $(X,Y)\in \fX^\heartsuit(F)$, we have $\eta_2(I-XY)=1$ as $\det(I-XY)$ is a unit
so that the transfer factors agree. The lemma follows. 
\end{proof}
As noted above, $\fX_k(\calo_F)\subset \mathfrak{X}^\heartsuit(\calo_F)$ and $\X_k(\calo_F) \subset \X^{\heartsuit,-1}(\calo_F)$, so that this lemma suffices to descend the orbital integrals to the infinitesimal setting for all $k\geq1$. It remains to consider orbits contained in the compliment of $\X^{\heartsuit,1}(\calo_F)\cup \X^{\heartsuit,-1}(\calo_F)$ in the case $k=0$.
\quash{\[
\fc_{\nu}(X,Y) =  -\nu\left(\begin{array}{cc}(I+YX)(I-YX)^{-1}&-2Y(I-XY)^{-1}\\-2X(I-YX)^{-1}&(I+XY)(I-XY)^{-1}\end{array}\right). 
\]
so that 
\[
{\omega}(y,w) =\eta\left(\det(w|w\fc_{\nu}(YX)|\ldots|w\fc_{\nu}(YX)^{n-1})\right),
\]
where we abuse notation and use $\fc_{\nu}$ for the corresponding Cayley transform on $\fgl_n(F)$. An elementary calculation (cf. \cite[proof of Lemma 3.5]{ZhangFourier}) shows that
\[
\det(w|w\fc_{\nu}(YX)|\ldots|w\fc_{\nu}(YX)^{n-1}) = 2^n\nu^{n(n-1)/2}\det(w|w(YX)|\ldots|w(YX)^{n-1}),
\]}
\subsubsection{Remaining orbits}\label{Section: descent on linear side}
Since we only consider $k=0$, we drop the subscript. To complete the reduction to the Lie algebra, we need to adapt several of the results of \cite[Section 8]{Lesliedescent} to our setting. We shall adopt the notation from that source with minimal further comment, referring to the corresponding results when the proofs are essentially the same. 

Suppose that $(y,w)\in\X(F)\times F_n$ is strongly regular, and that $y$ is \emph{strongly compact} in the sense of \cite[Sec. 4]{Lesliedescent}. In fact, we can go ahead and assume that $(y,w)\in \X(\calo_F)\times \calo_{F,n}$ since only orbits containing such elements contribute to the fundamental lemma. Under this assumption, it is shown in \emph{ibid.} that there is a canonical decomposition $y= y_{as}y_{tn}$ with $y_{as},y_{tn}\in \X(\calo_F)$ generalizing the classical topological Jordan decomposition.  

Denote the natural representation of $\G'=\GL_{2n}$ by $V_{2n}$. Recall that there is an element $\epsilon\in \G'(F)$ such that $\ep^2=1$, and the eigenspace decomposition of $\epsilon$ on $V_{2n}$ is 
\[
V_{2n} = W_1\oplus W_{-1},
\]
with $\dim(V_\nu) =n$ for $\nu=\pm1$. The involution $\theta = \Ad(\ep)$ on $\G'$ satisfies that $\rH' ={\G'}^\theta$. 


We now assume that $y\notin \X^{\heartsuit,\nu}(F)$ for either $\nu=\pm1$. Let $y = y_{as}y_{tn}$ be the topological Jordan decomposition of $y\in \X(\calo_F)$.  There is a decomposition 
\[
V_{2n} = V_0\oplus V_1\oplus V_{-1}
\]
with $V_1$ (resp. $V_{-1}$) is the $1$-eigenspace (resp. $-1$-eigenspace) of $y_{as}$ and $V_0$ is the orthogonal compliment of $V_1\oplus V_{-1}$ in $V_{2n}$. By our assumption $\dim(V_{\pm1})>0$. We now decompose the absolutely semi-simple part $y_{as}=\ga\cdot z_{as}$ where  
\[
\ga = (I_{V_0}, I_{V_1}, -I_{V_{-1}})\in \G'_{y_{as}}(\calo_F),
\]
 and 
 \[
 z_{as}=(y_{as}|_{V_0}, I_{V_1}, I_{V_{-1}})\in \G'_{y_{as}}(\calo_F).
 \]
 In particular,
 \[
 \G'_\ga=\GL(V_0\oplus V_1)\times \GL(V_{-1})\subset \G'
 \]
 and $\G'_{y_{as}}\subset \G'_\ga$ is a Levi subgroup of $\G'_\ga$.  
  It is clear that both $\ga, z_{as}\in \X(\calo_F)$. If we set $z=\ga^{-1}y,$ then $z=z_{as}y_{tu}\in \G'_{\ga}(\calo_F)$ is the topological Jordan decomposition of $z$. 
 \begin{Lem}\cite[Lemma 8.2]{Lesliedescent}\label{Lem: new decomp linear}
  With the notation as above, we have that $z\in \X_{\ga}^{rss}(F)$ where $\X_{\ga}:=\G'_{\ga}/\rH'_\ga$ denotes the descendant of $\X$ at $\ga$.
 \end{Lem}

Decomposing $V_{2n}$ with respect to the action of $\ga$, we make an abuse of notation and write
 \[
V_{2n}=V_1\oplus V_{-1}.
 \]
 effectively ignoring the previous decomposition $V_{2n}=V_0\oplus V_1\oplus V_{-1}$. The descendant $(\G'_\ga,\rH'_\ga)$ decomposes as a product
 \[
 (\G'_\ga,\rH'_\ga) = (\G'_{\ga,1},\rH'_{\ga,1})\times  (\G'_{\ga,-1},\rH'_{\ga,-1}) 
 \]
 where for both $\nu=\pm1$,
 \[
(\G'_{\ga,\nu},\rH'_{\ga,\nu}) =(\GL(V_{\nu}), \GL(V_{\nu,1})\times \GL(V_{\nu,-1})).
 \]
Here $V_{\nu,\nu'}=\{w\in V_{2n}: \ga w=\nu w,\:\: \ep w=\nu' w\}$. We have the associated symmetric spaces
 \[
 \X_{1}=\GL(V_{1})/ \GL(V_{1,1})\times \GL(V_{1,-1})
 \]
 and
 \[
 \X_{-1}=\GL(V_{-1})/ \GL(V_{-1,1})\times \GL(V_{-1,-1}).
 \]
If we set $\mathcal{L}=\mathcal{L}_1\oplus \mathcal{L}_{-1}$ for the standard lattice in $\fgl(V_{2n})$ inducing the hyperspecial subgroups $\rH'(\calo_F)\subset \G'(\calo_F)$, we have the decomposition
 \[
 \mathcal{L}_1 = (\mathcal{L}_1\cap V_{1,1})\oplus (\mathcal{L}_1\cap V_{-1,1})
 \] 
 and
 \[
\mathcal{L}_{-1}= (\mathcal{L}_{-1}\cap V_{1,-1})\oplus (\mathcal{L}_{-1}\cap V_{-1,-1}).
\]These lattices give rise to a hyperspecial subgroup $\rH'_\ga(\calo_F)$. 
\begin{Lem}\cite[Lemma 8.3]{Lesliedescent}\label{Lem: into the reg locus linear}
Writing $y=\ga z= (y_1,-y_{-1})\in \X_1(\calo_F) \times \X_{-1}(\calo_F)$, we have
\[
-y_{-1}\in \X_{-1}^{\heartsuit,1}(\calo_F)\:\text{ and }\:y_1\in \X_{1}^{\heartsuit,-1}(\calo_F).
\]
Additionally, the product decomposition of $(\G'_\ga,\rH'_\ga)$ induces a decomposition of the centralizer
\begin{equation}\label{eqn: stabilizer decomposition linear}
\rH'_y = \rH'_{y,1}\times \rH'_{y,-1}\subset \GL(V_{1}) \times \GL(V_{-1}).   
\end{equation}
\end{Lem}

Suppose now that $(y,w)\in \X(\calo_F)\times \calo_{F,n}$ is strongly regular and set $y=\ga z$ to be the previously derived decomposition. Decomposing $V_{2n} = V_1\oplus V_{-1}$, we write
\[
w = w_1+w_{-1}\in V_1(\calo_F)\oplus V_{-1}(\calo_F).
\]
\begin{Lem}\label{Lem: descent on weight factor}
For $\nu\in \{\pm1\}$, we have $(\nu y_\nu, w_{\nu})\in \X_{\nu}(\calo_F)\times \calo_{F,n_\nu}$ is strongly regular, where $n_\nu = \dim(V_\nu)$.
\end{Lem}
\begin{proof}
Let us set
\[
V_\nu^w:=\sspan\{w_\nu(\nu y_\nu)^i: i\in\{0,\ldots, n_{\nu}-1\}\}\subset V_\nu.
\]
Note that this is the same space as 
$
\sspan\{w_\nu(\nu y_\nu)^i: i\in \zz_{\geq0} \}.
$
Since $(y,w)$ is strongly regular, we have
\[
\sspan\{wy^i: i\in\{0,\ldots, n-1\}\}=V_{2n}= V_1\oplus V_{-1}\supset V_1^w\oplus V_{-1}^w,
\]
the claim follows by noting that for all $i\in \zz_{\geq0}$
\[
wy^i = [w_1 , w_{-1}]\begin{psmatrix}
y_1^i&\\&(-y_{-1})^i
\end{psmatrix} = [w_1 
y_1^i, w_{-1}(-y_{-1})^i]\in V_1^w\oplus V_{-1}^w.
\]
Thus, $V_{2n} = V_1\oplus V_{-1}\subset V_1^w\oplus V_{-1}^w$ showing that $V_\nu^w=V_\nu$. This is precisely the claim.
\end{proof}

\begin{Lem}
Suppose that $(y,w)\in \X(\calo_F)$ is strongly regular and $h\in \rH'(F)$. If $y=\ga z= (y_1,-y_{-1})$ is the decomposition just described, and $h= (h_1,h_{-1})\in \rH'_\ga(F)\subset \G'_\ga(F)=\GL(V_1)\times \GL(V_{-1})$ then for $s\in \cc$
\begin{equation}
    \int_{\rH_y'(F)}\Phi_0\left(wth^{(2)}\right)|t|^s\eta(t)dt= \prod_{\nu}\int_{\rH_{y,\nu}'(F)}\Phi_{0,\nu}\left(w_\nu t_\nu h_\nu^{(2)} \right)|t_\nu|^s\eta(t_\nu)dt_\nu,
\end{equation}
where $dt=dt_1dt_{-1}$ is the normalized decomposition of the Haar measure on $\rH'_{y}(F) = \rH'_{y,1}(F)\times \rH'_{y,-1}(F)$, and $\Phi_{0,\nu}$ denotes the characteristic function of $\calo_{F,n_\nu}$.
\end{Lem}
\begin{proof}
This is an integral of Tate type, in that if we write $A=R(y)$, then
\begin{align*}
     F[A]&\overset{\sim}{\lra} F_n \\
     p(A)&\longmapsto  wp(A),
\end{align*}
and $\rH'_y(F)\iso \T_A\cong F[A]^\times$, where the first isomorphism is via the projection onto the second factor. Thus, for any $f\in C_c^\infty(F_n)$, we may consider $f(w-)$ as a function on $F[A]$ realizing
\[
\int_{\T_A(F)}f(wt)|t|^s\eta(t)dt
\]
as a Tate integral. In particular, it converges for $\Re(s)>0$, admits meromorphic continuation, and is holomorphic at $s=0$ when $\eta$ is non-trivial. Thus, the integral obtained converges absolutely so that the claim makes sense.

The decomposition of the integral is immediate from this convergence and the product decompositions above.
\end{proof}
We now consider descent of the orbital integrals.
\begin{Lem}\label{Lem: linear descent final}
With the notation as in the previous lemma, we have
\[
\Orb^{\rH',\underline{\eta}}_{s_0,s}(\bfun_{\X(\calo_F)}\otimes\Phi_0,(y,w))=\prod_\nu\Orb^{\rH',\underline{\eta}}_{s_0,s}(\bfun_{\X_\nu(\calo_F)}\otimes\Phi_{0,\nu},(\nu y_\nu,w_\nu))
\] 
and 
\[
{\omega}(y,w) ={\omega}(y_1,w_1){\omega}(-y_{-1},w_{-1}).
\]
\end{Lem}
\begin{proof}
Since our Haar measures are chosen so that the indicated hyperspecial maximal compact subgroups have volume 1, we easily compute 
\begin{equation}\label{OI decomp 1}
\Orb^{\rH',\underline{\eta}}_{s_0,s}(\bfun_{\X(\calo_F)}\otimes\Phi_0,(y,w))=\displaystyle\sum_{\substack{y'\in \mathcal{H}'(\calo_F)\backslash \rH'(F)\cdot y\\y'\in \X(\calo_F)}}  \eta_{s_0,s}(h)\int_{\rH_y'(F)}\Phi_0\left(wth^{(2)}\right)\eta_0(t)|t|^{s_0}dt,
\end{equation}
where $h^{-1}\cdot y = y'$ and $\eta_{s_0,s}(h):=\eta_0(h^{(2)})|h|^{s_0}\eta_2(h)|h|^s$ 

Arguing as in the proof of Proposition 4.11 and Lemma 8.3 of \cite{Lesliedescent}, if $y= \ga z$ and $y'=\ga'z'$ we may then assume that $\ga=\ga'$, so that $z,z'\in \G'_\ga(F)$ and that $h\in \rH'_\ga(F)$. Note that this implies $z$ and $z'$ lie in the same $\rH_\ga'(F)$-orbit. Applying Lemmas \ref{Lem: into the reg locus linear} and \ref{Lem: descent on weight factor}, we see that \eqref{OI decomp 1} equals
\begin{align*}
    \displaystyle\sum_{\substack{y'=(y_1',-y_{-1}')\in \mathcal{H}_\ga'(\calo_F)\backslash \rH'_\ga(F)\cdot y\\y'\in \X(\calo_F)}} & \eta_s(h_1)\eta_s(h_{-1})\int_{\rH_{y,1}'(F)}\prod_{\nu=\pm1}\Phi_{0,\nu}\left(w_\nu t_\nu h_\nu \right)\eta_0(t_\nu)|t_\nu|^{s_0}dt_\nu \\
    &=\prod_\nu\left(\displaystyle\sum_{\substack{ y_\nu'\in \mathcal{H}'_{\ga,\nu}(\calo_F)\backslash \rH'_{\ga,\nu}(F)\cdot y_\nu\\y_\nu'\in \X_{\nu}(\calo_F)}}  \eta_s(h_\nu)\Phi_{0,\nu}\left(w_\nu t_\nu h_\nu \right)\eta_0(t_\nu)|t_\nu|^{s_0}dt_\nu\right).
\end{align*}
This gives the first result.

Applying Lemmas \ref{Lem: into the reg locus linear} and \ref{Lem: descent on weight factor}, the pairs $(\nu y_\nu,w_\nu)$ lie in the very regular loci of the respective descent varieties, and the relevant determinant is block diagonal. This gives the identity on weight factors.
\end{proof}
Combining this result with Lemmas \ref{Lem: into the reg locus linear} and \ref{Lem: descent linear heart} now implies that 
\begin{align}\label{eqn: descent product formula}
    \Orb^{\eta}_{(y,w)}\left(\bfun_{\X(\calo_F)},\Phi_{0}, s\right)&=\prod_\nu\Orb^{\rH',\underline{\eta}}_{s}(\bfun_{\X_\nu(\calo_F)}\otimes\Phi_{0,\nu},\nu y_\nu)\nonumber\\&=\prod_\nu\Orb_{s}^{\rH',\ul{\eta}}(\bfun_{\mathfrak{X}_\nu(\calo_F)}\otimes\Phi_{0,\nu},(X_\nu,Y_\nu)).
\end{align}

\subsection{Descent on the unitary side}\label{Section: descent unitary} In this section, we let $\G= \U(V_{2n})$ be the quasi-split unitary group of rank $2n$ associated to the split form $\tau_{2n}$. Then we may chose decompositions
\[
V_{2n} = V_n\oplus V_n = L\oplus L^\ast,
\] with $L$ a maximal isotypic subspace of $V_{2n}$, inducing two subgroups $\rH_{si} = \U(V_n)\times \U(V_n)$ and $\rH_{ii} = \Res_{E/F}\GL(L)$. We let $\ep_{\bullet}\in \GL(V_{2n})$ be the element of order two whose eigenspace decompositions produce the preceding decompositions
\[
V_{2n} = V_{\bullet,1}\oplus V_{\bullet,-1}.
\] If we set $\theta_{\bullet} = \Ad(\ep_{\bullet})$, then $\rH_{\bullet} = \G^{\theta_{\bullet}}$. For $k\geq0$, we also have the congruence pairs $(\G_{\bullet,k},\rH_{\bullet,k})$ discussed in \S \ref{Section: fundamental lemmas on variety}.

Setting $\Q_{{\bullet},k} = \G_k/\rH_{{\bullet}_k}$ for ${\bullet}\in \{si,ii\}$, we are interested in the stable orbital integrals
\begin{equation}\label{eqn: orbital int unitary fund}
 \SO^{\rH_{\bullet}}(\bfun_{\Q_{\bullet,k}(\calo_F)},x) = \int_{(\rH_{{\bullet},x}\bs \rH_{\bullet})(F)}\bfun_{\Q_{\bullet,k}(\calo_F)}(h^{-1}\cdot x)dh,   
\end{equation}
where $x\in\Q^{rss}_{\bullet}(F)$. On the Lie algebra level, we have the decompositions
\[
\fg = \fh_{\bullet}\oplus \fq_{\bullet},
\]
and we are interested in the stable orbital integrals
\begin{equation}\label{eqn: orbital int unitary fund lie}
\SO^{\rH_{\bullet}}(\bfun_{\fq_{\bullet,k}(\calo_F)},X) = \int_{(\rH_{{\bullet},x}\bs \rH_{\bullet})(F)}\bfun_{\fq_{\bullet,k}(\calo_F)}(h^{-1}\cdot X)dh,
\end{equation}
where $X\in \fq_{\bullet}^{rss}$ and where $\fq_{\bullet,k}(\calo_F) = \fq_{\bullet}\cap\vp^k\fg(\calo_F)$.

The goal of this section is to complete the reduction of Theorems \ref{Thm: fundamental lemma si} and \ref{Thm: fundamental lemma ii} to the corresponding linearized versions \ref{Thm: fundamental lemma s-i Lie} and \ref{thm:lie i-i}. This requires a descent argument analogous to the one discussed in \S \ref{Sec: descent linear} on the linear side. The argument is similar to the one in the previous section, and is carried out in detail  in \cite[Section 8]{Lesliedescent} when $\bullet= si$. We include a sketch in the inert-inert case, as the argument from \emph{ibid.} goes through without essential change.

  As in the linear case, we define the \textbf{very regular locus} of $\fq_\bullet$ 
\[
\fq_\bullet^\heartsuit(\calo_F)=\{X\in \fq_\bullet^{rss}-D_{1}(\calo_F): \overline{X}\notin D_{1}(k) \},
\]
where $\overline{X}$ denotes the image of $X$ under the modular reduction map. Similarly, we define the $\nu$-\textbf{very regular locus} of $\Q_{\bullet}(\calo_F)$
\[
\Q_{\bullet}^{\heartsuit,\nu}(\calo_F)=\{x\in\Q_{\bullet}^{rss}(\calo_F): \overline{x}\notin D_{\nu}(k)\}.
\]
 It is easy to see that  for any $k\geq 1$,  $\fq_{\bullet,k}(\calo_F)\subset \mathfrak{q}_\bullet^\heartsuit(\calo_F)$ and $\Q_{\bullet,k}(\calo_F) \subset \Q^{\heartsuit,-1}_\bullet(\calo_F)$, with (cf. \cite[Lemma A.5.13(4)]{KalethaPrasad})
\[
\Q_{\bullet,k}(\calo_F)=\fc_{-1}\left(\mathfrak{q}_{\bullet,k}(\calo_F)\right).
\]

We also define $\Q_{\bullet}^{\heartsuit, \nu}(F)$ to be the open subset of $\Q_{\bullet}^{rss}(F)$ consisting of elements in the same stable orbit as an element in $\Q_{\bullet}^{\heartsuit,\nu}(\calo_F)$; we similarly define $\fq_{\bullet}^{\heartsuit}$. The following lemma is a special case of the proof of \cite[Proposition 7.9]{Lesliedescent} in the split-inert case.
\begin{Lem}\label{Lem: descent unitary heart}
Suppose that $x=\fc_{\nu}(X)\in\Q_{\bullet}^{\heartsuit,\nu}(F)$. Then for any $\ka\in \mathfrak{C}(\rH_{\bullet,x},\rH_{\bullet};F)^D$
\[
 \Orb^{\rH_{\bullet},\ka}(\bfun_{\Q_{\bullet,k}(\calo_F)},x)= \Orb^{\rH_{\bullet},\ka}(\bfun_{\fq_{\bullet,k}(\calo_F)},X)
\]
for $k\geq 0$, where we note that $\rH_{\bullet,x} = \rH_{\bullet,X}$ to identify the cohomology groups.
\end{Lem}

As in the linear case, this lemma suffices to descend the orbital integrals to the infinitesimal setting for all $k\geq1$. It remains to consider orbits contained in the compliment of $\Q_{\bullet}^{\heartsuit,1}(\calo_F)\cup \Q_{\bullet}^{\heartsuit,-1}(\calo_F)$ in the case $k=0$.

\subsubsection{Remaining orbits}
Since we only consider $k=0$, we drop the subscript. When $x\notin \Q^{\heartsuit,\nu}(F)$ for either $\nu=\pm1$, we descend stable orbital integrals on $\Q_{\bullet}$ to the very regular locus of various descendant varieties as in \cite{Lesliedescent}. That is, suppose that $x\in \Q^{rss}_{\bullet}(\calo_F)$ such that 
\[
\overline{x}\in (D_1\cap D_{-1})(k),
\] and let $x=x_{as}x_{tu}$ be the topological Jordan decomposition from \cite[Section 3]{Lesliedescent}. If $V_0\oplus V_1\oplus V_{-1}$ is the eigenspace decomposition of $V_{2n}$ for $x_{as}$, 
we have
\begin{align}\label{central product}
  \G_{x_{as}}= \G_{x_{as}}'\times \U(V_1)\times \U(V_{-1}),  
\end{align}
where $\G_{x_{as}}'$ is a certain reductive group (cf. \cite[Lemma 5.13]{Lesliedescent}). Corollary 4.8 of \emph{ibid} that there exists $g\in \G_{x_{as}}(\calo_F)$ such that
\[
s_{\bullet}(g) = g\theta_{\bullet}(g)^{-1}=x_{as}.
\] Using the decomposition (\ref{central product}), we write $g= (g',g_{1},g_{-1})\in \G_{x_{as}}(\calo_F)$, so that
\[
s_{\\bullet}(g) = (s(g'),s(g_1),s(g_{-1}))= (x_{as}|_{V_0},I_{V_{1}},-I_{V_{-1}})=x_{as}.
\]We now decompose the absolutely semi-simple part $x_{as}=\ga\cdot y_{as}$
where  
\[
\ga = (I_{V_0}, I_{V_1}, -I_{V_{-1}})\in \G_{x_{as}}(\calo_F),
\]
 and 
 \[
 y_{as}=(x_{as}|_{V_0}, I_{V_1}, I_{V_{-1}})\in \G_{x_{as}}(\calo_F).
 \]
 In particular,
 \[
 \G_\ga=\U(V_0\oplus V_1)\times \U(V_{-1})\subset \G,
 \]
 and $\G_{x_{as}}\subset \G_\ga$ is a twisted Levi subgroup of $\G_\ga$.  
  It is clear that both $\ga, y_{as}\in \Q_{\bullet}(\calo_F)$. If we set $y=\ga^{-1}x,$ then $y=y_{as}x_{tu}\in \G_{\ga}(\calo_F)$ is the topological Jordan decomposition of $y$. We know that $y\in \Q_{\bullet}(F)$.
 \begin{Lem}\cite[Lemma 8.2]{Lesliedescent}\label{Lem: new decomp}
  With the notation as above, $y\in \Q_{\ga}^{rss}(F)$.
 \end{Lem}

Decomposing $V_{2n}$ with respect to the action of $\ga$, we make a slight abuse of notation and write
 \[
V_{2n}=W_1\oplus W_{-1},
 \]
 where $W_1=V_0\oplus V_1$ and $W_{-1}=V_{-1}$. The descendant $(\G_\ga,\rH_{\bullet,\ga})$ decomposes as a product
 \[
 (\G_\ga,\rH_{\bullet,\ga}) = (\G_{\ga,1},\rH_{\bullet,\ga,1})\times  (\G_{\ga,-1},\rH_{\bullet,\ga,-1}) 
 \]
 where for both $\nu=\pm1$ an easy exercise generalizing \cite[Lemma 5.13]{Lesliedescent} shows that $W_{\nu} = W_\nu\cap V_{\bullet,1}\oplus W_\nu\cap V_{\bullet,1}$ and
 \[
(\G_{\ga,\nu},\rH_{\bullet,\ga,\nu}) =\begin{cases}
    (\U(W_{\nu}), \U(W_{\nu,1})\times \U(W_{\nu,-1})) &:\bullet = si,\\
(\U(W_{\nu}), \Res_{E/F}\GL(W_{\nu}\cap L)) &:\bullet = ii.
\end{cases}
 \]
Here $W_{\nu,\nu'}=\{v\in V_{2n}: \ga v=\nu v,\:\: \ep_{\bullet} v=\nu' v\}$. We have the associated symmetric spaces
 \[
 \Q_{\bullet,\nu}=\G_{\ga,\nu}/\rH_{\bullet,\ga,\nu}.
 \]
 Proposition 4.5 of \cite{Lesliedescent} implies that each of these Hermitian spaces are split, so that we obtain a hyperspecial subgroup $\rH_{\bullet,\ga}(\calo_F)$. 
\begin{Lem}\label{Lem: into the reg locus}
Writing $x=\ga y= (x_1,-x_{-1})\in \Q_{\bullet,1}(\calo_F) \times \Q_{\bullet,-1}(\calo_F)$, we have
\[
-x_{-1}\in \Q_{\bullet,-1}^{\heartsuit,1}(\calo_F)\:\text{ and }\:x_1\in \Q_{\bullet,1}^{\heartsuit,-1}(\calo_F).
\]
\end{Lem}

The product decomposition of $(\G_\ga,\rH_\ga)$ induces a decomposition of the centralizer
\begin{equation}\label{eqn: stabilizer decomposition}
\rH_x = \rH_{\bullet,x,1}\times \rH_{\bullet,x,-1}\subset \G_\ga=\U(W_{1}) \times \U(W_{-1}).   
\end{equation}
The following is a generalization of \cite[Lemma 8.3]{Lesliedescent} for $\bullet =si$, and the argument given there holds for the inert-inert case as well.
 \begin{Lem}\label{eqn: almost there general}
 Let $x\in \Q_{\bullet}(\calo_F)$ be regular semi-simple such that $x\notin \Q^{\heartsuit,\nu}_{\bullet}(F)$ for either $\nu=\pm1$. With the notation as above, for any $\ka\in \mathfrak{C}(\rH_{\bullet,x},\rH_{\bullet};F)^D$
   \begin{equation*}
 \Orb^{\rH_{\bullet},\ka}(\bfun_{\Q_{\bullet}(\calo_F)},x)= \Orb^{\rH_{\bullet, \ga,1},\ka_1}(\bfun_{\Q_{\bullet, 1}(\calo_F)}, x_1)\cdot \Orb^{\rH_{\bullet, \ga,-1},\ka_{-1}}(\bfun_{\Q_{\bullet,-1}(\calo_F)},{- x_{-1}}),
 \end{equation*}
 where $\ka\mapsto(\ka_1,\ka_{-1})\in \prod_{\nu}\mathfrak{C}(\rH_{\bullet,x,\nu},\rH_{\bullet,\nu};F)^D$.
 \end{Lem}

 \subsection{Proof of Proposition \ref{Prop: descent fundamentals}} We first prove the descent of the fundamental lemmas in Theorems \ref{Thm: fundamental lemma stable Lie} and \ref{Thm: fundamental lemma s-i Lie}. It is easy to see that it suffices to assume $x\in \Q_{\bullet}^{rss}(\calo_F)$ matches $y\in \X^{rss}(\calo_F)$ (cf. \cite[Section 6.4]{Lesliedescent}).
 
First suppose $x\in \Q^{\heartsuit,\nu}_{\bullet}(F)$ for at least one value of $\nu=\pm1$; it is immediate that this forces $y\in \X^{\heartsuit,\nu}(F)$. Combining Lemma \ref{Lem: descent unitary heart} with that in Lemma \ref{Lem: descent linear heart}, we see that Theorems \ref{Thm: fundamental lemma si} and \ref{Thm: fundamental lemma ii} follow from Theorems  \ref{Thm: fundamental lemma s-i Lie} and \ref{thm:lie i-i} on the $\heartsuit$-loci. As noted above, this resolves the descent for $k\geq 1$. 
 
We now assume $k=0$. Suppose now that $x\notin  \Q^{\heartsuit,\nu}_{\bullet}(F)$ for either $\nu=\pm1$. With notations as above, we may write
 \[
 x= (x_1,-x_{-1})\in \Q_{\bullet,1}(\calo_F) \times \Q_{\bullet,-1}(\calo_F), \text{ and }y= (y_1,-y_{-1})\in \X_1(\calo_F) \times \X_{-1}(\calo_F).
 \]
 It is easy to see that $\nu x_\nu$ matches $\nu y_\nu$. Combining this Lemmas \ref{Lem: into the reg locus} and \ref{eqn: almost there general} on the unitary side with  Lemmas \ref{Lem: into the reg locus linear}, \ref{Lem: descent on weight factor}, and \ref{Lem: linear descent final} on the linear side, we reduce two the previous case (on the $\heartsuit$-locus) on the descendant varieties. This completes the proof that Theorems \ref{Thm: fundamental lemma s-i Lie} and \ref{thm:lie i-i} imply Theorems  \ref{Thm: fundamental lemma si} and \ref{Thm: fundamental lemma ii}, respectively.

We now consider the descent to Theorem \ref{Thm: fundamental lemma endoscopic Lie} in the endoscopic case. The argument is identical to the previous one, with the character $\ka = \varepsilon$. The only thing to check is the transfer factors. But the calculation in Lemma \ref{Lem: tranfer factor cayley} shows that if $x= \fc_{\nu}(x_1,x_2)\in \Q_{ii}(F)$, then
\[
\De_\varepsilon(x) = \eta(-2(I-x_2x_1))\De_\varepsilon(x_1,x_2).
\]
This sign is trivial on $\Q_{ii}^{\heartsuit,\nu}(F)$, so that the transfer factors agree on this locus. Otherwise, we apply Lemma \ref{eqn: almost there general} and obtain the appropriate equality on the descendant, and the claim follows. \qed

\bibliographystyle{alpha}

\bibliography{bibs}

\end{document}